\documentclass[letterpaper,10pt,oneside,headinclude,footinclude,mpinclude]{scrbook}

\usepackage[ includeall,
  marginparwidth=2in,
  marginparsep=0.25in,
  width=7.5in,
  height=9in
]{geometry}

\usepackage[osf]{mathpazo}
\usepackage[euler-digits]{eulervm}

\makeatletter
\def\evm@hbar{\evm@hslash}
\makeatother
\usepackage{mathrsfs}
\DeclareMathAlphabet{\mathcal}{OMS}{cmsy}{m}{n}

\usepackage{mathtools, amsmath, amsthm, amssymb, tikz-cd, xparse, tensor, longtable}

\usepackage{microtype}

\usepackage{hyperref, cleveref}

\usepackage{lipsum}

\usepackage[T1]{fontenc}

\usepackage[style=alphabetic,maxbibnames=5]{biblatex}

\usepackage{sidenotes}
\newcommand{\note}[1]{\sidenote{\footnotesize \color{black} #1}}
\renewcommand{\footnote}{\note}

\usepackage[autooneside=false]{scrlayer-scrpage}
\automark[section]{chapter}
\ohead{\thepage}
\ihead{\rightmark}
\ofoot[]{}
\usepackage{xcolor}
\definecolor{accent}{HTML}{914E0F}
\definecolor{link}{HTML}{105C9C}
\definecolor{RubineRed}{HTML}{ED017D}

\definecolor{myblue}{HTML}{30A4FF}
\definecolor{mypurple}{HTML}{AD2BF0}
\definecolor{mygold}{HTML}{F08813}
\definecolor{mygreen}{HTML}{01F556}
\definecolor{myred}{HTML}{FA3F2D}

\usepackage[export]{adjustbox}
\usepackage{graphicx}
\graphicspath{{fig/}}
\usepackage{subcaption}

\setkomafont{disposition}{}
\setkomafont{part}{\Huge\scshape}
\setkomafont{chapter}{\Large\scshape}
\setkomafont{section}{\color{accent}\scshape}
\setkomafont{subsection}{\color{accent}\itshape}

\newtheoremstyle{mythm}%
  {3pt}%
  {3pt}%
  {}%
  {}%
  {\scshape}%
  {:}%
  {.5em}%
  {}%
\theoremstyle{mythm}
\newtheorem{prop}{Proposition}[chapter]
\newtheorem{thm}[prop]{Theorem}
\newtheorem{lem}[prop]{Lemma}
\newtheorem{cor}[prop]{Corollary}
\newtheorem{conj}[prop]{Conjecture}
\newtheorem{defn}[prop]{Definition}
\newtheorem{remark}[prop]{Remark}

\newtheorem{ex}[prop]{Example}

\newtheorem*{introdefn}{Definition}
\newtheorem*{introconj}{Conjecture}

\crefname{thm}{theorem}{theorems}
\crefname{lem}{lemma}{lemmas}
\crefname{prop}{proposition}{propositions}

\newcommand{\defemph}[1]{{\scshape \color{accent} #1}}

\hypersetup{
  colorlinks=true,
  citecolor=link,
  linkcolor=link,
  menucolor=link,
  linktocpage=true,
  urlcolor=myred
}

\def\leftfun#1{\mathopen{}\left#1}
\def\rightfun#1{\right#1\mathclose{}}

\newcommand{\lie}[1]{\mathfrak{#1}}
\newcommand{\slg}{\operatorname{SL}_2(\mathbb C)}
\newcommand{\glg}{\operatorname{GL}_2(\mathbb C)}
\newcommand{\pslg}{\operatorname{PSL}_2(\mathbb C)}
\NewDocumentCommand{\qgrp}{O{\xi}}{\mathcal{O}_{#1}}
\newcommand{\qgrplong}[1]{\mathcal{O}_{#1}(\slg^*)}
\DeclareMathOperator{\gl}{\operatorname{GL}}
\newcommand{\charvar}[1]{\mathfrak{X}(#1)}

\NewDocumentCommand{\weyl}{O{\xi}}{\mathcal{W}_{#1}}
\newcommand{\vecfunc}{\mathcal{J}}
\newcommand{\doubfunc}{\mathcal{T}}
\newcommand{\mirrorfunc}{\mathcal{M}}
\newcommand{\cent}{\mathcal{Z}}
\newcommand{\Rmat}{\mathcal{R}}
\newcommand{\Smat}{\mathcal{S}}
\newcommand{\irrmodname}{V}
\NewDocumentCommand{\irrmod}{m}{\irrmodname\leftfun(#1\rightfun)}
\newcommand{\doubmodname}{W}
\NewDocumentCommand{\doubmod}{m O{}}{\doubmodname_{#2}\leftfun(#1\rightfun)}
\newcommand{\nr}{\ensuremath N} 
\newcommand{\units}[1]{\Gamma_{#1}}

\NewDocumentCommand{\wtmod}{m}{\irrmodname_{#1}}
\newcommand{\jkmod}{\wtmod{\nr-1}}

\newcommand{\invl}[1]{\mathrm{#1}}
\newcommand{\vecinv}{\invl{J}}
\newcommand{\doubinv}{\invl{T}}

\newcommand{\F}{\widetilde{F}}

\newcommand{\clifford}[1]{\mathfrak{C}_{#1}}
\newcommand{\cliffordbar}[1]{\overline{\mathfrak{C}}_{#1}}
\newcommand{\cliffordd}[1]{\mathscr{C}_{#1}}
\newcommand{\opspace}[1]{\mathfrak{H}_{#1}}
\newcommand{\opspacebar}[1]{\overline{\mathfrak{H}}_{#1}}
\newcommand{\opspaced}[1]{\mathscr{H}_{#1}}

\newcommand{\sidebraid}{\rotatebox[origin=c]{270}{$B$}}

\newcommand{\unit}{\mathbold{1}}
\newcommand{\catformat}[1]{\mathsf{#1}}
\newcommand{\modc}[1]{{#1}\catformat{-Mod}}
\newcommand{\catl}[1]{\mathscr{#1}}

\NewDocumentCommand{\tang}{D[]{}}{\catformat{Tang}_{#1}}
\NewDocumentCommand{\tangfr}{D[]{}}{\catformat{Tang}_{#1}^{\mathrm{fr}}}

\newcommand{\shape}{\operatorname{Sh}_0}
\newcommand{\shapeext}{\operatorname{Sh}}
\newcommand{\shaperad}{\sqrt[\nr]{\shapeext}}
\newcommand{\tangsh}{\tang[\shape]}
\newcommand{\tangshe}{\tang[\shapeext]}
\newcommand{\vect}[1]{\catformat{Vect}_{#1}}

\newcommand{\fr}{\operatorname{fr}}

\newcommand{\evup}[1]{\operatorname{ev}^\uparrow_{#1}}
\newcommand{\coevup}[1]{\operatorname{coev}^\uparrow_{#1}}
\newcommand{\evdown}[1]{\operatorname{ev}^\downarrow_{#1}}
\newcommand{\coevdown}[1]{\operatorname{coev}^\downarrow_{#1}}

\newcommand{\proj}{\catformat{Proj}}

\newcommand{\wtmodc}{\catl{W}}
\newcommand{\dwtmodc}{\catl{D}}

\newcommand{\str}{\operatorname{str}}
\newcommand{\ptr}{\operatorname{tr}}
\newcommand{\ptrr}[1]{\operatorname{tr}^r_{#1}}

\newcommand{\qtr}{\operatorname{qtr}}
\newcommand{\modtr}{{\mathop{\mathbf{t}}}}
\newcommand{\moddim}[1]{\mathbf{d}(#1)}

\newcommand{\reidone}{{\scshape RI}}
\newcommand{\reidtwo}{{\scshape RII}}
\newcommand{\reidthree}{{\scshape RIII}}

\NewDocumentCommand{\brd}{D[]{}D[]{}}{\catformat{Braid}_{#1}^{#2}}
\NewDocumentCommand{\brdsh}{D[]{}}{\brd[\shape][#1]}
\NewDocumentCommand{\brdshe}{D[]{}}{\brd[\shapeext][#1]}

\newcommand{\burau}{\mathcal{B}}
\newcommand{\lf}{\operatorname{lf}}

\NewDocumentCommand{\homol}{D[]{1} m D[]{} }{\operatorname{H}_#1(#2)^{#3}}
\NewDocumentCommand{\cohomol}{D[]{1} m D[]{} }{\operatorname{H}^#1(#2)^{#3}}

\newcommand{\HH}{\mathbb{H}}
\newcommand{\isom}{\operatorname{Isom}}
\newcommand{\crossratio}[4]{\left[ #1 : #2 : #3 : #4 \right]}

\DeclareDocumentCommand{\qfac}{m o m}{%
  \IfValueTF{#2}{%
    \left( #1 ; #2 \right)_{#3}%
  }{%
      \left( #1 \right)_{#3}%
  }
}
\DeclareDocumentCommand{\qlog}{m m}{L\leftfun(#1 \middle | #2 \rightfun)}
\DeclareDocumentCommand{\qlogn}{m m}{\Lambda\leftfun(#1 \middle | #2 \rightfun)}
\newcommand{\qlogsumname}{\mathfrak{S}}
\DeclareDocumentCommand{\qlogsum}{m}{\qlogsumname\leftfun(#1 \rightfun)}

\newcommand{\CC}{\mathbb{C}}
\newcommand{\ZZ}{\mathbb{Z}}
\newcommand{\RR}{\mathbb{R}}
\renewcommand{\hom}{\operatorname{Hom}}
\newcommand{\defeq}{\mathrel{:=}}
\newcommand{\iso}{\cong}
\newcommand{\id}{\operatorname{id}}
\newcommand{\tr}{\operatorname{tr}}
\newcommand{\op}{\mathrm{op}}
\newcommand{\cop}{\mathrm{cop}}
\DeclareMathOperator{\im}{im}
\DeclareMathOperator{\spec}{Spec}
\DeclareMathOperator{\End}{End}
\newcommand{\divalg}[1]{\operatorname{Div}\left(#1\right)}
\newcommand{\extp}{\bigwedge}

\newcommand{\U}{{\qgrp[i]}}

\input{my-front.tex}
\begin{document}
\frontmatter
\title{$\mathrm{SL}_2(\mathbb{C})$-holonomy invariants of links}

\author{Calvin McPhail-Snyder}

\degreeyear{2021}

\degreesemester{Spring}

\degree{Doctor of Philosophy}

\chair{Professor Nicolai Yu.\@ Reshetikhin}

\othermembers{Professor Vera V.\@ Serganova \\ Professor Raphael Bousso}

\field{Mathematics}

\campus{Berkeley}

\maketitle
\copyrightpage

\begin{abstract}
  The Reshetikhin-Turaev construction is a method of obtaining invariants of links (and other topological objects) via the representation theory of quantum groups.
It underlies quantum invariants such as the Jones polynomial and its many generalizations.
These invariants are algebraic in nature but are conjectured to detect important information about the geometry of links.
In this thesis we explore these connections using an enhanced version of the RT construction.

The geometry of a link complement can be described by a representation of its fundamental group into a Lie group, equivalently the holonomy of a flat Lie algebra-valued connection.
Our invariants take this data as input, so we call them holonomy invariants.
The case of trivial holonomy recovers the ordinary RT construction.

We consider holonomy representations into $\operatorname{SL}_2(\mathbb C)$, which are closely related to hyperbolic geometry.
In order to define our invariants we consider a particular coordinate system on the space of representations in terms of diagrams we call shaped tangles.
We show that these coordinates are closely related to the shape parameters of a certain ideal triangulation (the octahedral decomposition) of the link complement.

Using shaped tangles we define a family of holonomy invariants $\mathrm{J}_N$ indexed by integers $N \ge 2$, which we call the nonabelian quantum dilogarithm.
They can be interpreted as a noncommutative deformation of Kashaev's quantum dilogarithm (equivalently, the $N$th colored Jones polynomial at a $N$th root of unity) or of the ADO invariants, depending on the eigenvalues of the holonomy.
Our construction depends in an essential way on representations of quantum $\mathfrak{sl}_2$ at $q = \xi$ a primitive $2N$th root of unity.
We show that $\mathrm{J}_N$ is defined up to a power of $\xi$ and does not depend on the gauge class of the holonomy.

Afterwards we introduce a version of the quantum double construction for the holonomy invariants.
We show that the quantum double $\mathrm{T}_N$ of the nonabelian dilogarithm $\mathrm{J}_N$ admits a canonical normalization with no phase ambiguity.
Finally, we prove that in the case $N = 2$ the doubled invariant $\mathrm{T}_2$ computes the Reidemeister torsion of the link complement twisted by the holonomy representation.

\end{abstract}

\begin{dedication}
  In memory of\par
John Horton Conway\par
and of\par
Sir Vaughan Frederick Randal Jones

\end{dedication}

\begin{dedication}
  \begin{center}
\begingroup%
  \makeatletter%
  \providecommand\color[2][]{%
    \errmessage{(Inkscape) Color is used for the text in Inkscape, but the package 'color.sty' is not loaded}%
    \renewcommand\color[2][]{}%
  }%
  \providecommand\transparent[1]{%
    \errmessage{(Inkscape) Transparency is used (non-zero) for the text in Inkscape, but the package 'transparent.sty' is not loaded}%
    \renewcommand\transparent[1]{}%
  }%
  \providecommand\rotatebox[2]{#2}%
  \newcommand*\fsize{\dimexpr\f@size pt\relax}%
  \newcommand*\lineheight[1]{\fontsize{\fsize}{#1\fsize}\selectfont}%
  \ifx\svgwidth\undefined%
    \setlength{\unitlength}{292.61779404bp}%
    \ifx\svgscale\undefined%
      \relax%
    \else%
      \setlength{\unitlength}{\unitlength * \real{\svgscale}}%
    \fi%
  \else%
    \setlength{\unitlength}{\svgwidth}%
  \fi%
  \global\let\svgwidth\undefined%
  \global\let\svgscale\undefined%
  \makeatother%
  \begin{picture}(1,0.35948866)%
    \lineheight{1}%
    \setlength\tabcolsep{0pt}%
    \put(0,0){\includegraphics[width=\unitlength,page=1]{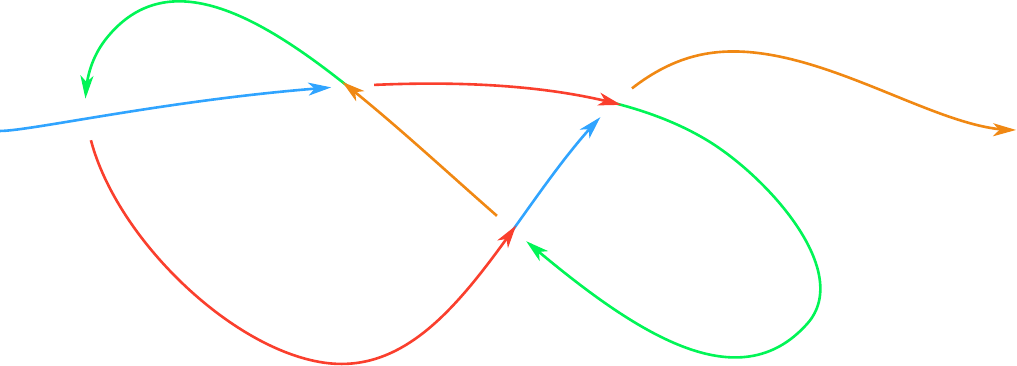}}%
  \end{picture}%
\endgroup%

  \end{center}
\end{dedication}

\tableofcontents

\begin{acknowledgements}
  I would first like to thank my advisor Nicolai Reshetikhin for his invaluable guidance and support throughout graduate school, whether it was suggesting interesting problems, pointing out that other problems weren't interesting, giving homework exercises, or leaving me alone to get things done.

Of the mathematics I learned in graduate school, 40\% was from independent reading, 20\% was from formal classes, and 40\% was from hanging out at seminars and reading groups and asking questions.
This especially goes for almost all of the geometric topology I've learned.
As such, I would like to thank the UC Berkeley mathematics (and physics) departments for providing a wonderful environment to learn math in, especially to everyone who put in the work to organize something.
In particular, I would like to thank all of my fellow graduate students who answered my dumb questions about knots and manifolds, such as Kyle Miller, Mike Klug, Julian Chaidez, Nic Brody, Alois Cerbu, Kevin Donoghue, Ethan Dlugie, Morgan Weiler, and I'm sure many others.

Thank you to the Lox Bagel Club for helping me push my limits, and especially for the bagel reps.

Buddy was a constant source of inspiration during the writing of this thesis.
In particular, he was always there to remind me that sometimes you just need to talk a walk, take a nap, or get a belly rub.

Crystal, you already know that this thesis would have been much worse without your love and support, but I'm putting it in writing anyway.

\end{acknowledgements}
\chapter{Typographical notes}
The content of this thesis is essentially identical to the version submitted to the University of California.
However, the design of the manuscript is significantly different; because the extensive use of margin notes and figures is not compatible with the UC Berkeley thesis formatting requirements, I decided to produce two versions.
I think that this version both looks better and is easier to read, but a more traditional format is available via the UC Berkeley library.
(I sincerely thank Paul Voijta for maintaining the \texttt{ucbthesis} package used to produce the official copy.)
The only differences between this version and the UCB version are the chapter numbering and the inclusion of this note.

Marginal notes%
\note{
  This is a marginal note!
}
are used for comments that are relevant but would distract from the flow of the argument.
Narrow figures are also placed in the margin, both to put them closer to the relevant text and to avoid large blocks of white space from centering them.
Gold small caps are used for definitional emphasis: we \defemph{emphasize} terms when they are defined, or in an introductory section to flag key terms to be defined later.

I was particularly influenced by two books when designing this thesis.
\emph{Concrete Mathematics} \cite{Graham1994} makes extensive use of margin notes.
I have also copied their use of the AMS Euler fonts for mathematics.
My second influence is a beautifully typeset textbook \cite{Fisher1972} on college algebra I acquired from a family member.
I borrowed a number of features from its design, including the use of {\color{accent} gold} as an accent color both for the text and in figures (such as \cref{fig:wirtinger-presentation}).

\chapter{List of symbols}
\newcommand{\symboldescr}[2]{$#1$ & #2\\}
This list is not exhaustive; for example, notation that is only used in one chapter is mostly omitted.
\begin{longtable}{cl}
  \symboldescr{\slg}{Special linear group of $2\times 2$ complex matrices with determinant $1$}
  \symboldescr{\slg^*}{Poisson dual group of $\slg$}
  \symboldescr{L}{Link in $S^3$, frequently with a representation $\rho : \pi_1(S^3 \setminus L) \to \slg$}
  \symboldescr{\nr}{Fixed integer $\ge 2$}
  \symboldescr{\xi}{$2\nr$th root of unity $\exp(\pi i /\nr)$}
  \symboldescr{\modc A}{Category of finite-dimensional $A$-modules}
  \symboldescr{\vect \Bbbk}{Category of finite-dimensional $\Bbbk$-vector spaces}
  \symboldescr{\qgrp[q]}{Quantized function algebra of $\slg^*$ (i.e.~quantum $\lie{sl}_2$)}
  \symboldescr{\cent_0}{Central subalgebra of $\qgrp$ with $\spec \cent_0 = \slg^*$}
  \symboldescr{\weyl[q]}{Extended $q^2$-Weyl algebra.}
  \symboldescr{\Rmat}{Outer automorphism of $\qgrp$ given by conjugation by the universal $R$-matrix}
  \symboldescr{\tau}{Map $V \otimes W \to W \otimes V$ given by $\tau(v \otimes w) = w \otimes v$}
  \symboldescr{\Smat}{Outer automorphism $\tau \Rmat$}
  \symboldescr{\pi(D)}{Fundamental group of a tangle diagram (complement)}
  \symboldescr{\Pi(D)}{Fundamental groupoid of a tangle diagram (complement)}
  \symboldescr{\tang}{Category of oriented framed tangle diagrams}
  \symboldescr{\chi}{Shape, a tuple $\chi = (a, b, \lambda)$ of complex numbers}
  \symboldescr{\chi}{Extended shape, a tuple $\chi = (a,b, \mu)$ with $\mu^\nr = \lambda$}
  \symboldescr{\chi}{Shape, a central character of $\weyl$}
  \symboldescr{\chi}{$\cent_0$-character induced from a shape by the embedding $\qgrp \to \weyl$}
  \symboldescr{g^{\pm}(\chi)}{Upper and lower holonomy of the shape $\chi$}
  \symboldescr{\tangsh}{Category of (oriented framed) shaped tangle diagrams}
  \symboldescr{\tangshe}{Category of (oriented framed) extended shaped tangle diagrams}
  \symboldescr{X}{Biquandle}
  \symboldescr{B}{Braiding of a biquandle}
  \symboldescr{\sidebraid}{Sideways braiding of a biquandle}
  \symboldescr{\alpha}{Twist map of a biquandle}
  \symboldescr{\tang[X]}{Category of tangles colored by the biquandle $X$}
  \symboldescr{S}{Braiding of a biquandle representation}
  \symboldescr{\catl{C}}{A pivotal category}
  \symboldescr{\modtr}{Modified trace}
  \symboldescr{\moddim{V}}{Modified dimension of $V$, i.e.\@ $\modtr_V \id_V$}
  \symboldescr{\units{n}}{Group of complex $n$th roots of unity}
  \symboldescr{\irrmod{\chi}}{Simple $\qgrp$-module with character $\chi$}
  \symboldescr{\qlog{B,A}{m}}{Unnormalized cyclic quantum dilogarithm}
  \symboldescr{\qlogn{B,A}{m}}{Normalized cyclic quantum dilogarithm}
  \symboldescr{\vecfunc}{Functor defining the nonabelian quantum dilogarithm}
  \symboldescr{\vecinv(L)}{Nonabelian quantum dilogarithm of the extended $\slg$-link $L$}
  \symboldescr{\overline{\vecfunc}}{Mirror of $\vecfunc$}
  \symboldescr{\boxtimes}{External tensor product}
  \symboldescr{\doubmod{\chi}}{External tensor product $\irrmod{\chi} \boxtimes \irrmod{\chi}^*$}
  \symboldescr{\doubfunc}{Quantum double of $\vecfunc$}
  \symboldescr{\doubinv(L)}{Invariant of $L$ associated to $\doubfunc$}
  \symboldescr{\tau(L, \rho)}{Reidemeister torsion of the complement of the link $L$ twisted by $\rho$ }
  \symboldescr{\burau}{Twisted reduced Burau representation on locally-finite homology}
\end{longtable}

\chapter{Introduction}

\section{Background and motivation}
This thesis lies in the intersection of two important ideas in low-dimensional topology.
We do not attempt to give a comprehensive description of either, but only mention the parts of the story relevant to us:
\begin{description}
  \item[hyperbolic topology]
    \cite{Thurston1980,Thurston1982}
    Topological $3$-manifolds can be equipped with uniform geometries%
    \note{
      The $2$-dimensional case may be familiar: spheres are positively curved, tori are flat, and tori with more than one handle are negatively curved.
      This is a consequence of the Gauss-Bonnet theorem.
    }
    in an essentially canonical way.
    By studying these geometries we can understand the topology.
    The most important case is \defemph{hyperbolic}, when the geometry has constant negative curvature.
  \item[quantum topology]
    \cite{Jones1985polynomial,Witten1989,Reshetikhin1990,Reshetikhin1991,Ohtsuki2002book,Turaev2016}
    Certain special classes of quantum field theory are topological and do not depend on the metric, so they lead to topological invariants.
    These invariants can be mathematically constructed via the representation theory of quantum groups (and related algebra) and are a powerful tool for distinguishing knots, links, and $3$-manifolds.
\end{description}

Invariants coming from quantum topology have the useful property that they can frequently be described in an entirely algebraic and combinatorial manner, sometimes very simply.%
\note{
  For example, the author has taught the Jones polynomial to (advanced) high school students after only a week and a half of instruction in knot theory.
}
On the other hand, this algebraic description makes them much harder to interpret in terms of fundamental topological properties: it is not really clear what it \emph{means} that the trefoil has Jones polynomial $-q^{-4} + q^{-3} + q^{-1}$.
This is in contrast to the geometric viewpoint, which seems to give more insight into the actual structure of knots and links.

\subsection{The volume conjecture}

Despite this, it seems that quantum invariants still have something to say about geometry:
\begin{introconj}[Volume Conjecture \cite{Kashaev1997,Murakami2001}]
  \label{conj:volume-conjecture}
  Let $K$ be a hyperbolic knot in $S^3$ and let $J_\nr(K)$ be the $\nr$th colored Jones polynomial of $K$ evaluated at $q = \exp(\pi i / \nr)$, normalized so that $J_\nr$ of the unknot is $1$.
  Then
  \[
    \lim_{\nr \to \infty} \frac{\log|J_\nr(K)|}{\nr} = \frac{\operatorname{Vol}(K)}{2\pi}
  \]
  where $\operatorname{Vol}(K)$ is the hyperbolic volume of $S^3 \setminus K$.
\end{introconj}
Even though $J_\nr(K)$ is defined entirely combinatorially (for example, via skein relations) it seems to detect fundamental geometric information about the $K$ such as the volume.
There are many generalizations \cite{Murakami2010,Chen2018} of this conjecture which include more topological objects (links and $3$-manifolds) and more geometric data (the Chern-Simons invariant.)
There is substantial numerical and theoretical evidence for this conjecture, but a general proof seems a long way off.

\subsection{Holonomy invariants}

The construction in this thesis is part of a program \cite{Kashaev2005,Geer2013,Blanchet2020} to understand relationships between quantum and hyperbolic topology, including conjectures like the Volume Conjecture.

The idea is to use geometric data in the form of representations $\pi_1(M) \to G$ into a Lie group $G$ to build more powerful quantum invariants.
We can equivalently describe $\rho$ as the holonomy of a flat $\lie g$-valued connection, where $\lie g$ is the Lie algebra of $G$, so we call these holonomy invariants:

\begin{introdefn}
  A \defemph{holonomy invariant} is a function on pairs $(M, \rho)$ where $M$ is a manifold and $\rho : \pi_1(M) \to G$ is a representation into a Lie group $G$.
  It should depend only on the conjugacy class of the representation: if $\rho' = g \rho g^{-1}$ for some fixed $g \in G$, then the holonomy invariant for $(L, \rho)$ and $(L, \rho')$ should agree.
\end{introdefn}
This definition should be understood informally: we might require some extra structure, or might only allow certain classes of representation.
We will usually use ``holonomy invariant'' to mean a quantum holonomy invariant, i.e.~one that is somehow analogous to the Jones polynomial or related to quantum field theory.

We focus on the case $M = S^3 \setminus L$ a link complement and $G = \slg$, which is closely related to hyperbolic geometry because the isometry group of hyperbolic $3$-space is $\pslg$:
one way to describe the hyperbolic structure on a manifold $M$ is via the \defemph{holonomy representation} $\rho : \pi_1(M) \to \pslg$.
Hyperbolic manifolds have a unique (up to conjugation) $\rho$ corresponding to their complete finite-volume hyperbolic structure.

More generally, studying representations of $3$-manifolds into $\slg$ is an important idea in geometric topology.
The collection of all such representations%
\note{
  Actually, usually the character variety is only the part corresponding to irreducible representations modulo conjugacy.
  This provides one motivation for our admissibility condition later, although we also allow reducible representations if they are diagonalizable.
}
modulo conjugacy is called the \defemph{character variety} $\charvar{L}$ of $L$.
We can think of a holonomy invariant for $L$ as a function on the space $\charvar{L}$, and we can think of (some) ordinary quantum invariants as the values of a holonomy invariant at certain points of $\charvar{L}$.

In particular, for any link $L$ the space $\charvar{L}$ contains distinguished points corresponding to representations of the form
\[
  \alpha_{t}(x) = 
  \begin{bmatrix}
    t & 0 \\
    0 & t^{-1}
  \end{bmatrix}
\]
for any meridian $x$ of $\pi_1(S^3 \setminus L)$.
Since these have abelian image, we call them abelian representations.

As discussed later in the introduction, we can think of the colored Jones polynomials as corresponding to $\alpha_{\pm 1}$ and the ADO invariants as corresponding to $\alpha_{t}$ for $t \ne \pm 1$.
Our nonabelian quantum dilogarithm is a generalization of these constructions to nonabelian representations.
From the point of view of the character variety, the nonabelian quantum dilogarithm is a function on (a cover of) $\charvar{L}$ that recovers ordinary quantum invariants at the abelian representations.

\subsection{Torsions}
The Reidemeister torsion is a classical, well-understood invariant of links.
From our perspective, the torsion $\tau(L, \rho)$ depends on both $L$ and a representation $\rho : \pi_1(S^3 \setminus L) \to G$ into a matrix group $G$, so we can think of it as a holonomy invariant.
In fact, we expect that the torsion can be understood as a \emph{quantum} holonomy invariant.

One reason this would be interesting is that there is much known about the relationship between the torsion and geometric and topological properties of knots and links, especially for nonabelian representations $\rho$.
By understanding the relationship between nonabelian torsions and quantum invariants, we can better understand the relationship between quantum invariants and other properties knots and links.

As a first step, in \cref{ch:torions} we prove that the $\slg$-torsion of a link complement can be recovered from the nonabelian quantum dilogarithm.
This generalizes the well-known construction of the Alexander polynomial as the invariant associated to quantum $\lie{sl}_2$ at $q = i$ a fourth root of unity.

\section{How to read this thesis}
Structuring mathematical documents is a very difficult task: it is hard to find the right balance between the order dictated by logical implications and the order dictated by actual understanding.
I have tried to emphasize understanding, which means there are occasional logical gaps that are filled in later.
These are usually noted, and the reader who wants to see the details first can always skip around.
I suggest the following reading order:
\begin{enumerate}
  \item Read this introduction.
  \item Read \cref{ch:prelim}: it's a bit technical, but it explains a lot of choices that otherwise seem quite odd.%
    \note{
      Here I speak from personal experience.
    }
  \item Read \cref{sec:tangle-diagrams} to fix some basic conventions on how to represent $\slg$-links.
  \item Read the overviews for each chapter, which include statements (sometimes informal) of the major results.
  \item Read the other sections of the chapters as necessary.
\end{enumerate}

Another issue is that readers are diverse in their backgrounds and interests.
I am not sure who (if anyone) will be reading this thesis, and I do not know exactly what they know, so I am forced to make some guesses.
I expect a reader of this thesis to know at least something about (or be prepared to look up) the following topics.
\begin{description}
  \item[Knot theory]
    Definitions of knots and links, Reidemeister moves, orientations and framings, as in \cite{Rolfsen2003}.
    Some familiarity with braid groups and Burau representations thereof \cite{Birman1974} would be useful, especially for \cref{ch:torions}.
    We also need some facts about fundamental groups of link complements, but these are mostly given in detail.
    To understand the \emph{motivation} for holonomy invariants and \cref{sec:octahedral-decompositions} it would be good to know something about the geometrization of $3$-manifolds \cite{Thurston1980} and/or hyperbolic knot theory \cite{Purcell2020}.%
    \note{
      If you don't know much about hyperbolic topology, maybe this thesis will be your motivation to learn some!
    }
  \item[Representation theory]
    Representations of algebras over a field, irreducible representations, multiplicity spaces, the double centralizer theorem.
    A nice, concise reference is \cite{Etingof2011}.
    It would also be good to have some familiarity with the representation theory of $\mathfrak{sl}_2$.
  \item[Reshetikhin-Turaev invariants] 
    This thesis is a generalization of the RT construction \cite{Reshetikhin1990} of link invariants from quantum groups, so it would be good to understand the basics.
    My favorite reference is \cite{Kassel1997}, which may be a bit difficult to obtain.
    Other sources are \cite{Kassel1995,Turaev2016,Bakalov2000}.
    Note that we only need the construction of link invariants: the more elaborate surgery TQFT \cite{Reshetikhin1991} is not required.
  \item[Algebraic geometry]
    Because we are studying a nontrivial \emph{commutative} algebra (the center of $\qgrp$) we need to use some very basic concepts from algebraic geometry.
    Specifically, we discuss the space $\spec R$ whose points are prime ideals of $R$.
    (Actually, we only consider the closed/geometric points, which correspond to maximal ideals.)
    The \defemph{Zariski topology} on $\spec R$ has closed sets $V(I) = \{ J \in \spec R : I \subseteq J\}$.
    For the classical case $R = \CC[x_1, \dots, x_n]$, the closed sets are zero sets of systems of polynomial equations.
\end{description}

\section{Statement of resutls}
\label{sec:our-results}

Let $\nr \ge 2$ be an integer.
We write $\xi$ for a primitive $2\nr$th root of unity, which we can take explicitly to be $\exp(\pi i /\nr)$.

\begin{marginfigure}
  \centering
  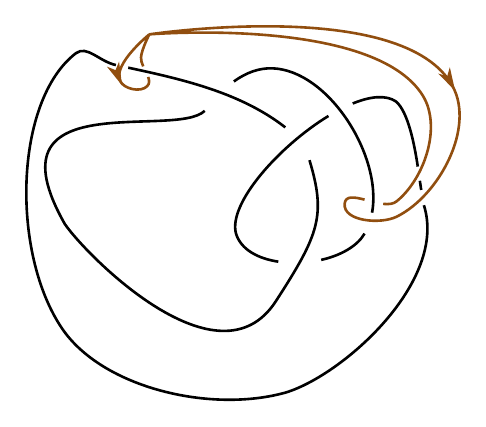
  \caption{Two meridians of the figure-eight knot.}
  \label{fig:meridian-example}
\end{marginfigure}

A \defemph{$\slg$-link} is a link $L \hookrightarrow S^3$ along with a representation $\rho : \pi_1(S^3 \setminus L) \to \slg$ of its fundamental group, which we consider up to conjugation.
We allow both the case that $\rho$ is irreducible (which is geometrically the most interesting) and that $\rho$ is completely reducible, which corresponds to standard quantum invariants (see \cref{sec:abelian-limit}).
However, we \emph{exclude} the case where $\rho$ is reducible but not decomposable, i.e.\@ when $\rho$ is conjugate to a representation with upper-triangular image
\[
  \rho(x) =
  \begin{pmatrix}
    \lambda & * \\
    0 & \lambda^{-1}
  \end{pmatrix}.
\]
If $\rho$ is not of this form, we say it is \defemph{completely reducible}.%
\note{
  This is the usual meaning of ``completely reducible'' for a group representation.
}

Recall that $\pi_1(S^3 \setminus L)$ has distinguished generators called \defemph{meridians} that wrap around a single strand, for example in \cref{fig:meridian-example}.
Each component $L_i$ of $L$ has many meridians, but up to orientation they are all conjugate.
Let $x_i$ be a meridian for the $i$th component of $L$.
We call the eigenvalues $\lambda_i^{\pm 1}$ of $\rho(x_i)$ the \defemph{eigenvalues} of $L_i$.
We say that an $\slg$-link is \defemph{enhanced} if for each component $L_i$ we choose a complex number $\mu_i$ with $\mu_i^\nr$ an eigenvalue of $L_i$.%
\note{
  This definition depends on $\nr$.
  To enhance a link for all $\nr$ at once we could instead pick a value of $\log \lambda_i$.
}
We call $\mu_i$ a \defemph{fractional eigenvalue}.
We require the eigenvalues of $L$ to satisfy
\begin{itemize}
  \item $\lambda_i \ne \pm 1$, or
  \item $\lambda_i = (-1)^{\nr + 1}$ and $\mu_i = \xi^{\nr -1} = - \xi^{-1}$.
\end{itemize}
If the representation $\rho$ is completely reducible and the fractional eigenvalues satisfy the condition above, we say that the pair $(L, \rho)$ is \defemph{admissible}.
From now on, when we say $\slg$-link we mean \emph{admissible} $\slg$-link.

Recall that an element of $\slg$ is \defemph{parabolic}%
\note{
  More geometrically, an element of $\slg$ is parabolic when it is $\pm I$ or when it has exactly one fixed point acting on $\hat \CC$ via fractional linear transformations.
  We discuss this more in \cref{ch:shaped-tangles}.
}
if it has trace $\pm 2$.
The admissibility condition for the eigenvalues has to do with parabolic elements: if the meridians are non-parabolic, then we can choose any fractional eigenvalues we want.
If they are parabolic, we have to be careful: some signs are not permitted,%
\note{
  We could change coordinates so that the allowed eigenvalues were uniform in $\nr$.
  We have chosen not to do this for notational simplicity, especially since we expect $\vecfunc_\nr$ to behave differently depending on the parity of $\nr$.
}
and we have to pick certain canonical fractional eigenvalues.

In this thesis, we construct two families of invariants (parametrized by $\nr$) of enhanced framed oriented $\slg$-links:
\begin{enumerate}
  \item The \defemph{nonabelian dilogarithm} $\vecinv_\nr(L)$, which is a complex number defined up to multiplication by a $2\nr$th root of unity.
    It does not depend on the framing of $L$ (up to a $2\nr$th root of unity).
  \item The \defemph{quantum double} of the nonabelian dilogarithm, denoted $\doubinv_{\nr}(L)$, which is defined with no scalar ambiguity.
    It also satisfies
    \[
      \doubinv_{\nr}(L) = \vecinv_{\nr}(L)\vecinv_{\nr}(\overline{L})
    \]
    up to a $2\nr$th root of unity.
    Here $\overline{L}$ is the mirror image (\cref{def:mirror-image}) of $L$.
\end{enumerate}
Both $\vecinv$ and $\doubinv$ depend only on the conjugacy class of the representation $\rho$.
They are defined on a sort of $\nr$-fold cover of the character variety, so we can think of them as holonomy invariants.

We expect that $\vecinv_{\nr}$ and $\doubinv_\nr$ are related to other geometric invariants like the torsion.
We prove one such correspondence in \cref{ch:torions}:
when $\nr = 2$,
\[
  \doubinv_2(L) = \tau(L, \rho),
\]
where $\tau(L, \rho)$ is the Reidemeister torsion of $S^3 \setminus L$ twisted by the representation $\rho$.

It would be interesting to better understand $\doubinv_{\nr}$ as a function on the character variety: is it meormorphic? Rational?
Our result gives a partial answer in the case $\nr =2$, because it implies that $\doubinv_2$ does not depend on the choice of fractional eigenvalue and can be written as
\[
  \doubinv_2(L) = 
  \frac{\Delta(L,\rho)}{2 - \lambda - \lambda^{-1}}
\]
for $\lambda \in \CC^{\times}$ and a \emph{polynomial} function $\Delta(L,\rho)$ of the matrix coefficients of $\rho$.
On the other side, it is known that \cite{Baseilhac2011} the torsion is a rational function on the character variety of a hyperbolic knot.

\subsection{Constructing the invariants}
Both invariants correspond to functors $\vecfunc$ and $\doubfunc$ from the category $\tangshe$ of \defemph{extended shaped tangles} to a pivotal category.%
\note{
  Specifically, we map them to the pivotal category $\modc \qgrp$ of modules for the quantum group and variants thereof.
}
We usually leave $\nr$ implicit.
To compute them, we use the following process:
\begin{enumerate}
  \item Pick a diagram $D$ of the extended $\slg$-link $L = (L, \rho, \{\mu_i\})$.
  \item Write the representation $\rho$ in terms of \defemph{shapes} associated to segments of $D$.%
    \note{
      This perspective assumes we already know what $\rho$ is.
      In many cases of interest, we want to \emph{compute} $\rho$ from an abstract characterization.
      For example, any hyperbolic link has a unique complete finite-volume hyperbolic structure which corresponds to a unique (up to conjugacy) $\rho$.
      In this case, we can use shaped diagrams to compute $\rho$ directly.
    }
    It is possible that this may require first conjguating (gauge-transforming) $\rho$ to avoid some singular cases, but this can always be done.
    The fractional eigenvalues $\{\mu_i\}$ of $L$ correspond to \defemph{extended} shapes.
  \item Cut open the diagram $D$ at a segment to obtain a \defemph{$(1,1)$-tangle} $T$ whose closure is $L$.
  \item Apply $\vecfunc$ to obtain an endomorphism $\vecfunc(T)$ of a simple object $V$ in a pivotal category.
  \item Take the \defemph{modified trace} of $\vecfunc(T)$ to obtain a scalar $\modtr \vecfunc(T)$, which is the invariant $\vecinv(L)$.
    Because $\vecfunc(T)$ is a morphism of a \emph{simple} object $V$, by Schur's Lemma it satisfies $\vecfunc(T) = \left\langle \vecfunc(T) \right\rangle \id_{V}$ for a scalar $\left\langle \vecfunc(T) \right\rangle$, and the modified trace can be computed as
    \[
      \modtr_{V} \vecfunc(T) = \left\langle \vecfunc(T) \right\rangle \moddim{V}
    \]
    where $\moddim V = \modtr_V \id_{V}$ is the \defemph{modified dimension} of $V$.
\end{enumerate}
We show that the construction above does not depend on
\begin{itemize}
  \item the gauge class of $\rho$,
  \item the choice of diagram $D$, or
  \item the choice of where to cut the diagram (i.e.\@ the choice of representative $(1,1)$-tangle).
\end{itemize}
The obvious thing to do would be to skip steps 3--5 and simply compute the image $\vecfunc(D)$ of $D$ in our pivotal category.
Because the quantum dimensions vanish in our cases of interest, this would result in uniformly zero invariants.
The theory of modified traces fixes this.

In order to define $\vecfunc$ we do need to make one family of arbitrary choices: to fix bases of the modules we use we need to take some $\nr$th roots.
In later terminology, we need to pick a \defemph{radical}.
As a consequence, the braidings defining $\vecfunc$ (hence the invariant $\vecinv$) are defined only up to a power of $\xi$.
We expect that these extra choices can be systematized; this should be related to a \defemph{flattening} of the ideal triangulation of the knot complement, as in \cite{Zickert2009}.
As discussed in \cref{ch:doubles}, the functor $\doubfunc$ can be defined with no phase ambiguity; we can think of this as cancelling the anomaly of $\vecfunc$.

Our definition of $\vecfunc$ and $\doubfunc$ is a special case of the construction of \cite{Blanchet2020}, which works for any \defemph{representation} of the \defemph{extended shape biquandle} in a pivotal category assigning strands to simple objects whose \defemph{modified dimensions} are \defemph{gauge-invariant}.
$\vecfunc$ and $\doubfunc$ are examples of such representations; another is the representation discussed in \cref{sec:adjoint-rep}, which was used to define the invariant of \cite{Kashaev2005,Chen2019,McPhailSnyderUnpub2}.

\subsection{Taking the abelian limit}
\label{sec:abelian-limit}
We mostly focus on $\slg$-links whose holonomy representation is irreducible;
this case is novel and is geometrically the most interesting.
However, $\vecinv_{\nr}(L, \rho)$ still makes sense%
\note{
  There is actually a technical issue with defining it for reducible representations $\rho$, as discussed in \cref{sec:pinched-crossing-braiding-overview,sec:pinched-crossings-braiding}.
}
when $\rho$ is reducible.

When $\rho$ sends every meridian to $(-1)^{\nr +1}$, up to normalization the invariant $\vecinv_{\nr}(L, \rho)$ is exactly Kashaev's quantum dilogarithm invariant \cite{Kashaev1995}, equivalently the $\nr$th colored Jones polynomial \cite{Murakami2001} at a $\nr$th root of unity.
Up to a scalar, which in this context only affects the dependence of $\vecinv_{\nr}(L)$ on the framing of $L$, the braiding defining $\vecinv$ and the braiding defining the quantum dilogarithm are both uniquely defined by the property that the intertwine the outer $S$-matrix  $\Smat$.
Because  $\vecinv_{\nr}(L)$ is framing-independent (up to a power of $\xi$) and assigns the value $1$ to the unknot it must agree with the Kashaev invariant.

Similarly, when the images of the meridians are simultaneously conjugate to matrices of the form
\[
  \begin{bmatrix}
    \lambda_i & * \\
    0 & \lambda_i^{-1}
  \end{bmatrix}
\]
for $\lambda_i \ne \pm 1$ we recover the semi-cyclic invariants of \cite{Geer2013}; when the upper-right entries are all $0$, i.e.\@ when the representation is decomposable, we obtain the $\nr$th ADO invariant \cite{Akutsu1992}.
We discuss this further in \cref{thm:abelian-limit}.

\section{Summary of the chapters}
\begin{description}
  \item[\Cref{ch:prelim}]
    We introduce quantum $\mathfrak{sl}_2$, an algebra $\qgrp[q] = \qgrplong q$.%
    \note{
      We prefer the notation $\qgrp[q]$ to $\mathcal{U}_q$ because we use slightly nonstandard generators.
      This normalization makes $\qgrp[q]$ a quantization of the algebra of functions on the Lie group $\slg^*$, instead of a quantization of the universal enveloping algebra $\mathcal{U}_q(\mathfrak{sl}_2)$.
      See \cref{sec:quantum-group} for more information.
    }
    We discuss its representation theory at $q = \xi$ a root of unity, how this relates to holonomy invariants, and how to get braidings from it.
    We also introduce a presentation of $\qgrp[q]$ in terms of a Weyl algebra $\weyl[q]$.
  \item[\Cref{ch:shaped-tangles}]
    We fix conventions on tangle diagrams and describe how to represent $\slg$-links in terms of shaped tangle diagrams.
    We show that these are closely related to the octahedral decomposition of the link complement.
  \item[\Cref{ch:functors}]
    Following \citeauthor{Blanchet2020} \cite{Blanchet2020} we give the general theory of holonomy invariants in terms of representations of biquandles.
  \item[\Cref{ch:algebras}]
    We construct a biquandle representation from $\qgrp$ that defines the functor $\vecfunc$.
    Along with \cref{ch:prelim,ch:shaped-tangles,ch:functors} this completes the construction of the invariant $\vecinv_\nr(L)$.
    In the process, we show that the braiding for $\qgrp$ is closely related to the octahedral decompositions of \cref{ch:shaped-tangles}.
  \item[\Cref{ch:doubles}]
    We give a graded version of the quantum double construction.
    By applying it to $\vecfunc$, we define $\doubfunc$ and the associated holonomy invariant $\doubinv_{\nr}(L)$ and prove that is can be defined without scalar ambiguity.
  \item[\Cref{ch:torions}]
    We define the twisted Reidemeister torsion in terms of twisted Burau representations of braid group(oids).
    By using a Schur-Weyl duality between the Burau representation and $\doubinv_2$ we prove that it computes the torsion.
  \item[\Cref{ch:modified-traces}]
    We explain how to construct the modified traces used in the rest of the thesis.
  \item[\Cref{ch:quantum-dilogarithm}]
    We discuss and prove some properties of the cyclic quantum dilogarithm used in \cref{ch:algebras} to compute the braidings.
\end{description}

\mainmatter
\addtocounter{chapter}{-1}
\chapter{Preliminaries on the quantum group}
\label{ch:prelim}
The key mathematical ingredient in this thesis is the quantum group associated to $\lie{sl}_2$ and a particular presentation of it in terms of a $q$-Weyl algebra.
We therefore begin with a brief overview of this algebra, pointing out some features and their consequences.
By doing so, we motivate most of the technical constructions of the first three chapters.

\section{The quantized function algebra}
\label{sec:quantum-group}
\begin{defn}
  The \defemph{quantized function algebra} $\qgrp[q] = \qgrp[q](\slg^*)$ of $\slg^*$ is the algebra over $\mathbb{C}[q, q^{-1}]$ with generators $K^{\pm 1}, E, F$ and relations
  \[
    KK^{-1} = 1, \quad KE = q^2 EK, \quad KF = q^{-2} FK, \quad EF - FE = (q - q^{-1})(K - K^{-1}).
  \]
  It is a Hopf algebra with coproduct
  \[
    \Delta(K) = K \otimes K, \quad \Delta(E) = E \otimes K + 1 \otimes E, \quad \Delta(F) = F \otimes 1 + K^{-1} \otimes F,
  \]
  counit
  \[
    \epsilon(K) = 1, \quad \epsilon(E) = \epsilon(F) = 0, \]
  and antipode
  \[
    S(E) = - EK^{-1}, \quad S(F) = - KF, \quad S(K) = K^{-1}.
  \]
\end{defn}

$\qgrp[q]$ is isomorphic (over a slightly different ring of scalars) to the usual quantum group $\mathcal{U}_{q}(\lie{sl}_2)$, but the above normalization is more convenient for our purposes.\note{To recover the conventions of \cite{Kassel1995,Blanchet2020}, replace $E$ by $(q - q^{-1})E$ and simiarly for $F$.}
In the limit $q \to 1$, $\mathcal{U}_q$ instead recovers the universal enveloping algebra of $\lie{sl}_2$.
The group $\slg^*$ appearing above is a matrix group closely related to $\slg$.

\begin{defn}
  The \defemph{Poisson dual group}%
  \note{$\slg$ is a \defemph{Poisson-Lie group:} a Lie group with a Poisson bracket on its algebra of functions.
    This gives the Lie algebra $\mathfrak{sl}_2$ the structure of a \defemph{Poisson-Lie bialgebra}; taking the dual of this structure gives a different Lie algebra $\mathfrak{sl}_2^*$, whose associated Lie group is $\slg^*$.
  See \cite[Chapter 2]{Etingof2002} for more details.}
  of $\slg$ is the group
  \[
    \slg^* =
    \left\{
      \left(
        \begin{bmatrix}
          \kappa & 0 \\
          \phi & 1
        \end{bmatrix},
        \begin{bmatrix}
          1 & \epsilon \\
          0 & \kappa
        \end{bmatrix}
      \right)
      \middle | \kappa \ne 0
    \right\}
    \subseteq \glg \times \glg.
  \]
\end{defn}

When $q = 1$, $\qgrp[1]$ is the commutative algebra $\mathbb C[\slg^*]$ of functions on the group $\slg^*$.
For example, the coproduct on $K$ and $E$ corresponds to the group multiplication
\[
  \begin{bmatrix}
    1 & \Delta(E) \\
    0 & \Delta(K)
  \end{bmatrix}
  =
  \begin{bmatrix}
    1 & E \otimes K + 1 \otimes E \\
    0 & K \otimes K
  \end{bmatrix}
  =
  \begin{bmatrix}
    1 & E \\
    0 & K
  \end{bmatrix}
  \otimes
  \begin{bmatrix}
    1 & E \\
    0 & K
  \end{bmatrix}
\]
and similarly for $K$ and $KF$.\note{For more on how Hopf algebras relate to algebraic groups like $\slg^*$, see \cite[\S III.4]{Kassel1995}.}
Since $\qgrp[q]$ is a noncommutative analogue of a commutative object $\CC[\slg^*]$ recovered at $q = 1$, we call it a quantization.\note{For a more physical notation, set $q = e^\hbar$ where $\hbar$ is Plank's constant.}

\begin{prop}
  When $q$ is not a root of unity, the center of $\qgrp[q]$ is a polynomial algebra generated by the Casimir element
  \[
    \Omega = EF + q^{-1}K + q K^{-1} = FE + qK + q^{-1}K^{-1}.
  \]
\end{prop}
\begin{proof}
  \autocite[Theorem VI.4.8]{Kassel1995}.
\end{proof}

For generic $q$, the representation theory of $\qgrp[q]$ is essentially the same as the classical theory of representations of $\lie{sl}_2$.\footnote{Formally, the module categories of $\qgrp[q]$ and $\mathcal{U}(\lie{sl}_2)$ have isomorphic Grothendieck rings.}
This is a consequence of the above proposition: when $q$ is not a root of unity, the center of $\qgrp[q]$ is a polynomial algebra, so simple $\qgrp[q]$-modules are essentially classified by a single number, just like the highest-weight classification of representations of $\lie{sl}_2$.

However, when $q = \xi$ is a $2\nr$th root of unity the center of $\qgrp$ becomes much larger.
In particular, it contains the algebra $\CC[\slg^*]$.
As a consequence, the representation theory of $\qgrp$ becomes much more complex; instead of a few discrete parameters, simple $\qgrp$-modules will be parameterized by points of $\slg^*$ (plus some additional data).

\begin{prop}
  $\cent_0 \defeq \CC[K^{\pm \nr}, E^\nr, F^\nr]$ is a central subalgebra of $\qgrp$.
  The center $\cent \defeq Z(\qgrp)$ of $\qgrp$ is generated by $\cent_0$ and the Casimir $\Omega$, subject to the relation
  \begin{equation}
    \label{eq:casimir-relation}
    \operatorname{Cb}_\nr(\Omega) = E^\nr F^\nr - (K^\nr + K^{-\nr}).
  \end{equation}
  Here $\operatorname{Cb}_\nr$ is the $\nr$th renormalized Chebyshev polynomial, determined by
  \[
    \operatorname{Cb}_\nr(t + t^{-1}) = t^{\nr} + t^{-\nr}.
  \]
\end{prop}

\begin{defn}
  A \defemph{$\cent_0$-character}  is an algebra homomorphism $\chi : \cent_0 \to \CC$, equivalently a (closed) point of $\spec \cent_0$.
  $\chi$ is determined by its values on $K^\nr$, $E^\nr$, and $F^\nr$, which we typically denote by
  \begin{equation}
    \label{eq:char-notation}
    \kappa = \chi(K^\nr), \quad \epsilon = \chi(E^\nr), \quad \phi = \chi(K^\nr F^\nr)
  \end{equation}
  The set of $\cent_0$-characters is a group, with mutiplication given by
  \[
    (\chi_1\chi_2)(x) =  (\chi_1 \otimes \chi_2)(\Delta(x)).
  \]
\end{defn}

\begin{prop}
  The map sending a $\cent_0$-character to the group element
  \[
    \left(
      \begin{bmatrix}
        \chi(K^\nr) & 0 \\
        \chi(K^\nr F^\nr) & 1
      \end{bmatrix},
      \begin{bmatrix}
        1 & \chi(E^\nr) \\
        0 & \chi(K^\nr)
      \end{bmatrix}
    \right) \in \slg^*
  \]
  is an isomorphism of algebraic groups.
\end{prop}
From now on, we identify a $\cent_0$-character $\chi$ and the corresponding point of $\slg^*$.
We can think of $\qgrp$ as being a bundle of algebras over $\slg^*$: the fiber over a point $\chi \in \slg^*$ is the algebra  $\qgrp / \ker \chi$, where $\ker \chi$ is the ideal generated by the kernel of $\chi$.

In turn, this gives a $\slg^*$-grading on $\qgrp$-modules.
Suppose for simplicity that $V$ is a simple $\qgrp$-module of finite dimension over $\CC$.
Then by Schur's lemma the central subalgebra  $\cent_0$ must act by scalars on $V$, so there is a homomorphism $\chi : \cent_0 \to \CC$ given by
 \[
   z \cdot v = \chi(z) v, \ \ z \in \cent_0.
\]

We think of $V$ as living in degree $\chi \in \slg^*$.
The $\slg^*$-grading on modules is the key ingredient underlying the construction of holonomy invariants: we can assign a strand of a link with holonomy $\chi$ a module with character $\chi$.
The braiding discussed in the next section is similarly compatible.

We are interested in $\slg$-representations of link complements, not $\slg^*$-representations.
As shown by \cite{Kashaev2005,Blanchet2020}, this problem can be fixed by thinking of $\slg^*$ as a factorization of $\slg$ into lower- and upper-triangular parts.
We give a version of this construction in \cref{ch:shaped-tangles,ch:functors}.

Finally, we discuss the role of the Casimir element $\Omega$.
Because $\Omega$ is central, a simple, finite-dimensional module $V$ will be described not just by the $\cent_0$-character $\chi$, but also by the scalar $\omega$ by which $\Omega$ acts.
From the relation (\ref{eq:casimir-relation}), we see that
\[
  \operatorname{Cb}_r(\omega) = \left( \epsilon \phi/\kappa - \kappa - 1/\kappa \right)
  =
  - \tr
  \begin{bmatrix}
    \kappa & -\epsilon \\
    \phi & (1- \epsilon\phi)/\kappa
  \end{bmatrix}
\]
where $\kappa = \chi(K^r)$, etc., as in \eqref{eq:char-notation}, and the matrix above is (conjugate to) the holonomy around a meridian colored by $\chi$.

In particular, the action of $\Omega$ is related to an $\nr$th root of the eigenvalues of the holonomy, so in general our invariants will depend on a link $L$, a $\slg$-representation of its complement, \emph{and} a choice of roots of the eigenvalues of the representation.
In the boundary-parabolic case generalizing the colored Jones polynomial the eigenvalues are $\pm 1$ and there is a canonical choice of root.

\section{Braidings for the quantum group}
\label{sec:holonomy-braidings}
It is well-known that the quantum group $\qgrp[q]$ is quasitriangular, but strictly speaking this is false \cite{Reshetikhin1995}.
What is true is that an $\hbar$-adic version $\qgrp[\hbar]$ of $\qgrp[q]$ defined over formal power series in $\hbar$ (with $q = e^\hbar$) is quasitriangular: there is a \defemph{universal $R$-matrix} $\mathbf R$ in (an appropriate completion of) $\qgrp[\hbar] \otimes \qgrp[\hbar]$ that intertwines the coproduct and opposite coproduct
\[
  \mathbf R \Delta = \Delta^{\op} \mathbf R
\]
and satisfies the Yang-Baxter equation\note{The Yang-Baxter equation becomes the \reidthree{} relation on braids if we write it in terms of $\tau \mathbf R$ instead of $\mathbf R$.}
\[
  \mathbf R_{12} \mathbf R_{13} \mathbf R_{23}
  =
  \mathbf R_{23} \mathbf R_{13} \mathbf R_{12}
\]
where by $\mathbf R_{ij}$ we mean $\mathbf R$ embedded in the $i$th and $j$th tensor factors.
As a consequence, for any $\qgrp[\hbar]$-modules $V,W$ the map $\tau \mathbf R : V \otimes W \to W \otimes V$ gives a braiding, where $\tau(v \otimes w) = w \otimes v$ is the flip map and by $\mathbf R$ we mean the map induced by multiplication by $\mathbf R$.

Because $\mathbf R$ does not lie in $\qgrp[q]^{\otimes 2}$, the algebra $\qgrp[q]$ is not quasitriangular.
However, when $V$ and $W$ are finite-dimensional  $\qgrp[q]$-modules the action of $\mathbf R$ still makes sense, because $E$ and $F$ act nilpotenly on $V$ and $W$.
This leads to the usual construction of a braiding on $\modc{\qgrp[q]}$ underlying the colored Jones polynomials.

When $q = \xi$ is a root of unity things are not so simple.
Because $E$ and $F$ no longer need to act nilpotently%
\note{It is possible to consider \defemph{semi-cyclic} modules where at least one acts nilpotently, which leads to the ADO invariants \cite{Akutsu1992} and generalizations thereof \cite{Geer2013}.
However, geometrically interesting representations (as discussed in \cref{ch:shaped-tangles}) are essentially never semi-cyclic.}
on finite-dimensional modules the universal $R$-matrix fails to converge.
It is still possible to obtain a braiding structure on $\qgrp$, but we need to take a more roundabout approach.

Following \citeauthor{Kashaev2004} \cite{Kashaev2004}, we instead focus on the automorphism $\Rmat$ of  $\qgrp[\hbar]^{\otimes 2}$ given by conjugation by the universal $R$-matrix:
\[
  \Rmat(x) \defeq \mathbf R x \mathbf R^{-1}.
\]
In the $\hbar$-adic case this is an inner automorphism, so it is simpler to consider $\mathbf R$.
However, unlike $\mathbf R$, the automorphism $\Rmat$ makes sense even for $q = \xi$ a root of unity.
We can therefore use the autormorphism $\Rmat$ to define a braiding on $\qgrp$.

\begin{thm}
  \label{thm:outer-R-mat-def}
  Set $W = 1 - K^{-\nr} E^{\nr} \otimes F^{\nr} K^{\nr} \in \cent_0^{\otimes 2} \subset \qgrp^{\otimes 2}$.
  There is an algebra homomorphism
  \[
    \Rmat : \qgrp^{\otimes 2} \to \qgrp^{\otimes 2}[W^{-1}]
  \]
  defined uniquely by
  \begin{align*}
    \Rmat(1 \otimes K)
    &=
    (1 \otimes K) (1 - \xi^{-1} K^{-1} E \otimes F K)
    \\
    \Rmat(E \otimes 1)
    &=
    E \otimes K
    \\
    \Rmat(1 \otimes F)
    &=
    K^{-1} \otimes F
    \\
    \intertext{and}
    \Rmat(\Delta(u))
    &=
    \Delta^{\op}(u), \quad u \in \qgrp.
  \end{align*}
  It satisfies the Yang-Baxter equations
  \begin{align*}
    (\Delta \otimes 1) \Rmat(u \otimes v) &= \Rmat_{13} \Rmat_{23} (\Delta(u) \otimes v) \\
    (1 \otimes \Delta) \Rmat(u \otimes v) &= \Rmat_{13} \Rmat_{12} (u \otimes \Delta(v))
  \end{align*}
  and the identity relations
  \begin{align*}
    (\epsilon \otimes 1) \Rmat(u \otimes v) &= \epsilon(u) v \\
    (1 \otimes \epsilon) \Rmat(u \otimes v) &= \epsilon(v) u 
  \end{align*}
  where $\epsilon$ is the counit of $\qgrp$.
\end{thm}
\begin{proof}
  See \cite{Kashaev2004}.
\end{proof}

\begin{defn}
  We call the automorphism $\Rmat$ the \defemph{outer $R$-matrix} of $\qgrp$.
  Writing $\tau$ for the flip $\tau(x \otimes y) = y \otimes x$, we also use the closely related map
  \[
    \Smat \defeq \tau \Rmat
  \]
  which we call the \defemph{outer $S$-matrix}.
\end{defn}
It is not too hard to see (and we will show later for the Weyl presentation) that $\Rmat$ descends to a homomorphism of $\cent_0^{\otimes 2}$.
In particular, it acts on pairs of $\cent_0$-characters: given two such characters $\chi_1, \chi_2$, there are unique $\chi_1', \chi_2'$ such that
\[
  (\chi_1' \otimes \chi_2')\Rmat = \chi_1 \otimes \chi_2.
\]
equivalently
\[
  (\chi_2' \otimes \chi_1')\Smat = \chi_1 \otimes \chi_2.
\]
The map $(\chi_1 , \chi_2) \to (\chi_2', \chi_1')$ is an example of an algebraic structure callled a biquandle, wich are discussed more generally in \cref{sec:biquandles}.

Because $\Rmat$ is defined only when $W$ is invertible it is a partially defined biquandle, which is annoying but can be resolved by the theory of generic biquandles \cite[\S 5]{Blanchet2020}.
We discuss these issues in detail for the closely related shape biquandle of \cref{sec:shaped-tangle-diagrams}, which is obtained from the presentation of $\qgrp$ in terms of a Weyl algebra given in the next section.

The automorphism $\Rmat$ describes the braiding on the algebra $\qgrp$, but to construct invariants of links the critical ingredient is a braiding on the \emph{modules}.
Suppose we have some $\qgrp$-modules $V_1, V_2, V_1', V_2'$ with characters $\chi_i, \chi_i'$.
An $R$-matrix is a linear map $R = R_{\chi_1, \chi_2}$ that intertwines $\Rmat$ in the the sense that
\begin{equation}
  \label{eq:inner-R-matrix-cond}
  R(u \cdot x) = \Rmat(u) \cdot R(x) \text{ for every } x \in V_1 \otimes V_2, u \in \qgrp^{\otimes 2}.
\end{equation}
For $\qgrp[q]$ we could obtain \emph{the} matrix $R$ as the action of the universal $R$-matrix, but at $q = \xi$ there is no universal $R$-matrix.
Instead, we say that \emph{an} $R$-matrix for the modules $V_i$ is a linear map satisfying \eqref{eq:inner-R-matrix-cond}.
To construct holonomy invariants of links we need to compute $R$-matrices for all $\chi_1, \chi_2$.

\begin{defn}
  Suppose the set $X$ parametrizes isomorphism classes of $\qgrp$-modules%
  \note{
    $X$ will be an $r$-fold cover of a Zariski open dense subset of $\slg$.
    For now, it is fine to to replace $X$ with $\slg$ for the purpose of understanding this definition.
  }
  and there is a projection $\pi$ from $X$ to the set of $\cent_0$-characters.
  Let $\chi_1, \chi_2 \in X$ and let $(\chi_2', \chi_1')$ be related to $(\chi_1, \chi_2)$ by the braiding (i.e., by the biquandle structure on $X$) as above.
  A \defemph{holonomy $R$-matrix} is a family of $\qgrp$-modules $\{V_\chi\}_{\chi \in X}$ and a family $\{R_{\chi_1, \chi_2}\}_{\chi_1, \chi_2 \in X}$ of linear maps
  \[
    R_{\chi_1, \chi_2} : V_{\chi_1} \otimes V_{\chi_2} \to V_{\chi_1'} \otimes V_{\chi_2'}
  \]
  satisfying \eqref{eq:inner-R-matrix-cond} when the $\chi_i$ are interpreted as $\cent_0$-characters and a colored version of the Yang-Baxter equation (equivalently, the braid relation or the \reidthree{} move).
\end{defn}

This definition is somewhat imprecise: see \cref{def:yang-baxter-model} for details.
Similarly, the term ``holonomy $R$-matrix'' is somewhat misleading because it really refers to a \emph{family} of objects; for this reason we usually call this structure a \defemph{model} of the biquandle $X$.

When the modules $V_{\chi}$ are simple, it is not hard to show (\cref{thm:R-matrix-exists}) that solutions to \eqref{eq:inner-R-matrix-cond} are determined up to a scalar.
As a consequence there is a canonical holonomy $R$-matrix up to an overall scalar, which we can think of as a projective representation.
To obtain useful link invariance we need to lift this to a genuine representation by finding an appropriate normalization of the $R$-matrices.

Unfortunately, this is a significant technical problem, in particular because we need to normalize a \emph{family} of $R$-matrices $R_{\chi_1, \chi_2}$, where $\chi_1, \chi_2$ range over $\slg^*$.\note{Strictly speaking the problem is worse, because as mentioned above simple $\qgrp$-modules rely on more data than just the $\cent_0$-characters $\chi_i$.}
We resolve it in two different ways, both of which involve presenting $\qgrp$ in terms of a Weyl algebra (discussed in the next section).

Our first method is to use the Weyl presentation to explicitly compute the matrix coefficients of $R$ and compute its determinant; by using a continuity argument, we can show that an appropriate normalization satisfies the colored Yang-Baxter relation.
\Cref{ch:algebras} is devoted to this computation.
This method only reduces the scalar ambiguity to a power of $\xi$ because the computation of the matrix coefficients requires some extra choices.

In connection with torsions, a different approach is possible.
We show in \cref{ch:doubles} that the tensor product $R_{\chi_1, \chi_2} \boxtimes \overline{R}_{\chi_1, \chi_2}$ of the $R$-matrix and an appropriate mirror image preserves a (family of) vectors, so we can normalize the doubled $R$-matrix by preserving the vectors.
This allows us to define the functor $\doubfunc$ underlying the doubled invariant without any scalar ambiguity and leads to the Schur-Weyl duality of \cref{ch:torions}.

\section{The quantum Weyl algebra}
\label{sec:weyl-alg}
It is convenient to use a different presentation of $\qgrp$ in terms of a $q$-Weyl algebra.
The original motivation was to obtain $q$-difference equations for the coefficients of the $R$-matix.
In the process we express $\cent_0$-characters in terms of their values on the center of the Weyl algebra (at $q = \xi$ a root of unity).
These coordinates turn out to have a direct geometric interpretation in terms of ideal triangulations of the link complement, as discussed in \cref{sec:octahedral-decompositions}.

\begin{defn}
  The \defemph{cyclic Weyl algebra}\note{To recover the usual Weyl relation $[X,Y] = 2\hbar$, set $x = e^{\hbar X}$, $y = e^{\hbar Y}$, and $q = e^\hbar$.} is the algebra $\weyl[q]^0$ given by
  \[
    \weyl[q]^0 = \CC[q, q^{-1}] \langle x, y | xy = q^2 yx \rangle.
  \]
  We assume that $x$ and $y$ are invertible.
  We usually use the \defemph{extended Weyl algebra} $\weyl[q] = \weyl[q]^0[z, z^{-1}]$, which has an additional central generator $z$.
\end{defn}

\begin{prop}
  \label{thm:weyl-embedding}
  There is an algebra homomorphism $\phi : \qgrp[q] \to \weyl[q]$ given by
  \begin{align*}
    K &\mapsto x 
      &
    E &\mapsto qy (z - x)
      &
    F &\mapsto y^{-1}(1 - z^{-1} x^{-1} )
  \end{align*}
  which acts on the Casimir by
  \[
    \Omega \mapsto qz + (qz)^{-1}.
  \]
\end{prop}
\begin{prop}
  At a root of unity $q = \xi$ the center of $\weyl$ is generated by $x^\nr$, $y^\nr$, and $z$.
  The automorphism $\phi$ takes the center of $\qgrp$ to the center of $\weyl$.
  Explicitly,
  \begin{align*}
    \phi(K^\nr)
    &=
    x^\nr
    \\
    \phi(E^r)
    &=
    y^\nr(x^\nr - z^\nr)
    \\
    \phi(F^r)
    &=
    y^{-\nr}(1 - z^{-\nr} x^{-\nr})
  \end{align*} 
\end{prop}
\begin{proof}
  $K^\nr$ is obvious and $F^\nr$ follows from the same reasoning as $E^\nr$.
  For $E^\nr$, notice that
  \begin{align*}
    \phi(E^\nr)
    &=
    \left( \xi y (z - x) \right)^{\nr}
    \\
    &=
    (\xi y)^2 (z - \xi^2 x) (z - x) \left( y (z - x ) \right)^{\nr-2}
    \\
    &\dots
    =
    (\xi y)^r \prod_{k=0}^{\nr-1} (z - \xi^{2k} x)
  \end{align*}
  All the terms in the product except $z^\nr$ and $x^\nr$ vanish.
  The coefficient of $z^\nr$ is clearly $1$, while the coefficient of $x^\nr$ is $(-1)^{\nr}$ times $\xi$ raised to the power
  \[
    \sum_{k=0}^{\nr-1} (2 k) =  \nr(\nr-1)
  \]
  so that
  \[
    \phi(E^\nr)
    = -y^\nr( z^\nr + (-1)^{\nr} \xi^{\nr(\nr-1)} x^\nr)
    = -y^\nr( z^\nr + (-1)^{\nr} (-1)^{\nr-1} x^\nr)
    = y^\nr( x^\nr - z^\nr). \qedhere
  \]
\end{proof}

\begin{defn}
  A \defemph{shape} is an algebra homomorphism
  \[
    \chi : \CC[x^{\pm \nr},y^{-\pm \nr}, z^{\pm \nr}] \to \CC
  \]
  By pulling back along $\phi$ we obtain a $\cent_0$-character $\phi^* \chi$.
  We usually abuse notation and denote both $\chi$ and $\phi^* \chi$ by $\chi$.
  The map $\chi$ is determined by its values on the generators $x^\nr$, $y^\nr$, and $z^\nr$, which we usually denote by
  \begin{equation}
    \label{eq:weyl-char-notation}
    \chi(x^\nr) = a, \quad
    \chi(y^\nr) = b, \quad
    \chi(z^\nr) = \lambda.
  \end{equation}

  An \defemph{extended shape} is a homomorphism
  \[
    \chi : Z(\weyl) = \CC[x^{\pm \nr},y^{-\pm \nr}, z^{\pm 1}] \to \CC
  \]
  from the center of $\weyl$ to $\CC$.
  Equivalently, an extended shape is a lift of a shape given by a choice of $\nr$th root of $\lambda = \chi(z)$.
  We usually write  $\chi(z) = \mu$ for this root.
  Just as before, pulling back along the embedding $\phi$ determines a $\cent$-character $\phi^* \chi$ that we also denote by $\chi$.
\end{defn}
The term ``shape'' refers to the geometric interpretation of the parameters $a, b,$ and $\lambda$ in terms of the shape parameters of an ideal triangulation of the link complement, as given in \cref{sec:octahedral-decompositions}.

The extension of a shape corresponds to the extra choice of a root of the eigenvalue required to define the holonomy invariants.
The motivation is that extended shapes parametrize isomorphism classes of simple $\weyl$-modules.
Pulling back along $\phi$ we obtain the family of simple $\qgrp$-modules we want to use in the construction of our holonomy invariants.

When pulling back along $\phi$, we associate the character $\chi$ in \eqref{eq:weyl-char-notation} to the group element
\begin{equation}
  \label{eq:weyl-char-image}
  \left(
    \begin{bmatrix}
      a & 0 \\
      (a - 1/\lambda)/b
    \end{bmatrix}
    ,
    \begin{bmatrix}
      1 & (a - \lambda)b \\
      0 & a
    \end{bmatrix}
  \right)
  \in \slg^*
\end{equation}
where $\lambda = \mu^\nr = \chi(z^\nr)$.
Not all of $\slg^*$ can be written in terms of shapes, that is in the form \eqref{eq:weyl-char-image}, but we can get around this in two different ways.
One is to use the fact that what we really care about are conjugacy classes (gauge classes) of representations of link complements.
By using the methods of \cite{Blanchet2020} we show (\cref{thm:links-presentable}) that every admissible representation is conjugate to one that can be expressed in terms of shapes.
We can then define a holonomy invariant for arbitrary representations by first gauge transforming.

A more geometric method is to observe that, in practice, interesting representations come from solving the gluing equations of \cref{sec:shaped-tangle-diagrams}.
For example, any nondegenerate solution to the gluing equations with $\lambda = \pm 1$ gives the complete finite-volume hyperbolic structure of the link complement.%
\note{
  These should exist for generic diagrams of a hyperbolic link, but the question is somewhat subtle: see the discussion at the end of \cref{sec:gluing-eqs}.
}
By definition the holonomy of such representations can be expressed in terms of Weyl characters.

By pulling back the outer $R$-matrix $\Rmat$ to an automorphism of (the division algebra) of $\weyl$ we get a braiding on the shapes.
We derive this action in detail in \cref{sec:weyl-braiding}, and before that give it without proof in \cref{sec:shaped-tangle-diagrams}.

\chapter{Shaped Tangles}
\label{ch:shaped-tangles}

\section*{Overview}

As usual for the Reshetikhin-Turaev construction, we think about links in terms of their diagrams.
More generally, by allowing links with boundaries (tangles) we obtain a braided monoidal category $\tang$ of tangles, which will correspond to the braided monoidal category of $\qgrp$-modules in the standard way.
To extend this to holonomy invariants, we consider tangle diagrams decorated with complex numbers we call \defemph{shapes} or \defemph{shape parameters}.
The shapes are closely related to the complex dihedral angles (shape parameters) of a certain ideal triangulation of the knot complement.

The shape parameters are a particular set of coordinates on the $\slg$-representation variety of a punctured disc.
They arise as central characters of the algebra $\weyl$, which is a \defemph{quantum cluster algebra} in the sense of \cite{Fock2009}.
Cluster algebras are related to ideal triangulations of surfaces, such as those in \cref{fig:build-ideal-octahedron} used to build the octahedral decomposition.

There have been a number of papers relating cluster algebras and hyperbolic structures on links, such as \cite{Hikami2014,Cho2020}, going back to the original work of \cite{Kashaev1995}.
They are also closely related to the segment variables of \cite{Kim2018}, and \cite{Baseilhac2004} use (in our language) ratios of shape parameters as the arguments of their cyclic quantum dilogarithms.

\subsection{Structure of the chapter}
In \cref{sec:tangle-diagrams} we establish basic conventions on tangles and their diagrams, then show how to label these diagrams by shapes and explain how they are related to $\slg$ representations of the tangle complements.
In \cref{sec:hyperbolic-links} we give a very rapid overview of the hyperbolic geometry of link complements.
Finally, in \cref{sec:octahedral-decompositions} we connect shaped tangles to hyperbolic geometry by explaining how the shape coordinates of a tangle are related to a certain ideal triangulation of its complement, the \defemph{octahedral decomposition}.

\section{Tangle diagrams}
\label{sec:tangle-diagrams}
We will construct our quantum holonomy invariants using categories of modules, and to match this we want a topological category to describe things like knots and links.

\subsection{Tangles}
\begin{defn}
  For $m,n \in \ZZ_{\ge 0}$, let $\mathfrak{T}_{m,n}$ be the space $[0,1]^3$ with $m$ marked points%
  \note{We should pick the locations of the marked points and for all: maybe put point $k$ at $(k/(m+1), 0, 0)$ or similar. The exact locations do not really matter.}
  in $\{0\} \times [0,1]^2$  and $n$ marked points in $\{1\} \times [0,1]^2$.
  We think of the $m$ points as incoming and on the left, and the $n$ points as outgoing and on the right.

  An \defemph{$(m,n)$-tangle} $T$ is a piecewise-linearizable%
  \note{We want to avoid wild tangles, which appear if only require the map to be continuous.
  What we call a tangle is more specifically a tame tangle.}
  embedding 
  \[
    f_T : [0,1]^{\amalg \, k} \amalg (S^1)^{\amalg \, l} \to \mathfrak{T}_{m,n}
  \]
  of $k$ disjoint copies of the interval $[0,1]$ and $l$ of the circle $S^1$ into $\mathfrak{T}_{m,n}$.
  We require $f_T$ to restrict to a bijection between the boundary of the intervals and the marked points.

  In this case we say that $T$ has $k$ \defemph{open components}, $l$ \defemph{closed components}, and $k + l$ \defemph{components}.
  We consider two tangles $T_1$ and $T_2$  equivalent if there is an ambient isotopy rel boundary of $\mathfrak{T}_{m,n}$ between them.

  A \defemph{link} is a $(0,0)$-tangle, and a \defemph{knot} is a $(0,0)$-tangle with one component.
  We consider the empty tangle to be the unique $(0,0)$-tangle with zero components.
\end{defn}

\begin{marginfigure}
  \centering
  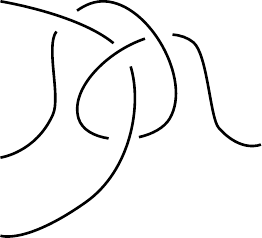
  \caption{A $(3,1)$-tangle.}
  \label{fig:tangle-example}
\end{marginfigure}

\begin{defn}
  \label{def:tangle-cat}
  Tangles form a monoidal category with
  \begin{description}
    \item[objects] nonnegative numbers.
    \item[morphisms] from $m$ to $n$ the space of $(m,n)$-tangles.
      To compose $T_1 : m \to n$ and $T_2 : n \to k$, we glue their ambient spaces along $\{1\} \times [0,1]^2$ inside $\mathfrak{T}_{m,n}$ and $\{0\} \times [0,1]^2$ inside $\mathfrak{T}_{n,k}$, identifying the $n$ marked points.
      (Afterward we should rescale to obtain a tangle in an interval of the right length.)
    \item[monoidal product] horizontal composition.
      In more detail, if $T_1 : m_1 \to n_1$ and $T_2 : m_2 \to n_2$ are two tangles, we can get a new tangle $T_1 \otimes T_2 : m_1 + m_2 \to n_1 + n_2$ by gluing $[0,1] \times \{1\} \times [0,1]$ inside $\mathfrak{T}_{m_1, n_1}$ to $[0,1] \times \{0\} \times [0,1]$ inside $\mathfrak{T}_{m_2, n_2}$, then rescaling so the marked points in the boundary go where they should.
  \end{description}
\end{defn}

\begin{figure}
  \centering
\begingroup%
  \makeatletter%
  \providecommand\color[2][]{%
    \errmessage{(Inkscape) Color is used for the text in Inkscape, but the package 'color.sty' is not loaded}%
    \renewcommand\color[2][]{}%
  }%
  \providecommand\transparent[1]{%
    \errmessage{(Inkscape) Transparency is used (non-zero) for the text in Inkscape, but the package 'transparent.sty' is not loaded}%
    \renewcommand\transparent[1]{}%
  }%
  \providecommand\rotatebox[2]{#2}%
  \newcommand*\fsize{\dimexpr\f@size pt\relax}%
  \newcommand*\lineheight[1]{\fontsize{\fsize}{#1\fsize}\selectfont}%
  \ifx\svgwidth\undefined%
    \setlength{\unitlength}{243.48140717bp}%
    \ifx\svgscale\undefined%
      \relax%
    \else%
      \setlength{\unitlength}{\unitlength * \real{\svgscale}}%
    \fi%
  \else%
    \setlength{\unitlength}{\svgwidth}%
  \fi%
  \global\let\svgwidth\undefined%
  \global\let\svgscale\undefined%
  \makeatother%
  \begin{picture}(1,0.73555946)%
    \lineheight{1}%
    \setlength\tabcolsep{0pt}%
    \put(0,0){\includegraphics[width=\unitlength,page=1]{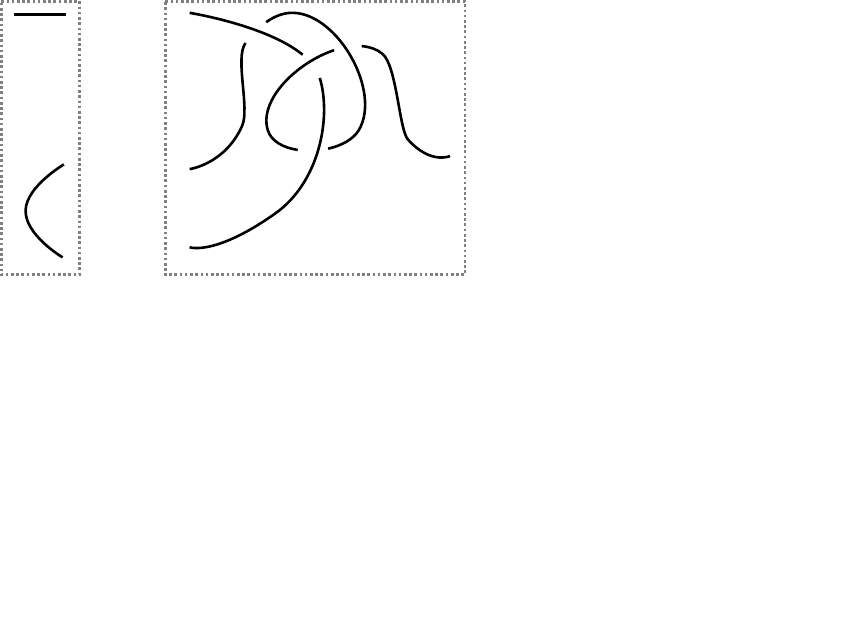}}%
    \put(0.13405882,0.56459943){\makebox(0,0)[lt]{\lineheight{1.25}\smash{\begin{tabular}[t]{l}$\circ$\end{tabular}}}}%
    \put(0.57448152,0.56459943){\makebox(0,0)[lt]{\lineheight{1.25}\smash{\begin{tabular}[t]{l}$=$\end{tabular}}}}%
    \put(0,0){\includegraphics[width=\unitlength,page=2]{tangle-composition.pdf}}%
    \put(0.13284346,0.15555846){\makebox(0,0)[lt]{\lineheight{1.25}\smash{\begin{tabular}[t]{l}$\otimes$\end{tabular}}}}%
    \put(0,0){\includegraphics[width=\unitlength,page=3]{tangle-composition.pdf}}%
    \put(0.57448152,0.15183657){\makebox(0,0)[lt]{\lineheight{1.25}\smash{\begin{tabular}[t]{l}$=$\end{tabular}}}}%
    \put(0,0){\includegraphics[width=\unitlength,page=4]{tangle-composition.pdf}}%
  \end{picture}%
\endgroup%

  \caption{Tangle composition and tensor product}
  \label{fig:tangle-composition-example}
\end{figure}

\subsection{Tangle diagrams}
To make the connection with representation theory we want a more combinatorial description of tangles.
\begin{defn}
  An \defemph{$(m,n)$-tangle diagram} is the projection of an $(m,n)$-tangle to $[0,1]^2$ with finitely many double points (and no other self-intersections), all of which are transverse.
  At each double point, we record which strand went above the other.
  Tangle diagrams form a monoidal category just as tangles do.
\end{defn}

\begin{defn}
  \label{def:diagram-terms}
  Let $D$ be a tangle diagram.
  We fix some terminology on parts of $D$.
  \begin{itemize}
    \item An \defemph{arc} starts at an undercrossing, continues past overcrossings, and ends at the next undercrossing.
  For example, the diagram of the trefoil knot in \cref{fig:trefoil-wirtinger} has three arcs, each colored differently.
    \item
    Each arc is composed of \defemph{segments}, which start and end at crossings (regardless of whether they are over or undercrossings.)
    The diagram in \cref{fig:trefoil-wirtinger} has six segments.%
  \note{If we think of a tangle diagram as a graph, so that the crossings are vertices with $4$ edges, then the edges of this graph are the segements.}
  \item The diagram cuts the plane into \defemph{regions} separated from each other by the segments.
    The diagram in \cref{fig:trefoil-wirtinger} has five regions.
  \end{itemize}
\end{defn}

We have already seen examples: we drew the tangle in \cref{fig:tangle-example} as a tangle diagram.
Any (tame) tangle can be isotoped so that its projection gives a tangle diagram, and it is well-known that two tangle diagrams represent the same tangle exactly when they are related by local moves:
\begin{thm}
  Two tangle diagrams represent isotopic tangles if and only if they are related by isotopies of $[0,1]^2$ fixing the boundary and the \defemph{Reidemeister moves} of \cref{fig:reidemeister-moves}.
\end{thm}

Notice that we do \emph{not} say that two tangle diagrams are equal if they are related by Reidemeister moves, just that they are equivalent.
In the next section we will associate each tangle diagram a system of gluing equations, which transform nontrivially if the crossings of the diagram change.
Later we will label our diagrams by parameters that are only generically defined, so we will have to be careful as to what moves are permitted.
For both reasons we want to think of a diagram as a fixed object.

\begin{figure}
  \centering
\begingroup%
  \makeatletter%
  \providecommand\color[2][]{%
    \errmessage{(Inkscape) Color is used for the text in Inkscape, but the package 'color.sty' is not loaded}%
    \renewcommand\color[2][]{}%
  }%
  \providecommand\transparent[1]{%
    \errmessage{(Inkscape) Transparency is used (non-zero) for the text in Inkscape, but the package 'transparent.sty' is not loaded}%
    \renewcommand\transparent[1]{}%
  }%
  \providecommand\rotatebox[2]{#2}%
  \newcommand*\fsize{\dimexpr\f@size pt\relax}%
  \newcommand*\lineheight[1]{\fontsize{\fsize}{#1\fsize}\selectfont}%
  \ifx\svgwidth\undefined%
    \setlength{\unitlength}{360.90743637bp}%
    \ifx\svgscale\undefined%
      \relax%
    \else%
      \setlength{\unitlength}{\unitlength * \real{\svgscale}}%
    \fi%
  \else%
    \setlength{\unitlength}{\svgwidth}%
  \fi%
  \global\let\svgwidth\undefined%
  \global\let\svgscale\undefined%
  \makeatother%
  \begin{picture}(1,0.59377907)%
    \lineheight{1}%
    \setlength\tabcolsep{0pt}%
    \put(0,0){\includegraphics[width=\unitlength,page=1]{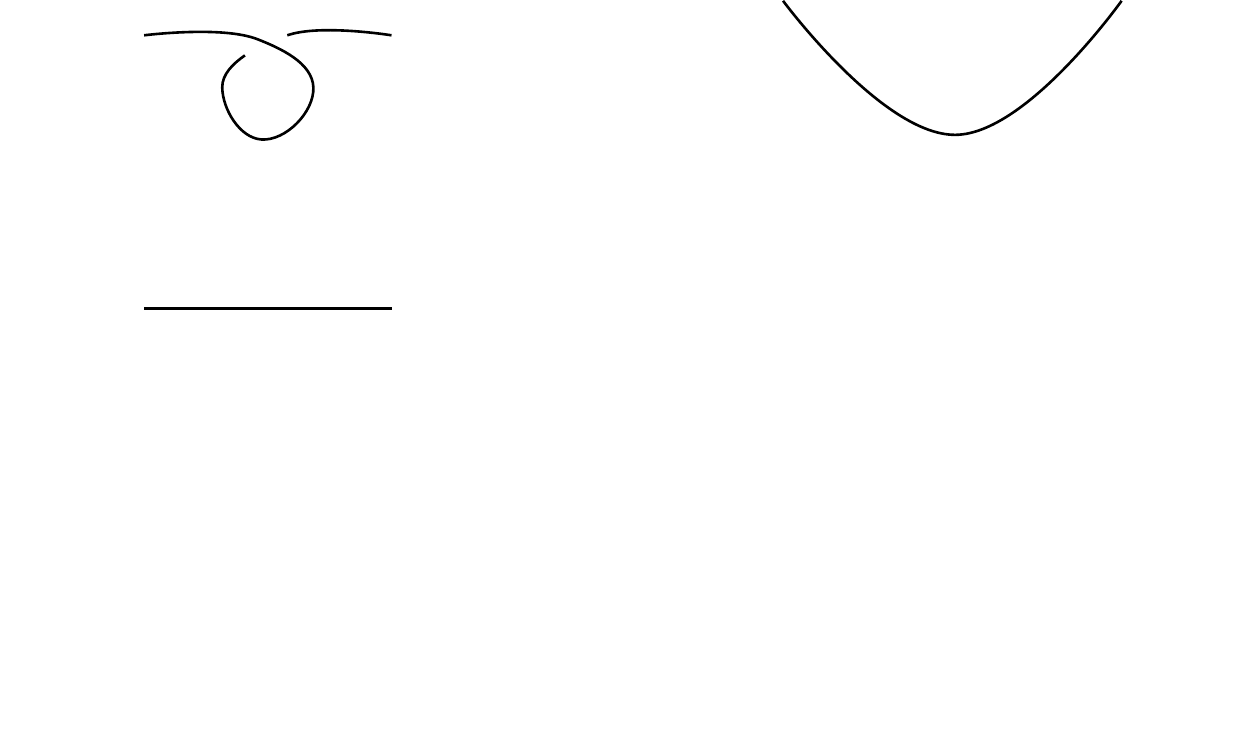}}%
    \put(0.47642591,0.10088182){\makebox(0,0)[lt]{\lineheight{1.25}\smash{\begin{tabular}[t]{l}$\overset{=}{\text{(\reidthree{})}}$\end{tabular}}}}%
    \put(0,0){\includegraphics[width=\unitlength,page=2]{reidemeister-moves.pdf}}%
    \put(0.73375186,0.42920934){\makebox(0,0)[lt]{\lineheight{1.25}\smash{\begin{tabular}[t]{l}$\overset{=}{\text{(\reidtwo{})}}$\end{tabular}}}}%
    \put(0.19241511,0.42608497){\makebox(0,0)[lt]{\lineheight{1.25}\smash{\begin{tabular}[t]{l}$\overset{=}{\text{(\reidone{})}}$\end{tabular}}}}%
    \put(0,0){\includegraphics[width=\unitlength,page=3]{reidemeister-moves.pdf}}%
  \end{picture}%
\endgroup%

  \caption{The Reidemeister moves}
  \label{fig:reidemeister-moves}
\end{figure}

\begin{remark}
  \label{rem:left-to-right}
  We write tangle diagrams left-to-right, so composition is horizontal and tensor product is vertical, as shown in \cref{fig:tangle-composition-example}.
  This is a slightly nonstandard convention in quantum topology, but we find it much more convenient to read left-to-right than vertically, especially since there are two conventions on which order to read vertical diagrams.
\end{remark}

\begin{marginfigure}
  \centering
  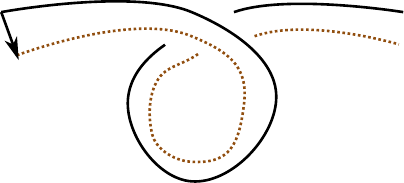
  \caption{The blackboard framing: the tip of the normal vector traces out the {\color{accent} gold} path.}
  \label{fig:blackboard-framing}
\end{marginfigure}
\begin{marginfigure}
  \centering
  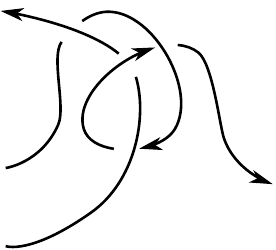
  \caption{An oriented tangle, specifically a morphism $(-, +, +) \to (+)$.}
  \label{fig:tangle-example-oriented}
\end{marginfigure}

To keep track of labelings we need to orient our tangles, and we also need to keep track of the framing.%
\note{This is exactly the data needed to make into a pivotal category as in \cref{def:pivotal-cat}.}
\begin{defn}
  A \defemph{framed} tangle is a tangle along with a trivalization of the normal bundle of each component.
  As usual, we can frame a tangle diagram by using the \defemph{blackboard framing} of \cref{fig:blackboard-framing}.
  The \reidone{} move does not preserve the blackboard framing, so two framed tangle diagrams are equivalent if and only if they are related by isotopy and the \reidtwo{} and \reidthree{} moves.
  A \defemph{oriented} tangle is a tangle along with an orientation of each component.
  We record the orientation of a tangle diagram as in \cref{fig:tangle-example-oriented}.
\end{defn}

Tangle diagrams form categories just like the topological category of tangles (and its obvious extensions to oriented and framed tangles) given in \cref{def:tangle-cat}.
\begin{defn}
  We write $\tang[]$ for the monoidal category of oriented framed tangle diagrams, with
  \begin{description}
    \item[objects] tuples $(\epsilon_1, \epsilon_2, \dots, \epsilon_m)$ of signs in $\{+, -\}$.
    \item[morphisms] tangle diagrams. A morphism $(\epsilon_1, \dots, \epsilon_m) \to (\epsilon_1, \dots, \epsilon_n)$ is an oriented framed $(m,n)$-tangle diagram with orientations matching the signs:
      on the left-hand (domain) side, a $+$ corresponds to an outgoing strand and a $-$ to an incoming strand.
      On the right-hand side the opposite holds.
      Compositon is given by horizontal composition of diagrams.
    \item[monoidal product] given by vertical composition of diagrams.
  \end{description}
\end{defn}

\begin{defn}
The orientation on a tangle diagram lets us separate crossings into two types: \defemph{positive} and \defemph{negative}.
These are denoted by $\sigma$ and $\widetilde \sigma$ in \cref{fig:tangle-generators}.
(We prefer the notation $\widetilde \sigma$ to $\sigma^{-1}$ because negative and inverse braidings are distinct when considering tangles colored by a biquandle.)
\end{defn}

Tangle diagrams can be built via under composition and tensor product out of the generators in \cref{fig:tangle-generators}.
\begin{figure}
  \centering
  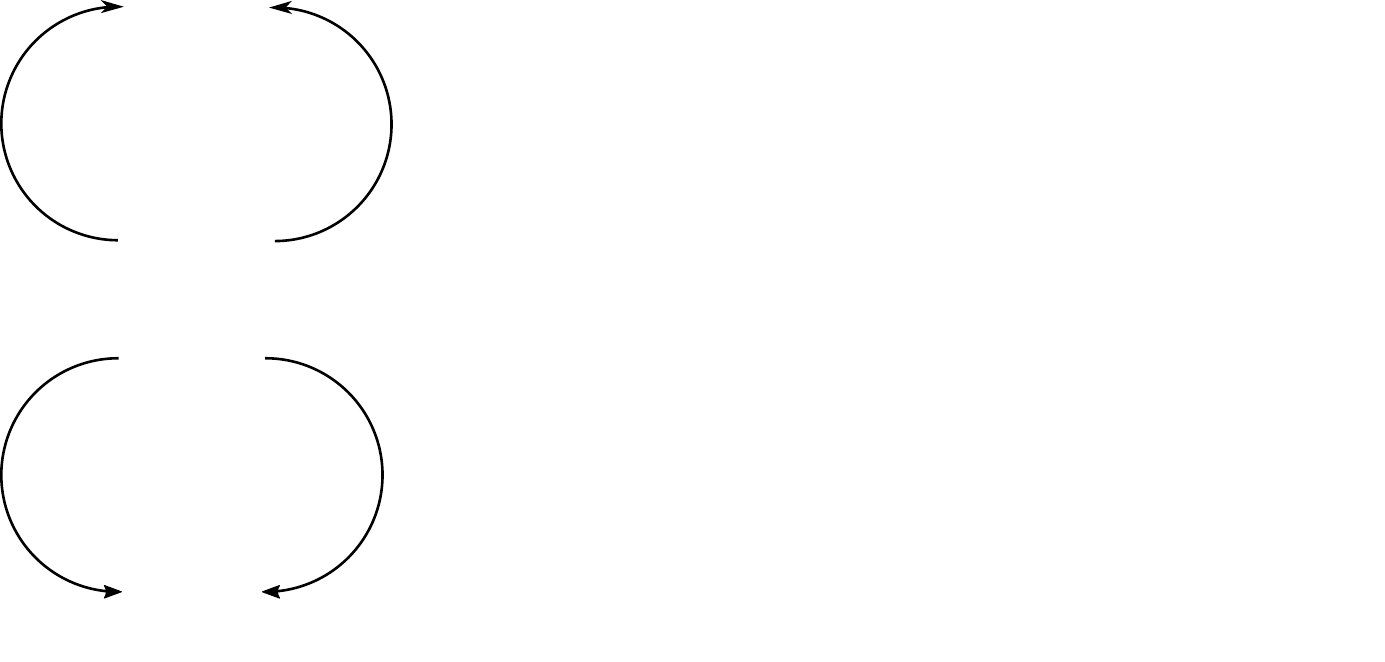
  \caption{Generators of oriented tangle diagrams.}
  \label{fig:tangle-generators}
\end{figure}
In more algebraic language, the generators are maps
\begin{align*}
  &\coevup{} : (\,) \to (+,-)
  &
  &\evup{} : (-,+) \to (\,)
  &
  &\sigma : (+, +) \to (+,+)
  \\
  &\coevdown{} : (\,) \to (-,+)
  &
  &\evdown{} : (+,-) \to (\,)
  &
  &\widetilde{\sigma} : (+, +) \to (+,+)
\end{align*}
where $(\,)$ is the empty tuple.
We can obtain braidings $(+,-) \to (-,+)$, etc.~by rotating $\sigma$ and $\overline{\sigma}$ using the evaluation and coevaluation morphisms, as in \cref{fig:sideways-braidings}.%
\note{
  Our later condition of \defemph{sideways invertibility} (\cref{def:yang-baxter-model}) is related to such rotations.
}

We can always isotope a tangle diagram so that it is composed out of exactly these generators.
This is sometimes called a \defemph{Morse-link presentation} of the tangle, because it has isolated critical points and you prove it exists via Morse theory.
We can view such a presentation as a certain kind of decorated planar graph.

\subsection{The fundamental groupoid of a tangle}
\label{sec:fundamental-groupoid}
Our aim is to study links $L$ along with representations $\pi_1(S^3 \setminus L) \to \slg$ of their complements, so we first establish some conventions on fundamental groups.

\begin{defn}
  Let $T$ be a $(m,n)$ tangle.
  Its \defemph{fundamental group} $\pi(D)$ is the fundamental group $\pi_1(\mathfrak{T}_{(m,n)} \setminus T)$ of the complement of $T$.
  When $D$ is a tangle diagram we similarly define $\pi(D)$ to be the fundamental group of the underlying tangle.

  For $G$ a group, a \defemph{$G$-tangle} is a tangle $T$ along with a representation $\rho : \pi(T) \to G$.
  We say two representations $\rho_1, \rho_2$ are \defemph{conjugate} or \defemph{gauge-equivalent} if there is some $x \in G$ with%
  \note{If $\rho$ is the holonomy of a flat $\mathfrak{g}$-connection, then a gauge transformation of the connection corresponds to conjugation of $\rho$.}
  \[
    x \rho_1(g) x^{-1} = \rho_2(g) \text{ for all } g \in G.
  \]
  In this case we say that $\rho_2$ was obtained by a \defemph{gauge transformation} of $\rho_1$.
\end{defn}

\begin{defn}
  \label{def:wirtinger-presentation}

  Let $T$ be a tangle represented by an oriented diagram $D$.
  The \defemph{Wirtinger presentation} of $\pi(T)$ has one generator for every arc (\cref{def:diagram-terms}) of $D$ and a conjugation relation at each crossing.
  We write $\pi(D)$ for the Wirtinger presentation of $\pi(T)$ associated to the diagram $D$.
\end{defn}

\begin{figure}
  \centering
\begingroup%
  \makeatletter%
  \providecommand\color[2][]{%
    \errmessage{(Inkscape) Color is used for the text in Inkscape, but the package 'color.sty' is not loaded}%
    \renewcommand\color[2][]{}%
  }%
  \providecommand\transparent[1]{%
    \errmessage{(Inkscape) Transparency is used (non-zero) for the text in Inkscape, but the package 'transparent.sty' is not loaded}%
    \renewcommand\transparent[1]{}%
  }%
  \providecommand\rotatebox[2]{#2}%
  \newcommand*\fsize{\dimexpr\f@size pt\relax}%
  \newcommand*\lineheight[1]{\fontsize{\fsize}{#1\fsize}\selectfont}%
  \ifx\svgwidth\undefined%
    \setlength{\unitlength}{156.94872665bp}%
    \ifx\svgscale\undefined%
      \relax%
    \else%
      \setlength{\unitlength}{\unitlength * \real{\svgscale}}%
    \fi%
  \else%
    \setlength{\unitlength}{\svgwidth}%
  \fi%
  \global\let\svgwidth\undefined%
  \global\let\svgscale\undefined%
  \makeatother%
  \begin{picture}(1,0.77254016)%
    \lineheight{1}%
    \setlength\tabcolsep{0pt}%
    \put(0,0){\includegraphics[width=\unitlength,page=1]{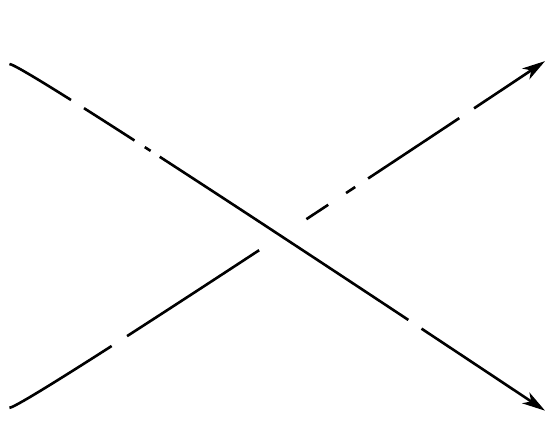}}%
    \put(0.63876859,0.01803184){\makebox(0,0)[lt]{\lineheight{1.25}\smash{\begin{tabular}[t]{l}$w_1$\end{tabular}}}}%
    \put(0.87770008,0.44810853){\makebox(0,0)[lt]{\lineheight{1.25}\smash{\begin{tabular}[t]{l}$w_{2'} = w_1^{-1}w_2 w_1$\\\end{tabular}}}}%
    \put(0,0){\includegraphics[width=\unitlength,page=2]{wirtinger-presentation.pdf}}%
    \put(-0.00634662,0.47200171){\makebox(0,0)[lt]{\lineheight{1.25}\smash{\begin{tabular}[t]{l}$w_1$\end{tabular}}}}%
    \put(0.23258493,0.01803184){\makebox(0,0)[lt]{\lineheight{1.25}\smash{\begin{tabular}[t]{l}$w_2$\end{tabular}}}}%
  \end{picture}%
\endgroup%

  \caption{Generators of the Writinger presentation near a crossing. Notice that the generator $w_1$ corresponding to the over-arc is the same on each side.}
  \label{fig:wirtinger-presentation}
\end{figure}

\begin{ex}
  \begin{marginfigure}
    \centering
\begingroup%
  \makeatletter%
  \providecommand\color[2][]{%
    \errmessage{(Inkscape) Color is used for the text in Inkscape, but the package 'color.sty' is not loaded}%
    \renewcommand\color[2][]{}%
  }%
  \providecommand\transparent[1]{%
    \errmessage{(Inkscape) Transparency is used (non-zero) for the text in Inkscape, but the package 'transparent.sty' is not loaded}%
    \renewcommand\transparent[1]{}%
  }%
  \providecommand\rotatebox[2]{#2}%
  \newcommand*\fsize{\dimexpr\f@size pt\relax}%
  \newcommand*\lineheight[1]{\fontsize{\fsize}{#1\fsize}\selectfont}%
  \ifx\svgwidth\undefined%
    \setlength{\unitlength}{131.53002882bp}%
    \ifx\svgscale\undefined%
      \relax%
    \else%
      \setlength{\unitlength}{\unitlength * \real{\svgscale}}%
    \fi%
  \else%
    \setlength{\unitlength}{\svgwidth}%
  \fi%
  \global\let\svgwidth\undefined%
  \global\let\svgscale\undefined%
  \makeatother%
  \begin{picture}(1,0.87829149)%
    \lineheight{1}%
    \setlength\tabcolsep{0pt}%
    \put(0,0){\includegraphics[width=\unitlength,page=1]{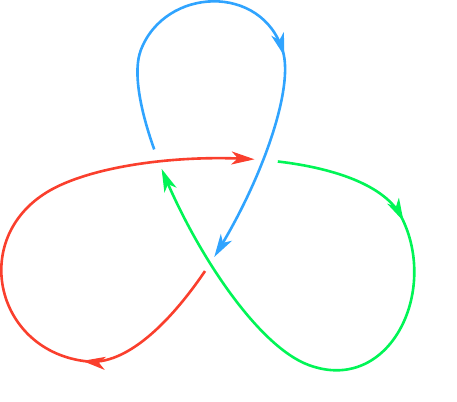}}%
    \put(0.64753263,0.76899158){\color[rgb]{0.18823529,0.64313725,1}\makebox(0,0)[lt]{\lineheight{1.25}\smash{\begin{tabular}[t]{l}$w_1$\end{tabular}}}}%
    \put(0.90412803,0.42686437){\color[rgb]{0.00392157,0.96078431,0.3372549}\makebox(0,0)[lt]{\lineheight{1.25}\smash{\begin{tabular}[t]{l}$w_2$\end{tabular}}}}%
    \put(0.1628524,-0.00079459){\color[rgb]{0.98039216,0.24705882,0.17647059}\makebox(0,0)[lt]{\lineheight{1.25}\smash{\begin{tabular}[t]{l}$w_3$\end{tabular}}}}%
  \end{picture}%
\endgroup%

    \caption{Generators of the fundamental group of the trefoil complement.}
    \label{fig:trefoil-wirtinger}
  \end{marginfigure}
  The trefoil group has three generators $w_1, w_2, w_3$ as shown in \cref{fig:trefoil-wirtinger}.
  The relations from the crossings are (clockwise from top-right)
  \begin{align*}
    w_3 w_1 &= w_2 w_1
            &
    w_1 w_2 &= w_2 w_3
            &
    w_2 w_3 &= w_3 w_1.
  \end{align*}
\end{ex}
Since each arc corresponds to a generator of $\pi(D)$, we can describe $G$-tangles by labeling their arcs by elements of $G$ satisfying the Wirtinger relations.
In the language of \cref{ch:functors}, these are tangles colored by the conjugation quandle of $G$.

Howver, it turns out to be much more convenient to describe things in terms of the fundamental \emph{groupoid}, for two reasons:
\begin{enumerate}
  \item The braiding on the quantum group does not correspond to the conjugation quandle of $\slg$, but a slightly more complicated structure (a biquandle factorization of it.)
    We discuss this issue in \cref{ch:functors}.
  \item Using the groupoid representation allows for a direct geometric interpretation of the matrix coefficients of the holonomy, as explained in \cref{sec:shaped-tangle-diagrams}.
\end{enumerate}
\begin{defn}
  A \defemph{groupoid} is a category in which all morphisms are invertible.
  A group can be viewed as a groupoid with one object.
\end{defn}

\begin{ex}
  The fundamental group $\pi_1(X, p_0)$ of a topological space $X$ with basepoint $p_0 \in X$ is the group of homotopy classes of loops based at $p_0$.
  We can equivalently view $\pi_1(X, p_0)$ as a category with a single object $p_0$ and morphisms homotopy classes of paths starting and ending at $p_0$.

  From this perspective, we can generalize to the \defemph{fundamental groupoid} $\Pi_1(X, X_0)$ of $X$ with basepoints $X_0 \subseteq X$.
  This is a category with
  \begin{description}
    \item[objects] elements of $X_0$ .
    \item[morphisms] $p_0 \to p_1$ homotopy classes of paths $p_0$ to $p_1$.
  \end{description}
  We recover the fundamental group as the special case $\pi_1(X, p_0) = \Pi_1(X, \{p_0\})$ with a single basepoint.
\end{ex}

\begin{defn}
  \begin{marginfigure}
\begingroup%
  \makeatletter%
  \providecommand\color[2][]{%
    \errmessage{(Inkscape) Color is used for the text in Inkscape, but the package 'color.sty' is not loaded}%
    \renewcommand\color[2][]{}%
  }%
  \providecommand\transparent[1]{%
    \errmessage{(Inkscape) Transparency is used (non-zero) for the text in Inkscape, but the package 'transparent.sty' is not loaded}%
    \renewcommand\transparent[1]{}%
  }%
  \providecommand\rotatebox[2]{#2}%
  \newcommand*\fsize{\dimexpr\f@size pt\relax}%
  \newcommand*\lineheight[1]{\fontsize{\fsize}{#1\fsize}\selectfont}%
  \ifx\svgwidth\undefined%
    \setlength{\unitlength}{140.23033333bp}%
    \ifx\svgscale\undefined%
      \relax%
    \else%
      \setlength{\unitlength}{\unitlength * \real{\svgscale}}%
    \fi%
  \else%
    \setlength{\unitlength}{\svgwidth}%
  \fi%
  \global\let\svgwidth\undefined%
  \global\let\svgscale\undefined%
  \makeatother%
  \begin{picture}(1,0.37438403)%
    \lineheight{1}%
    \setlength\tabcolsep{0pt}%
    \put(0,0){\includegraphics[width=\unitlength,page=1]{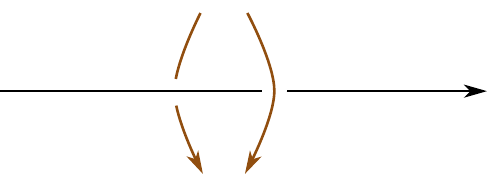}}%
    \put(0.57762114,0.26741715){\makebox(0,0)[lt]{\lineheight{1.25}\smash{\begin{tabular}[t]{l}$x_i^+$\end{tabular}}}}%
    \put(0.25672045,0.26741715){\makebox(0,0)[lt]{\lineheight{1.25}\smash{\begin{tabular}[t]{l}$x_i^-$\end{tabular}}}}%
    \put(0,0){\includegraphics[width=\unitlength,page=2]{tangle-groupoid-generators.pdf}}%
    \put(0.64180094,0.10696684){\makebox(0,0)[lt]{\lineheight{1.25}\smash{\begin{tabular}[t]{l}$i$\end{tabular}}}}%
  \end{picture}%
\endgroup%

    \caption{Each strand $i$ is associated to two morphisms (paths) $x_i^+$ and $x_i^-$ in $\Pi(D)$.}
    \label{fig:tangle-groupoid-generators}
  \end{marginfigure}
  \begin{marginfigure}
\begingroup%
  \makeatletter%
  \providecommand\color[2][]{%
    \errmessage{(Inkscape) Color is used for the text in Inkscape, but the package 'color.sty' is not loaded}%
    \renewcommand\color[2][]{}%
  }%
  \providecommand\transparent[1]{%
    \errmessage{(Inkscape) Transparency is used (non-zero) for the text in Inkscape, but the package 'transparent.sty' is not loaded}%
    \renewcommand\transparent[1]{}%
  }%
  \providecommand\rotatebox[2]{#2}%
  \newcommand*\fsize{\dimexpr\f@size pt\relax}%
  \newcommand*\lineheight[1]{\fontsize{\fsize}{#1\fsize}\selectfont}%
  \ifx\svgwidth\undefined%
    \setlength{\unitlength}{144bp}%
    \ifx\svgscale\undefined%
      \relax%
    \else%
      \setlength{\unitlength}{\unitlength * \real{\svgscale}}%
    \fi%
  \else%
    \setlength{\unitlength}{\svgwidth}%
  \fi%
  \global\let\svgwidth\undefined%
  \global\let\svgscale\undefined%
  \makeatother%
  \begin{picture}(1,0.67307788)%
    \lineheight{1}%
    \setlength\tabcolsep{0pt}%
    \put(0,0){\includegraphics[width=\unitlength,page=1]{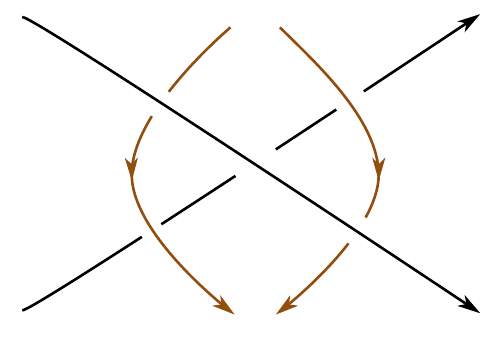}}%
    \put(-0.0032803,0.63811369){\makebox(0,0)[lt]{\lineheight{1.25}\smash{\begin{tabular}[t]{l}$1$\end{tabular}}}}%
    \put(-0.0032803,0.03553828){\makebox(0,0)[lt]{\lineheight{1.25}\smash{\begin{tabular}[t]{l}$2$\end{tabular}}}}%
    \put(0.97038872,0.04534311){\makebox(0,0)[lt]{\lineheight{1.25}\smash{\begin{tabular}[t]{l}$1'$\end{tabular}}}}%
    \put(0.97038872,0.64305345){\makebox(0,0)[lt]{\lineheight{1.25}\smash{\begin{tabular}[t]{l}$2'$\end{tabular}}}}%
    \put(0.07148831,0.32196937){\makebox(0,0)[lt]{\lineheight{1.25}\smash{\begin{tabular}[t]{l}$x_1^- x_2^+$\end{tabular}}}}%
    \put(0.79194551,0.32196937){\makebox(0,0)[lt]{\lineheight{1.25}\smash{\begin{tabular}[t]{l}$x_{2'}^+ x_{1'}^-$\end{tabular}}}}%
  \end{picture}%
\endgroup%

    \caption{Deriving the middle relation at a crossing.}
    \label{fig:tangle-groupoid-relations}
  \end{marginfigure}
  Let $D$ be a tangle diagram, and recall the definition (\cref{def:diagram-terms}) of the regions and segments of $D$.
  The \defemph{fundamental groupoid} $\Pi(D)$ of $D$ has
  \begin{description}
    \item[objects] regions of $D$.
    \item[morphisms] words in certain generators.
  For any two adjacent regions separated by a segment $i$ we assign two generators $x_i^+$ and $x_i^-$, which we think of as going above and below the segment, as in \cref{fig:tangle-groupoid-generators}.
  (We orient the generators to match the orientation of $D$.)
  \end{description}
  At each crossing there are three relations
  \begin{align*}
    x_1^- x_2^- &= x_{2'}^- x_{1'}^-
                &
    x_1^- x_2^+ &= x_{2'}^+ x_{1'}^-
                &
    x_1^+ x_2^+ &= x_{2'}^+ x_{1'}^+
  \end{align*}
  obtained from paths below, between, and above the strands; we have shown the paths corresponding to the between relation in \cref{fig:tangle-groupoid-relations}.
\end{defn}

\begin{figure}
  \centering
\begingroup%
  \makeatletter%
  \providecommand\color[2][]{%
    \errmessage{(Inkscape) Color is used for the text in Inkscape, but the package 'color.sty' is not loaded}%
    \renewcommand\color[2][]{}%
  }%
  \providecommand\transparent[1]{%
    \errmessage{(Inkscape) Transparency is used (non-zero) for the text in Inkscape, but the package 'transparent.sty' is not loaded}%
    \renewcommand\transparent[1]{}%
  }%
  \providecommand\rotatebox[2]{#2}%
  \newcommand*\fsize{\dimexpr\f@size pt\relax}%
  \newcommand*\lineheight[1]{\fontsize{\fsize}{#1\fsize}\selectfont}%
  \ifx\svgwidth\undefined%
    \setlength{\unitlength}{388.34960175bp}%
    \ifx\svgscale\undefined%
      \relax%
    \else%
      \setlength{\unitlength}{\unitlength * \real{\svgscale}}%
    \fi%
  \else%
    \setlength{\unitlength}{\svgwidth}%
  \fi%
  \global\let\svgwidth\undefined%
  \global\let\svgscale\undefined%
  \makeatother%
  \begin{picture}(1,0.42270363)%
    \lineheight{1}%
    \setlength\tabcolsep{0pt}%
    \put(0,0){\includegraphics[width=\unitlength,page=1]{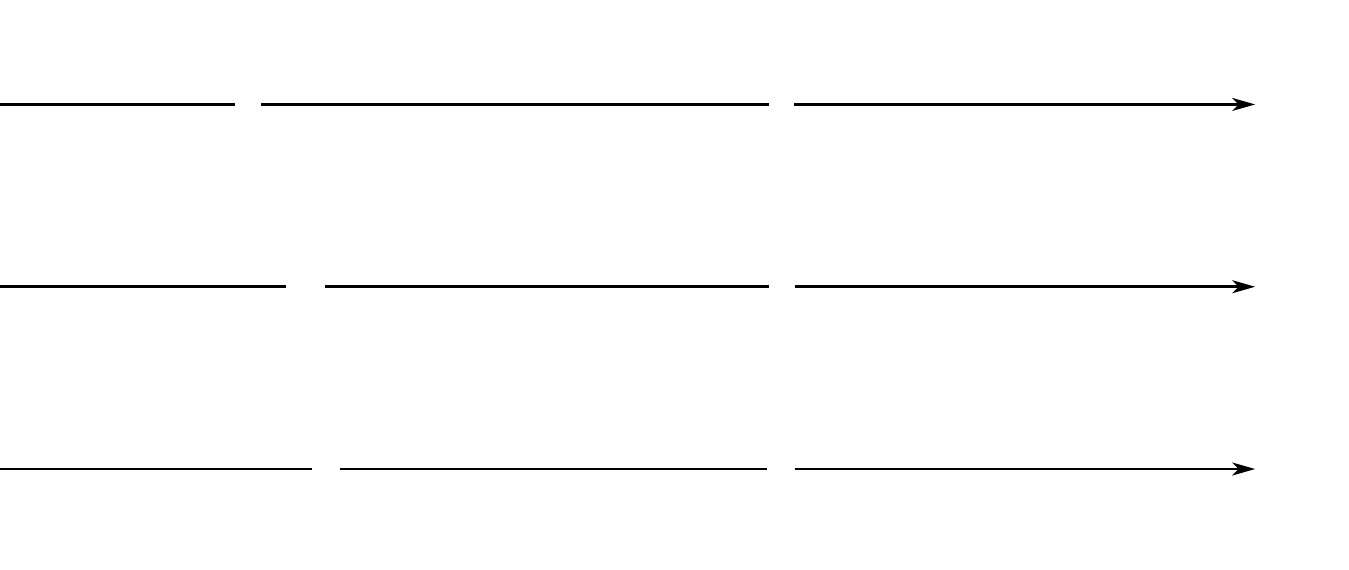}}%
    \put(0.9435563,0.34781885){\makebox(0,0)[lt]{\lineheight{1.25}\smash{\begin{tabular}[t]{l}$1$\end{tabular}}}}%
    \put(0.94355329,0.21006863){\makebox(0,0)[lt]{\lineheight{1.25}\smash{\begin{tabular}[t]{l}$2$\end{tabular}}}}%
    \put(0.9435563,0.07488117){\makebox(0,0)[lt]{\lineheight{1.25}\smash{\begin{tabular}[t]{l}$3$\end{tabular}}}}%
    \put(0,0){\includegraphics[width=\unitlength,page=2]{path-factorization.pdf}}%
    \put(0.16415622,0.00728742){\makebox(0,0)[lt]{\lineheight{1.25}\smash{\begin{tabular}[t]{l}$w_3$\end{tabular}}}}%
    \put(0.58903112,0.30663111){\makebox(0,0)[lt]{\lineheight{1.25}\smash{\begin{tabular}[t]{l}$x_1^+$\end{tabular}}}}%
    \put(0.59868735,0.17144363){\makebox(0,0)[lt]{\lineheight{1.25}\smash{\begin{tabular}[t]{l}$x_2^+$\end{tabular}}}}%
    \put(0.59868735,0.03625617){\makebox(0,0)[lt]{\lineheight{1.25}\smash{\begin{tabular}[t]{l}$x_3^+$\end{tabular}}}}%
    \put(0.42210433,0.11471753){\makebox(0,0)[lt]{\lineheight{1.25}\smash{\begin{tabular}[t]{l}$\left(x_3^-\right)^{-1}$\end{tabular}}}}%
  \end{picture}%
\endgroup%

  \caption{The path $w_3$ in $\pi(D)$ and the path $x_1^+ x_2^+ x_3^+ \left(x_3^- \right)^{-1} \left(x_2^+\right)^{-1} \left(x_1^+\right)^{-1}$ in $\Pi(D)$ are equivalent.}
  \label{fig:path-factorization}
\end{figure}
\begin{prop}
  \label{thm:path-factorization}
  There is an equivalence of categories $\Phi : \pi(D) \to \Pi(D)$.
\end{prop}
  We say that $\pi(D)$ is the \defemph{skeleton} of $\Pi(D)$.
  The skeleton of a groupoid always exists, and here we have chosen a particular form of it by using the Wirtinger presentation with a specified basepoint.
\begin{proof}
  To see how to define $\Phi$, consult \cref{fig:path-factorization}, which shows that 
  \[
    \Phi(w_3) = x_1^+ x_2^+ x_3^+ \left(x_1^+ x_2^+ x_3^-\right)^{-1}.
  \]
  In general, we view the Wirtinger generators as travelling over the diagram to get to their designated segment, after which they loop around it and return.
  There is more than one path to take from the basepoint at the top of the diagram, but this does not matter because of relations like $x_1^+ x_2^+ = x_{2'}^+ x_{1'}^+$ between the over paths.

  It is geometrically evident that $\Phi$ gives an equivalence of categories.%
  \note{When we say ``geometrically evident'' we mean ``evident to the author and annoying to explain the details''.}
  For details, see \cite[\S 3.5]{Geer2013}, in particular Lemma 3.4.
\end{proof}

\subsection{Shaped tangle diagrams}
\label{sec:shaped-tangle-diagrams}

We can now describe our preferred description of $\slg$-tangles.
We leave two technical questions unanswered:
\begin{enumerate}
  \item Are the shapes compatible with Reidemeister moves?
  \item Can every $\slg$-tangle be represented as a shaped tangle?
\end{enumerate}
We address these issues in \cref{ch:functors}.

\begin{defn}
  We say that a tangle diagram $D$ is \defemph{pre-shaped} if
  \begin{enumerate}
    \item each segment is assigned a nonzero complex number $b \in \CC^\times$, and
    \item each connected component is assigned a nonzero complex number $\lambda \in \CC^\times$.
  \end{enumerate}
  We refer to these numbers as the \defemph{$b$-variables} and \defemph{eigenvalues} of the diagram.
  We typically assume the segments of the diagram are indexed by a set $I$ and write $b_i$ and $\lambda_i$ for the variables assigned to segment $i \in I$.

  At every crossing of the diagram we assign \defemph{$a$-variables} to the surrounding half-segments.
  At a positive crossing with labels $1, 2, 1', 2'$ (as in \cref{fig:shaped-positive-crossing}) we set
  \begin{marginfigure}
\begingroup%
  \makeatletter%
  \providecommand\color[2][]{%
    \errmessage{(Inkscape) Color is used for the text in Inkscape, but the package 'color.sty' is not loaded}%
    \renewcommand\color[2][]{}%
  }%
  \providecommand\transparent[1]{%
    \errmessage{(Inkscape) Transparency is used (non-zero) for the text in Inkscape, but the package 'transparent.sty' is not loaded}%
    \renewcommand\transparent[1]{}%
  }%
  \providecommand\rotatebox[2]{#2}%
  \newcommand*\fsize{\dimexpr\f@size pt\relax}%
  \newcommand*\lineheight[1]{\fontsize{\fsize}{#1\fsize}\selectfont}%
  \ifx\svgwidth\undefined%
    \setlength{\unitlength}{139.47438812bp}%
    \ifx\svgscale\undefined%
      \relax%
    \else%
      \setlength{\unitlength}{\unitlength * \real{\svgscale}}%
    \fi%
  \else%
    \setlength{\unitlength}{\svgwidth}%
  \fi%
  \global\let\svgwidth\undefined%
  \global\let\svgscale\undefined%
  \makeatother%
  \begin{picture}(1,0.50689042)%
    \lineheight{1}%
    \setlength\tabcolsep{0pt}%
    \put(0,0){\includegraphics[width=\unitlength,page=1]{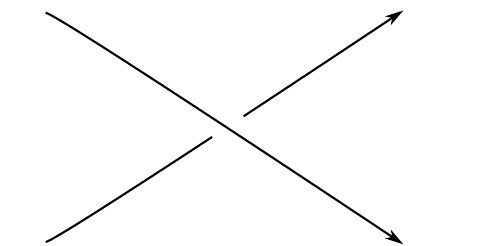}}%
    \put(-0.00248481,0.48041241){\makebox(0,0)[lt]{\lineheight{1.25}\smash{\begin{tabular}[t]{l}$\chi_1$\end{tabular}}}}%
    \put(-0.00285672,0.00811644){\makebox(0,0)[lt]{\lineheight{1.25}\smash{\begin{tabular}[t]{l}$\chi_2$\end{tabular}}}}%
    \put(0.84847629,0.48074305){\makebox(0,0)[lt]{\lineheight{1.25}\smash{\begin{tabular}[t]{l}$\chi_2'$\end{tabular}}}}%
    \put(0.84810438,0.00844703){\makebox(0,0)[lt]{\lineheight{1.25}\smash{\begin{tabular}[t]{l}$\chi_1'$\end{tabular}}}}%
  \end{picture}%
\endgroup%

    \caption{Labels for shape parameters at a positive crossing.}
    \label{fig:shaped-positive-crossing}
  \end{marginfigure}
  \begin{equation}
    \label{eq:positive-a-relations}
    \begin{aligned}
      a_1
    &=
    \frac{b_2 - \lambda_1 b_1 }{b_{2'} - b_1}
    &
    a_{1'}
    &=
    \frac{\lambda_2 b_2 - \lambda_1 b_{1'}}{\lambda_2 b_{2'} - b_{1'}}
    \\
    a_2
    &=
    \frac{1/\lambda_1 b_{1'} - 1/\lambda_2 b_2}{1/\lambda_1 b_1 - 1/b_2}
    &
    a_{2'}
    &=
    \frac{1/b_1 - 1/b_{2'}}{1/b_{1'} - 1/\lambda_2b_{2'}}
    \end{aligned}
  \end{equation}
  Similarly at a negative crossing (as in \cref{fig:shaped-negative-crossing}) we set
  \begin{marginfigure}
\begingroup%
  \makeatletter%
  \providecommand\color[2][]{%
    \errmessage{(Inkscape) Color is used for the text in Inkscape, but the package 'color.sty' is not loaded}%
    \renewcommand\color[2][]{}%
  }%
  \providecommand\transparent[1]{%
    \errmessage{(Inkscape) Transparency is used (non-zero) for the text in Inkscape, but the package 'transparent.sty' is not loaded}%
    \renewcommand\transparent[1]{}%
  }%
  \providecommand\rotatebox[2]{#2}%
  \newcommand*\fsize{\dimexpr\f@size pt\relax}%
  \newcommand*\lineheight[1]{\fontsize{\fsize}{#1\fsize}\selectfont}%
  \ifx\svgwidth\undefined%
    \setlength{\unitlength}{139.47438812bp}%
    \ifx\svgscale\undefined%
      \relax%
    \else%
      \setlength{\unitlength}{\unitlength * \real{\svgscale}}%
    \fi%
  \else%
    \setlength{\unitlength}{\svgwidth}%
  \fi%
  \global\let\svgwidth\undefined%
  \global\let\svgscale\undefined%
  \makeatother%
  \begin{picture}(1,0.50689042)%
    \lineheight{1}%
    \setlength\tabcolsep{0pt}%
    \put(0,0){\includegraphics[width=\unitlength,page=1]{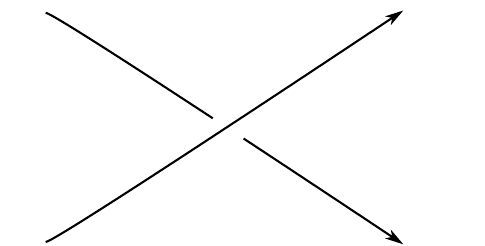}}%
    \put(-0.00248481,0.48041241){\makebox(0,0)[lt]{\lineheight{1.25}\smash{\begin{tabular}[t]{l}$\chi_1$\end{tabular}}}}%
    \put(-0.00285672,0.00811644){\makebox(0,0)[lt]{\lineheight{1.25}\smash{\begin{tabular}[t]{l}$\chi_2$\end{tabular}}}}%
    \put(0.84847629,0.48074305){\makebox(0,0)[lt]{\lineheight{1.25}\smash{\begin{tabular}[t]{l}$\chi_2'$\end{tabular}}}}%
    \put(0.84810438,0.00844703){\makebox(0,0)[lt]{\lineheight{1.25}\smash{\begin{tabular}[t]{l}$\chi_1'$\end{tabular}}}}%
  \end{picture}%
\endgroup%

    \caption{Labels for shape parameters at a negative crossing.}
    \label{fig:shaped-negative-crossing}
  \end{marginfigure}
  \begin{equation}
    \label{eq:negative-a-relations}
    \begin{aligned}
      a_1
    &=
    \frac{1/b_2 - 1/\lambda_1 b_1}{1/b_{2'} -1/b_{1'}}
    &
    a_{1'}
    &=
    \frac{1/\lambda_2 b_2 - 1/\lambda_1 b_{1'}}{1/\lambda_2 b_2 - 1/b_{1'}}
    \\
    a_2
    &=
    \frac{\lambda_1 b_{1'} - \lambda_2 b_2}{\lambda_1 b_1 - b_2}
    &
    a_{2'}
    &=
    \frac{b_{1'} - \lambda_2b_{2'}}{b_1 - b_{2'}}
    \end{aligned}
  \end{equation}
  It is possible that the right-hand side of one of these is of the form $0/0$, so that the corresponding $a$-variable is indeterminate.
  (We consider $0$ and $\infty$ as determinate but disallowed values.)
  Such a crossing is said to be \defemph{pinched}, and we must additionally specify the values of the $a$-variables at a pinched crossing.
  They are discussed in more detail in \cref{sec:pinched-crossings}.
\end{defn}

\begin{defn}
  In a pre-shaped diagram each segment is assigned two $a$-variables by the crossings at either end.
  The \defemph{gluing equations} of a pre-shaped diagram assert that the two $a$-variables of each segment agree.
  A \defemph{shaped} diagram is a pre-shaped diagram whose $a$-variables are not $0$ or $\infty$ and satisfy the gluing equations.
  We call the variables $\chi = (a, b, \lambda)$ assigned to the segments of a shaped diagram \defemph{shapes}.%
  \note{Previously we said a shape was a central character of the Weyl algebra.
  We will explain the connection in \cref{ch:algebras}.}
\end{defn}

The definition of the $a$-variables is somewhat complicated and as yet unmotivated.
We will see in \cref{sec:octahedral-decompositions} that the gluing equations for the $a$-variables are exactly the gluing equations associated to the octahedral decomposition of the tangle complement \cite{Kim2018}.
In \cref{ch:algebras} we will derive the relations of \cref{eq:positive-a-relations,eq:negative-a-relations} from the action of the outer $R$-matrix $\Rmat$ on the center of the Weyl algebra.

As promised before, every shaped tangle is associated to a representation of its fundamental groupoid.
\begin{defn}
  \label{def:diagram-holonomy-rep}
  Let $\chi = (a, b, \lambda)$ be a shape.
  The \defemph{upper holonomy} and \defemph{lower holonomy} of the shape $\chi$ are the matrices
  \begin{align*}
    g^+(\chi)
    &\defeq
    \begin{pmatrix}
      a & 0 \\
      (a - 1/\lambda)/b & 1
    \end{pmatrix}
    &
    g^-(\chi)
    &\defeq
    \begin{pmatrix}
      1 &  (a - \lambda) b\\
      0 & a
    \end{pmatrix}
  \end{align*}
  The \defemph{holonomy representation} of the fundamental groupoid of a shaped tangle diagram $D$ is the representation $\Pi(D) \to \glg$ which assigns the upper path across a segment with shape $\chi$ to $g^+(\chi)$ and the lower path to $g^-(\chi)$.
  We similarly refer to the induced representation $\pi(D) \to \slg \subset \glg$ as the holonomy representation.
\end{defn}

Strictly speaking the holonomy representation of $\Pi(D)$ is a map to $\glg$, not $\slg$.
Dividing by $\sqrt a$ would fix this, but introduce the problem of keeping track of the choice of square root.
In any case the restriction to $\pi(D)$ lands in $\slg$.

\subsection{Pinched crossings}
\label{sec:pinched-crossings}
In the equations (\ref{eq:positive-a-relations}--\ref{eq:negative-a-relations}) the $a$-variables are given by rational functions, which can have singularities.
\begin{prop}
  \label{thm:pinched-crossing}
  A crossing of a pre-shaped tangle diagram is \defemph{pinched}%
  \note{
    We borrow this term from \cite{Kim2018}.
    At a pinched crossing the locations of the ideal vertices of the ideal octahedron are ``pinched'' together: the points $P_1$ and $P_2$ lie at $1$ for every tetrahedron.
  }
  if any of the following conditions hold:
  \begin{enumerate}
    \item $b_{2'} = b_1$
      \item $b_2 = \lambda_1 b_1$
      \item $\lambda_1 b_{1'} = \lambda_2 b_{2'}$
      \item  $b_{1'} = \lambda_2 b_{2'}$
  \end{enumerate}
  At any pinched crossing in a shaped diagram all four conditions must hold.
\end{prop}
\begin{proof}
  By checking \cref{eq:positive-a-relations,eq:negative-a-relations} we can see that if one equation holds, but not all of them, at least one $a$-variable will be $0$ or $\infty$.
  Since a shaped diagram has $a$-variables in $\CC\setminus\{0\}$, all four conditions must be true if any is.
\end{proof}

At a pinched crossing, we cannot use our explicit formula \cref{thm:braiding-factorization} for the coefficients of the braiding.
We discuss this further in \cref{sec:pinched-crossing-braiding-overview}.
It is still useful to allow pinched crossings because they correspond to the case where the holonomy representation is abelian: for example, by assigning every strand of a diagram the shape $\chi = (\lambda, b, \lambda^{-1})$, we get a representation in which every meridian is sent to
\[
  \begin{bmatrix}
    \lambda & -(\lambda - \lambda^{-1})b \\
    0 & \lambda^{-1}
  \end{bmatrix}
\]
which in turn corresponds to the ADO invariant.
When $\lambda = (\xi^{\nr-1})^{\nr} = (-1)^{\nr + 1}$ we instead obtain Kashaev's quantum dilogarithm invariant, equivalently \cite{Murakami2001} the colored Jones polynomial evaluated at $\xi^2$.

\section{Hyperbolic links and ideal tetrahedra}
\label{sec:hyperbolic-links}
We give a very brief introduction to hyperbolic knot theory.
For a more comprehensive (and probably more comprehensible) description, see \autocite{Purcell2020}.
The fundamental references for the hyperbolic geometry of $3$-manifolds are \citeauthor{Thurston1980}'s notes \cite{Thurston1980}.

\subsection{Hyperbolic geometry for link complements}
\begin{defn}
  \defemph{Hyperbolic $3$-space} is%
  \note{More specifically, this is the upper half-plane model of hyperbolic space.}
  \[
    \HH^3 = \{ (x + iy, t) \in \CC \times \RR | t > 0\}
  \]
  with metric of constant negative curvature
  \[
    ds^2 = \frac{dx^2 + dy^2 + dt^2}{t^2}.
  \]
  Its boundary is $\partial \HH^3 = \hat \CC = \CC \cup \{\infty\}$, with Euclidean metric.

  A \defemph{hyperbolic structure} on a $3$-manifold $M$ is%
  \note{To use more general language, a hyperbolic structure is a $(\HH^3, \isom(\HH^3))$-structure.}
  a cover by charts $\phi_{\alpha} : U_{\alpha} \to \HH^3$ whose transition maps are isometries of $\HH^3$.
  These charts induce a metric on $M$, and we say that the hyperbolic structure is \defemph{complete} if the induced metric is.
\end{defn}

Informally a hyperbolic $3$-manifold $M$ is one that is locally isometric to $\HH^3$, so that gluing together the pieces gives a smooth, complete manifold $M$.
We refer to \autocite[Chapter 3]{Purcell2020} for details.

For a more combinatorial perspective, we want to make sense of this in terms of a triangulation of $M$.
The triangulation gives a topological decription of $M$, and we extend this to a geometric description by putting each tetrahedron in $\HH^3$ and gluing together their faces using isometries.
If the isometries are compatible we get a smooth, complete hyperbolic structure on $M$.

We are interested in link complements, which are hyperbolic in a slightly different way.
If $L$ is a link in $S^3$, the complement $S^3 \setminus L$ has ``open'' boundaries at each component of $L$, which we call \defemph{cusps}.
More formally, we say a manifold $M$ is cusped if it is homeomorphic to the interior of a compact $3$-manifold with tori as boundaries; each boundary torus is called a cusp.
The corresponding triangulation description also has open parts.

\begin{defn}
  A link $L$ in $S^3$ is \defemph{hyperbolic} if it the complement $S^3 \setminus L$ can be given a complete hyperbolic metric of finite volume.%
  \note{
    We do not define the hyperbolic volume here: see \cite[Chapter 9]{Purcell2020} for details.
    We can think of the finite-volume condition as a kind of geometric compactness.
  }
\end{defn}

\begin{defn}
  An \defemph{ideal tetrahedron} is a tetrahedron with its $0$-skeleton (the \defemph{ideal vertices}) removted.
  An \defemph{ideal triangulation} of a link $L$ in $S^3$ is a triangulation of $S^3 \setminus L$ by ideal tetrahedra whose ideal vertices lie on the cusps.
\end{defn}

In an ideal triangulation the missing ideal vertices all lie on the link $L$.
It is not immediately clear%
\note{It certainly was not immediately clear to the author that ideal triangulations existed at all when he first learned about them!}
why ideal triangulations should exist, or how to find them.
We refer to \autocite[Chapter 1]{Purcell2020} for a comprehensive example, including constructions of relevant physical models.
In the next section we will describe a systematic way to build an ideal triangulation of a link complement starting from a link diagram.%
\note{The construction places a (twisted) ideal octahedron at each crossing of the diagram, so it is called an octahedral decomposition.}

\subsection{Describing ideal triangulations}
For now, suppose we have an ideal triangulation of a link complement.
To put a metric on $S^3 \setminus L$, we label the vertices of each ideal tetrahedron with points of $\CC \cup \{0\} = \hat \CC = \partial \HH^3$.
To keep track of orientations, we also want to pick an (arbitrary) ordering of the vertices.
\begin{defn}
  A \defemph{shaped} ideal tetrahedron is one whose vertices are totally ordered and labeled by points of $\hat \CC = \CC \cup \{\infty\}$.
  A \defemph{shaped ideal triangulation} is an ideal triangulation by shaped tetrahedra.
\end{defn}

When we glue two shaped tetrahedra together along a face, they will in general disagree about where the common vertices are located.
This disagreement is resolved by isometries called face maps.%

\begin{thm}
  The group of orientation-preserving isometries of $\HH^3$ is $\pslg$, which has a \emph{left} action on the boundary $\partial \HH^3 = \hat \CC$ via fractional linear transformations:
  \[
    z \cdot
    \begin{pmatrix}
      a & b \\
      c & d
    \end{pmatrix}
    =
    \frac{az + c}{bz + d}.
  \]
  In particular, an element of $\isom(\HH^3)$ is uniquely determined by its action on three points.
\end{thm}

\begin{remark}
  In order to match our topologically natural convention that the path $x$ followed by the path $y$ is $xy$, we need the image of the holonomy representation to act the left.%
  \note{We only need fractional linear transformations in this and the next section, while the compositions of paths shows up again in \cref{ch:functors,ch:torions}.}
  This forces a somewhat unfortunate convention on fractional linear transformations, as shown above.
\end{remark}
\note{We can think of the face maps as generating the transition maps of the hyperbolic structure.}

\begin{defn}
  Let $\mathcal{T}$ be a shaped ideal triangulation of the complement of $L$.
  If $F$ and $F'$ are two faces with vertices $p_1, p_2, p_3$ and $q_1, q_2, q_3$ which are identified in the gluing of $\mathcal{T}$, the \defemph{face map} between them is the element of $\isom(\HH^3) = \pslg$ sending $p_1, p_2, p_3$ to $q_1, q_2, q_3$.
  The \defemph{holonomy representation} of $\mathcal{T}$ is the representation $\pi(L) \to \pslg$ induced by the face maps.
\end{defn}
We see that a hyperbolic structure on the complement of $L$ corresponds to a representation $\pi(L) \to \pslg$.
The converse is true as well: given a representation $\pi(L) \to \pslg$ we get a hyperbolic structure on $S^3 \setminus L$.

\begin{defn}
  For four points $p_0, p_1, p_2, p_3 \in \hat \CC$, the \defemph{cross-ratio} is
  \[
    \crossratio{p_0}{p_1}{p_2}{p_3} = \frac{(p_0 - p_3)(p_1 - p_2)}{(p_0 - p_2)(p_1 - p_3)}.
  \]
\end{defn}

\begin{prop}
  \begin{enumerate}
    \item The cross-ratio $\crossratio{p_0}{p_1}{p_2}{p_3}$ is equal to $0$, $1$, or $\infty$ if and only if some of the points $p_i$ coincide,
    \item the cross-ratio is preserved by $\pslg$, and
    \item if  $q_0, \dots, q_3 \in \hat \CC$ are another set of four points, there is an element $f \in \pslg$ with $f(p_i) = q_i$ if and only if
  \[
    \crossratio{p_0}{p_1}{p_2}{p_3}
    =
    \crossratio{q_0}{q_1}{q_2}{q_3}.
  \]
  \end{enumerate}
In particular, the cross-ratio completely describes the oriented congruence class of an ideal tetrahedron.
\end{prop}

We can see this more geometrically by relating the cross-ratios to dihedral angles.
\begin{marginfigure}
  \centering
\begingroup%
  \makeatletter%
  \providecommand\color[2][]{%
    \errmessage{(Inkscape) Color is used for the text in Inkscape, but the package 'color.sty' is not loaded}%
    \renewcommand\color[2][]{}%
  }%
  \providecommand\transparent[1]{%
    \errmessage{(Inkscape) Transparency is used (non-zero) for the text in Inkscape, but the package 'transparent.sty' is not loaded}%
    \renewcommand\transparent[1]{}%
  }%
  \providecommand\rotatebox[2]{#2}%
  \newcommand*\fsize{\dimexpr\f@size pt\relax}%
  \newcommand*\lineheight[1]{\fontsize{\fsize}{#1\fsize}\selectfont}%
  \ifx\svgwidth\undefined%
    \setlength{\unitlength}{144bp}%
    \ifx\svgscale\undefined%
      \relax%
    \else%
      \setlength{\unitlength}{\unitlength * \real{\svgscale}}%
    \fi%
  \else%
    \setlength{\unitlength}{\svgwidth}%
  \fi%
  \global\let\svgwidth\undefined%
  \global\let\svgscale\undefined%
  \makeatother%
  \begin{picture}(1,1)%
    \lineheight{1}%
    \setlength\tabcolsep{0pt}%
    \put(0,0){\includegraphics[width=\unitlength,page=1]{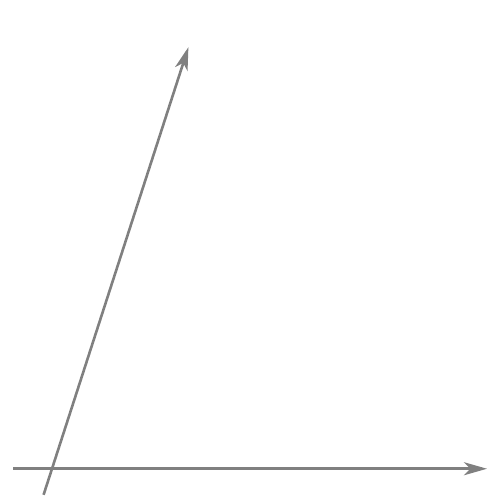}}%
    \put(0.03160913,0.08893871){\makebox(0,0)[lt]{\lineheight{1.25}\smash{\begin{tabular}[t]{l}$0$\end{tabular}}}}%
    \put(0.43968993,0.08821324){\makebox(0,0)[lt]{\lineheight{1.25}\smash{\begin{tabular}[t]{l}$1$\end{tabular}}}}%
    \put(0.56703962,0.50375857){\makebox(0,0)[lt]{\lineheight{1.25}\smash{\begin{tabular}[t]{l}$z$\end{tabular}}}}%
    \put(0,0){\includegraphics[width=\unitlength,page=2]{dihedral-angle.pdf}}%
    \put(0.57128527,0.92448159){\makebox(0,0)[lt]{\lineheight{1.25}\smash{\begin{tabular}[t]{l}$\infty$\end{tabular}}}}%
    \put(0,0){\includegraphics[width=\unitlength,page=3]{dihedral-angle.pdf}}%
  \end{picture}%
\endgroup%

  \caption{An ideal tetrahedron with vertices $0,\infty,1,z$. The edges are circles intersecting the boundary sphere $\hat \CC$ at right angles.}
  \label{fig:dihedral-angle}
\end{marginfigure}

\begin{defn}
  Let $T$ be a shaped tetrahedron with $p_0, p_1, p_2, p_3$.
  The \defemph{edge invariant} or \defemph{dihedral angle} of the edge from $p_0$ to $p_1$ is the cross-ratio
  \[
    \crossratio{p_0}{p_1}{p_2}{p_3} = \frac{(p_0 - p_3)(p_1 - p_2)}{(p_0 - p_2)(p_1 - p_3)}.
  \]
  We say that the tetrahedron $T$ is \defemph{flat} when the cross-ratio is real and \defemph{degenerate} when it is $0$, $1$, or $\infty$.
\end{defn}

To understand this definition, consider the special case where the vertices of $T$ are $0$, $\infty$, $1$, and $z$, as shown in \cref{fig:dihedral-angle}.
When standing at $0$, multiplication by  $z$ moves from the point $1$ to the point $z$, so we say the edge $0\infty$ has (complex) dihedral angle $z$.

\begin{prop}
  Let $T$ be a shaped ideal tetrahedron with vertices at $p_0, p_1, p_2, p_3$, and set $z = \crossratio{p_0}{p_1}{p_2}{p_3}$.
  Then the edge invariants of $T$ are given by \cref{fig:edge-invariants}, where 
  \[
    z^{\circ} \defeq \frac{1}{1-z} \text{ and }
    z^{\circ \circ} \defeq (z ^{\circ})^{\circ} = 1 - \frac{1}{z}.
  \]
\end{prop}

\begin{marginfigure}
  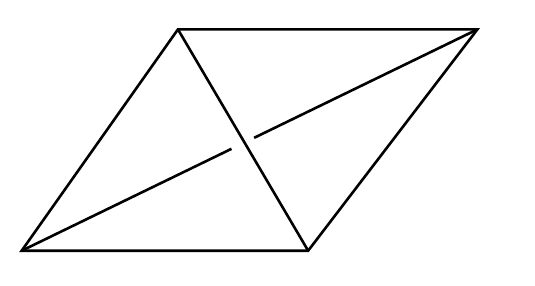
  \caption{Dihedral angles of an ideal tetrahedron.}
  \label{fig:edge-invariants}
\end{marginfigure}

Notice that $(z^{\circ \circ})^{\circ} = z$, so the $\circ$ symbols work modulo $3$.
Using a prime $'$ is the standard notation in this context (such as \cite[Section 1]{Zickert2009}) but conflicts with our use of $'$ for tangle diagram parameters.

\begin{proof}
  The invariant of the edge from $p_0$ to $p_1$ holds by definition.
  We can derive the remaining invariants by acting by fractional linear transformations: recall that given any four points $p_0, \dots, p_3 \in \hat \CC$ there is a unique element of $\pslg$ sending them to $0, \infty, 1, z$.
  Such an argument is used in the proof of \cite[Lemma 4.6]{Purcell2020}.
\end{proof}

\subsection{The gluing and completeness equations}
\label{sec:gluing-eqs}

Finally, we turn to the question of how to determine whether a shaped ideal triangulation corresponds to a smooth, complete hyperbolic structure.
This is determined by a set of algebraic equations called the gluing and completeness relations.

\begin{defn}
  Let $\mathcal T$ be a shaped triangulation of $S^3 \setminus L$.
  For an edge $e$, let $T(e)$ be the set of tetrahedra glued to $e$.
  The \defemph{gluing equation} for $e$ says that
  \[
    \prod_{\tau \in T(e)}  \tau(e)^{\pm \epsilon(\tau,e)} = 1
  \]
  where $\tau(e)$ is the shape parameter associated to $e$ and $\epsilon(\tau, e) \in \{\pm 1\}$ accounts for the orientation of $e$.
  The gluing equations for $T$ are the set of gluing equations for each edge of $T$.
\end{defn}

Roughly speaking the gluing equations for $T$ say that the angles inside of $S^3 \setminus L$ make sense when circling around an edge.
However, we need a second set of equations to account for the behavior when circling around a cusp.
It is possible to express these in terms of the dihedral angles, but more straightforward to use the holonomy maps.%
\note{Since our preferred coordinate system gives us holonomy maps for free, we do not lose anything by taking this perspective.}

\begin{defn}
  \label{def:boundary-parabolic}
  An element $g \in \pslg$ is \defemph{parabolic} if it acts by a pure translation on $\hat \CC$, equivalently if it is conjugate to the pure translation
  \[
    \begin{pmatrix}
      1 & 0 \\
      1 & 1
    \end{pmatrix} : 
    z \mapsto z + 1.
  \]
  We say that the triangulation $T$ is \defemph{complete} if the holonomy around each cusp is parabolic.
  By \cite[Theorem 4.10]{Purcell2020} this is equivalent to requiring that the induced metric on the manifold obtained by gluing $T$ is complete.%
  \note{If the holonomy around a cusp acted by scaling, then we would spiral in or our of the cusp as we circled aroudn it.
  To get a complete metric, this cannot happen, so the holonomy must be a pure translation.}
  
  More generally, we can consider incomplete metrics, sometimes called \defemph{pseudo-hyperbolic} structures.
  We say that the triangulation $T$ is \defemph{$\lambda$-deformed} if the holonomy map around a cusp is conjugate to the map%
  \note{In the language of \cite{Kim2018} we would call this a $\lambda^2$-deformed solution.}
  \[
    \begin{pmatrix}
      \lambda & 0 \\
      0 & \lambda^{-1}
    \end{pmatrix}:
    z \mapsto \lambda^2 z.
  \]
More generally, we could consider the case where each cusp is $\lambda_i$-deformed for a different value of $\lambda_i$.
\end{defn}

In a shaped tangle diagram (say of a link for simplicity), cusps are in one-to-one correspondence with connected components.
We will see shortly that a diagram component with eigenvalue $\lambda$ corresponds to a $\lambda$-deformed octahedral decomposition.

\begin{defn}
  We say that a shaped triangulation $T$ is \defemph{consistent} if it satisfies the gluing equations.
  If additionally it is complete and does not contain any flat tetrahedra (equivalently, if none of the dihedral angles are real or $\infty$) we say that it is \defemph{geometric}.
\end{defn}

We emphasize that we do not need geometric tetrahedra to define our link invariants.
However, they are still an important special case.
For example, to find the distinguished complete, finite-volume hyperbolic structure on a link complement it suffices to compute a geometric triangulation with parabolic holonomy around the cusps.
For shape diagrams this means that any shaped diagram of a link with eigenvalues $\pm 1$ and all dihedral angles (as given in \cref{table:positive-crossing-data,table:negative-crossing-data}) non-real gives the distinguished hyperbolic structure.

While this is a sufficient condition, it is not necessary, and exactly when a geometric choice of shapes exists for a given diagram is a subtle question.
For example, any diagram with a kink (as in \cref{fig:blackboard-framing}) will always have a pinched crossing where the dihedral angles lie in $\{0, 1, \infty\}$, and for links it is possible to draw diagrams \cite[Proposition 4.7]{Kim2018} with finite-volume parabolic holonomy that contain pinched crossings.

One way to avoid some of these problems is to consider braids instead of tangles.
\citeauthor{Cho2020} \cite{Cho2020} discus the existence of (in our language) non-pinched solutions to the gluing equations.
It would be interesting to understand their methods in connection with the shape parameters, since they use a different (but closely related) system of coordinates.

\section{Octahedral decompositions}
\label{sec:octahedral-decompositions}
We have called the decorations $\chi = (a, b, \lambda)$ of our diagrams ``shapes'' but have not yet justified this term.
In this section, we explain how to interpret these parameters in terms of the complex dihedral angles (shape parameters) of the octahedral decomposition of a link complement.
We refer extensively to the detailed description of this triangulation by \citeauthor{Kim2018} \autocite{Kim2018}.

\subsection{Building the octahedra}
The octahedral decomposition of a link complement puts an ideal octahedron at each crossing of a diagram of the link.
To motivate some features of this description and make it easier to see the connection with our description of the holonomy of shaped tangle diagrams, we explain a perspective related to ideal triangulations of discs.\note{This has something to do with cluster algebras, and later in Chapter \ref{ch:algebras} we will present quantum $\lie{sl}_2$ in terms of a Weyl algebra, that is a quantum cluster algebra. We hope to make these connections more precise in the future.}

Every link $L$ can be represented as the closure of a braid $\beta$.
If we view $\beta$ as an element of the mapping class group of the $n$-punctured disc $D_n$, then the complement $S^3 \setminus L$ of $L$ is the mapping torus\note{If $f : \Sigma \to \Sigma$ is a homeomorphism, them mapping torus of $f$ is the space $\Sigma \times [0,1]$ modulo the relation $(x,0) \sim (f(x),1)$.} of $\beta$.
If we ideally triangulate $D_n$ and interpret the action of $\beta$ in terms of this triangulation, we can get an ideal triangulation of the mapping torus of $\beta$, that is of $S^3 \setminus L$.

\begin{figure}
  \centering
  \subcaptionbox{The initial triangulation.\label{fig:triangulated-braiding-0}}{ \def\svgwidth{2.5in} 
\begingroup%
  \makeatletter%
  \providecommand\color[2][]{%
    \errmessage{(Inkscape) Color is used for the text in Inkscape, but the package 'color.sty' is not loaded}%
    \renewcommand\color[2][]{}%
  }%
  \providecommand\transparent[1]{%
    \errmessage{(Inkscape) Transparency is used (non-zero) for the text in Inkscape, but the package 'transparent.sty' is not loaded}%
    \renewcommand\transparent[1]{}%
  }%
  \providecommand\rotatebox[2]{#2}%
  \newcommand*\fsize{\dimexpr\f@size pt\relax}%
  \newcommand*\lineheight[1]{\fontsize{\fsize}{#1\fsize}\selectfont}%
  \ifx\svgwidth\undefined%
    \setlength{\unitlength}{215.81824493bp}%
    \ifx\svgscale\undefined%
      \relax%
    \else%
      \setlength{\unitlength}{\unitlength * \real{\svgscale}}%
    \fi%
  \else%
    \setlength{\unitlength}{\svgwidth}%
  \fi%
  \global\let\svgwidth\undefined%
  \global\let\svgscale\undefined%
  \makeatother%
  \begin{picture}(1,1.03093509)%
    \lineheight{1}%
    \setlength\tabcolsep{0pt}%
    \put(0.366521,0.92532356){\color[rgb]{0,0,0}\makebox(0,0)[lt]{\lineheight{1.25}\smash{\begin{tabular}[t]{l}$P_+$\end{tabular}}}}%
    \put(0.38424276,0.03278304){\color[rgb]{0,0,0}\makebox(0,0)[lt]{\lineheight{1.25}\smash{\begin{tabular}[t]{l}$P_-$\end{tabular}}}}%
    \put(0.1486229,0.51022397){\color[rgb]{0,0,0}\makebox(0,0)[lt]{\lineheight{1.25}\smash{\begin{tabular}[t]{l}$P_2$\end{tabular}}}}%
    \put(0.59268972,0.51326773){\color[rgb]{0,0,0}\makebox(0,0)[lt]{\lineheight{1.25}\smash{\begin{tabular}[t]{l}$P_1$\end{tabular}}}}%
    \put(0,0){\includegraphics[width=\unitlength,page=1]{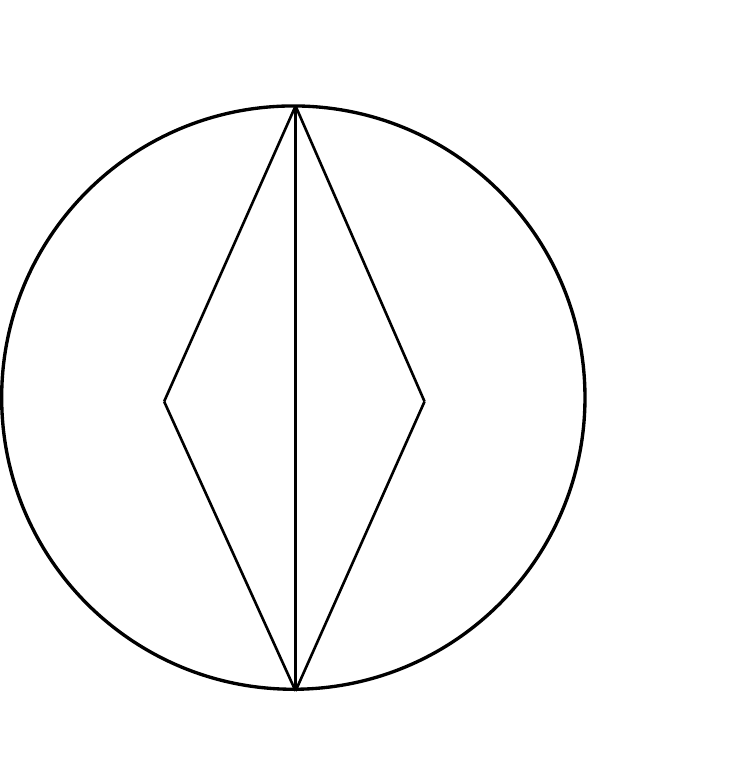}}%
  \end{picture}%
\endgroup%
 }%
  \hfill
  \subcaptionbox{Building a tetrahedron on top of a quadrilateral.\label{fig:triangulated-braiding-1}}{ \def\svgwidth{2.5in} 
\begingroup%
  \makeatletter%
  \providecommand\color[2][]{%
    \errmessage{(Inkscape) Color is used for the text in Inkscape, but the package 'color.sty' is not loaded}%
    \renewcommand\color[2][]{}%
  }%
  \providecommand\transparent[1]{%
    \errmessage{(Inkscape) Transparency is used (non-zero) for the text in Inkscape, but the package 'transparent.sty' is not loaded}%
    \renewcommand\transparent[1]{}%
  }%
  \providecommand\rotatebox[2]{#2}%
  \newcommand*\fsize{\dimexpr\f@size pt\relax}%
  \newcommand*\lineheight[1]{\fontsize{\fsize}{#1\fsize}\selectfont}%
  \ifx\svgwidth\undefined%
    \setlength{\unitlength}{215.81824493bp}%
    \ifx\svgscale\undefined%
      \relax%
    \else%
      \setlength{\unitlength}{\unitlength * \real{\svgscale}}%
    \fi%
  \else%
    \setlength{\unitlength}{\svgwidth}%
  \fi%
  \global\let\svgwidth\undefined%
  \global\let\svgscale\undefined%
  \makeatother%
  \begin{picture}(1,1.03093509)%
    \lineheight{1}%
    \setlength\tabcolsep{0pt}%
    \put(0.366521,0.92532356){\color[rgb]{0,0,0}\makebox(0,0)[lt]{\lineheight{1.25}\smash{\begin{tabular}[t]{l}$P_+$\end{tabular}}}}%
    \put(0.38424276,0.03278304){\color[rgb]{0,0,0}\makebox(0,0)[lt]{\lineheight{1.25}\smash{\begin{tabular}[t]{l}$P_-$\end{tabular}}}}%
    \put(0.1486229,0.51022397){\color[rgb]{0,0,0}\makebox(0,0)[lt]{\lineheight{1.25}\smash{\begin{tabular}[t]{l}$P_2$\end{tabular}}}}%
    \put(0.59268972,0.51326773){\color[rgb]{0,0,0}\makebox(0,0)[lt]{\lineheight{1.25}\smash{\begin{tabular}[t]{l}$P_1$\end{tabular}}}}%
    \put(0,0){\includegraphics[width=\unitlength,page=1]{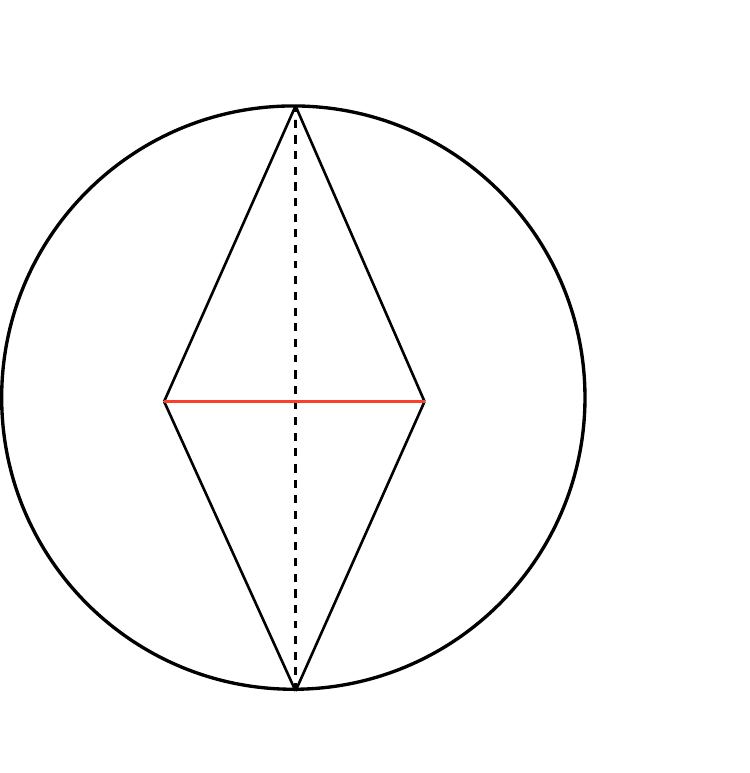}}%
  \end{picture}%
\endgroup%
 }%
  \hfill
  \subcaptionbox{Adding two more tetrahedra.\label{fig:triangulated-braiding-2}}{ \def\svgwidth{2.5in} 
\begingroup%
  \makeatletter%
  \providecommand\color[2][]{%
    \errmessage{(Inkscape) Color is used for the text in Inkscape, but the package 'color.sty' is not loaded}%
    \renewcommand\color[2][]{}%
  }%
  \providecommand\transparent[1]{%
    \errmessage{(Inkscape) Transparency is used (non-zero) for the text in Inkscape, but the package 'transparent.sty' is not loaded}%
    \renewcommand\transparent[1]{}%
  }%
  \providecommand\rotatebox[2]{#2}%
  \newcommand*\fsize{\dimexpr\f@size pt\relax}%
  \newcommand*\lineheight[1]{\fontsize{\fsize}{#1\fsize}\selectfont}%
  \ifx\svgwidth\undefined%
    \setlength{\unitlength}{215.81824493bp}%
    \ifx\svgscale\undefined%
      \relax%
    \else%
      \setlength{\unitlength}{\unitlength * \real{\svgscale}}%
    \fi%
  \else%
    \setlength{\unitlength}{\svgwidth}%
  \fi%
  \global\let\svgwidth\undefined%
  \global\let\svgscale\undefined%
  \makeatother%
  \begin{picture}(1,1.03093509)%
    \lineheight{1}%
    \setlength\tabcolsep{0pt}%
    \put(0.366521,0.92532356){\color[rgb]{0,0,0}\makebox(0,0)[lt]{\lineheight{1.25}\smash{\begin{tabular}[t]{l}$P_+$\end{tabular}}}}%
    \put(0.38424276,0.03278304){\color[rgb]{0,0,0}\makebox(0,0)[lt]{\lineheight{1.25}\smash{\begin{tabular}[t]{l}$P_-$\end{tabular}}}}%
    \put(0.1486229,0.51022397){\color[rgb]{0,0,0}\makebox(0,0)[lt]{\lineheight{1.25}\smash{\begin{tabular}[t]{l}$P_2$\end{tabular}}}}%
    \put(0.59268972,0.51326773){\color[rgb]{0,0,0}\makebox(0,0)[lt]{\lineheight{1.25}\smash{\begin{tabular}[t]{l}$P_1$\end{tabular}}}}%
    \put(0,0){\includegraphics[width=\unitlength,page=1]{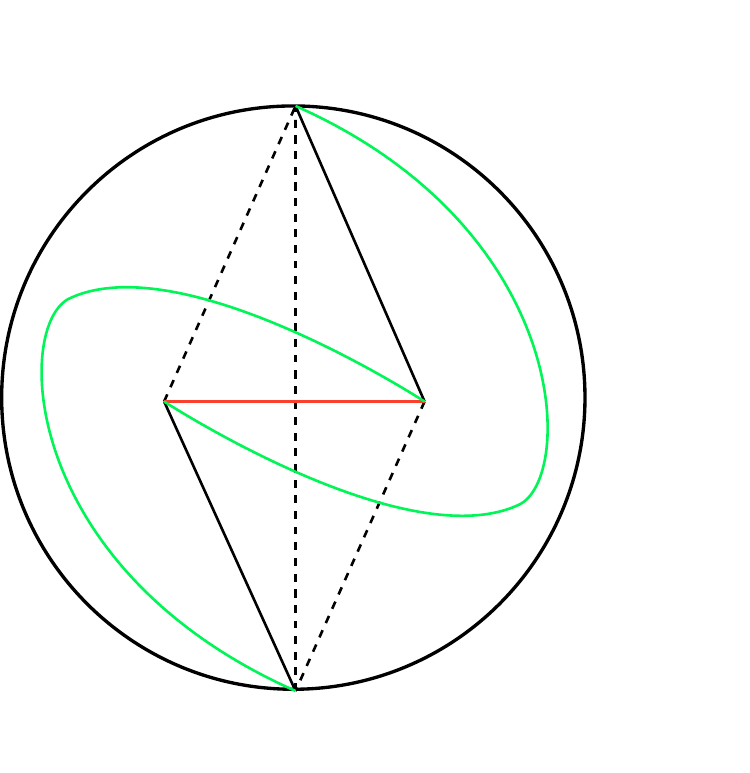}}%
  \end{picture}%
\endgroup%
 }%
  \hfill
  \subcaptionbox{The final result.\label{fig:triangulated-braiding-total}}{ \def\svgwidth{2.5in} 
\begingroup%
  \makeatletter%
  \providecommand\color[2][]{%
    \errmessage{(Inkscape) Color is used for the text in Inkscape, but the package 'color.sty' is not loaded}%
    \renewcommand\color[2][]{}%
  }%
  \providecommand\transparent[1]{%
    \errmessage{(Inkscape) Transparency is used (non-zero) for the text in Inkscape, but the package 'transparent.sty' is not loaded}%
    \renewcommand\transparent[1]{}%
  }%
  \providecommand\rotatebox[2]{#2}%
  \newcommand*\fsize{\dimexpr\f@size pt\relax}%
  \newcommand*\lineheight[1]{\fontsize{\fsize}{#1\fsize}\selectfont}%
  \ifx\svgwidth\undefined%
    \setlength{\unitlength}{215.81824493bp}%
    \ifx\svgscale\undefined%
      \relax%
    \else%
      \setlength{\unitlength}{\unitlength * \real{\svgscale}}%
    \fi%
  \else%
    \setlength{\unitlength}{\svgwidth}%
  \fi%
  \global\let\svgwidth\undefined%
  \global\let\svgscale\undefined%
  \makeatother%
  \begin{picture}(1,1.03093509)%
    \lineheight{1}%
    \setlength\tabcolsep{0pt}%
    \put(0.366521,0.92532356){\color[rgb]{0,0,0}\makebox(0,0)[lt]{\lineheight{1.25}\smash{\begin{tabular}[t]{l}$P_+$\end{tabular}}}}%
    \put(0.38424276,0.03278304){\color[rgb]{0,0,0}\makebox(0,0)[lt]{\lineheight{1.25}\smash{\begin{tabular}[t]{l}$P_-$\end{tabular}}}}%
    \put(0.1486229,0.51022397){\color[rgb]{0,0,0}\makebox(0,0)[lt]{\lineheight{1.25}\smash{\begin{tabular}[t]{l}$P_2$\end{tabular}}}}%
    \put(0.59268972,0.51326773){\color[rgb]{0,0,0}\makebox(0,0)[lt]{\lineheight{1.25}\smash{\begin{tabular}[t]{l}$P_1$\end{tabular}}}}%
    \put(0,0){\includegraphics[width=\unitlength,page=1]{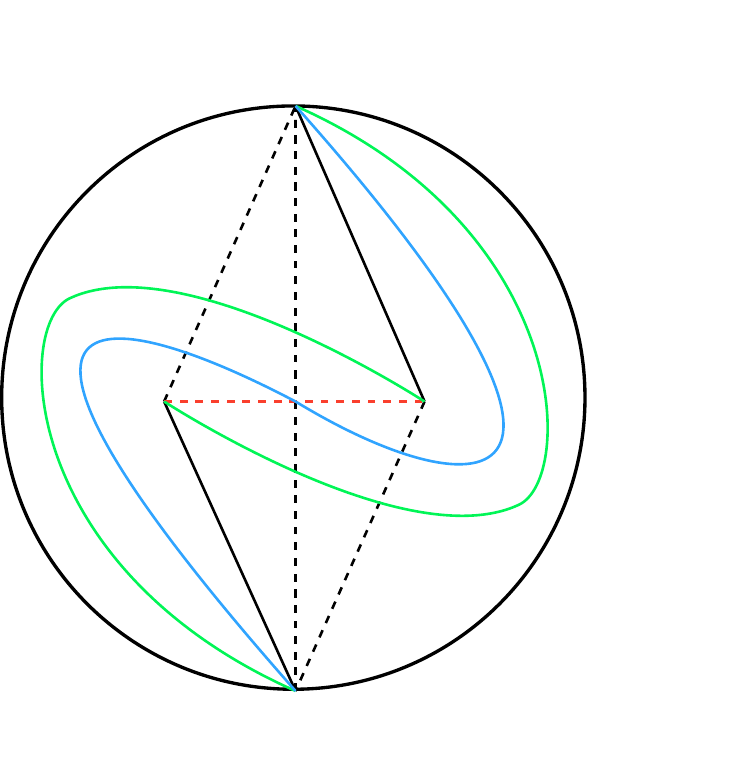}}%
  \end{picture}%
\endgroup%
 }%
  \caption{Building an ideal octahedron.}
  \label{fig:build-ideal-octahedron}
\end{figure}
We describe this process in \cref{fig:build-ideal-octahedron}.
We start with the triangulation in \cref{fig:triangulated-braiding-0}.
For simplicity we consider a single crossing at a time, so we only need to consider two punctures $P_1$ and $P_2$ (plus two auxiliary punctures $P_+$ at the top and $P_-$ at the bottom.)
We think of these punctures as corresponding to strands oriented out of the page.%
\note{It's straightforward to extend this picture to any number of interior punctures by gluing copies of \cref{fig:triangulated-braiding-0} along the vertical edges.}

We can modify ideal triangulations by flipping the diagonal of a quadrilateral.
From a $3$-dimensional perspective, we are attaching the final edge of a tetrahedron above its base.
In \cref{fig:triangulated-braiding-1} we add a {\color{myred} red} edge to build an ideal tetrahedron $P_2 P_- P_+ P_1$.
We then add two {\color{mygreen} green} edges, building two more tetrahedra.
Finally, we add the {\color{myblue} blue} edge to finish.

Ignoring the interior dashed edges, which are now below the tetrahedra we have added, we have a new, twisted copy of the triangulation \ref{fig:triangulated-braiding-0}.
By rotating $P_1$ above $P_2$, we pull the green edges taut and obtain our original picture, but with the points $P_1$ and $P_2$ swapped.
In the process, we have braided the point $P_1$ over the point $P_2$.
This corresponds to a positive braiding in our conventions, assuming that the strands are oriented out of the page in \cref{fig:build-ideal-octahedron}.

The resulting ideal polyhedron is composed of four ideal tetrahedra, so we call it an \defemph{ideal octahedron.}
We give a side view of the octahedron in \cref{fig:labeled-octahedron}.
However, the alert reader will notice that these pictures do not match!
The octahedron of \cref{fig:labeled-octahedron} has two extra points $P_-'$ and $P_+'$ and two extra edges.

To fix this, we need to identify the points $P_+$ and $P_+'$ by pulling them above the octahedron and identifying the edges $P_1 P_+$ and $P_1 P_+'$.
After doing the same thing below the octahedron with $P_-$ and $P_-'$, we obtain a \defemph{twisted ideal octahedron} \autocite[Figures 3--4]{Kim2018}.
We typically refer to both as ideal octahedra.

It is not hard to see \autocite[Section 3.1]{Kim2018} that by placing such an ideal octahedron at each crossing of a diagram of a link $L$ we obtain an triangulation of $S^3 \setminus (L \cup \{P_+, P_-\})$, where $P_{+}$ and $P_-$ are the extra ideal points above and below the octahedra.
In particular, this procedure gives us a combinatorial method for constructing ideal triangulations of link complements from a diagram of the link, although with many more tetrahedra than necessary.
For example, the octahedral decomposition of the figure-eight knot has $4\cdot 4 = 16$ tetrahedra, while the optimal triangulation\autocite[Chapter 1]{Purcell2020} of its complement has only $2$.

\subsection{Shaped ideal octahedra}
We can now describe the relationship between the parameters of a shaped tangle diagram and the octahedral decomposition.

\begin{marginfigure}
\begingroup%
  \makeatletter%
  \providecommand\color[2][]{%
    \errmessage{(Inkscape) Color is used for the text in Inkscape, but the package 'color.sty' is not loaded}%
    \renewcommand\color[2][]{}%
  }%
  \providecommand\transparent[1]{%
    \errmessage{(Inkscape) Transparency is used (non-zero) for the text in Inkscape, but the package 'transparent.sty' is not loaded}%
    \renewcommand\transparent[1]{}%
  }%
  \providecommand\rotatebox[2]{#2}%
  \newcommand*\fsize{\dimexpr\f@size pt\relax}%
  \newcommand*\lineheight[1]{\fontsize{\fsize}{#1\fsize}\selectfont}%
  \ifx\svgwidth\undefined%
    \setlength{\unitlength}{158.20389175bp}%
    \ifx\svgscale\undefined%
      \relax%
    \else%
      \setlength{\unitlength}{\unitlength * \real{\svgscale}}%
    \fi%
  \else%
    \setlength{\unitlength}{\svgwidth}%
  \fi%
  \global\let\svgwidth\undefined%
  \global\let\svgscale\undefined%
  \makeatother%
  \begin{picture}(1,0.97607433)%
    \lineheight{1}%
    \setlength\tabcolsep{0pt}%
    \put(0,0){\includegraphics[width=\unitlength,page=1]{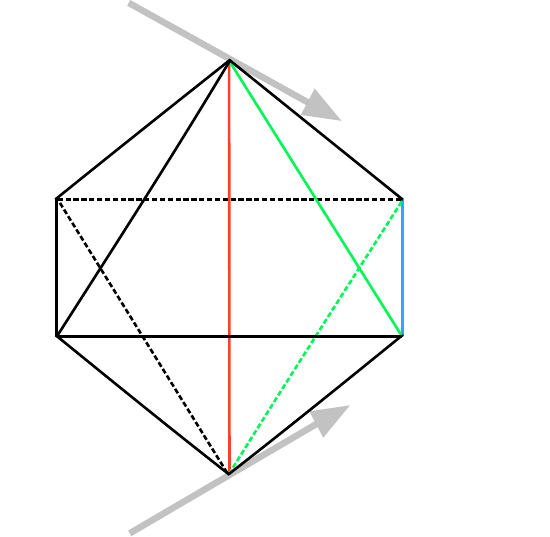}}%
    \put(0.42169059,0.89590065){\makebox(0,0)[lt]{\lineheight{1.25}\smash{\begin{tabular}[t]{l}$P_1$\end{tabular}}}}%
    \put(0.42138088,0.04403442){\makebox(0,0)[lt]{\lineheight{1.25}\smash{\begin{tabular}[t]{l}$P_2$\end{tabular}}}}%
    \put(-0.00378982,0.61365632){\makebox(0,0)[lt]{\lineheight{1.25}\smash{\begin{tabular}[t]{l}$P_-$\end{tabular}}}}%
    \put(-0.00629627,0.37869737){\makebox(0,0)[lt]{\lineheight{1.25}\smash{\begin{tabular}[t]{l}$P_+$\end{tabular}}}}%
    \put(0.75336644,0.61375074){\makebox(0,0)[lt]{\lineheight{1.25}\smash{\begin{tabular}[t]{l}$P_+'$\end{tabular}}}}%
    \put(0.75085999,0.37879177){\makebox(0,0)[lt]{\lineheight{1.25}\smash{\begin{tabular}[t]{l}$P_-'$\end{tabular}}}}%
  \end{picture}%
\endgroup%

  \caption{An ideal octahedron at a positive crossing, modified from \cite[Figure 9a]{Kim2018}.}
  \label{fig:labeled-octahedron}
\end{marginfigure}
Suppose that we are at a positive crossing.
There are four tetrahedra in \cref{fig:labeled-octahedron}, which we call the \defemph{back}, \defemph{left}, \defemph{right}, and \defemph{front} tetrahedra; their vertices are given in the second column of \cref{table:positive-crossing-data}.
We think about looking into the octahedron from in front of the braiding, to match the disc picture.
Each tetrahedron thinks that $P_+$ and $P_-$ are located at $\infty$ and at $0$, respectively, but they disagree where the points $P_1$ and $P_2$ are located.

Label the shapes of the diagram as in \cref{fig:shaped-positive-crossing}, and write $\chi_i = (a_i, b_i, \lambda_i)$ as usual.
The shapes assign the geometric data of \cref{table:positive-crossing-data} to the octahedron.
These assignments are chosen to be compatible with the holonomy of the shaped diagram, as we will show in \cref{prop:positive-face-maps-agree}.
\note{
  We will see later that the coordinates of the points $P_1$ and $P_2$ are essentially the ``segment variables'' \autocite{Kim2018} of the diagram.
  We have chosen to vary the locations of $P_1$ and $P_2$ and fix $P_+$ and $P_-$, which is the opposite of the convention in \cite{Kim2018}.
}
\begin{table}
  \[
    \begin{array}{c|c|cccc|c}
      \text{tetrahedron} & \text{points} & P_1 & P_2 & \text{angle } z_i^{\circ} \\
      \hline
      \tau_b & P_2P_-P_+P_1 & -1/\lambda_1 b_1 & -1/b_2 & b_2/\lambda_1 b_1 \\
      \tau_l & P_2P_+P_-'P_1 & -1/\lambda_1 b_{1'} & -1/\lambda_2 b_2 & \lambda_1 b_{1'}/\lambda_2 b_2 \\
      \tau_r & P_2P_+'P_-P_1 & -1/b_1 & -1/b_{2'} & b_1 / b_{2'} \\
      \tau_f & P_2P_-'P_+'P_1 & -1/b_{1'} & -1/\lambda_2 b_{2'} & \lambda_2 b_{2'}/b_{1'}
    \end{array}
  \]
  \caption{Geometric data associated to the positive crossing in \cref{fig:labeled-octahedron}.}
  \label{table:positive-crossing-data}
\end{table}

The last column of \cref{table:positive-crossing-data} gives the complex dihedral angle of the horizontal edge $P_- P_+$ of the corresponding tetrahedron, equivalently of the vertical {\color{myred} red} edge.
As a sanity check, the gluing equation
\[
  z_b^\circ z_r^\circ z_l^\circ z_f^\circ = 1
\]
of the vertical edge is satisfied for any shapes $\chi_i$.
We will see later that the angles of the external vertical edges ($P_1 P_-'$, etc.) are associated with the $a$-parameters of the shapes.

We can use this data to compute the face maps of the ideal triangulation.
It is somewhat easier to think about this from the disc picture of the octahedron.
In \cref{fig:triangulated-braiding-top-face-map}, we glue $\tau_b$ and $\tau_l$ along the shaded face $P_1P_2P_+$.
The tetrahedra agree on the location of $P_+$, but disagree on the locations of $P_1$ and $P_2$, and the face map fixes this.

We get similar face maps for all the internal gluings.
As promised, these maps agree with the holonomy representation defined for the shapes.
\begin{marginfigure}
  {\def\svgwidth{2.5in}
\begingroup%
  \makeatletter%
  \providecommand\color[2][]{%
    \errmessage{(Inkscape) Color is used for the text in Inkscape, but the package 'color.sty' is not loaded}%
    \renewcommand\color[2][]{}%
  }%
  \providecommand\transparent[1]{%
    \errmessage{(Inkscape) Transparency is used (non-zero) for the text in Inkscape, but the package 'transparent.sty' is not loaded}%
    \renewcommand\transparent[1]{}%
  }%
  \providecommand\rotatebox[2]{#2}%
  \newcommand*\fsize{\dimexpr\f@size pt\relax}%
  \newcommand*\lineheight[1]{\fontsize{\fsize}{#1\fsize}\selectfont}%
  \ifx\svgwidth\undefined%
    \setlength{\unitlength}{215.81824493bp}%
    \ifx\svgscale\undefined%
      \relax%
    \else%
      \setlength{\unitlength}{\unitlength * \real{\svgscale}}%
    \fi%
  \else%
    \setlength{\unitlength}{\svgwidth}%
  \fi%
  \global\let\svgwidth\undefined%
  \global\let\svgscale\undefined%
  \makeatother%
  \begin{picture}(1,1.03093509)%
    \lineheight{1}%
    \setlength\tabcolsep{0pt}%
    \put(0.366521,0.92532356){\color[rgb]{0,0,0}\makebox(0,0)[lt]{\lineheight{1.25}\smash{\begin{tabular}[t]{l}$P_+$\end{tabular}}}}%
    \put(0.38424276,0.03278304){\color[rgb]{0,0,0}\makebox(0,0)[lt]{\lineheight{1.25}\smash{\begin{tabular}[t]{l}$P_-$\end{tabular}}}}%
    \put(0.1486229,0.51022397){\color[rgb]{0,0,0}\makebox(0,0)[lt]{\lineheight{1.25}\smash{\begin{tabular}[t]{l}$P_2$\end{tabular}}}}%
    \put(0.59268972,0.51326773){\color[rgb]{0,0,0}\makebox(0,0)[lt]{\lineheight{1.25}\smash{\begin{tabular}[t]{l}$P_1$\end{tabular}}}}%
    \put(0,0){\includegraphics[width=\unitlength,page=1]{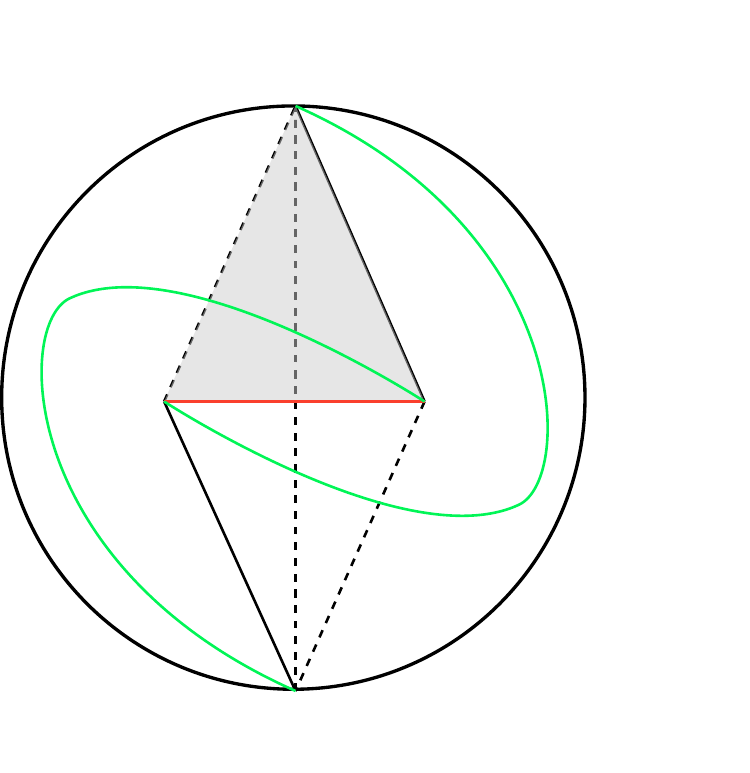}}%
  \end{picture}%
\endgroup%
}
  \caption{The face corresponding to $g^+(\chi_2)$.}
  \label{fig:triangulated-braiding-top-face-map}
\end{marginfigure}

\begin{prop}
  \label{prop:positive-face-maps-agree}
  The face maps (as elements of $\pslg$) of the octahedron in \cref{fig:labeled-octahedron} agree with the holonomies assigned to the diagram complement by the shape parameters.
\end{prop}
\begin{proof}
  If we think of the face map in \cref{fig:triangulated-braiding-top-face-map} as going from $\tau_b$ to $\tau_l$, then it represents the positive holonomy of strand $2$, which should be mapped to $g^+(\chi_2)$.
Observe that for any $z \in \hat \CC$,
\[
  z \cdot g^+(\chi_2) =
  z \cdot
    \begin{pmatrix}
      a & 0 \\
      (a - 1/\lambda)/b & 1
    \end{pmatrix}
    =
  a_2 z + \frac{a_2}{b_2} - \frac{1}{\lambda_2 b_2}.
\]
In particular, we see that $g^+(\chi_2)$ fixes $\infty$, maps $-1/ b_2$ to $-1/\lambda_2 b_2$, and maps $-1/\lambda_1 b_1$ to 
\[
  (-1/\lambda_1 b_1) \cdot g^+(\chi_2) = a_2 \left( \frac{1}{b_2} -\frac{1}{\lambda_1 b_1} \right)   - \frac{1}{\lambda_2 b_2} = -\frac{1}{\lambda_1 b_1'}.
\]
Because fractional linear transformations are totally determined by their action on three points of $\hat \CC$, we conclude that the face map agrees with $g^+(\chi_2)$.

The negative holonomy of strand $2$ does not correspond directly to a face map, but the face map going from $\tau_b$ to $\tau_r$ similarly corresponds to the \emph{inverse} negative holonomy of $\chi_1$.
We see that the transformation
\[
  z \cdot g^-(\chi_1)^{-1}
  = z \cdot
  \begin{pmatrix}
    1 & -(1 + \lambda_1/a_1) b_1 \\
    0 & 1/a_1
  \end{pmatrix}
  =
  \left[
    -b_1 - \frac{\lambda_1 b_1}{a_1} + \frac{1}{za_1}
  \right]^{-1}
\]
preserves $0$, maps $1/\lambda_1 b_1$ to $-1/b_1$, and maps $-1/\lambda_2 b_2$ to 
\[
  (-1/\lambda_2 b_2) \cdot g^-(\chi_1)^{-1}
  =
  \left[
    -b_1 - \frac{\lambda_1 b_1}{a_1} - \frac{b_2}{a_1}
  \right]^{-1}
  =
  - \frac{1}{b_2'}.
\]

There is a parallel characterization of the holonomies on the other side of the crossing.
For example, $g^+(\chi_2')$ corresponds to the gluing map between $\tau_r$ and $\tau_f$, and correspondingly acts on the vertices of $\tau_r$ by
\begin{align*}
  \infty \cdot g^+(\chi_2')
  &=
  \infty
  \\
  (-1/b_2') \cdot g^+(\chi_2')
  &=
  -1/\lambda_2 b_2'
  \\
  (-1/b_1) \cdot g^+(\chi_2')
  &=
  -\frac{a_2'}{b_1} + \frac{a_2'}{b_2'} - \frac{1}{\lambda_2 b_2'}
  =
  -1/b_1'
\end{align*}
and similarly the face map gluing $\tau_l$ to $\tau_f$ is $g^-(\chi_1')^{-1}$.
\end{proof}

\begin{figure}
  \centering
  \subcaptionbox{An ideal octahedron at a negative crossing, modified from \cite[Figure 9b]{Kim2018}.\label{fig:labeled-octahedron-negative}}{ \def\svgwidth{2.5in} 
\begingroup%
  \makeatletter%
  \providecommand\color[2][]{%
    \errmessage{(Inkscape) Color is used for the text in Inkscape, but the package 'color.sty' is not loaded}%
    \renewcommand\color[2][]{}%
  }%
  \providecommand\transparent[1]{%
    \errmessage{(Inkscape) Transparency is used (non-zero) for the text in Inkscape, but the package 'transparent.sty' is not loaded}%
    \renewcommand\transparent[1]{}%
  }%
  \providecommand\rotatebox[2]{#2}%
  \newcommand*\fsize{\dimexpr\f@size pt\relax}%
  \newcommand*\lineheight[1]{\fontsize{\fsize}{#1\fsize}\selectfont}%
  \ifx\svgwidth\undefined%
    \setlength{\unitlength}{159.66983414bp}%
    \ifx\svgscale\undefined%
      \relax%
    \else%
      \setlength{\unitlength}{\unitlength * \real{\svgscale}}%
    \fi%
  \else%
    \setlength{\unitlength}{\svgwidth}%
  \fi%
  \global\let\svgwidth\undefined%
  \global\let\svgscale\undefined%
  \makeatother%
  \begin{picture}(1,0.98590168)%
    \lineheight{1}%
    \setlength\tabcolsep{0pt}%
    \put(0,0){\includegraphics[width=\unitlength,page=1]{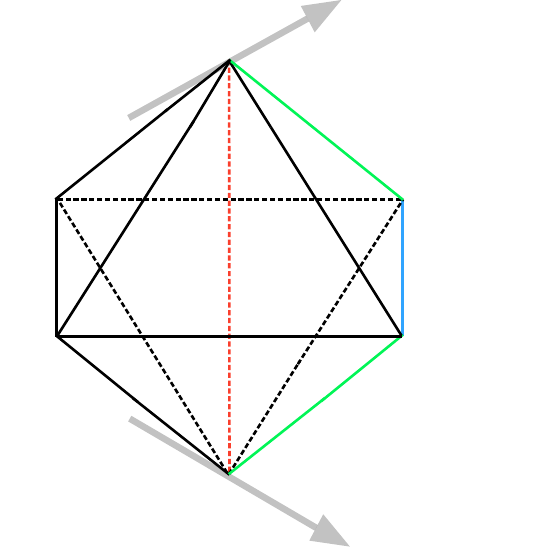}}%
    \put(0.3614527,0.90646408){\makebox(0,0)[lt]{\lineheight{1.25}\smash{\begin{tabular}[t]{l}$P_2$\end{tabular}}}}%
    \put(0.36114584,0.06241891){\makebox(0,0)[lt]{\lineheight{1.25}\smash{\begin{tabular}[t]{l}$P_1$\end{tabular}}}}%
    \put(-0.00375502,0.62681106){\makebox(0,0)[lt]{\lineheight{1.25}\smash{\begin{tabular}[t]{l}$P_+$\end{tabular}}}}%
    \put(-0.00623846,0.39400929){\makebox(0,0)[lt]{\lineheight{1.25}\smash{\begin{tabular}[t]{l}$P_-$\end{tabular}}}}%
    \put(0.74644972,0.62690462){\makebox(0,0)[lt]{\lineheight{1.25}\smash{\begin{tabular}[t]{l}$P_-'$\end{tabular}}}}%
    \put(0.74396628,0.39410282){\makebox(0,0)[lt]{\lineheight{1.25}\smash{\begin{tabular}[t]{l}$P_+'$\end{tabular}}}}%
  \end{picture}%
\endgroup%
}
  \hfill
  \subcaptionbox{Building via flips.\label{fig:triangulated-braiding-negative-total}}{ \def\svgwidth{2.5in} 
\begingroup%
  \makeatletter%
  \providecommand\color[2][]{%
    \errmessage{(Inkscape) Color is used for the text in Inkscape, but the package 'color.sty' is not loaded}%
    \renewcommand\color[2][]{}%
  }%
  \providecommand\transparent[1]{%
    \errmessage{(Inkscape) Transparency is used (non-zero) for the text in Inkscape, but the package 'transparent.sty' is not loaded}%
    \renewcommand\transparent[1]{}%
  }%
  \providecommand\rotatebox[2]{#2}%
  \newcommand*\fsize{\dimexpr\f@size pt\relax}%
  \newcommand*\lineheight[1]{\fontsize{\fsize}{#1\fsize}\selectfont}%
  \ifx\svgwidth\undefined%
    \setlength{\unitlength}{217.05317688bp}%
    \ifx\svgscale\undefined%
      \relax%
    \else%
      \setlength{\unitlength}{\unitlength * \real{\svgscale}}%
    \fi%
  \else%
    \setlength{\unitlength}{\svgwidth}%
  \fi%
  \global\let\svgwidth\undefined%
  \global\let\svgscale\undefined%
  \makeatother%
  \begin{picture}(1,1.02506955)%
    \lineheight{1}%
    \setlength\tabcolsep{0pt}%
    \put(0.37012526,0.9200589){\color[rgb]{0,0,0}\makebox(0,0)[lt]{\lineheight{1.25}\smash{\begin{tabular}[t]{l}$P_+$\end{tabular}}}}%
    \put(0.3877462,0.03259649){\color[rgb]{0,0,0}\makebox(0,0)[lt]{\lineheight{1.25}\smash{\begin{tabular}[t]{l}$P_-$\end{tabular}}}}%
    \put(0.15346685,0.50732103){\color[rgb]{0,0,0}\makebox(0,0)[lt]{\lineheight{1.25}\smash{\begin{tabular}[t]{l}$P_2$\end{tabular}}}}%
    \put(0.59500714,0.51034747){\color[rgb]{0,0,0}\makebox(0,0)[lt]{\lineheight{1.25}\smash{\begin{tabular}[t]{l}$P_1$\end{tabular}}}}%
    \put(0,0){\includegraphics[width=\unitlength,page=1]{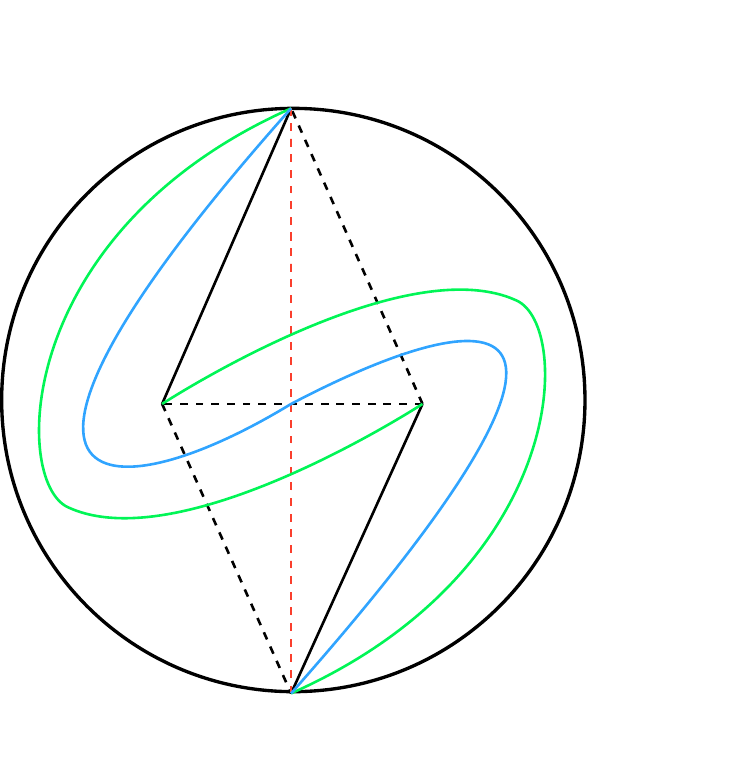}}%
  \end{picture}%
\endgroup%
}
  \caption{The octahedron associated to a negative crossing.}
  \label{fig:negative-octahedron}
\end{figure}

We now give a similar description of negative crossings, as shown in \cref{fig:labeled-octahedron-negative}.
We assign the octahedron at a negative crossing the geometric data of \cref{table:negative-crossing-data}, which is almost identical to \cref{table:positive-crossing-data}.
To derive it, it helps to use the alternative description of the octahedron in \cref{fig:triangulated-braiding-negative-total}.
For example, from \cref{fig:triangulated-braiding-negative-total} we see that the face map gluing $\tau_b$ and $\tau_l$ should correspond to $g^-(\chi_2)$, and we confirm that
\begin{align*}
  0 \cdot g^-(\chi_2)
  &=
  0,
  &
  (-1/b_2) \cdot g^-(\chi_2)
  &=
  -\frac{1}{\lambda_2 b_2},
  &
  (-1/\lambda_1 b_1) \cdot g^-(\chi_2)
  &=
  -\frac{1}{\lambda_1 b_1'}.
\end{align*}

\begin{table}
  \[
    \begin{array}{c|c|cccc|c}
      \text{tetrahedron} & \text{points} & P_1 & P_2 & \text{angle } z_i^{\circ} \\
      \hline
      \tau_b & P_1P_+P_-P_2 & -1/\lambda_1 b_1 & -1/b_2 & b_2/\lambda_1 b_1 \\
      \tau_l & P_1P_-P_+'P_2 & -1/\lambda_1 b_{1'} & -1/\lambda_2 b_2 & \lambda_1 b_{1'} / \lambda_2 b_2\\
      \tau_r & P_1P_-'P_+P_2 & -1/b_1 & -1/b_{2'} & b_1/b_{2'} \\
      \tau_f & P_1P_+'P_-'P_2 & -1/b_{1'} & -1/\lambda_2 b_{2'} & \lambda_2 b_{2'}/b_{1'}
    \end{array}
  \]
  \caption{Geometric data associated to the negative crossing in \cref{fig:labeled-octahedron-negative}.
  Notice that the only difference from \cref{table:positive-crossing-data} is the second column.}
  \label{table:negative-crossing-data}
\end{table}

\begin{thm}
  Let $L$ be a link in $S^3$ presented by a shaped diagram $D$.
  The representation $\pi(L) \to \pslg$ induced by the face maps of the octahedral decomposition of $L$ corresponding to $D$ agrees with the holonomy representation $\pi(L) \to \slg$ corresponding to the shaped diagram.
\end{thm}
\begin{proof}
  It is enough to check the generators of $\Pi(D)$.
  We checked the paths near a positive crossing in \cref{prop:positive-face-maps-agree}, and negative crossings follow by similar computations; we checked one just before the statement of the theorem.
\end{proof}

\subsection{Gluing shaped octahedra}
In \cref{sec:shaped-tangle-diagrams} we said that a shaped diagram was one satisfying certain gluing equations on the $a$-variables \cref{eq:positive-a-relations,eq:negative-a-relations}.
In this section we gave a different notion of gluing equation coming from hyperbolic geometry: any assignment of $b$-variables and eigenvalues $\lambda$ to a shaped tangle assigns complex dihedral angles to the ideal octahedra, and we can ask whether these angles satisfy the gluing equations when the octahedra are glued together.

\citeauthor{Kim2018} \cite{Kim2018} work out the gluing and completeness equations of the octahedral decomposition in detail and show that they can be described by one equation for each segment of the diagram and for each region, including the case of $\lambda$-deformed (incomplete) structures.
In terms of the shape coordinates the segment equations corresponding to the consistency equations for the $a$-variables, while the regional equations are automatically satisfied.

\begin{thm}
  Let $D$ be a pre-shaped diagram of a link.
  The octahedral decomposition of the complement of $D$ satisfies the gluing equations if and only if $D$ is a shaped diagram, i.e.\@ if the gluing equations of $D$ hold.
\end{thm}
\begin{proof}
  This is exactly \cite[Proposition 4.1]{Kim2018} once we make the appropriate identifications.
  We give details in one case.

  Consider an over-under crossing as in \cref{fig:over-under-crossing}.
  \begin{marginfigure}
    \centering
\begingroup%
  \makeatletter%
  \providecommand\color[2][]{%
    \errmessage{(Inkscape) Color is used for the text in Inkscape, but the package 'color.sty' is not loaded}%
    \renewcommand\color[2][]{}%
  }%
  \providecommand\transparent[1]{%
    \errmessage{(Inkscape) Transparency is used (non-zero) for the text in Inkscape, but the package 'transparent.sty' is not loaded}%
    \renewcommand\transparent[1]{}%
  }%
  \providecommand\rotatebox[2]{#2}%
  \newcommand*\fsize{\dimexpr\f@size pt\relax}%
  \newcommand*\lineheight[1]{\fontsize{\fsize}{#1\fsize}\selectfont}%
  \ifx\svgwidth\undefined%
    \setlength{\unitlength}{140.23033333bp}%
    \ifx\svgscale\undefined%
      \relax%
    \else%
      \setlength{\unitlength}{\unitlength * \real{\svgscale}}%
    \fi%
  \else%
    \setlength{\unitlength}{\svgwidth}%
  \fi%
  \global\let\svgwidth\undefined%
  \global\let\svgscale\undefined%
  \makeatother%
  \begin{picture}(1,0.37438403)%
    \lineheight{1}%
    \setlength\tabcolsep{0pt}%
    \put(0,0){\includegraphics[width=\unitlength,page=1]{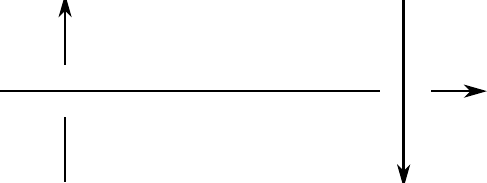}}%
    \put(0.16045031,0.32210222){\makebox(0,0)[lt]{\lineheight{1.25}\smash{\begin{tabular}[t]{l}$2'$\end{tabular}}}}%
    \put(0.16045031,0.02794331){\makebox(0,0)[lt]{\lineheight{1.25}\smash{\begin{tabular}[t]{l}$2$\end{tabular}}}}%
    \put(0.85573497,0.29415882){\makebox(0,0)[lt]{\lineheight{1.25}\smash{\begin{tabular}[t]{l}$3$\end{tabular}}}}%
    \put(0.85573497,0.02674161){\makebox(0,0)[lt]{\lineheight{1.25}\smash{\begin{tabular}[t]{l}$3'$\end{tabular}}}}%
    \put(-0.05348344,0.16045025){\makebox(0,0)[lt]{\lineheight{1.25}\smash{\begin{tabular}[t]{l}$1$\end{tabular}}}}%
    \put(0.4278675,0.2139337){\makebox(0,0)[lt]{\lineheight{1.25}\smash{\begin{tabular}[t]{l}$1'$\end{tabular}}}}%
    \put(1.01644576,0.16129458){\makebox(0,0)[lt]{\lineheight{1.25}\smash{\begin{tabular}[t]{l}$1''$\end{tabular}}}}%
  \end{picture}%
\endgroup%

    \caption{An over-under crossing with labeled segments.}
    \label{fig:over-under-crossing}
  \end{marginfigure}%
  By \cref{eq:positive-a-relations}, the $a$-variables assigned to the left and right sides are
  \[
    a_{1', L} = \frac{\lambda_2 b_2 - \lambda_1 b_{1'}}{\lambda_2 b_{2'} - b_{1'}}
    \text{ and }
    a_{1', R} = \frac{1/\lambda_3 b_{3} - 1/\lambda_1 b_{1'}}{1/\lambda_3 b_3 - 1/b_{1'}}.
  \]
  Using \cref{table:positive-crossing-data} we can interpret
  \[
    a_{1', L} = \lambda_1 \frac{1 - \lambda_2 b_2/\lambda_1 b_{1'}}{1 - \lambda_2 b_{2'}/b_{1'}}
    =
    \lambda_1 z_{l, L} z_{f, L}^{\circ \circ} 
  \]
  as $\lambda_1$ times the dihedral angle of the edge $P_1 P_-$ of the left-hand octahedron (associated to the crossing $1, 2, 1', 2'$).
  Write $\hat \gamma_{L} = z_{l, L} z_{f, L}^{\circ \circ} $ for this angle.
  We can similarly interpet
  \[
    a_{1', R} = \lambda_1 \frac{1 - \lambda_1 b_{1'} /\lambda_3 b_{3}}{1 - b_{1'}/\lambda_3 b_3}
    = \lambda_1^{-1} z_{l, R} z_{b, R}^{\circ \circ}
  \]
  as $\lambda_1^{-1}$ times the dihedral angle of the edge $P_2 P^+$ of the right-hand tetrahedron; write $\alpha_{R}$ for this angle.

  The $a$-consistency equation for the edge $1'$ now reads
  \[
    \lambda_1 \hat \gamma_{L} = \lambda_1^{-1} \alpha_{R} \text{, equivalently } \frac{\alpha_{R}}{\hat \gamma_{L}} = \lambda_1^2,
  \]
  which is exactly the hyperbolicity equation associated to that segment by \cite[Figure 14a and equation (10)]{Kim2018}.
  A similar check works for the remaining types of segments (over-over, etc.)

  As discussed in \cite[Proposition 4.4]{Kim2018}, the region equations are automatically satisfied when we use the $b$-variables.
  In more detail, the region variables are the dihedral angles of the hoizontal edges of the octahedra.
  Each corner of a region corresponds to a crossing, and we associate region variables at a crossing as in \cref{fig:region-variables}.
  \begin{marginfigure}
    \centering
\begingroup%
  \makeatletter%
  \providecommand\color[2][]{%
    \errmessage{(Inkscape) Color is used for the text in Inkscape, but the package 'color.sty' is not loaded}%
    \renewcommand\color[2][]{}%
  }%
  \providecommand\transparent[1]{%
    \errmessage{(Inkscape) Transparency is used (non-zero) for the text in Inkscape, but the package 'transparent.sty' is not loaded}%
    \renewcommand\transparent[1]{}%
  }%
  \providecommand\rotatebox[2]{#2}%
  \newcommand*\fsize{\dimexpr\f@size pt\relax}%
  \newcommand*\lineheight[1]{\fontsize{\fsize}{#1\fsize}\selectfont}%
  \ifx\svgwidth\undefined%
    \setlength{\unitlength}{144bp}%
    \ifx\svgscale\undefined%
      \relax%
    \else%
      \setlength{\unitlength}{\unitlength * \real{\svgscale}}%
    \fi%
  \else%
    \setlength{\unitlength}{\svgwidth}%
  \fi%
  \global\let\svgwidth\undefined%
  \global\let\svgscale\undefined%
  \makeatother%
  \begin{picture}(1,0.39591226)%
    \lineheight{1}%
    \setlength\tabcolsep{0pt}%
    \put(0,0){\includegraphics[width=\unitlength,page=1]{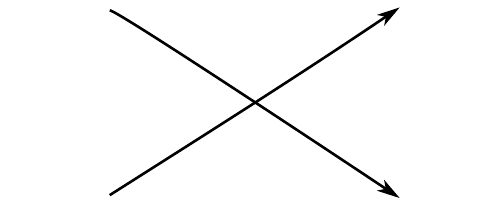}}%
    \put(0.14389246,0.37523126){\makebox(0,0)[lt]{\lineheight{1.25}\smash{\begin{tabular}[t]{l}$\chi_1$\end{tabular}}}}%
    \put(0.14360196,0.00633945){\makebox(0,0)[lt]{\lineheight{1.25}\smash{\begin{tabular}[t]{l}$\chi_2$\end{tabular}}}}%
    \put(0.80854481,0.37548953){\makebox(0,0)[lt]{\lineheight{1.25}\smash{\begin{tabular}[t]{l}$\chi_2'$\end{tabular}}}}%
    \put(0.80825435,0.00659769){\makebox(0,0)[lt]{\lineheight{1.25}\smash{\begin{tabular}[t]{l}$\chi_1'$\end{tabular}}}}%
    \put(0.17332669,0.16641966){\makebox(0,0)[lt]{\lineheight{1.25}\smash{\begin{tabular}[t]{l}$b_2/\lambda_1 b_1$\end{tabular}}}}%
    \put(0.63987008,0.16533341){\makebox(0,0)[lt]{\lineheight{1.25}\smash{\begin{tabular}[t]{l}$\lambda_2 b_{2'}/b_{1'}$\end{tabular}}}}%
    \put(0.40616686,0.31924482){\makebox(0,0)[lt]{\lineheight{1.25}\smash{\begin{tabular}[t]{l}$b_{1}/b_{2'}$\end{tabular}}}}%
    \put(0.33645312,0.00955987){\makebox(0,0)[lt]{\lineheight{1.25}\smash{\begin{tabular}[t]{l}$\lambda_1 b_{1'}/\lambda_2 b_{2}$\end{tabular}}}}%
  \end{picture}%
\endgroup%

    \caption{Region variables at a crossing.
    The assignment does not depend on the type of crossing.}
    \label{fig:region-variables}
  \end{marginfigure}%
  The gluing equation for each region says that the product of the variables for a region is $1$.
  It is not hard to see that this is automatically true for any pre-shaped diagram because the region variables are ratios of the $b$-variables and eigenvalues.
\end{proof}

Above we used the fact that the dihedral angles of the vertical edges of each octahedron are related to the $a$-variables.
For later use we give the remaining associations.
At a positive crossing, \cref{table:positive-crossing-data} gives that
\begin{equation}
  \begin{aligned}
    \label{eq:vertical-angles-positive}
    z_b^{\circ \circ} z_r
    &=
    \frac{1 - b_{2'}/b_1}{1 - b_2/\lambda_1 b_1}
    =
    \frac{\lambda_1}{a_1}
    &
    z_l z_f^{\circ \circ}
    &=
    \frac{1 - \lambda_2 b_2/\lambda_1 b_{1'}}{1 - \lambda_2 b_{2'}/b_{1'}}
    =
    \frac{a_{1'}}{\lambda_1}
    \\
    z_b^{\circ \circ} z_l
    &=
    \frac{1 - \lambda_2 b_2/\lambda_1 b_{1'}}{1 - b_2/\lambda_1 b_1}
    =
    \lambda_2 a_2
    &
    z_r z_f^{\circ \circ}
    &=
    \frac{1 - b_{2'}/b_1}{1 - \lambda_2 b_{2'}/b_{1'}}
    =
    \frac{1}{\lambda_2 a_2'}
  \end{aligned}
\end{equation}
and similarly at a negative crossing \cref{table:negative-crossing-data} gives
\begin{equation}
  \label{eq:vertical-angles-negative}
  \begin{aligned}
    z_b z_r^{\circ \circ}
    &=
    \frac{1 - \lambda_1 b_1/b_2}{1 - b_1/b_{2'}}
    =
    {\lambda_1}{a_1}
    &
    z_l^{\circ \circ} z_f
    &=
    \frac{1 - b_{1'}/\lambda_2 b_{2'}}{1 - \lambda_1 b_{1'}/\lambda_2 b_{2}}
    =
    \frac{1}{\lambda_1 a_1'}
    \\
    z_b z_l^{\circ \circ}
    &=
    \frac{1 - \lambda_1 b_1/b_2}{1 - \lambda_1 b_{1'}/\lambda_2 b_{2}}
    =
    \frac{\lambda_2}{a_2}
    &
    z_r^{\circ \circ} z_f
    &=
    \frac{1 - b_{1'}/\lambda_2 b_{2'}}{1 - b_{1}/b_{2'}}
    =
    \frac{a_2'}{\lambda_2}
  \end{aligned}
\end{equation}

\chapter{Representations of shaped tangles}
\label{ch:functors}
\section*{Overview}
In this chapter we give a general theory of how to construct invariants of links via shaped tangle diagrams.
Our approach is an extension of the Reshetikhin-Turaev construction \cite{Reshetikhin1990}, with a few modifications.

Specifically, we explain how to define $\slg$-link invariants using a \defemph{representation} of the \defemph{shape biquandle} in a pivotal category.
We define these precisely in \cref{sec:biquandles}. 
Because of vanishing quantum dimensions, we also need a \defemph{modified trace} on the category of modules, as discussed in \cref{sec:modified-traces} and \cref{ch:modified-traces}.
For now, we can think of a modified trace as a family of linear maps
\[
  \modtr_V : \End_{\catl C}(V) \to \CC
\]
for objects $V$ of a (subcategory of) a pivotal category $\catl C$.
These traces should be cyclic and appropriately compatible with the usual quantum trace of $\catl C$ coming from the pivotal structure.

The majority of the chapter is devoted to the construction of these invariants (described below) and the proof that they are well-defined.
We also define some simple examples of representations: scalar representations, which we can think of as changes in normalization and correspond to biquandle $2$-cocycles in the sense of \cite{Kamada2018}, and adjoint representations, which were used in the original Kashaev-Reshetikhin construction \cite{Kashaev2005,Chen2019,McPhailSnyderUnpub2} of holonomy invariants.
In \cref{ch:algebras,ch:doubles} we construct the representations that give the nonabelian quantum dilogarithm and its double.

\subsection{Summary of the construction}
Let $L$ be an extended $\slg$-link $L$, and let $\catl C$ be a pivotal category, such as the category $\modc \qgrp$.
A representation of the extended shape biquandle in $\modc \qgrp$ gives a functor $\mathcal F : \tangshe \to \modc \qgrp$ via a slight extension of the usual RT construction.
Suppose in addition that we have a modified trace $\modtr$ on a subcategory%
\note{
  As discussed in \cref{ch:modified-traces}, this subcategory should be an ideal, meaning that it is closed under tensor products and retracts.
}
$I$ of $\catl C$ containing the image of $\mathcal{F}.$
Then we can define an invariant $\invl{F}(L)$ from $\mathcal{F}$ as follows:
\begin{enumerate}
  \item Express $L$ as an extended shaped tangle diagram $D$.
    (This may require a gauge transformation of the holonomy.)
  \item Cut open $D$ to obtain a $(1,1)$-tangle%
    \note{A $(1,1)$-tangle is a tangle with one input and one output.
      If we join these together, we get a link, and conversely we can obtain a $(1,1)$-tangle diagram by cutting open a link along some particular strand.
      $(1,1)$-tangles are closely related to long knots and long links.
    }
    diagram $T$.
  \item Use the functor $\mathcal{F}$ to view $T$ as a morphism $\mathcal{F}(T) : V \to V$ of an object $V$ of $\catl C$.
  \item Take the modified trace of $\mathcal{F}(T)$ and set
    \[
    \invl{F}(L) \defeq \modtr_V \mathcal{F}(T).
    \]
\end{enumerate}
Our goal in this chapter is to prove that $\invl{F}(L)$ is a well-defined invariant of $L$.
\begin{thm}
  \label{thm:invariant-construction}
  Let $\mathcal{F}$ be a representation of the extended shape biquandle (\cref{sec:biquandles}) and $\modtr$ a compatible trace (\cref{sec:modified-traces}).
  If $\modtr$ is gauge-invariant (\cref{sec:internal-gauge-transf}) and the representation $\mathcal{F}$ is \defemph{regular} and \defemph{absolutely simple} (\cref{def:abs-simple-regular}), the scalar $\invl{F}(L)$ defined above is a gauge-independent invariant of extended $\slg$-links.
\end{thm}
\begin{proof}
  This follows from \cref{thm:rep-gives-functor,thm:links-presentable,thm:gauge-invariant-trace}.
  We explain the idea of the proof below.
\end{proof}

\subsection{Some problems}
To prove \cref{thm:invariant-construction} we need to address the following technical issues:
\begin{description}
  \item[Reidemeister moves] 
    If the labeling of a diagram by shapes is to describe properties of the tangle (not just its diagrams) then the labellings must be compatible with the Reidemeister moves defining equivalence of tangle diagrams.
  \item[Existence of shapes]
    Can every $\slg$-link be represented as a shaped tangle?
  \item[Gauge equivlance]
    If two $\slg$-links are gauge equivalent, will our functor assign the same invariant to both?
  \item[Choice of cut]
    How do we know our invariants do not depend on the choice of representative $(1,1)$-tangle?
\end{description}
These issues are dealt with by the \defemph{generic biquandle factorizations} of \citeauthor{Blanchet2020} \cite{Blanchet2020}.
We emphasize the parts that are necessary for understanding $\slg$ holonomy invariants, as their construction is more general.
As part of the construction we introduce a \emph{partially defined} biquandle, which is related to excluding singular crossings of shaped links where the $a$-variables would be $0$ or $\infty$.
We de-emphasize the technical issues this causes; the skeptical reader can consult \cite[Section 5]{Blanchet2020} for the details.

\subsection{Some solutions}

It is common in knot theory to describe representations of a tangle (or knot or link) fundamental group using the Wirtinger presentation associated to a diagram of the tangle.
For technical reasons%
\note{
  The issue was already raised in \cref{ch:prelim}: the central subalgebra $\cent_0$ of $\qgrp$ is not the ring of functions on $\slg$, but instead the ring of functions on $\slg^*$.
  Furthermore, to match the $R$-matrices computed in \cref{ch:algebras} we need to restrict to characters of $\cent_0$ that pull back to the center of the Weyl algebra; these characters are exactly the shape biquandle.
}
we cannot use this description for $\slg$-tangles, but instead want to use a generalization, an algebraic structure called a \defemph{biquandle}.
In the language of \cite{Blanchet2020}, we show that the shape biquandle gives a generic factorization of the conjugation quandle of $\slg$.

Diagrams colored by the \defemph{shape biquandle} are exactly shaped tangles as defined in \cref{ch:shaped-tangles}, so they are compatible with Reidemeister moves.
Actually, this is only true generically: at pinched crossings (\cref{sec:pinched-crossings}) we could send the $a$-variables to zero, which we prohibit.
However, since most diagrams%
\note{By ``most'' we mean ``lying in a dense Zariski open subset''.}
allow the Reidemeister moves, we say that shaped tangles form a \defemph{generic biquandle}.
This resolves the first question.

We then turn to the remaining questions, which are answered together.
Given a $\slg$-link $L$ with a diagram $D$ we can always describe the representation $\rho : \pi(L) \to \slg$ by decorating the arcs of $D$ with elements of $\slg$.
(Formally, we say that $D$ is colored by the \defemph{conjugation quandle} of $\slg$.)
For generic $\rho$ we can pull these back to get a shaped diagram $D$ whose holonomy is $\rho$, as described in 
In certain singular cases this is not possible, but we show that we can always conjugate $\rho$ to avoid these.

We can therefore define our invariants on \emph{any} $\slg$-link by first gauge-transforming to a holonomy representation that lies in the image of the shape coordinates.
At a slightly more abstract level, recall that the representation variety of $L$ is the space of representations $\rho : \pi(L) \to \slg$, and the character variety is the representation variety modulo conjugation by  $\slg$.%
\note{
  There are geometric subtleties with taking this quotient of the representation variety, which we do not address here.
}
Shapings of diagrams of $L$ give a coordinate system on a Zariski open dense subset of the representation variety of $L$.
Because every representation is conjugate to one lying in this set, the shapes give coordinates on the character variety, which is our object of interest.

For this to be well-defined the invariant of shaped tangles cannot depend on the gauge class of shaped diagram used to represent a particular holonomy representation.
In \cref{sec:internal-gauge-transf} we show that this is the case, as long as the modified traces are appropriately gauge invariant.
A key idea in this proof is to understand gauge transformations in a diagram-theoretic way, so that gauge invariance is a consequence of the Reidemeister moves.
Finally, use of the modified trace introduced in \cref{sec:modified-traces} also ensures that the renormalized link invariants are independent of the choice of representative $(1,1)$-tangle (i.e.\@ the choice of where to cut $L$.)

\subsection{Structure of the chapter}
In \cref{sec:biquandles} we introduce biquandles as an algebraic structure, then discuss their representations in \cref{sec:biquandle-reps}, which lead to Reshetikhin-Turaev type functors for shaped tangles.
In \cref{sec:modified-traces} we discuss modified traces and how they relate to biquandle representations.
For now, we assume that traces with the required properties exist; an explanation of how to construct them is given in \cref{ch:modified-traces}.
We also show how to use modified traces to show that (under a few technical hypotheses satisfied in our examples) biquandle models automatically induce twists.
In \cref{sec:internal-gauge-transf} we explain how to view gauge transformations of $\slg$-links in terms of shaped tangle diagrams, which allows us to show that our invariants are well-defined and gauge-independent.
These provide all the ingredients for the proof of \cref{thm:invariant-construction}.

We give a simple example of a biquandle  representation in \cref{sec:adjoint-rep}, where we discuss the \defemph{adjoint} representation used to construct the Kashaev-Reshetikhin invariants of \cite{Kashaev2005,Chen2019,McPhailSnyderUnpub2}.
We conclude the chapter (\cref{sec:scalar-reps}) with a discussion of \defemph{scalar} representations of biquandles, which we can think of as changes in normalization.
They turn out to be equivalent to the biquandle $2$-cocycles of \cite{Kamada2018}.
Later in \cref{ch:algebras} we will define the representation $\vecfunc$, which gives the nonabelian quantum dilogarithm.
The functor $\doubfunc$ of \cref{ch:doubles} is another example of a representation.

\section{Biquandles}
\label{sec:biquandles}
There are various algebraic structures that describe labellings of tangle diagrams, including quandles, racks, and biquandles.
Biquandles%
\note{Actually, we need partially defined biquandles.}
are the most specific structure that capture the braiding on the quantized function algebra, so we focus on them.

\begin{defn}
  A \defemph{biquandle} \cite[Section 3]{Blanchet2020} is a set $X$ with a bijective map $B = (B_1, B_2) : X \times X \to X \times X$ which satisfies the following axioms.
  \begin{enumerate}
    \item $B$ satisfies the braid relation
      \[
        (\id_X \times B) (B \times \id_X ) (\id_X \times B)
        =
        (B \times \id_X ) (\id_X \times B) (B \times \id_X ).
      \]
    \item $B$ is sideways invertible: there is a (unique) bijection $\sidebraid : X \times X \to X \times X$ such that
       \[
         \sidebraid(B_1(x,y), x) = (B_2(x,y),y)
      \]
      for every $x,y \in X$.%
      \note{This is equivalent to saying that, for any $x \in X$, the maps $y \mapsto B_1(x,y)$ and $y \mapsto B_2(y,x)$ are bijections.}
    \item There is a bijection  $\alpha : X \to X$ such that
      \[
        \sidebraid(x,x) = (\alpha(x), \alpha(x)).
      \]
  \end{enumerate}
\end{defn}

\subsection{Biquandle colorings of diagrams}

\begin{defn}
  Let $X$ be a biquandle.
  An oriented, framed tangle diagram $D$ is said to be \defemph{$X$-colored} if each segment is assigned a color in $X$, such that the colors at the crossings are compatible with the map $B$.
  For example, at a labeled positive crossing (\cref{fig:biquandle-crossing-labels})
  we must have $B(x_1,x_2) = (x_2', x_1')$.
  At a negative crossing we would instead have $B^{-1}(x_1, x_2) = (x_2', x_1')$.

  \begin{marginfigure}
    \centering
\begingroup%
  \makeatletter%
  \providecommand\color[2][]{%
    \errmessage{(Inkscape) Color is used for the text in Inkscape, but the package 'color.sty' is not loaded}%
    \renewcommand\color[2][]{}%
  }%
  \providecommand\transparent[1]{%
    \errmessage{(Inkscape) Transparency is used (non-zero) for the text in Inkscape, but the package 'transparent.sty' is not loaded}%
    \renewcommand\transparent[1]{}%
  }%
  \providecommand\rotatebox[2]{#2}%
  \newcommand*\fsize{\dimexpr\f@size pt\relax}%
  \newcommand*\lineheight[1]{\fontsize{\fsize}{#1\fsize}\selectfont}%
  \ifx\svgwidth\undefined%
    \setlength{\unitlength}{113.47111416bp}%
    \ifx\svgscale\undefined%
      \relax%
    \else%
      \setlength{\unitlength}{\unitlength * \real{\svgscale}}%
    \fi%
  \else%
    \setlength{\unitlength}{\svgwidth}%
  \fi%
  \global\let\svgwidth\undefined%
  \global\let\svgscale\undefined%
  \makeatother%
  \begin{picture}(1,0.700154)%
    \lineheight{1}%
    \setlength\tabcolsep{0pt}%
    \put(0,0){\includegraphics[width=\unitlength,page=1]{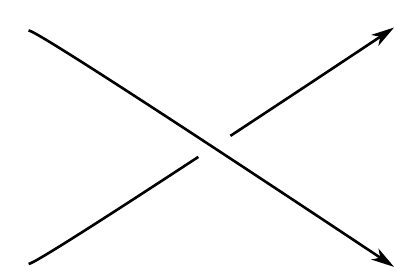}}%
    \put(-0.02199761,0.61980597){\makebox(0,0)[lt]{\lineheight{1.25}\smash{\begin{tabular}[t]{l}$x_1$\end{tabular}}}}%
    \put(-0.02199761,0.02494095){\makebox(0,0)[lt]{\lineheight{1.25}\smash{\begin{tabular}[t]{l}$x_2$\end{tabular}}}}%
    \put(1.01571146,0.61980597){\makebox(0,0)[lt]{\lineheight{1.25}\smash{\begin{tabular}[t]{l}$x_2'$\end{tabular}}}}%
    \put(1.01571146,0.02494095){\makebox(0,0)[lt]{\lineheight{1.25}\smash{\begin{tabular}[t]{l}$x_1'$\end{tabular}}}}%
  \end{picture}%
\endgroup%

    \caption{Biquandle elements $(x_2', x_1') = B(x_1, x_2)$ at a positive crossing.}
    \label{fig:biquandle-crossing-labels}
  \end{marginfigure}

  $X$-colored diagrams form a monoidal category denoted $\tang[X]$ whose
  \begin{description}
    \item[objects]
      are tuples
      \[
        ( (x_1, \epsilon_1), (x_2, \epsilon_2), \dots, (x_n, \epsilon_n) )
      \]
      of pairs of elements $x_i \in X$ and signs.
      We frequently drop the signs from the notation when they are clear.
    \item[morphisms] are $X$-colored tangle diagrams.
      Two diagrams can only be composed if their colors and orientations match.
    \item[tensor product]
      is given by vertical stacking, just like for tangle diagrams.
  \end{description}
\end{defn}

\begin{marginfigure}
\begingroup%
  \makeatletter%
  \providecommand\color[2][]{%
    \errmessage{(Inkscape) Color is used for the text in Inkscape, but the package 'color.sty' is not loaded}%
    \renewcommand\color[2][]{}%
  }%
  \providecommand\transparent[1]{%
    \errmessage{(Inkscape) Transparency is used (non-zero) for the text in Inkscape, but the package 'transparent.sty' is not loaded}%
    \renewcommand\transparent[1]{}%
  }%
  \providecommand\rotatebox[2]{#2}%
  \newcommand*\fsize{\dimexpr\f@size pt\relax}%
  \newcommand*\lineheight[1]{\fontsize{\fsize}{#1\fsize}\selectfont}%
  \ifx\svgwidth\undefined%
    \setlength{\unitlength}{139.24332333bp}%
    \ifx\svgscale\undefined%
      \relax%
    \else%
      \setlength{\unitlength}{\unitlength * \real{\svgscale}}%
    \fi%
  \else%
    \setlength{\unitlength}{\svgwidth}%
  \fi%
  \global\let\svgwidth\undefined%
  \global\let\svgscale\undefined%
  \makeatother%
  \begin{picture}(1,0.59943705)%
    \lineheight{1}%
    \setlength\tabcolsep{0pt}%
    \put(0,0){\includegraphics[width=\unitlength,page=1]{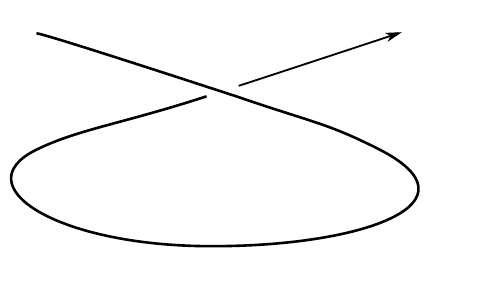}}%
    \put(0.0036189,0.52962032){\makebox(0,0)[lt]{\lineheight{1.25}\smash{\begin{tabular}[t]{l}$x$\end{tabular}}}}%
    \put(0.84644978,0.53396044){\makebox(0,0)[lt]{\lineheight{1.25}\smash{\begin{tabular}[t]{l}$x$\end{tabular}}}}%
    \put(0.35911169,0.01792614){\makebox(0,0)[lt]{\lineheight{1.25}\smash{\begin{tabular}[t]{l}$\alpha(x)$\end{tabular}}}}%
  \end{picture}%
\endgroup%

  \caption{Biquandle colorings at a kink.}
  \label{fig:RI-coloring}
\end{marginfigure}

When we think of biquandles as colorings for tangle diagrams, we reveal the motivation for each axiom:
\begin{enumerate}
  \item The colors need to be compatible with the \reidthree{} move.
  \item Because $B$ is a bijection, if we pick the colors on the left of a crossing we determine them on the right (and vice-versa.)
    This axiom says that, in addition, picking the colors on the top determines the colors on the bottom (and vice-versa.)
  \item While we allow the colorings to change in a \reidone-type diagram, they can only change in a fairly regular way, as in \cref{fig:RI-coloring}.
    (Relaxing this condition would give a birack.)
\end{enumerate}
Our motivating example of a biquandle is the shape biquandle, although technically speaking it is only partially defined.
Before addressing this, we give a simpler example.

\subsection{Conjugation quandles and the fundamental group}
\label{sec:wirtinger-presentation}
As discussed in \cref{sec:fundamental-groupoid}, every tangle diagram gives a Wirtinger presentation of the fundamental group of the complement, as shown in \cref{fig:wirtinger-presentation}.
From a biquandle perspective, we have ``input'' elements $w_1$ and $w_2$ and ''output'' elements $w_{2'} = w_1^{-1} w_2 w_1$ and $w_{1'} = w_1$.
\begin{defn}
  \label{def:conjugation-quandle}
  Let $G$ be a group.
  The \defemph{conjugation quandle}%
  \note{In the special case $B(g_1, g_2) = (g_2', g_1)$ where the overstrand does not change we call our structure a quandle, instead of a biquandle.}
  of $G$ is the biquandle on the set $G$ given by
  \[
    B(g_1, g_2) = (g_1^{-1} g_2 g_1, g_1).
  \]
  When we say a tangle diagram is colored by $G$, we mean it is colored by the conjugation quandle of $G$.
\end{defn}
\begin{prop}
  \label{thm:group-colors-exist}
  Let $T$ be a link with a diagram $D$.
  Then representations $\rho : \pi(T) \to G$ of its fundamental group are in bijection with colorings of $D$ by (the conjugation quandle of) $G$.
\end{prop}
\begin{proof}
  By assigning an element of $G$ to each segment of the diagram, we specify the image of our homomorphism $\rho$.
  Requiring them to satisfy the relations of the conjugation quandle is exactly the same as requiring them to sasisfy the relations of the Wirtinger presentation of $\pi(D) \iso \pi(T)$.
  Conversely, given a tangle and a representation $\rho : \pi(T) \to G$ we can pick an oriented diagram $D$ of $T$ and then the image of appropriate meridians of $T$ under $\rho$ gives a coloring of $D$ by $G$.
\end{proof}

The obvious way to get holonomy invariants of $G$-links would be to construct functors from $\tang[G]$ to a linear category, say the category $\modc H$ of modules over a Hopf algebra $H$.
\note{In our language, a Hopf $G$-coalgebra $H$ gives a model of the conjugation quandle of $G$ in the category $\modc H$.}
These correspond to objects that \citeauthor{Turaev2010} \cite{Turaev2010} calls \defemph{Hopf group-coalgebras}.
Unfortunately, we are only aware of examples of this type when $G$ is a discrete or finite group.
For Lie groups such as $G = \slg$ we need to take a different approach, which requires a more complicated description of the holonomy.

\subsection{Group factorizations}
The necessary generalization of the colorings of a tangle corresponds to our earlier generalization of Wirtinger presentation of the fundamental group to the fundamental groupoid.
Specifically, our goal is to describe representations $\pi(D) \to G$ in terms of $\Pi(D)$ by using the equivalence $\Phi$ of \cref{fig:path-factorization,thm:path-factorization}.
One way to do this is to split $G$ into upper and lower parts $G^{\pm}$ and assign each strand a pair of elements $g_i^{\pm} \in G^{\pm}$.
This gives a representation $\Pi(D) \to G^+ \times G^-$, and if we choose things appropriately pulling back along $\Phi$ will give a representation $\pi(D) \to G$.

Strictly speaking we do not use this perspective in our construction because the shape biquandle does not correspond to a group, as discussed below.
However, we include it as motivation for the shape coordinates.

\begin{defn}
  A \defemph{group factorization} is a triple $(G, \overline G, G^*)$ of groups, with $G$ a normal subgroup of $\overline G$, along with maps $g^+, g^- : G^* \to \overline G$ such that the map $\psi : G^* \to \overline G$
  \[
    \psi(a) = g^+(a) g^-(a)^{-1}
  \]
  restricts to a bijecton $G^* \to G$.
\end{defn}

\begin{ex}
  Set $\overline{G} = \glg$,
  \begin{align*}
    T
    &\defeq
    \left\{
      \begin{pmatrix}
        \kappa & -\epsilon \\
        0 & \kappa^{-1}
      \end{pmatrix}
    \right\} \subseteq \glg,
    \\
    T^*
    &\defeq
    \left\{
      \left(
      \begin{pmatrix}
        1 & 0 \\
        0 & \kappa
      \end{pmatrix}
      ,
      \begin{pmatrix}
        \kappa & \epsilon \\
        0 & 1
      \end{pmatrix}
      \right)
    \right\} \subseteq \glg \times \glg,
  \end{align*}
  and let $g^+, g^- : G^* \to \overline{G}$ be the inclusions of the first and second factors, respectively.
  Then $(T, T^*, \overline{G})$ is a group factorization, and the map $\psi$ acts by
  \[
    \psi:
      \left(
      \begin{pmatrix}
        1 & 0 \\
        0 & \kappa
      \end{pmatrix}
      ,
      \begin{pmatrix}
        \kappa & \epsilon \\
        0 & 1
      \end{pmatrix}
      \right)
      \mapsto
      \begin{pmatrix}
        \kappa & -\epsilon \\
        0 & \kappa^{-1}
      \end{pmatrix}.
  \]
\end{ex}
This group factorization corresponds to semi-cyclic invariants \cite{Geer2013}, but to get geometrically interesting representations, we need a more general notion.
Instead of requiring $\psi$ to be a bijection, we only require it to be \emph{generically} bijective.

\begin{defn}
  A \defemph{generic group factorization} is a triple $(G, \overline G, G^*)$ of groups, with $G$ a normal subgroup of $\overline G$, along with maps $g^+, g^- : G^* \to \overline G$ such that the map $\psi : G^* \to \overline G$
  \[
    \psi(a) = g^+(a) g^-(a)^{-1}
  \]
  restricts to a generic bijection $G^* \to U$.
\end{defn}

\begin{ex}
  \label{def:poisson-dual-sl2}
  Recall the \defemph{Poisson dual group} of $\slg$:
  \[
    \slg^* \defeq \left\{ \left(
      \begin{bmatrix}
        \kappa & 0 \\
        \phi & 1
      \end{bmatrix}, %
      \begin{bmatrix}
        1 & \epsilon \\
        0 & \kappa
      \end{bmatrix} %
  \right) \right\} \subseteq \glg \times \glg .
  \]
  Set $G = \slg$ and $\overline G = \glg \times \glg$, and let $g^+, g^- : G^* \to \overline G $ be the inclusions of the first and second factors.
  Then the map $\psi$ acts by
  \begin{equation}
    \label{eq:sl2factorizaton}
    \psi :
    \left(
      \begin{bmatrix}
        \kappa & 0 \\
        \phi & 1
      \end{bmatrix}, %
      \begin{bmatrix}
        1 & \epsilon \\
        0 & \kappa
      \end{bmatrix} %
    \right)
    \mapsto
    \begin{bmatrix}
      \kappa & - \epsilon \\
      \phi & (1- \epsilon\phi)/\kappa
    \end{bmatrix}
  \end{equation}
  The image of $\psi$ is the set $U$ of matrices with $1,1$ entry nonzero, which is a Zariski open dense subset of $\slg$, so we obtain a generic factorization of $\slg$.
\end{ex}

The factorization of $\slg$ in terms of $\slg^* = \operatorname{Spec} \cent_0$ is directly related to the braiding on $\qgrp$ given in \cref{ch:prelim}, and it can be used \cite{Blanchet2020,McPhailSnyder2020} to construct invariants of $\slg$-links.
The category of $\qgrp$-weight modules (where $\cent_0$ acts diagonalizably) is $\slg^*$-graded, but we would expect invariants of $\slg$-links to come from an $\slg$-graded category.
The factorization perspective explains why can use $\qgrp$-modules to construct these invariants, but as a consequence we have to replace the conjugation quandle of $\slg$ with a factorization in terms of $\slg^*$.

Strictly speaking we do not use group factorizations here, as mentioned earlier.
In order to compute the matrix coefficients of the braiding and use the geometric ideas of \cref{ch:shaped-tangles} we want to consider the subset
\[
  \operatorname{Sh} \defeq
  \left\{
    \left(
      \begin{bmatrix}
        a & 0 \\
        (a - 1/\lambda)/b & 1
      \end{bmatrix}
      ,
      \begin{bmatrix}
        1 & (a - \lambda)b \\
        0 & a
      \end{bmatrix}
    \right)
  \right\}
  \subseteq \slg^*
\]
corresponding to $\cent_0$-characters that pull back to center of $\weyl$.
This is \emph{not} a subgroup, so we do not have a group factorization.
Nonetheless, it is large enough that it still provides a factorization of the conjugation quandle of $\slg$, which is \emph{almost} the factorization associated to $\slg^*$.
The map $\psi$ acts by 
\[
  \psi :
  \left(
    \begin{bmatrix}
      a & 0 \\
      (a - 1/\lambda)/b & 1
    \end{bmatrix}
    ,
    \begin{bmatrix}
      1 & (a - \lambda)b \\
      0 & a
    \end{bmatrix}
  \right)
  \mapsto
  \begin{bmatrix}
    a & -(a - \lambda) b \\
    (a - 1/\lambda) /b  &  \lambda + \lambda^{-1} - a
  \end{bmatrix}
\]
and we think of the matrix on the right as the holonomy around a strand with shape $\chi = (a, b, \lambda)$.

In both cases, we have a coordinate system on a proper dense subset of $\slg$.
Because this is a proper subset, we need to use \defemph{generic} biquandles to describe it, as discussed in \cref{sec:generic-biquandles}.
This causes some technical difficulties, but because the set is dense we can work around them.

\subsection{Generic biquandles}
\label{sec:generic-biquandles}
Unfortunately our previous definiton of biquandle is too restrictive to describe shaped tangles.
The reason is that the shapes can fail to be defined.
For example, if $b_2 = \lambda_1 b_1$ at a positive crossing but $b_2' \ne b_1$, by \cref{eq:positive-a-relations} we have
\[
  a_1 = \frac{b_1 - \lambda_1 \beta_1}{b_{2'}-  b_1} = 0.
\]
Since we only allow $a_1 \ne 0$, the biquandle corresponding to the shapes is not defined at such a crossing.
To allow such examples we need the notion of a generically defined biquandle.

\begin{defn}
  Let $X$ be a topological space and let $\mathscr{D}$ be the set of open dense subsets of  $X$.%
  \note{In our examples, $X$ is always a subset of $\CC^n$ with the Zariski topology.}
  \begin{itemize}
    \item A \defemph{partial map} $f : A \to B$ is a map $f : A' \to B$ for $A' \subseteq A$.
    \item A property of (some) elements $x \in X$ is \defemph{generically true} if it holds for all $x \in Z$, where $Z \in \mathscr{D}$ is some open dense subset.
    \item
      A \defemph{generic bijection} $f : X \to X$ is  is a bijection $f : Z_1 \to Z_2$ for some $Z_1, Z_2 \in \mathscr{D}$ such that for every $Z \in \mathscr{D}$, the sets $f(Z \cap Z_1)$ and $f^{-1}(Z \cap Z_2)$ lie in $\mathscr{D}$.
  \end{itemize}
\end{defn}

\begin{defn}
  \label{def:generic-biquandle}
  A \defemph{generic biquandle} $(X,B)$ is a biquandle in which $B$ and the sideways map $\sidebraid$ are only partially defined, and in which the following hold:
  \begin{enumerate}
    \item For any $x_1, x_2, x_{1'}, x_{2'} \in X$, the following are equivalent:
      \begin{enumerate}
        \item $(x_{2'}, x_{1'}) = B(x_1, x_2)$
        \item $(x_1, x_2) = B^{-1}(x_{2'}, x_{1'})$
        \item $(x_{1'}, x_2) = \sidebraid(x_{2'}, x_1)$
        \item $(x_{2'}, x_1) = \sidebraid^{-1}(x_{1'}, x_2)$
      \end{enumerate}
  \item For any $x_0 \in X$, the map $x \mapsto B_1(x_0, x)$ is a generic bijection, and similarly for the seven other maps given by the components of $B$, $B^{-1}$, $\sidebraid$, and $\sidebraid^{-1}$.
  \end{enumerate}

  For any generic biquandle $X$ we define the category $\tang[X]$ of $X$-colored diagrams just as before.
  However, we now restrict the Reidemeister moves: they are only allowed when they do not introduce crossings for which the map $B$ is not defined.
\end{defn}

\subsection{The shape biquandle}
\begin{defn}
  \label{def:shape-biquandle}
  Let $\shape = \CC^\times \times \CC^\times \times \CC^\times$ be the set of triples of nonzero complex numbers.
  We usually denote elements of $\shape$ as $\chi = (a, b, \lambda)$.
  The \defemph{shape biquandle} is the generic biquandle on $\shape$ given by
  \begin{align*}
    a_{1'}
      &=
      a_1 A^{-1}
      &
      a_{2'}
      &=
      a_2 A
      \\
      b_{1'}
      &=
      b_2
      \left(
        \frac{\lambda_1}{\lambda_2} - a_2 \left( \lambda_1 - \frac{b_2}{b_1} \right)
      \right)^{-1}
      &
      b_{2'}
      &=
      b_1
      \left(
        1 - a_1^{-1}\left( \lambda_1 - \frac{b_2}{b_1} \right)
      \right)
      \\
      \lambda_{1'}
      &= \lambda_1
      &
      \lambda_{2'}
      &= \lambda_2
  \end{align*}
  where $B(\chi_1, \chi_2) = (\chi_{2'}, \chi_{1'})$, $\chi_i = (a_i, b_i, \lambda_i)$, and
  \[
    A = 1 + \frac{b_1}{b_2} (a_1 - \lambda_1)(1 - a_2^{-1} \lambda_2^{-1}).
  \]
  Here the $\chi_i$ correspond to segments of a knot diagram as in \cref{fig:shaped-positive-crossing}.

  We refer to the elements $\chi$ of $X$ as \defemph{shapes}, and write $\tangsh$ for the category of \defemph{shaped tangles}, that is (oriented framed) tangles colored by the shape biquandle.
\end{defn}

The rules for $B$ are equivalent to the gluing equations \eqref{eq:positive-a-relations} and \eqref{eq:negative-a-relations}: we can derive them by solving for the primed variables in terms of the unprimed variables.
It is similarly possible to compute the sideways map $\sidebraid$, but we don't need the formula.

We do, however, compute the twist map $\alpha$.
For any shape $\chi= (a, b, \lambda)$, the unique shape $\alpha(\chi)$ with $B(\chi, \alpha(\chi)) = (\chi, \alpha(\chi))$ is given by
\[
  \alpha(a, b, \lambda) = 
  \left(
    \lambda + \lambda^{-1} - a^{-1},
    \lambda b,
    \lambda
  \right).
\]

\begin{thm}
  The shape biquandle is a generic biquandle.
\end{thm}
\begin{proof}
  We first show that $B$ satisfies the colored braid relation.
  We could check this directly, although this is somewhat tedious.
  A more systematic approach is to use the fact that the map $B$ is determined by the action of the outer $S$-matrix $\Smat$ on central characters $\chi : Z(\weyl) \to \CC$.
  Because $\Smat$ satisfies the colored braid relation and acts bijectively on characters, $B$ does as well.

  To see that the shape biquandle is generic, observe that $\shape$ is a Zariski open dense subset of $\CC^3$, i.e.\@ an element of $\mathscr{D}$.
  The maps on the parameters $a, b, \lambda$ defining the shape biquandle are rational functions, so the maps defining $B, \sidebraid$, and their inverses are rational maps, so they take elements of $\mathscr{D}$ to elements of $\mathscr{D}$.
  It follows that they are generic bijections as required.
\end{proof}

It is not hard to extend this definition to the case of extended $\slg$-links: we just have to keep track of some extra $\nr$th roots.
\begin{defn}
  An \defemph{extended shape} is a shape $\chi = (a, b, \lambda)$ along with a choice $\mu$ of $\nr$th root $\mu^\nr = \lambda$ of the eigenvalue $\lambda$.
  We call $\mu$ a \defemph{fractional eigenvalue}.
  We consider only diagrams of admissible links (as defined in \cref{sec:our-results}), so we assume that if $\lambda = \pm 1$ we have $\mu = - \xi^{-1}$, so $\lambda = (-1)^{\nr +1}$.
  Because the shape biquandle permutes the eigenvalues $\lambda_i$ extended shapes also form a biquandle, which we call the \defemph{extended shape biquandle}.
\end{defn}
The nonabelian quantum dilogarithm is defined via a representation of the extended shape biquandle constructed in \cref{ch:algebras}.

\subsection{Shape presentations of \texorpdfstring{$\slg$}{SL\_2(C)}-links}
We can now make explicit what it means to represent a $\slg$-link in terms of a shaped link diagram.
\begin{defn}
  \label{def:holonomy-functor}
  Every shaped tangle represents a $\slg$-colored tangle, in the sense that there is a functor
  \[
    \Psi : \tangsh \to \tang[\slg].
  \]
  If $D$ is a shaped tangle diagram, the holonomy representation (\cref{def:diagram-holonomy-rep}) of $D$ gives a representation $\rho : \pi(D) \to \slg$.
  By \cref{thm:group-colors-exist}, there is a $\slg$-coloring $D'$ of the underlying diagram of $D$, and we define $\Psi(D) = D'$.
\end{defn}
$\Psi$ is a special case of the functor $\mathcal{Q}$ defined in \cite[Theorem 3.9]{Blanchet2020}.
We do not actually need the functoriality of $\Psi$,  so we do not prove it.

If $D$ is the diagram of a $\slg$-link $L$, by \cref{thm:group-colors-exist} we can color the arcs of $D$ by elements $g_i \in \slg$ describing the images of the Writinger generators under the holonomy representation.
We can then solve for shape parameters that give the same holonomy.
In more abstract language we are computing an inverse image of the diagram under the functor $\Psi$.

For example, in \cref{fig:solving-for-shapes} the {\color{accent} gold} path wrapping around segment $1$ has holonomy $g_1 = \rho(w_1) \in \slg$, where $w_1$ is the corresponding Wirtinger generator.
We need to choose the shape $\chi_1$ so that
\[
  \psi(\chi_1) = 
  g^+(\chi_1) g^-(\chi_1)^{-1}
  =
  \begin{bmatrix}
    a_1 & -(a_1 - \lambda_1) b_1 \\
    (a_1 - 1/\lambda_1) /b_1  &  \lambda_1 + \lambda^{-1}_1 - a_1
  \end{bmatrix}
  =
  g_1.
\]
Once we choose $a_1$, $b_1$, and $\lambda_1$, we need to choose $\chi_2$ so that
\[
  g^+(\chi_1) g^+(\chi_2) g^-(\chi_2)^{-1} g^+(\chi_1)^{-1} = g_2.
\]
These formulas come from the equivalence between $\pi(D)$ and $\Pi(D)$ described in \cref{fig:path-factorization}.
\begin{figure}
  \centering
\begingroup%
  \makeatletter%
  \providecommand\color[2][]{%
    \errmessage{(Inkscape) Color is used for the text in Inkscape, but the package 'color.sty' is not loaded}%
    \renewcommand\color[2][]{}%
  }%
  \providecommand\transparent[1]{%
    \errmessage{(Inkscape) Transparency is used (non-zero) for the text in Inkscape, but the package 'transparent.sty' is not loaded}%
    \renewcommand\transparent[1]{}%
  }%
  \providecommand\rotatebox[2]{#2}%
  \newcommand*\fsize{\dimexpr\f@size pt\relax}%
  \newcommand*\lineheight[1]{\fontsize{\fsize}{#1\fsize}\selectfont}%
  \ifx\svgwidth\undefined%
    \setlength{\unitlength}{171.68959808bp}%
    \ifx\svgscale\undefined%
      \relax%
    \else%
      \setlength{\unitlength}{\unitlength * \real{\svgscale}}%
    \fi%
  \else%
    \setlength{\unitlength}{\svgwidth}%
  \fi%
  \global\let\svgwidth\undefined%
  \global\let\svgscale\undefined%
  \makeatother%
  \begin{picture}(1,0.77655278)%
    \lineheight{1}%
    \setlength\tabcolsep{0pt}%
    \put(0.19856942,0.24769805){\makebox(0,0)[lt]{\lineheight{1.25}\smash{\begin{tabular}[t]{l}$g_2$\end{tabular}}}}%
    \put(0,0){\includegraphics[width=\unitlength,page=1]{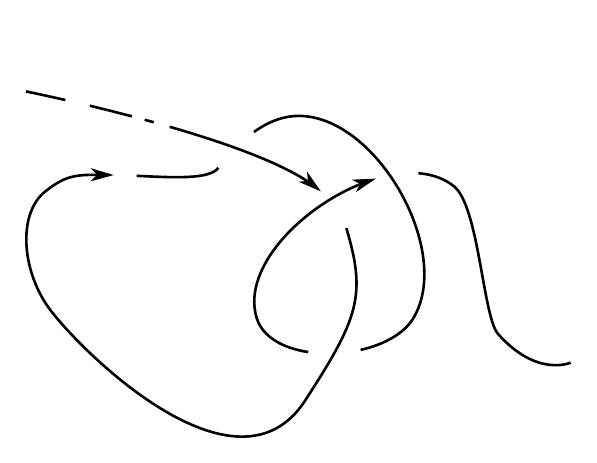}}%
    \put(0.13301686,0.68151395){\makebox(0,0)[lt]{\lineheight{1.25}\smash{\begin{tabular}[t]{l}$g_1$\end{tabular}}}}%
    \put(0,0){\includegraphics[width=\unitlength,page=2]{solving-for-shapes.pdf}}%
  \end{picture}%
\endgroup%

  \caption{Solving for the shapes $\chi_1, \chi_2$ in terms of the holonomy $g_1, g_2$.}
  \label{fig:solving-for-shapes}
\end{figure}
Continuing in this manner, it is straightforward if tedious to compute the shapes $\chi_i$ of every segment of the diagram.%
\note{
  As mentioned before, for geometrically interesting representations this process becomes much easier: instead of solving for the shape parameters in terms of $\rho$, we determine the shape parameters from gluing equations and thereby determine $\rho$.
}

Since not every matrix in $\slg$ lies in the image of $\psi$, some representations $\rho$ might not be expressible in terms of shapes.
However, we can always conjugate away from these singular cases:

\begin{thm}
  \label{thm:links-presentable}
  Let $D_0$ be a diagram colored by the conjugation quandle of $\slg$ corresponding to an admissible $\slg$-link $L$.
  Then there is a shaped diagram $D$ with $\Psi(D) = D_0$, i.e.\@ whose holonomy representation is the representation of $D_0$ induced by the colors.
\end{thm}
\begin{marginfigure}
\begingroup%
  \makeatletter%
  \providecommand\color[2][]{%
    \errmessage{(Inkscape) Color is used for the text in Inkscape, but the package 'color.sty' is not loaded}%
    \renewcommand\color[2][]{}%
  }%
  \providecommand\transparent[1]{%
    \errmessage{(Inkscape) Transparency is used (non-zero) for the text in Inkscape, but the package 'transparent.sty' is not loaded}%
    \renewcommand\transparent[1]{}%
  }%
  \providecommand\rotatebox[2]{#2}%
  \newcommand*\fsize{\dimexpr\f@size pt\relax}%
  \newcommand*\lineheight[1]{\fontsize{\fsize}{#1\fsize}\selectfont}%
  \ifx\svgwidth\undefined%
    \setlength{\unitlength}{142.77801132bp}%
    \ifx\svgscale\undefined%
      \relax%
    \else%
      \setlength{\unitlength}{\unitlength * \real{\svgscale}}%
    \fi%
  \else%
    \setlength{\unitlength}{\svgwidth}%
  \fi%
  \global\let\svgwidth\undefined%
  \global\let\svgscale\undefined%
  \makeatother%
  \begin{picture}(1,0.6189484)%
    \lineheight{1}%
    \setlength\tabcolsep{0pt}%
    \put(0,0){\includegraphics[width=\unitlength,page=1]{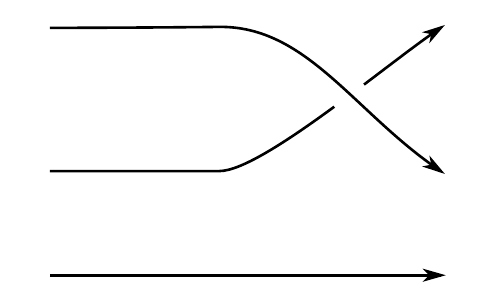}}%
    \put(0.04832677,0.56279219){\makebox(0,0)[lt]{\lineheight{1.25}\smash{\begin{tabular}[t]{l}$1$\end{tabular}}}}%
    \put(0.04832677,0.27388216){\makebox(0,0)[lt]{\lineheight{1.25}\smash{\begin{tabular}[t]{l}$\lambda$\end{tabular}}}}%
    \put(0.04832677,0.0637658){\makebox(0,0)[lt]{\lineheight{1.25}\smash{\begin{tabular}[t]{l}$\lambda^2$\end{tabular}}}}%
    \put(0.91505685,0.0637658){\makebox(0,0)[lt]{\lineheight{1.25}\smash{\begin{tabular}[t]{l}$\lambda^2$\end{tabular}}}}%
    \put(0.91505685,0.27388216){\makebox(0,0)[lt]{\lineheight{1.25}\smash{\begin{tabular}[t]{l}$\lambda$\end{tabular}}}}%
    \put(0.91505685,0.56279219){\makebox(0,0)[lt]{\lineheight{1.25}\smash{\begin{tabular}[t]{l}$1$\end{tabular}}}}%
  \end{picture}%
\endgroup%

  \caption{Shaping of a braid diagram corresponding to an abelian holonomy representation with eigenvalue $\lambda$.}
  \label{fig:abelian-coloring}
\end{marginfigure}
\begin{proof}
  Write $\rho$ for the holonomy representation of $L$ and fix a diagram $D$ of $L$.
  The representation $\rho$ decorates the strands of $D$ with shapes, but it is possible that 
  We consider three cases:
  \begin{enumerate}
    \item $\rho$ is irreducible,
    \item $\rho$ is reducible and the meridians have eigenvalues $\lambda \ne \pm 1$,
    \item $\rho$ is reducible and the meridians have eigenvalues $\lambda = \pm 1$.
  \end{enumerate}
  If $\rho$ is irreducible, we can conjugate it so that the image of every meridian has nonzero entries.
  
  We assume that $L$ is admissible, so $\rho$ is a completely reducible representation.
  As such, if $\rho$ is reducible it is conjugate to a representation of the form
  \[
    \begin{bmatrix}
      \lambda_i & 0 \\
      0 & \lambda_i^{-1}
    \end{bmatrix}.
  \]
  Consider the shapes $\chi_k = (\lambda_i, \lambda_i^k, \lambda_i)$ for integers $k$.
  It is not hard to check that $B(\chi_k, \chi_{k+1}) = (\chi_k, \chi_{k+1})$, and we can use this to find a coloring $D$ of the diagram with the shapes $\chi_k$.

  This is easiest to see using a braid diagram, as in \cref{fig:abelian-coloring}.
  When every component has eigenvalue $\lambda$, we assign a strand at ``depth'' $k$ the parameter $b = \lambda^{k-1}$.
  This works even if different components have different eigenvalues.
  This coloring is preserved by Reidemeister moves, which shows how to find such a coloring of a general diagram $D$:
  transform $D$ to a braid diagram%
  \note{
    Since $D$ is a tangle diagram, strictly speaking here we mean transforming it to a braid-like diagram, possibly with some through-strands or strands that return to the same side.
    Really all that is important is that the strands of $D$ all lie at distinct heights in the diagram and only change heights at crossings.
  }
  pick an abelian coloring as indicated, then transform back to $D$.

  If $\lambda_i \ne 1$, we can check that the holonomy around any meridian of the diagram will be conjugate to a matrix
  \[
    \begin{bmatrix}
      \lambda_i & 0 \\
      b & \lambda_i^{-1}
    \end{bmatrix}
  \]
  for some \emph{fixed} $b$, and because $\lambda_i \ne \lambda_i^{-1}$ these matrices are conjugate to diagonal matrices as required.
  If $\lambda_i = \pm 1$, then the any completely reducible representation is the constant representation $\rho(x) = \pm I_2$.
  In this case a coloring by characters $(\pm 1, (-1)^k, \pm 1)$ works and has holonomy $\rho$.
\end{proof}

\section{Representations of biquandles}
\label{sec:biquandle-reps}
Now that we know what a tangle diagram colored by a biquandle is, we can discuss how to map them to pivotal categories.

\subsection{Pivotal categories}
\begin{defn}
  \label{def:pivotal-cat}
  Let $\catl C$ be a (strict)%
  \note{As is customary, we ignore the distinction between a monoidal and strict monoidal category by the ritual invocation of Mac Lane's coherence theorem \cite[Chapter 7]{Lane1997}.  }
  monoidal category with product $\otimes$ and unit object~$\unit$.
  Under $\otimes$, the space $\Bbbk = \End_{\catl C}(\unit)$ becomes a commutative monoid, and the hom spaces of $\catl C$ are all $\Bbbk$-bimodules.
  We assume that the left and right actions of $\Bbbk$ agree.
  We call $k$ the \defemph{scalars}%
  \note{
    By defining the scalars in this roundabout way we capture both categories like $\tangsh$ (trivial scalars) and $\modc A$ for $A$ a Hopf algebra over $\CC$ (scalars are $\CC$).
  }
  of $\catl C$.

  We say that $\catl C$ is a \defemph{pivotal} category if it has duals.
  In more detail, we require that for every object $V$ of $\catl C$, there is a \defemph{dual object} $V^*$ and morphisms
  \begin{align*}
      &\evup V : V \otimes V^* \to \Bbbk
      &
      &\coevup V : \Bbbk \to V \otimes V^*
      \\
      &\evdown V : V^* \otimes V \to \Bbbk
      &
      &\coevdown V : \Bbbk \to V^* \otimes V
  \end{align*}
  coherent in the usual way.
  For example, we can use these define maps $\phi_X : X \to X^{**}$ that form a monoidal natural isomorphism; see \cref{sec:modified-invariants} and \cite{Geer2013a} for details.%
  \note{
  A pivotal category is a slight weakening of a spherical category \cite{nlab:pivotal_category,Barrett1999} where the twists need not be trivial.
}
\end{defn}
The arrows decorating the evaluation and coevaluation morphism correspond to the orientations of the cups and caps of \cref{fig:tangle-generators}.
The shaped tangle category $\tangsh$ is a pivotal category with scalars the trivial monoid.
In \cref{sec:modified-traces} we give a pivotal structure on  $\modc \qgrp$.

We get biquandle representations by an extension of the Reshetikhin-Turaev construction.
In more straightforward language, we construct them by specifying their values on the oriented generators corresponding to those shown in \cref{fig:tangle-generators}.
The cups and caps come from the pivotal structure on the target category $\catl C$, so the key ingredient is the braiding.

\subsection{The RT functor}

\begin{marginfigure}
  \centering
  
  \caption{The right positive twist $\theta_x^R~:~V_x~\to~V_x$.}
  \label{fig:twist-right}
\end{marginfigure}
\begin{marginfigure}
  \centering
\begingroup%
  \makeatletter%
  \providecommand\color[2][]{%
    \errmessage{(Inkscape) Color is used for the text in Inkscape, but the package 'color.sty' is not loaded}%
    \renewcommand\color[2][]{}%
  }%
  \providecommand\transparent[1]{%
    \errmessage{(Inkscape) Transparency is used (non-zero) for the text in Inkscape, but the package 'transparent.sty' is not loaded}%
    \renewcommand\transparent[1]{}%
  }%
  \providecommand\rotatebox[2]{#2}%
  \newcommand*\fsize{\dimexpr\f@size pt\relax}%
  \newcommand*\lineheight[1]{\fontsize{\fsize}{#1\fsize}\selectfont}%
  \ifx\svgwidth\undefined%
    \setlength{\unitlength}{139.24332333bp}%
    \ifx\svgscale\undefined%
      \relax%
    \else%
      \setlength{\unitlength}{\unitlength * \real{\svgscale}}%
    \fi%
  \else%
    \setlength{\unitlength}{\svgwidth}%
  \fi%
  \global\let\svgwidth\undefined%
  \global\let\svgscale\undefined%
  \makeatother%
  \begin{picture}(1,0.59943705)%
    \lineheight{1}%
    \setlength\tabcolsep{0pt}%
    \put(0,0){\includegraphics[width=\unitlength,page=1]{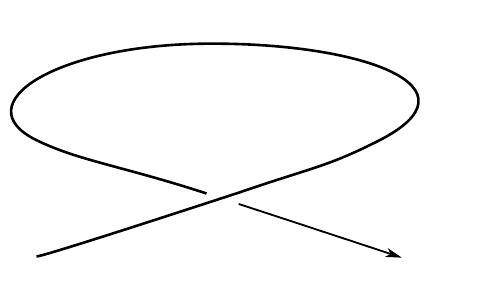}}%
    \put(0.0036189,0.06652875){\makebox(0,0)[lt]{\lineheight{1.25}\smash{\begin{tabular}[t]{l}$x$\end{tabular}}}}%
    \put(0.84644978,0.07086888){\makebox(0,0)[lt]{\lineheight{1.25}\smash{\begin{tabular}[t]{l}$x$\end{tabular}}}}%
    \put(0.35911169,0.53396044){\makebox(0,0)[lt]{\lineheight{1.25}\smash{\begin{tabular}[t]{l}$\alpha^{-1}(x)$\end{tabular}}}}%
  \end{picture}%
\endgroup%

  \caption{The left positive twist $\theta_x^L~:~V_x~\to~V_x$.}
  \label{fig:twist-left}
\end{marginfigure}

\begin{defn}
  \label{def:yang-baxter-model}
  Let $\catl C$ be a pivotal category and $(X,B)$ a biquandle.
  A \defemph{model} of $X$ in $\catl C$ is a family $\left\{V_x \right\}_{x \in X}$ of objects indexed by $X$ and a family of \defemph{braidings}, isomorphisms
  \[
    S_{x_1,x_2} : V_{x_1} \otimes V_{x_2} \to V_{x_{2'}} \otimes V_{x_{1'}}
  \]
  where $B(x_1, x_2) = (x_{2'}, x_{1'})$.
  The braidings must satisfy the colored braid relation
  \[
    (S_{-,-} \otimes \id_{-})
    (\id_{-} \otimes S_{-,-})
    (S_{-,-} \otimes \id_{-})
    =
    (\id_{-} \otimes S_{-,-})
    (S_{-,-} \otimes \id_{-})
    (\id_{-} \otimes S_{-,-})
  \]

  We consider two more technical conditions on our biquandle models.
  For $(x_{2'}, x_{1'}) = B(x_1, x_2)$, the model induces the maps 
  \begin{align*}
    &s_L^+(x_{2'}, x_1) : V_{x_{2'}} \otimes V_{x_1} \to V_{x_{1'}} \otimes V_{x_2}
    \\
    &s_R^-(x_{2'}, x_1) : V_{x'1} \otimes V_{x_2} \to V_{x_{2'}} \otimes V_{x_1}
  \end{align*}
  shown in \cref{fig:sideways-braidings}, which we call \defemph{sideways braidings}.
  The \defemph{positive twists} are shown in \cref{fig:twist-right,fig:twist-left}.
  We say that a model
  \begin{itemize}
    \item is \defemph{sideways invertible} if the sideways braidings $s_L^+$ and $s_R^-$ are inverses, and
    \item \defemph{induces a twist} if $\theta_x^R = \theta_x^L$ for all $x \in X$.
  \end{itemize}

\end{defn}

\begin{figure}
  \centering
  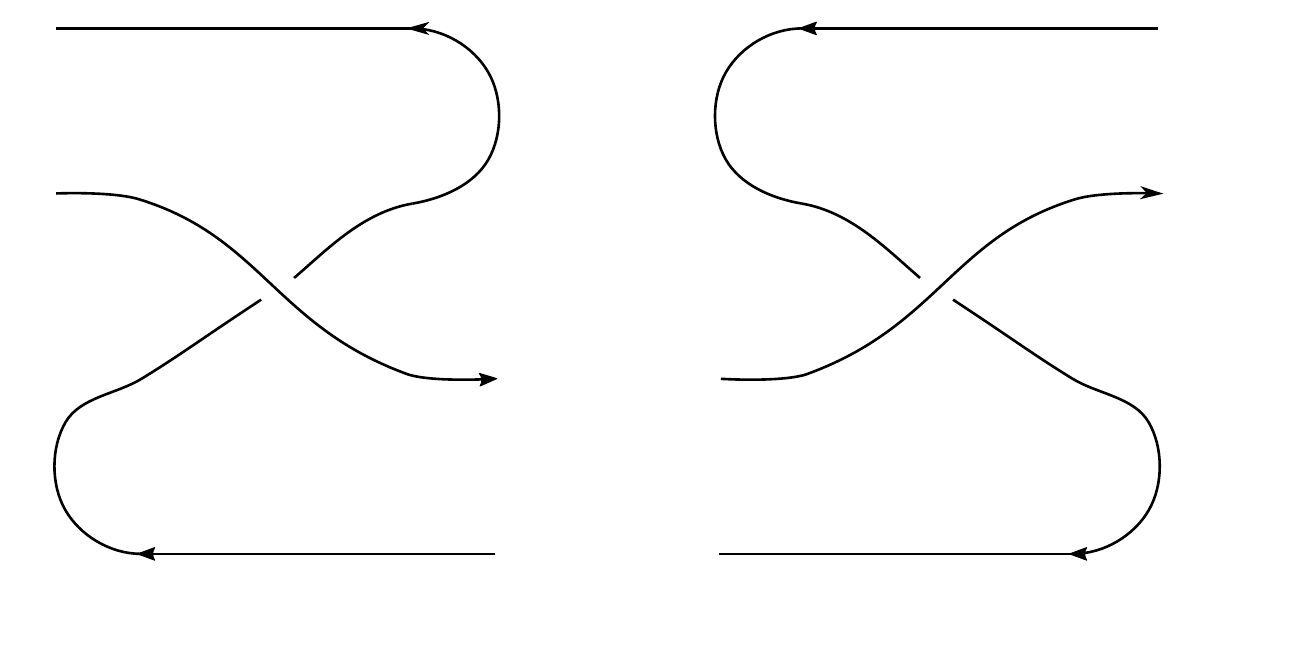
  \caption{Sideways braidings.}
  \label{fig:sideways-braidings}
\end{figure}
We can think of the twist condition as requiring the image of the model to be spherical, and if a model induces a twist the analogous negative twists agree as well.
Sideways invertibility is a sort of compatibility condition between taking duals and the braidings: even if the strands of a crossing aren't oriented exactly as \cref{fig:tangle-generators} we can still define an appropriate braiding.

\begin{thm}
  \label{thm:rep-gives-functor}
  A \defemph{representation} of a biquandle $(X,B)$ in a pivotal category $\catl C$ is a model $(V,S)$ of $(X,B)$ in $\catl C$ which is sideways invertible and induces a twist.
  Given such a representation, there is a unique monoidal functor
  \[
    \mathcal{F} : \tang[X] \to \catl C
  \]
  sending generators of $\tang[X]$ to the analogous morphisms
  \begin{align*}
  &\mathcal{F}(\coevup{x}) = \coevup{V_x}
  &
  &\mathcal{F}(\evup{x}) = \evup{V_x}
  &
  &\mathcal{F}(\sigma_{x_1,x_2}) = S_{x_1, x_2}
  \\
  &\mathcal{F}(\coevdown{x}) = \coevdown{V_x}
  &
  &\mathcal{F}(\evdown{x}) = \evdown{V_x}
  &
  &\mathcal{F}(\sigma_{x_{2'},x_{1'}}^{-1}) = S_{x_{2'}, x_{1'}}^{-1}
  \end{align*}
  In particular, $\mathcal{F}$ respects the colored \reidtwo{} and \reidthree{} relations.
\end{thm}
\begin{proof}
  The proof is basically the same as in the case where $X$ is trivial.
  The details are given in \cite[Theorem 4.2]{Blanchet2020}.
\end{proof}

We frequently blur the distinction between the functor induced by a representation and the representation itself.

\subsection{The adjoint representation}
\label{sec:adjoint-rep}

For now, we can give a simple example of a representation
Let $\chi$ be an extended shape, which we can think of as a central character of $\weyl$.
By pulling back along the embedding $\phi : \qgrp \to \weyl$ we can think of $\chi$ as a central character of $\qgrp$.
Writing $\ker \chi$ for the ideal generated by the kernel of $\chi$, we see that $\qgrp / \ker \chi$ is a $\qgrp$-module.
It is not simple, but it is a simple $(\qgrp,\qgrp)$  \emph{bimodule}.

\begin{ex}
  The \defemph{adjoint} model of the extended shape biquandle is given by assigning a character $\chi$ the module $\qgrp / \chi$ and a crossing $(\chi_1, \chi_2) \to (\chi_{2'}, \chi_{1'})$ to the automorphism
  \[
    \mathcal{S} : \qgrp/\ker \chi_1 \otimes \qgrp/\ker \chi_2 \to \qgrp /\ker \chi_{2'} \otimes \qgrp / \ker \chi_{1'}.
  \] 
  It is not hard to show that this model is a representation taking values in the category of $(\qgrp, \qgrp)$-bimodules.
\end{ex}
In the terminology of \cref{def:abs-simple-regular}, this representation is regular and absolutely simple,%
\note{
  These technical conditions (discussed later) are needed to prove that the invariant associated to a representation is gauge-independent.
}
since $\qgrp /\ker \chi$ is a simple bimodule.
The original Kashaev-Reshetikhin invariant \cite{Kashaev2005} is defined in terms of this representation, which is discussed in more detail in \cite{Chen2019,McPhailSnyderUnpub2}.

One advantage of the adjoint representation is that it is straightforward to define and we avoid the scalar ambiguities in the braiding that cause major problems in \cref{ch:algebras}.
However, the topological interpretation of the adjoint invariant is less clear, while the quantum dilogarithm can be seen either as a generalization of the colored Jones polynomial or as a torsion.

We can identify $\qgrp / \ker \chi$ with the endomorphism algebra of a simple $\qgrp$-module with character $\chi$ (which is defined explicitly in \cref{def:weyl-irrmod}).
For this reason, we might expect that the adjoint representation is the double of the nonabelian quantum dilogarithm, as is the case for the usual RT construction.
However, this is not the case: the holonomy quantum double is a more elaborate construction given in \cref{ch:doubles}.%

\section{Modified traces in pivotal categories}
\label{sec:modified-traces}
To obtain nonzero invariants from the non-semisimple category $\modc \qgrp$ we need to use \defemph{modified traces}.
We define these axiomatically and explain their use in this section.
For more on the construction of modified traces see \cref{ch:modified-traces}.

\subsection{\texorpdfstring{$\modc \qgrp$}{O xi-mod} as a pivotal category}
$\qgrp$ is a pivotal Hopf algebra with pivot $K^{\nr -1}$.
This means that the category $\modc \qgrp$ of finite-dimensional $\qgrp$-modules becomes a pivotal category by setting $V^* = \hom_\CC(V, \CC)$ and
\begin{align*}
  \coevup V &: \CC \to V \otimes V^* & &1 \mapsto \sum_{j} v_j \otimes v^j
  \\
  \coevdown V &: \CC \to V^* \otimes V & &1 \mapsto \sum_{j} v^j \otimes K^{\nr - 1} \cdot v_j
  \\
  \evup V &: V^* \otimes V \to \CC & &v \otimes f \mapsto f(v)
  \\
  \evdown V &: V \otimes V^* \to \CC & &f \otimes v \mapsto f(K^{1-\nr} \cdot v)
\end{align*}
where $\{v_j\}$ is any basis of $V$ and $\{v^j\}$ is the dual basis of $V^*$.

If $f : V \to V$ is an endomorphism of  $\modc \qgrp$, then the (right) \defemph{quantum trace} of $f$ is%
\note{
  Recall that maps in tangle categories are read left-to-right!
}
\[
  \qtr f \defeq \coevup V (f \otimes \id_{V^*}) \evdown V \in \hom(\CC, \CC) = \CC.
\]
\begin{marginfigure}
  \centering
\begingroup%
  \makeatletter%
  \providecommand\color[2][]{%
    \errmessage{(Inkscape) Color is used for the text in Inkscape, but the package 'color.sty' is not loaded}%
    \renewcommand\color[2][]{}%
  }%
  \providecommand\transparent[1]{%
    \errmessage{(Inkscape) Transparency is used (non-zero) for the text in Inkscape, but the package 'transparent.sty' is not loaded}%
    \renewcommand\transparent[1]{}%
  }%
  \providecommand\rotatebox[2]{#2}%
  \newcommand*\fsize{\dimexpr\f@size pt\relax}%
  \newcommand*\lineheight[1]{\fontsize{\fsize}{#1\fsize}\selectfont}%
  \ifx\svgwidth\undefined%
    \setlength{\unitlength}{108bp}%
    \ifx\svgscale\undefined%
      \relax%
    \else%
      \setlength{\unitlength}{\unitlength * \real{\svgscale}}%
    \fi%
  \else%
    \setlength{\unitlength}{\svgwidth}%
  \fi%
  \global\let\svgwidth\undefined%
  \global\let\svgscale\undefined%
  \makeatother%
  \begin{picture}(1,0.51999998)%
    \lineheight{1}%
    \setlength\tabcolsep{0pt}%
    \put(0,0){\includegraphics[width=\unitlength,page=1]{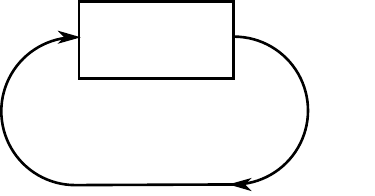}}%
    \put(0.40046397,0.38668968){\makebox(0,0)[lt]{\lineheight{1.25}\smash{\begin{tabular}[t]{l}$f$\end{tabular}}}}%
    \put(0.85435272,0.19833155){\makebox(0,0)[lt]{\lineheight{1.25}\smash{\begin{tabular}[t]{l}$V$\end{tabular}}}}%
  \end{picture}%
\endgroup%

  \caption{The right trace of a map $f~:~V~\to~V$.}
  \label{fig:right-trace}
\end{marginfigure}
The left trace is $\qtr f \defeq \coevdown V (\id_{V^*} \otimes f) \evup V$, which agrees with the right trace because $S(K^{\nr -1}) = K^{1 - \nr}$ (where $S$ is the antipode).
The \defemph{quantum dimension} of an object $V$ is $\qtr \id_V$.

\begin{lem}
  The quantum dimensions of the standard modules $\irrmod{\chi}$ of \cref{def:weyl-irrmod} are all zero, so the quantum trace of any endomorphism
  \[
    f : \irrmod{\chi_1} \otimes \cdots \otimes \irrmod{\chi_n} \to  \irrmod{\chi_1} \otimes \cdots \otimes \irrmod{\chi_n}
  \]
  is zero as well.
\end{lem}
\begin{proof}
  Let $\{v_m\}$ be the usual basis of $\irrmod{\chi} = \irrmod{\alpha, \beta, \mu}$.
  Then the quantum dimension is
  \[
    1 \mapsto \sum_{m} v_m v^m \mapsto \sum_{m} v_m K^{1 - \nr} v^m
    = \alpha^{\nr - 1} \sum_{m} v_m v^{m - 1} \mapsto 
    \alpha^{\nr - 1} \sum_{m} v^{m-1} (v_{m}) = 0.
  \]
  To see that traces of morphisms vanish, write
  \[
    \qtr f = \qtr\left( \ptr_{\irrmod{\chi_2} \otimes \cdots \otimes \irrmod{\chi_n}} f \right)
  \]
  where $g = \ptr_{\irrmod{\chi_2} \otimes \cdots \otimes \irrmod{\chi_n}} f$ is the partial trace of $f$.
  (See \cref{def:modified-trace}.)
  Since $g \in \End_{\catl C}(\irrmod{\chi_1})$ is an endomorphism of a simple object, we must have
  \[
    g = \lambda \id_{\irrmod{\chi_1}}
  \]
  for some scalar $\lambda$, so
  \[
    \qtr f = \qtr g = \lambda \qtr \id_{\irrmod{\chi_1}} = 0.\qedhere
  \]
\end{proof}

\subsection{Modified traces}
As a consequence, we cannot use $\qtr$ to define useful link invariants for $\modc \qgrp$.
However, there is a workaround.
Let $f : V \to V$ be an endomorphism of a simple%
\note{
  The scalars of $\CC$ as defined in \cref{def:pivotal-cat} might not form an algebraically closed field, a field, or even a ring.
  In practice, we will restrict our attention to examples where the module assigned to a strand is absolutely simple (see \cref{def:abs-simple-regular}) so that it satisfies Schur's Lemma by definition.
}
object $V$ in some $\CC$-pivotal category $\catl C$.
By Schur's Lemma there is a scalar $\langle f \rangle$ with
\[
  \langle f \rangle \id_V = f.
\]

Suppose $\mathcal{F}$ is the functor induced by a representation of a biquandle $X$ in $\catl C$ and that $\mathcal{F}$ maps strands to simple objects, and let $D$ be an $X$-colored diagram with no free ends.
Instead of evaluating $\mathcal{F}(D)$ directly, we can instead cut open $D$ to obtain a $(1,1)$-tangle%
\note{
  A $(1,1)$-tangle is a morphism $(x , \pm) \to (x, \pm)$ for some $x \in X$, i.e.\@ an $X$-colored tangle diagram with two free ends of matching orientation both labeled by $x$.
  To close it, we draw a strand from one free end to the other.
}
$T$ whose closure is $L$.
The image $\mathcal{F}(T)$ of $T$ under $\mathcal{F}$ will be an endomorphism of a simple object, and we can regard the scalar $\left\langle \mathcal{F}(T) \right\rangle$ as the value of $\mathcal{F}$ on $L$.

For this to be well-defined we need to make sure that $\left\langle \mathcal{F}(T) \right\rangle$ does not depend on our choice of where to cut $D$.
When every strand of $D$ is assigned the same module it is not hard to see that $\left\langle \mathcal{F}(T) \right\rangle$ is an invariant of $D$, but for a nontrivial biquandle $X$ this is less clear, and in general false.

One step towards making it work is to instead consider the number
\[
  \left\langle \mathcal{F}(T) \right\rangle \moddim{V}
\]
where $\moddim{V}$ is a scalar called the \defemph{modified dimension} of $V$.
It we choose the modified dimensions appropriately, then we can get gauge-invariant link invariants.
An elegant and general way to do this is the theory of modified traces \cite{Geer2018}, which we discuss in detail in \cref{ch:modified-traces}.
We give some basic definitions here.

An \defemph{ideal} in a pivotal category $\catl C$ is a full subcategory closed under tensor products and retracts.
(See \cref{def:pivotal-ideal} for details.)
For $\catl C = \modc H$ the category of modules of a pivotal Hopf algebra, the subcategory $\proj(\catl C)$ of projective $H$-modules is an ideal, and this is the example to keep in mind.

\begin{marginfigure}
  \centering
\begingroup%
  \makeatletter%
  \providecommand\color[2][]{%
    \errmessage{(Inkscape) Color is used for the text in Inkscape, but the package 'color.sty' is not loaded}%
    \renewcommand\color[2][]{}%
  }%
  \providecommand\transparent[1]{%
    \errmessage{(Inkscape) Transparency is used (non-zero) for the text in Inkscape, but the package 'transparent.sty' is not loaded}%
    \renewcommand\transparent[1]{}%
  }%
  \providecommand\rotatebox[2]{#2}%
  \newcommand*\fsize{\dimexpr\f@size pt\relax}%
  \newcommand*\lineheight[1]{\fontsize{\fsize}{#1\fsize}\selectfont}%
  \ifx\svgwidth\undefined%
    \setlength{\unitlength}{108bp}%
    \ifx\svgscale\undefined%
      \relax%
    \else%
      \setlength{\unitlength}{\unitlength * \real{\svgscale}}%
    \fi%
  \else%
    \setlength{\unitlength}{\svgwidth}%
  \fi%
  \global\let\svgwidth\undefined%
  \global\let\svgscale\undefined%
  \makeatother%
  \begin{picture}(1,0.55999998)%
    \lineheight{1}%
    \setlength\tabcolsep{0pt}%
    \put(0,0){\includegraphics[width=\unitlength,page=1]{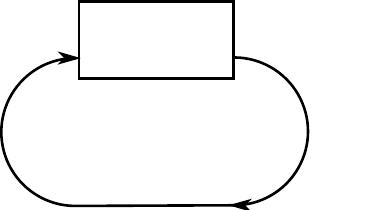}}%
    \put(0.40046397,0.42668968){\makebox(0,0)[lt]{\lineheight{1.25}\smash{\begin{tabular}[t]{l}$f$\end{tabular}}}}%
    \put(0.85435272,0.18277594){\makebox(0,0)[lt]{\lineheight{1.25}\smash{\begin{tabular}[t]{l}$W$\end{tabular}}}}%
    \put(0,0){\includegraphics[width=\unitlength,page=2]{right-ptrace.pdf}}%
    \put(0.83670859,0.47442176){\makebox(0,0)[lt]{\lineheight{1.25}\smash{\begin{tabular}[t]{l}$X$\end{tabular}}}}%
    \put(-0.04993946,0.47620762){\makebox(0,0)[lt]{\lineheight{1.25}\smash{\begin{tabular}[t]{l}$V$\end{tabular}}}}%
  \end{picture}%
\endgroup%

  \caption{The right partial trace of a map $f~:~V \otimes W ~\to~X \otimes W$.}
  \label{fig:right-ptrace}
\end{marginfigure}

\begin{defn}
  \label{def:modified-trace}
  Let $\catl C$ be a $\CC$-pivotal category.
  For $W$ an object of $\catl C$, the (right) \defemph{partial trace} on $W$ is the map
  \[
    \ptr_W : \hom_{\catl c}(V \otimes W, X \otimes W) \to \hom_{\catl C}{V,X}
  \]
  defined by
  \[
    \ptr_W(g) = (\id_X \otimes \evdown W) (g \otimes \id_{W^*})( \id_V \otimes \coevup W).
  \]
  Now let $I$ be a right ideal in $\catl C$.
  A right \defemph{modified trace} on $I$ is a family of $\CC$-linear functions
  \[
    \{\modtr_V : \hom_{\catl C}(V, V) \to \CC \}_{V \in I}
  \]
  for every object $V$ of $I$ that are
  \begin{enumerate}
    \item \defemph{compatible with partial traces:} If $V \in I$ and $W \in \catl C$, then for any $f \in \hom_{\catl C}(V \otimes W, V \otimes W)$,
      \[
        {\modtr}_{V \otimes W}(f) = {\modtr}_{V} \left( \ptr_W(f) \right)
      \]
    \item \defemph{cyclic:} If $U, V \in I$, then for any morphisms $f : V \to U$, $g: U \to V$, we have
      \[
        {\modtr}_V(gf) = {\modtr}_U(fg)
      \]
  \end{enumerate}
\end{defn}
We can similarly define left partial traces and left modified traces, and a modified trace is one that is both left and right.
A modified trace gives modified dimensions via
\[
  \moddim{V} \defeq \modtr_V \id_V.
\]

\begin{thm}
  \label{thm:modified-trace-exists}
  There is a nontrivial modified trace on the ideal of projective objects of $\modc \qgrp$.
  It assigns the modules $\irrmod{\chi}$ of \cref{def:weyl-irrmod} the renormalized dimensions
  \[
    \moddim{\irrmod{\chi}} =
    \begin{cases}
      \frac{\xi\mu - (\xi\mu)^{-1}}{(\xi\mu)^{\nr} - (\xi\mu)^{-\nr}}
 & \text{$\mu$ not a root of unity}
 \\
      1 & \text{$\mu$ a root of unity}
    \end{cases}
  \]
  where $\mu$ is the fractional eigenvalue of the extended shape $\chi$.
  In particular, the renormalized dimension only depends on the (fractional) eigenvalues of $\chi$, hence is invariant under gauge transformations.
\end{thm}
\begin{proof}
  See \cref{sec:mod-dim-weight-mods}.
\end{proof}

\begin{remark}
  Our trace differs%
  \note{
    We also have some extra factors of $\xi$ compared to \cite[(37)]{Blanchet2020}, but this comes from our convention that $\Omega$ acts by $\xi \mu + (\xi \mu)^{-1}$, not $\mu + \mu^{-1}$.
  }
  from that of \cite[Section 6.3]{Blanchet2020} by a factor of $\nr (-1)^{\nr +1}$.
  This normalization is more natural for the relationship with the torsion given in \cref{ch:torions}.
\end{remark}

\subsection{Link invariants from modified traces}
\label{sec:modified-invariants}

\begin{defn}
  Let $L$ be an extended $\slg$-link, represented as an extended shaped tangle diagram $D$.
  A \defemph{cutting presentation} of $D$ is an extended shaped $(1,1)$-tangle $T$ whose closure is $D$.

  Now let $\mathcal F : \tangshe \to \catl C$ be a functor for  $\catl C$ a pivotal category.%
  \note{
    This definition works for functors out of any colored tangle category $\tang[X]$.
    However, we have not given an intrinsic definition of what an $X$-colored link is for general biquandles $X$; we refer to \cite{Blanchet2020} for the general theory.
  }
  If $\modtr$ is a modified trace on an ideal $I$ of $\catl C$, 
  we say that $\modtr$ is \defemph{compatible} with $\mathcal{F}$ if the image of $\mathcal{F}$ lies in the ideal $I$.
  The \defemph{modified link invariant} defined by $\mathcal{F}$ and a compatible trace $\modtr$ is the scalar computed by
  \[
    \invl F(L) \defeq \modtr(\mathcal F(T))
  \]
  where $T$ is any cutting presentation of a diagram $D$ of $L$.
\end{defn}

\begin{thm}
  \label{thm:cutting-indep-link}
  The modified link invariant $\invl F(L)$ associated to $\mathcal{F}$ and  $\modtr$ is well-defined: it does not depend on the choice of representative tangle diagram $D$ or of the cutting presentation $T$.
\end{thm}
\begin{proof}
  Because $\mathcal{F}$ is a functor, it respects Reidemeister moves, so $\invl F(L)$ does not depend on the choice of diagram $D$.
  We can now apply \cref{thm:cutting-indep-diagram} to see that it also does not depend on the choice of $T$.
\end{proof}

\subsection{Twists}
As an immediate application, we use modified traces to show that for the class of biquandle models we consider the left and right twists agree automatically.
\begin{defn}
  \label{def:abs-simple-regular}
  Let $V$ be an object of a pivotal category $\catl C$ with scalars $\Bbbk$.
  We say that $V$ is
  \begin{itemize}
    \item \defemph{absolutely simple} if the map $\Bbbk \to \End_{\catl C}(V)$ sending $k$ to $k\cdot \id_{V}$ is a bijection, and
    \item \defemph{regular} if the functor $X \mapsto V \otimes X$ is a faithful endofunctor of $\catl C$, that is if the map $f \mapsto \id_{V} \otimes f$ is an injective map on $\hom$-sets.
  \end{itemize}
  A model $(V_-, S_{-,-})$ of a (possibly generic) biquandle $X$ in $\catl C$ is absolutely simple if $V_x$ is absolutely simple for every $x \in X$, and regular if every $V_x$ is regular.
\end{defn}

\begin{thm}
  \label{thm:twist-induced}
  Let $(V_-, S_{-,-})$ be an absolutely simple and regular model of a (generic) biquandle $X$ in $\catl C$.
  If this model admits a gauge-invariant modified trace $\modtr$ which has $\moddim{V_x} \ne 0$ for every $x \in X$ then it induces a twist.
\end{thm}
\begin{proof}
  We need to show that the images of \cref{fig:twist-right,fig:twist-left} under the functor $\mathcal{F}$ induced by the model agree.
  Because $V_x$ is absolutely simple and $\moddim{V_x} \ne 0$, it is enough to show that
  \[
    \left\langle \theta_x^R \right\rangle = \left\langle \theta_x^L \right\rangle.
  \]
  Recall that for endomorphisms $f$ of an absolutely simple object $V$, $f = \left\langle f \right\rangle \id_V$.
  Below we will write $\left\langle f \otimes g \right\rangle = \left\langle f \right\rangle \left\langle g \right\rangle$ for the tensor product $f \otimes g$ of two such endomorphisms.

  The trick is to compute the modified trace of $S_{x, \alpha(x)}$ (equivalently, the invariant of the framed unknot) in two different ways.
  Because the left and right partial traces are compatible with the partial trace, we have
  \[
    \left\langle \theta_x^R\right\rangle \moddim{V_x}
    =
    \invl{F} \left( \underset{\alpha(x)}{\overset{x}{\includegraphics[valign=c]{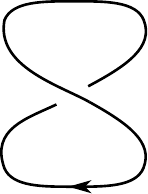}}} \right) 
    =
    \left\langle \theta_y^L \right\rangle \moddim{V_y}
  \]
  Here $y = \alpha(x)$ is determined by the biquandle and $\invl F$ is the modified diagram invariant associated to $\mathcal{F}$ and $\modtr$.
  The left and right sides of the relation correspond to cutting the diagram at the top and the bottom.
  Because $\modtr$ is gauge-invariant, $\moddim{V_x} = \moddim{V_y}$, and by hypothesis the modified dimensions are nonzero, so $\left\langle \theta_x^R\right\rangle = \left\langle \theta_y^L \right\rangle$.

  On the other hand, writing $z = \alpha^{-1}(x)$, we see that
  \begin{align*}
    \left\langle \theta_x^L \right\rangle
    =
    \left\langle
      \mathcal{F}
      \left(
        \setlength{\unitlength}{1in}
        \begin{picture}(1.01060,0.4)%
          \put(0,0){\includegraphics[valign=c]{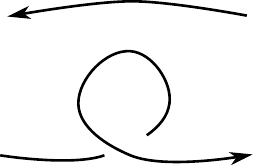}}
          \put(0.2,0.22){$x$}
          \put(0.3,-0.4){$x$}
          \put(0.7,0.0){$z$}
        \end{picture}
      \right)
    \right\rangle
    &=
    \left\langle
      \mathcal{F}
      \left(
        \setlength{\unitlength}{1in}
        \begin{picture}(1.2,0.6)
          \put(0.1,0){\includegraphics[valign=c]{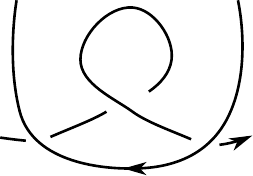}}
          \put(0,-0.2){$x$}
          \put(0,0.2){$x$}
          \put(0.6,-0.4){$y$}
          \put(0.3,-0.05){$y$}
          \put(0.85,0.15){$x$}
        \end{picture}
      \right)
    \right\rangle
    \\
    &=
    \left\langle \theta_y^L \right\rangle
    \left\langle
      \mathcal{F}
      \left(
        \setlength{\unitlength}{1in}
        \begin{picture}(1.1,0.4)
          \put(0.1,0){\includegraphics[valign=c]{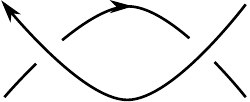}}
          \put(0,0.2){$x$}
          \put(0.55,0.3){$y$}
        \end{picture}
      \right)
    \right\rangle
    \\
    &=
    \left\langle \theta_y^L \right\rangle. \qedhere
  \end{align*}
\end{proof}

\section{Internal gauge transformations}
\label{sec:internal-gauge-transf}
As mentioned in the introduction, the essential ingredient in proving that our invariants are gauge-invariant is the representation of gauge transformations in terms of shaped diagrams.

\begin{defn}
  Let $D$ be a $(1,1)$ shaped tangle diagram, viewed as a morphism  $(\chi, +) \to (\chi, +)$.
  We call the diagrams $\Gamma_i^\pm(D, \psi)$ defined in \cref{fig:gauge-transf} \defemph{internal gauge transformations} of $D$.
  They are defined for any shape $\psi$ such that the relevant diagrams are well-defined.
\end{defn}

\begin{figure}
  \centering
  \subcaptionbox{$\Gamma_1^+(D, \gamma) \defeq D'$\label{fig:gauge-transf-i}}{ \def\svgwidth{2in} 
\begingroup%
  \makeatletter%
  \providecommand\color[2][]{%
    \errmessage{(Inkscape) Color is used for the text in Inkscape, but the package 'color.sty' is not loaded}%
    \renewcommand\color[2][]{}%
  }%
  \providecommand\transparent[1]{%
    \errmessage{(Inkscape) Transparency is used (non-zero) for the text in Inkscape, but the package 'transparent.sty' is not loaded}%
    \renewcommand\transparent[1]{}%
  }%
  \providecommand\rotatebox[2]{#2}%
  \newcommand*\fsize{\dimexpr\f@size pt\relax}%
  \newcommand*\lineheight[1]{\fontsize{\fsize}{#1\fsize}\selectfont}%
  \ifx\svgwidth\undefined%
    \setlength{\unitlength}{144bp}%
    \ifx\svgscale\undefined%
      \relax%
    \else%
      \setlength{\unitlength}{\unitlength * \real{\svgscale}}%
    \fi%
  \else%
    \setlength{\unitlength}{\svgwidth}%
  \fi%
  \global\let\svgwidth\undefined%
  \global\let\svgscale\undefined%
  \makeatother%
  \begin{picture}(1,1)%
    \lineheight{1}%
    \setlength\tabcolsep{0pt}%
    \put(0,0){\includegraphics[width=\unitlength,page=1]{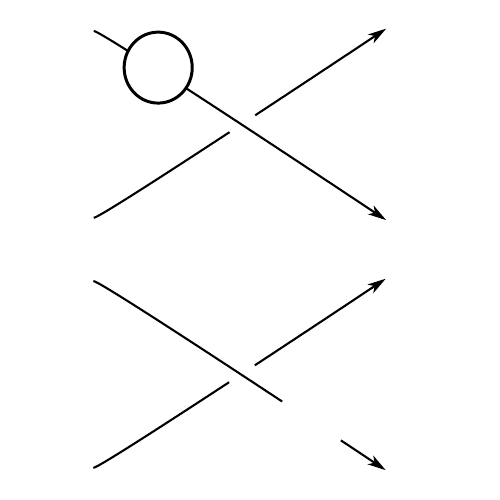}}%
    \put(0.29393686,0.84454894){\makebox(0,0)[lt]{\lineheight{1.25}\smash{\begin{tabular}[t]{l}$D$\end{tabular}}}}%
    \put(0,0){\includegraphics[width=\unitlength,page=2]{gauge-transf-i.pdf}}%
    \put(0.58431978,0.13866283){\makebox(0,0)[lt]{\lineheight{1.25}\smash{\begin{tabular}[t]{l}$D'$\end{tabular}}}}%
    \put(0.13270892,0.90100857){\makebox(0,0)[lt]{\lineheight{1.25}\smash{\begin{tabular}[t]{l}$\chi$\end{tabular}}}}%
    \put(0.13270892,0.56767529){\makebox(0,0)[lt]{\lineheight{1.25}\smash{\begin{tabular}[t]{l}$\gamma$\end{tabular}}}}%
    \put(0.13296831,0.42229327){\makebox(0,0)[lt]{\lineheight{1.25}\smash{\begin{tabular}[t]{l}$\chi$\end{tabular}}}}%
    \put(0.13296831,0.07229318){\makebox(0,0)[lt]{\lineheight{1.25}\smash{\begin{tabular}[t]{l}$\gamma$\end{tabular}}}}%
    \put(0.41666665,0.50520835){\makebox(0,0)[lt]{\lineheight{1.25}\smash{\begin{tabular}[t]{l}$\Rightarrow$\end{tabular}}}}%
  \end{picture}%
\endgroup%
 }%
  \hfill
  \subcaptionbox{$\Gamma_2^+(D,\gamma) \defeq D'$.\label{fig:gauge-transf-ii}}{ \def\svgwidth{2in} 
\begingroup%
  \makeatletter%
  \providecommand\color[2][]{%
    \errmessage{(Inkscape) Color is used for the text in Inkscape, but the package 'color.sty' is not loaded}%
    \renewcommand\color[2][]{}%
  }%
  \providecommand\transparent[1]{%
    \errmessage{(Inkscape) Transparency is used (non-zero) for the text in Inkscape, but the package 'transparent.sty' is not loaded}%
    \renewcommand\transparent[1]{}%
  }%
  \providecommand\rotatebox[2]{#2}%
  \newcommand*\fsize{\dimexpr\f@size pt\relax}%
  \newcommand*\lineheight[1]{\fontsize{\fsize}{#1\fsize}\selectfont}%
  \ifx\svgwidth\undefined%
    \setlength{\unitlength}{144bp}%
    \ifx\svgscale\undefined%
      \relax%
    \else%
      \setlength{\unitlength}{\unitlength * \real{\svgscale}}%
    \fi%
  \else%
    \setlength{\unitlength}{\svgwidth}%
  \fi%
  \global\let\svgwidth\undefined%
  \global\let\svgscale\undefined%
  \makeatother%
  \begin{picture}(1,1)%
    \lineheight{1}%
    \setlength\tabcolsep{0pt}%
    \put(0,0){\includegraphics[width=\unitlength,page=1]{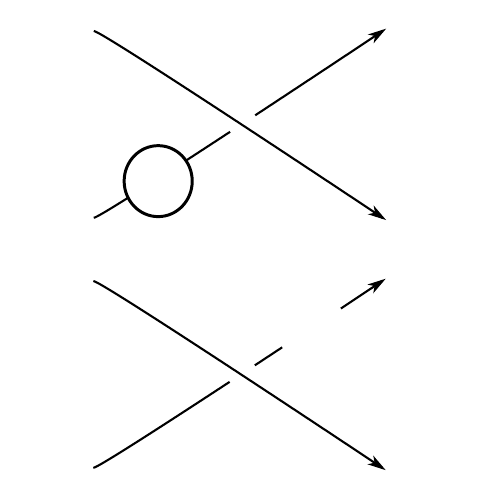}}%
    \put(0.29393686,0.61708542){\makebox(0,0)[lt]{\lineheight{1.25}\smash{\begin{tabular}[t]{l}$D$\end{tabular}}}}%
    \put(0,0){\includegraphics[width=\unitlength,page=2]{gauge-transf-ii.pdf}}%
    \put(0.58431978,0.32278558){\makebox(0,0)[lt]{\lineheight{1.25}\smash{\begin{tabular}[t]{l}$D'$\end{tabular}}}}%
    \put(0.13270892,0.90100857){\makebox(0,0)[lt]{\lineheight{1.25}\smash{\begin{tabular}[t]{l}$\gamma$\end{tabular}}}}%
    \put(0.13270892,0.56767529){\makebox(0,0)[lt]{\lineheight{1.25}\smash{\begin{tabular}[t]{l}$\chi$\end{tabular}}}}%
    \put(0.13296831,0.42229327){\makebox(0,0)[lt]{\lineheight{1.25}\smash{\begin{tabular}[t]{l}$\gamma$\end{tabular}}}}%
    \put(0.13296831,0.07229318){\makebox(0,0)[lt]{\lineheight{1.25}\smash{\begin{tabular}[t]{l}$\chi$\end{tabular}}}}%
    \put(0.41666665,0.50520835){\makebox(0,0)[lt]{\lineheight{1.25}\smash{\begin{tabular}[t]{l}$\Rightarrow$\end{tabular}}}}%
  \end{picture}%
\endgroup%
 }%
  \hfill
  \subcaptionbox{$\Gamma_1^-(D, \gamma) \defeq D'$\label{fig:gauge-transf-iii}}{ \def\svgwidth{2in} 
\begingroup%
  \makeatletter%
  \providecommand\color[2][]{%
    \errmessage{(Inkscape) Color is used for the text in Inkscape, but the package 'color.sty' is not loaded}%
    \renewcommand\color[2][]{}%
  }%
  \providecommand\transparent[1]{%
    \errmessage{(Inkscape) Transparency is used (non-zero) for the text in Inkscape, but the package 'transparent.sty' is not loaded}%
    \renewcommand\transparent[1]{}%
  }%
  \providecommand\rotatebox[2]{#2}%
  \newcommand*\fsize{\dimexpr\f@size pt\relax}%
  \newcommand*\lineheight[1]{\fontsize{\fsize}{#1\fsize}\selectfont}%
  \ifx\svgwidth\undefined%
    \setlength{\unitlength}{144bp}%
    \ifx\svgscale\undefined%
      \relax%
    \else%
      \setlength{\unitlength}{\unitlength * \real{\svgscale}}%
    \fi%
  \else%
    \setlength{\unitlength}{\svgwidth}%
  \fi%
  \global\let\svgwidth\undefined%
  \global\let\svgscale\undefined%
  \makeatother%
  \begin{picture}(1,1)%
    \lineheight{1}%
    \setlength\tabcolsep{0pt}%
    \put(0,0){\includegraphics[width=\unitlength,page=1]{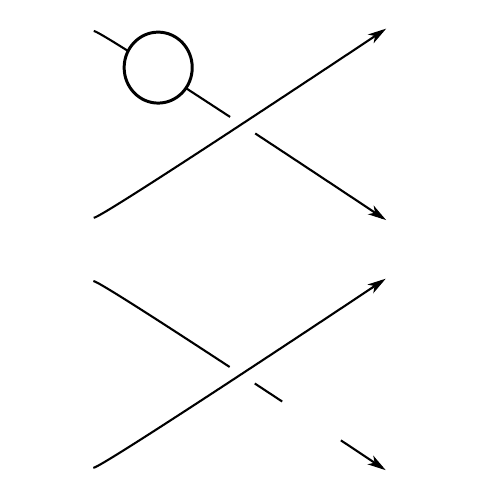}}%
    \put(0.29393686,0.84454894){\makebox(0,0)[lt]{\lineheight{1.25}\smash{\begin{tabular}[t]{l}$D$\end{tabular}}}}%
    \put(0,0){\includegraphics[width=\unitlength,page=2]{gauge-transf-iii.pdf}}%
    \put(0.58431978,0.13866283){\makebox(0,0)[lt]{\lineheight{1.25}\smash{\begin{tabular}[t]{l}$D'$\end{tabular}}}}%
    \put(0.13270892,0.90100857){\makebox(0,0)[lt]{\lineheight{1.25}\smash{\begin{tabular}[t]{l}$\chi$\end{tabular}}}}%
    \put(0.13270892,0.56767529){\makebox(0,0)[lt]{\lineheight{1.25}\smash{\begin{tabular}[t]{l}$\gamma$\end{tabular}}}}%
    \put(0.13296831,0.42229327){\makebox(0,0)[lt]{\lineheight{1.25}\smash{\begin{tabular}[t]{l}$\chi$\end{tabular}}}}%
    \put(0.13296831,0.07229318){\makebox(0,0)[lt]{\lineheight{1.25}\smash{\begin{tabular}[t]{l}$\gamma$\end{tabular}}}}%
    \put(0.41666665,0.50520835){\makebox(0,0)[lt]{\lineheight{1.25}\smash{\begin{tabular}[t]{l}$\Rightarrow$\end{tabular}}}}%
  \end{picture}%
\endgroup%
 }%
  \hfill
  \subcaptionbox{$\Gamma_2^-(D,\gamma) \defeq D'$\label{fig:gauge-transf-iv}}{ \def\svgwidth{2in} 
\begingroup%
  \makeatletter%
  \providecommand\color[2][]{%
    \errmessage{(Inkscape) Color is used for the text in Inkscape, but the package 'color.sty' is not loaded}%
    \renewcommand\color[2][]{}%
  }%
  \providecommand\transparent[1]{%
    \errmessage{(Inkscape) Transparency is used (non-zero) for the text in Inkscape, but the package 'transparent.sty' is not loaded}%
    \renewcommand\transparent[1]{}%
  }%
  \providecommand\rotatebox[2]{#2}%
  \newcommand*\fsize{\dimexpr\f@size pt\relax}%
  \newcommand*\lineheight[1]{\fontsize{\fsize}{#1\fsize}\selectfont}%
  \ifx\svgwidth\undefined%
    \setlength{\unitlength}{144bp}%
    \ifx\svgscale\undefined%
      \relax%
    \else%
      \setlength{\unitlength}{\unitlength * \real{\svgscale}}%
    \fi%
  \else%
    \setlength{\unitlength}{\svgwidth}%
  \fi%
  \global\let\svgwidth\undefined%
  \global\let\svgscale\undefined%
  \makeatother%
  \begin{picture}(1,1)%
    \lineheight{1}%
    \setlength\tabcolsep{0pt}%
    \put(0,0){\includegraphics[width=\unitlength,page=1]{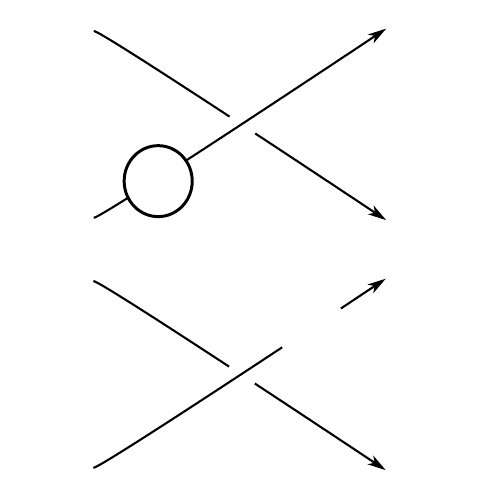}}%
    \put(0.29393686,0.61708542){\makebox(0,0)[lt]{\lineheight{1.25}\smash{\begin{tabular}[t]{l}$D$\end{tabular}}}}%
    \put(0,0){\includegraphics[width=\unitlength,page=2]{gauge-transf-iv.pdf}}%
    \put(0.58431978,0.32278558){\makebox(0,0)[lt]{\lineheight{1.25}\smash{\begin{tabular}[t]{l}$D'$\end{tabular}}}}%
    \put(0.13270892,0.90100857){\makebox(0,0)[lt]{\lineheight{1.25}\smash{\begin{tabular}[t]{l}$\gamma$\end{tabular}}}}%
    \put(0.13270892,0.56767529){\makebox(0,0)[lt]{\lineheight{1.25}\smash{\begin{tabular}[t]{l}$\chi$\end{tabular}}}}%
    \put(0.13296831,0.42229327){\makebox(0,0)[lt]{\lineheight{1.25}\smash{\begin{tabular}[t]{l}$\gamma$\end{tabular}}}}%
    \put(0.13296831,0.07229318){\makebox(0,0)[lt]{\lineheight{1.25}\smash{\begin{tabular}[t]{l}$\chi$\end{tabular}}}}%
    \put(0.41666665,0.50520835){\makebox(0,0)[lt]{\lineheight{1.25}\smash{\begin{tabular}[t]{l}$\Rightarrow$\end{tabular}}}}%
  \end{picture}%
\endgroup%
 }%
  \caption{Gauge transformations of diagrams.}
  \label{fig:gauge-transf}
\end{figure}

As suggested by the name, internal gauge transformations correspond to gauge transformations of the holonomy:
\begin{thm}
  \label{thm:gauge-transfs-are-gauge-transfs}
  Let $D$ be a shaped $(1,1)$-tangle diagram and $D' = \Gamma_i^\pm(D, \gamma)$ an internal gauge transformation of it.
  Write $\rho_D$ for the holonomy representation $\rho : \pi(D) \to \slg$ of the complement of $D$ as given in \cref{def:diagram-terms} and similarly for $D'$.
 Then the representations $\rho_D$ and $\rho_{D'}$ are conjugate.
\end{thm}
\begin{proof}
  Consider the case $D' = \Gamma_1^{+}(D, \gamma)$ ; the others are similar.
  Write $\mathfrak{D}$ and $\mathfrak{D}'$ for the two equivalent diagrams shown in \cref{fig:gauge-transf-i}.
  Below we mark two paths $q$ and $q'$ in the complements of $\mathfrak{D}$ and $\mathfrak{D}'$, respectively:
  \begin{center}
\begingroup%
  \makeatletter%
  \providecommand\color[2][]{%
    \errmessage{(Inkscape) Color is used for the text in Inkscape, but the package 'color.sty' is not loaded}%
    \renewcommand\color[2][]{}%
  }%
  \providecommand\transparent[1]{%
    \errmessage{(Inkscape) Transparency is used (non-zero) for the text in Inkscape, but the package 'transparent.sty' is not loaded}%
    \renewcommand\transparent[1]{}%
  }%
  \providecommand\rotatebox[2]{#2}%
  \newcommand*\fsize{\dimexpr\f@size pt\relax}%
  \newcommand*\lineheight[1]{\fontsize{\fsize}{#1\fsize}\selectfont}%
  \ifx\svgwidth\undefined%
    \setlength{\unitlength}{198.61078835bp}%
    \ifx\svgscale\undefined%
      \relax%
    \else%
      \setlength{\unitlength}{\unitlength * \real{\svgscale}}%
    \fi%
  \else%
    \setlength{\unitlength}{\svgwidth}%
  \fi%
  \global\let\svgwidth\undefined%
  \global\let\svgscale\undefined%
  \makeatother%
  \begin{picture}(1,0.3063213)%
    \lineheight{1}%
    \setlength\tabcolsep{0pt}%
    \put(0,0){\includegraphics[width=\unitlength,page=1]{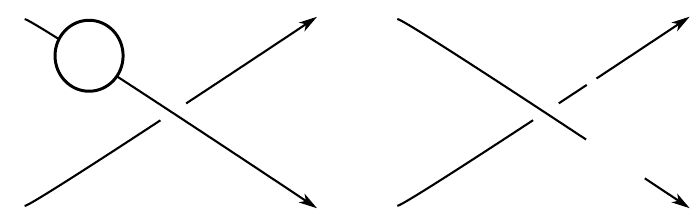}}%
    \put(0.11288385,0.21079759){\makebox(0,0)[lt]{\lineheight{1.25}\smash{\begin{tabular}[t]{l}$D$\end{tabular}}}}%
    \put(0,0){\includegraphics[width=\unitlength,page=2]{gauge-transf-proof-i.pdf}}%
    \put(0.86440195,0.06171942){\makebox(0,0)[lt]{\lineheight{1.25}\smash{\begin{tabular}[t]{l}$D'$\end{tabular}}}}%
    \put(-0.00401224,0.25173286){\makebox(0,0)[lt]{\lineheight{1.25}\smash{\begin{tabular}[t]{l}$\chi$\end{tabular}}}}%
    \put(-0.00401224,0.01005418){\makebox(0,0)[lt]{\lineheight{1.25}\smash{\begin{tabular}[t]{l}$\gamma$\end{tabular}}}}%
    \put(0.53715578,0.26736173){\makebox(0,0)[lt]{\lineheight{1.25}\smash{\begin{tabular}[t]{l}$\chi$\end{tabular}}}}%
    \put(0.53715578,0.01359903){\makebox(0,0)[lt]{\lineheight{1.25}\smash{\begin{tabular}[t]{l}$\gamma$\end{tabular}}}}%
    \put(0,0){\includegraphics[width=\unitlength,page=3]{gauge-transf-proof-i.pdf}}%
    \put(0.19641004,0.25053417){\color[rgb]{0.56862745,0.30588235,0.05882353}\makebox(0,0)[lt]{\lineheight{1.25}\smash{\begin{tabular}[t]{l}$q$\end{tabular}}}}%
    \put(0.8369079,0.25852545){\color[rgb]{0.56862745,0.30588235,0.05882353}\makebox(0,0)[lt]{\lineheight{1.25}\smash{\begin{tabular}[t]{l}$q'$\end{tabular}}}}%
  \end{picture}%
\endgroup%

  \end{center}
  Choose a path $p$ in the complement of the tangle diagram $D$ representing an element of $\pi(D)$.
  Because $D'$ has the same underlying tangle as $D$ (only the labels differ), $p$ also represents an element of $\pi(D')$.
  We see that the conjugations $qpq^{-1}$ and $q' p (q')^{-1}$ represent elements of $\pi(\mathfrak{D})$ and $\pi(\mathfrak{D}')$, respectively.

  Because the holonomy representation is compatible with Reidemeister moves,
  \[
    \rho_{\mathfrak{D}}(qpq^{-1}) = \rho_{\mathfrak{D}'}(q' p (q')^{-1} ).
  \]
  Because $q$ does not pass above or below any strands of the diagram, $\rho_{\mathfrak{D}}(qpq^{-1}) = \rho_D(p)$, and similarly
  \[
    \rho_{\mathfrak{D}'}(q' p (q')^{-1} ) = g^+(\psi') \rho_{D'}(p) g^{+}(\psi')^{-1},
  \]
  where $g^+(\gamma)$ is the holonomy corresponding to crossing above a strand labelled by $\gamma'$.
  (Here $\gamma' = B_1(\chi, \gamma)$.)
  In particular, we see that
  \[
    \rho_D(p) = g^+(\gamma') \rho_{D'}(p) g^{+}(\gamma')^{-1}
  \]
  so $\rho_D(p)$ and $\rho_{D'}(p)$ are conjugate.
  Since this holds for every path in the complement of $D$, i.e.\@ for a representative of every element of $\pi(D) = \pi(D')$, we conclude that $\rho_{D}$ and $\rho_{D'}$ are conjugate.
\end{proof}

The point of defining gauge transformations this way is that the diagrams in \cref{fig:gauge-transf} are equivalent via colored Reidemeister moves.
If $\mathcal F$ is any functor respecting these moves, the image of the diagrams under $\mathcal{F}$ give relationships between $\mathcal{F}(D)$ and $\mathcal{F}(\Gamma_i^{\pm}(D))$ that we can exploit to prove gauge invariance.

\begin{defn}
  Let $D = \id_\chi$ be the diagram consisting of a single strand colored by the extended shape $\chi$.
  Then any gauge transformation $\Gamma_i^{\pm}(\id_{\chi'}, \gamma)$ is of the form $\id_{\chi'}$ for some $\chi'$.
  We say that any two characters related this way are \defemph{gauge-equivalent}.

  We say a representation $(V,S)$ of the extended shape biquandle with a compatible trace $\modtr$ is \defemph{gauge-invariant} if $\moddim{V_\chi} = \moddim{V_{\chi'}}$ whenever $\chi$ and $\chi'$ are gauge-equivalent.
\end{defn}

\begin{thm}
  \label{thm:gauge-invariant-trace}
  Let $\mathcal{F}$ be a gauge-invariant representation of the extended shape biquandle.
  Then the diagram invariant associated to $\mathcal{F}$ is gauge-invariant in the sense that
  \[
    \invl{F}(D) = \invl{F}(\Gamma_i^{\pm}(D,\psi))
  \]
  for any internal gauge transformation of $D$.
\end{thm}
\begin{proof}[Proof sketch]
  Again, we prove only the case $D' = \Gamma_{1}^+(D, \gamma)$, as the others are similar.
  The trick is to consider the two diagrams
  \begin{equation}
    \label{eq:gauge-invariant-trace-proof-i}
\begingroup%
  \makeatletter%
  \providecommand\color[2][]{%
    \errmessage{(Inkscape) Color is used for the text in Inkscape, but the package 'color.sty' is not loaded}%
    \renewcommand\color[2][]{}%
  }%
  \providecommand\transparent[1]{%
    \errmessage{(Inkscape) Transparency is used (non-zero) for the text in Inkscape, but the package 'transparent.sty' is not loaded}%
    \renewcommand\transparent[1]{}%
  }%
  \providecommand\rotatebox[2]{#2}%
  \newcommand*\fsize{\dimexpr\f@size pt\relax}%
  \newcommand*\lineheight[1]{\fontsize{\fsize}{#1\fsize}\selectfont}%
  \ifx\svgwidth\undefined%
    \setlength{\unitlength}{144bp}%
    \ifx\svgscale\undefined%
      \relax%
    \else%
      \setlength{\unitlength}{\unitlength * \real{\svgscale}}%
    \fi%
  \else%
    \setlength{\unitlength}{\svgwidth}%
  \fi%
  \global\let\svgwidth\undefined%
  \global\let\svgscale\undefined%
  \makeatother%
  \begin{picture}(1,0.22388831)%
    \lineheight{1}%
    \setlength\tabcolsep{0pt}%
    \put(0,0){\includegraphics[width=\unitlength,page=1]{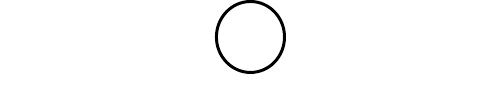}}%
    \put(0.4785244,0.12972216){\makebox(0,0)[lt]{\lineheight{1.25}\smash{\begin{tabular}[t]{l}$D$\end{tabular}}}}%
    \put(0,0){\includegraphics[width=\unitlength,page=2]{gauge-invariant-trace-proof-i.pdf}}%
    \put(0.03085591,0.14526666){\makebox(0,0)[lt]{\lineheight{1.25}\smash{\begin{tabular}[t]{l}$\chi$\end{tabular}}}}%
    \put(0.03085828,-0.00577504){\makebox(0,0)[lt]{\lineheight{1.25}\smash{\begin{tabular}[t]{l}$\gamma$\end{tabular}}}}%
    \put(0.91627315,0.14526683){\makebox(0,0)[lt]{\lineheight{1.25}\smash{\begin{tabular}[t]{l}$\chi$\end{tabular}}}}%
    \put(0.91627556,-0.00577478){\makebox(0,0)[lt]{\lineheight{1.25}\smash{\begin{tabular}[t]{l}$\gamma$\end{tabular}}}}%
  \end{picture}%
\endgroup%

  \end{equation}
  and
  \begin{equation}
    \label{eq:gauge-invariant-trace-proof-ii}
\begingroup%
  \makeatletter%
  \providecommand\color[2][]{%
    \errmessage{(Inkscape) Color is used for the text in Inkscape, but the package 'color.sty' is not loaded}%
    \renewcommand\color[2][]{}%
  }%
  \providecommand\transparent[1]{%
    \errmessage{(Inkscape) Transparency is used (non-zero) for the text in Inkscape, but the package 'transparent.sty' is not loaded}%
    \renewcommand\transparent[1]{}%
  }%
  \providecommand\rotatebox[2]{#2}%
  \newcommand*\fsize{\dimexpr\f@size pt\relax}%
  \newcommand*\lineheight[1]{\fontsize{\fsize}{#1\fsize}\selectfont}%
  \ifx\svgwidth\undefined%
    \setlength{\unitlength}{144bp}%
    \ifx\svgscale\undefined%
      \relax%
    \else%
      \setlength{\unitlength}{\unitlength * \real{\svgscale}}%
    \fi%
  \else%
    \setlength{\unitlength}{\svgwidth}%
  \fi%
  \global\let\svgwidth\undefined%
  \global\let\svgscale\undefined%
  \makeatother%
  \begin{picture}(1,0.27500001)%
    \lineheight{1}%
    \setlength\tabcolsep{0pt}%
    \put(0,0){\includegraphics[width=\unitlength,page=1]{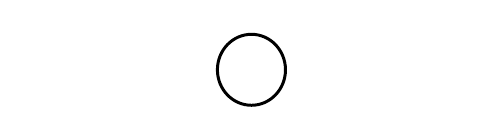}}%
    \put(0.47005763,0.11519503){\makebox(0,0)[lt]{\lineheight{1.25}\smash{\begin{tabular}[t]{l}$D'$\end{tabular}}}}%
    \put(0,0){\includegraphics[width=\unitlength,page=2]{gauge-invariant-trace-proof-ii.pdf}}%
    \put(0.02881498,0.02883111){\makebox(0,0)[lt]{\lineheight{1.25}\smash{\begin{tabular}[t]{l}$\gamma$\end{tabular}}}}%
    \put(0.02881498,0.12778947){\makebox(0,0)[lt]{\lineheight{1.25}\smash{\begin{tabular}[t]{l}$\chi$\\\end{tabular}}}}%
    \put(0.93506563,0.02883111){\makebox(0,0)[lt]{\lineheight{1.25}\smash{\begin{tabular}[t]{l}$\gamma$\end{tabular}}}}%
    \put(0.93506563,0.12778947){\makebox(0,0)[lt]{\lineheight{1.25}\smash{\begin{tabular}[t]{l}$\chi$\\\end{tabular}}}}%
  \end{picture}%
\endgroup%

  \end{equation}
  Call the diagram in \eqref{eq:gauge-invariant-trace-proof-i} $X$ and the diagram in \eqref{eq:gauge-invariant-trace-proof-ii} $X'$.
  Because $X$ and $X'$ are equivalent via Reidemeister moves, we must have $\mathcal{F}(X) = \mathcal{F}(X')$, hence
  \[
    \modtr \mathcal{F} (X) = \modtr \mathcal{F} (X').
  \]
  Because the modified trace is compatible with tensor products,
  \[
    \modtr \mathcal{F}(X) = \modtr \left( \mathcal{F}(D) \otimes \id_{\mathcal{F}(\psi)}\right) = \modtr \mathcal{F}(D) = \moddim{\mathcal{F}(\chi')} \left\langle \mathcal{F}(D) \right\rangle.
  \]
  On the other hand, because the representation $\mathcal{F}$ is regular and absolutely simple,
  \[
    \modtr \mathcal{F}(X') = \left\langle \mathcal{F}(D') \right\rangle \moddim{\mathcal{F}(\chi')}
  \]
  where $\chi'$ is the shape with $D' : \chi' \to \chi'$.
  Because $\modtr$ is gauge-invariant, $\moddim{\mathcal{F}(\chi)} = \moddim{\mathcal{F}(\chi')}$  and we conclude that $\modtr \mathcal{F}(D) = \modtr \mathcal{F}(D')$.
\end{proof}
This is a proof sketch because the biquandle is only partially defined, so some Reidmeister moves are undefined.
The details of this are dealt with in \cite[Appendix A]{Blanchet2020}.

\section{Scalar representations of biquandles}
\label{sec:scalar-reps}
A particularly simple class of biquandle representations are those taking values in scalars.
To emphasize the general nature of our results, we return to the case of an arbitrary biquandle $X$ instead of the shape biquandle.
For simplicity, we do not consider the case of partially-defined biquandles at all, although it should be easy to extend the theory to this case.

Our primary motivation for studying scalar invariants is that we can think of them as changes in normalization of holonomy invariants.
The non-holonomy analogue for quantum invariants is the dependence (or lack thereof) on the framing.
For example, the Kauffman bracket \cite{Kauffman1987} associates each framed link $L$ a Laurent polynomial $\left\langle L \right\rangle$ in $A^{\pm 1}$.
Changing the framing of $L$ multiplies $\left\langle L \right\rangle$ by a factor of $A^{\pm 3}$.

To correct this, we can change the normalization of the braiding: multiplying it by an appropriate power of $A$ will give an invariant (the Jones polynomial) that does not depend on the framing.
Even though the bracket polynomial and the Jones polynomial are not the same invariant, they are essentially the same.
Furthermore, this is essentially the only nontrivial change%
\note{
  The other normalization change is multiplying the polynomials by an overall constant.
  For example, sometimes it is convenient to say that the Jones polynomial of the unknot is $q + q^{-1}$, while sometimes it is better to say that it is $1$.
}
we can make: the $R$-matrix of quantum $\lie{sl}_2$ fixes the braiding up to an overall scalar, but we can change the scalar.
For applications to geometric topology, it is natural to choose this scalar so that the invariant does not depend on the framing.

For holonomy invariants the situation is more complicated.
Instead of a single braiding $S$ we have a family $\{S_{x,y}\}_{x,y \in X}$ of them indexed by the elements of a biquandle.
Given a scalar-valued braiding $\omega_{x,y}$, we can think of the product $\omega_{x,y} S_{x,y}$ as a change in the normalization $S$.
Framing-independence is no longer enough to define the scalar normalization: there could be two scalar-valued braidings $\omega$ and $\omega'$ that both take the value $1$ on the twists shown in \cref{fig:twist-right}.

It turns out that scalar-valued biquandle representations are classified by a known algebraic structure: we show that they are equivalent to the birack $2$-cocyles of \cite{Kamada2018}.
Representations with twist $1$ instead correspond to biquandle cocycles.
To cohomologus cocycles give different invariants, but we can show that they given the same invariants on links.

It follows that one way to understand scalar normalizations of the nonabelian quantum dilogarithm would be to compute the second cohomology group of the (extended) shape biquandle.
This seems like a difficult problem.
Because the shape biquandle is a factorization of the conjugation quandle of $\slg$, we could probably reduce the classification to computing the \emph{quandle} cohomology \cite{Birman1974} of the conjugation quandle of $\slg$, although this may not be much simpler.

\begin{defn}
Let $\Bbbk$ be a field, and write $\Bbbk^\times$ for its multiplicative group.%
\note{
  We could replace $\Bbbk^\times$ with any abelian group, which would be useful for examples like $\CC^\times / \units n$.
}
A \defemph{scalar representation} of a biquandle $X$ is one which sends every $x \in X$ to the $\Bbbk$-vector space $\Bbbk$.
\end{defn}
The category $\vect \Bbbk$ is pivotal and the usual trace of vector spaces gives a compatible modified trace, which is automatically gauge-invariant.
In addition, every scalar representation is absolutely simple and regular.

\begin{marginfigure}
  \centering
  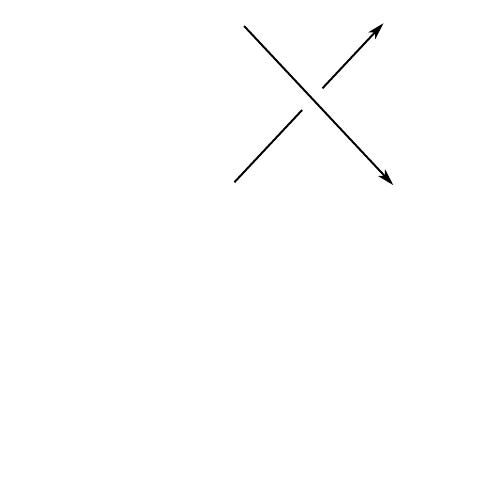
  \caption{Scalar representations of a biquandle at positive and negative crossings.}
  \label{fig:cocycle-assignments}
\end{marginfigure}
\begin{marginfigure}
  \centering
\begingroup%
  \makeatletter%
  \providecommand\color[2][]{%
    \errmessage{(Inkscape) Color is used for the text in Inkscape, but the package 'color.sty' is not loaded}%
    \renewcommand\color[2][]{}%
  }%
  \providecommand\transparent[1]{%
    \errmessage{(Inkscape) Transparency is used (non-zero) for the text in Inkscape, but the package 'transparent.sty' is not loaded}%
    \renewcommand\transparent[1]{}%
  }%
  \providecommand\rotatebox[2]{#2}%
  \newcommand*\fsize{\dimexpr\f@size pt\relax}%
  \newcommand*\lineheight[1]{\fontsize{\fsize}{#1\fsize}\selectfont}%
  \ifx\svgwidth\undefined%
    \setlength{\unitlength}{127.50840569bp}%
    \ifx\svgscale\undefined%
      \relax%
    \else%
      \setlength{\unitlength}{\unitlength * \real{\svgscale}}%
    \fi%
  \else%
    \setlength{\unitlength}{\svgwidth}%
  \fi%
  \global\let\svgwidth\undefined%
  \global\let\svgscale\undefined%
  \makeatother%
  \begin{picture}(1,0.54077601)%
    \lineheight{1}%
    \setlength\tabcolsep{0pt}%
    \put(0,0){\includegraphics[width=\unitlength,page=1]{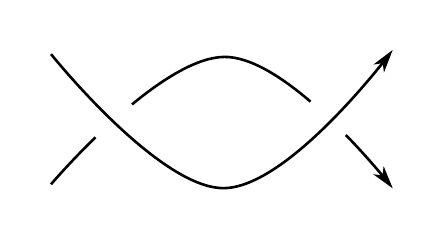}}%
    \put(0.02123498,0.42727488){\makebox(0,0)[lt]{\lineheight{1.25}\smash{\begin{tabular}[t]{l}$x_1$\end{tabular}}}}%
    \put(0.0222754,0.10385286){\makebox(0,0)[lt]{\lineheight{1.25}\smash{\begin{tabular}[t]{l}$x_2$\end{tabular}}}}%
    \put(0.89191869,0.42727242){\makebox(0,0)[lt]{\lineheight{1.25}\smash{\begin{tabular}[t]{l}$x_1$\end{tabular}}}}%
    \put(0.89295916,0.10385036){\makebox(0,0)[lt]{\lineheight{1.25}\smash{\begin{tabular}[t]{l}$x_2$\end{tabular}}}}%
    \put(0.46499628,0.43903634){\makebox(0,0)[lt]{\lineheight{1.25}\smash{\begin{tabular}[t]{l}$x_{2'}$\end{tabular}}}}%
    \put(0.46603666,0.04503073){\makebox(0,0)[lt]{\lineheight{1.25}\smash{\begin{tabular}[t]{l}$x_{1'}$\end{tabular}}}}%
  \end{picture}%
\endgroup%

  \caption{A scalar representation assigns the left-hand crossing $\omega(x_1, x_2)$ and the right-hand crossing $\omega(B(x_{2'}, x_{1'}))^{-1} = \omega(x_1, x_2)^{-1}$.}
  \label{fig:reidemeister-ii-scalar}
\end{marginfigure}
Because $\End_{\Bbbk}(\Bbbk^{\otimes 2})$ is canonically isomorphic to $\Bbbk$, the braiding a of scalar model is described by a family of scalars
\[
  \omega(x_1,x_2) \in \Bbbk^{\times}
\]
satisfying the colored braid relations.
In more detail, we assign a positive crossing $\sigma : (x_1, x_2) \to (x_{2'}, x_{1'})$ with incoming colors $(x_1, x_2)$ a scalar $\omega(x_1, x_2)$.
We take the convention that a \emph{negative} crossing is assigned $\omega(B^{-1}(x_1,x_2))^{-1}$, i.e.\@ the inverse of the value of $\omega$ on the \emph{outgoing} colors.
Both are shown in \cref{fig:cocycle-assignments}.
This convention seems somewhat unnatural, but as shown in \cref{fig:reidemeister-ii-scalar} it guarantees the \reidtwo{} relation.

The function $\omega$ must also satisfy the \reidthree{} relation, as shown in \cref{fig:cocycle-condition}.
Adopting the subscript/superscript notation shown there, it becomes
\begin{equation}
  \label{eq:cocycle-condition}
  1=
  \frac
  {\omega(y,z) \omega(x, z^y) \omega(x_{z^y}, y)}
  {\omega(x,y) \omega(x_y, z) \omega(y^x, z^{x_y})}
\end{equation}
which we can think of as a $2$-cocycle condition%
for maps $\omega : X \times X \to \Bbbk^\times$.

\begin{figure}
  \centering
  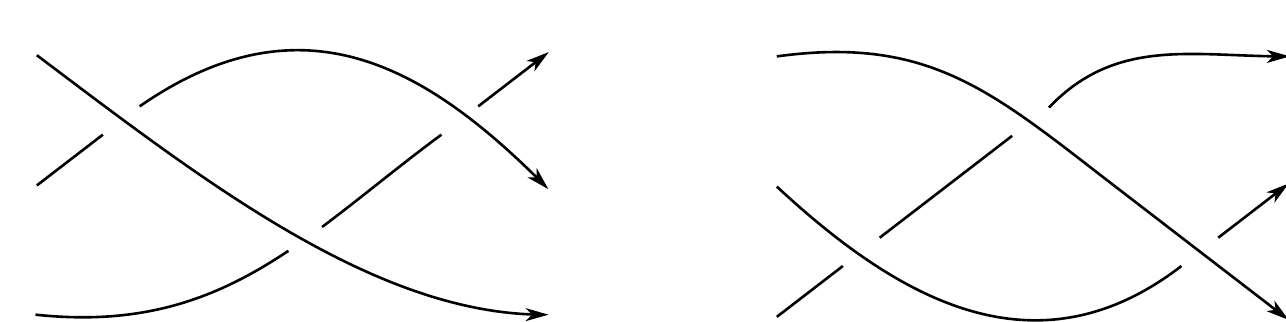
  \caption{Labels for the colored \reidthree{} relation as used in the $2$-cocycle relation.}
  \label{fig:cocycle-condition}
\end{figure}

Since $\omega$ comes from a biquandle representation $\mathcal{F}$, it satisfies a twist condition, specifically that
\begin{equation}
  \label{eq:cocycle-framing-condition}
  \omega(x, \alpha(x)) = \omega(\alpha^{-1}(x), x)^{-1}
\end{equation}
for every $x \in X$.
To derive this, note that
\[
  \omega(x, \alpha(x)) = \mathcal{F}(\theta_x^R) = \mathcal{F}(\theta_x^{L})
  = \omega(B^{-1}(\alpha^{-1}(x), x))^{-1}
  = \omega(\alpha^{-1}(x), x)^{-1}.
\]
\citeauthor{Kamada2018} \cite{Kamada2018} use a different convention on biquandle labels at crossings which makes the cocycle condition \eqref{eq:cocycle-condition} much simpler to state in general, but which is inconvenient for our purposes.
Roughly speaking, they emphasize the sideways braiding $\sidebraid$ over the braiding $B$.
Their convention also makes the twist condition \eqref{eq:cocycle-framing-condition} hold automatically.

\begin{defn}
  A function $\omega : X \times X \to \Bbbk^{\times}$ is a \defemph{biquandle $2$-cocycle} of $X$ with values in $\Bbbk^\times$ if it satisfies \cref{eq:cocycle-condition,eq:cocycle-framing-condition}.
  We say $\omega$ is \defemph{framing-independent} if $\omega(x, \alpha(x)) = 1$ for all $x$.
  We write $\operatorname{Z}^2(X;\Bbbk^\times)^{\fr}$ and $\operatorname{Z}^{2}(X;\Bbbk^\times)$ for the spaces of $2$-cocycles and framing-independent $2$-cocyles, respectively.
\end{defn}
\begin{thm}
  Every scalar representation $\mathcal{F}$ of a biquandle $X$ corresponds to a $2$-cocycle of $X$ and vice-versa.
  The cocycle is framing-independent if and only if $\mathcal{F}(\theta_x^R)$ is the identity for every $x$, equivalently if the corresponding invariant of framed $X$-colored links does not depend on the framing.
\end{thm}
\begin{proof}
  We explained how to get a $2$-cocycle from $X$, and it is not hard to see that the process works in reverse.
\end{proof}

As suggested by the terminology, there is a notion of \defemph{biquandle cohomology} \cite{Kamada2018}, which generalizes the quandle cohomology of \cite{Carter2003}.
In the notation of \cref{fig:cocycle-condition}, a $2$-coboundary is a map of the form
\begin{equation}
  \label{eq:coboundary-condition}
  d \phi(x,y) = \frac{\phi(x) \phi(y)}{ \phi(x_y) \phi(y^x)}
\end{equation}
for some $\phi : X \to \Bbbk^\times$.
It is not hard to see that any $1$-coboundary is a framing-independent $2$-cocycle.

\begin{defn}
  A biquandle $2$-cocycle is a \defemph{$2$-coboundary} if $\omega = d\phi$ for some $\phi : X \to \Bbbk^\times$, where $d\phi$ is given in \cref{eq:coboundary-condition}.
  Write $\operatorname{B}^2(X;\Bbbk^\times)$ for the space of $2$-cocycles.
  The \defemph{framed biquandle $2$-cohomology} of $X$ with coefficients in $\Bbbk^\times$ is the space
  \[
    \cohomol[2]{X; \Bbbk^\times}[\fr] \defeq \frac{\operatorname{Z}^2(X;\Bbbk^\times)^{\fr}}{\operatorname{B}^2(X;\Bbbk^\times)}.
  \]
  Similarly, the \defemph{biquandle $2$-cohomology} of $X$ is 
  \[
    \cohomol[2]{X; \Bbbk^\times} \defeq \frac{\operatorname{Z}^2(X;\Bbbk^\times)}{\operatorname{B}^2(X;\Bbbk^\times)}.
  \]
\end{defn}

The framed cohomology corresponds to what \cite{Kamada2018} call the \defemph{birack cohomology}, while the unframed version is their biquandle cohomology.

In general the representations corresponding to cohomologus cocycles will differ, but their values on links will agree.
Before proving this we introduce some notation.
Let $\mathcal{F}$ be a scalar representation of $X$ corresponding to a $2$-cocycle $\omega$.
We denote the link invariant given by $\mathcal{F}$ by $\omega(L)$.
Concretely, $\omega(L)$ can be computed by picking a diagram $D$ of $L$ with crossings $\{c_i\}_{i \in I}$.
If $c_i$ has incoming colors $(x_i,y_i)$, then
\[
  \omega(L) = \prod_{i \in I} W^{\epsilon_i}_\omega(x_i, y_i)
\]
where $\epsilon_i$ records the sign of the crossing and
\[
  W^{\epsilon_i}_\omega(x_i, y_i)=
  \begin{cases}
    \omega(x_i, y_i) & \epsilon_i = 1
    \\
    \omega(B(x_i, y_i))^{-1} & \epsilon_i = -1.
  \end{cases}
\]
is the scalar associated to the crossing by $\mathcal{F}$.

\begin{thm}
  The map $\omega \mapsto \omega(L)$ is a group homomorphism  $\operatorname{Z}^2(X;\Bbbk^\times)^{\fr} \to \Bbbk^\times$ which depends only on the cohomology class of $\omega$.
\end{thm}
As a corollary, we get a map $\cohomol[2]{X;\Bbbk^\times}[\fr] \to \Bbbk^\times$.
Since $\cohomol[2]{X;\Bbbk^\times}$ is a subgroup of $\cohomol[2]{X;\Bbbk^\times}[\fr]$, a similar result holds for the biquandle cohomology.

\begin{marginfigure}
  \centering
\begingroup%
  \makeatletter%
  \providecommand\color[2][]{%
    \errmessage{(Inkscape) Color is used for the text in Inkscape, but the package 'color.sty' is not loaded}%
    \renewcommand\color[2][]{}%
  }%
  \providecommand\transparent[1]{%
    \errmessage{(Inkscape) Transparency is used (non-zero) for the text in Inkscape, but the package 'transparent.sty' is not loaded}%
    \renewcommand\transparent[1]{}%
  }%
  \providecommand\rotatebox[2]{#2}%
  \newcommand*\fsize{\dimexpr\f@size pt\relax}%
  \newcommand*\lineheight[1]{\fontsize{\fsize}{#1\fsize}\selectfont}%
  \ifx\svgwidth\undefined%
    \setlength{\unitlength}{142.91485119bp}%
    \ifx\svgscale\undefined%
      \relax%
    \else%
      \setlength{\unitlength}{\unitlength * \real{\svgscale}}%
    \fi%
  \else%
    \setlength{\unitlength}{\svgwidth}%
  \fi%
  \global\let\svgwidth\undefined%
  \global\let\svgscale\undefined%
  \makeatother%
  \begin{picture}(1,0.68745838)%
    \lineheight{1}%
    \setlength\tabcolsep{0pt}%
    \put(0,0){\includegraphics[width=\unitlength,page=1]{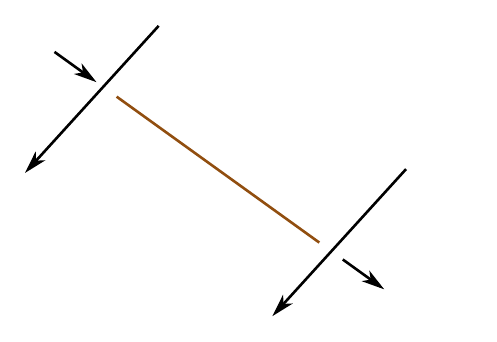}}%
    \put(0.26737791,0.49624671){\makebox(0,0)[lt]{\lineheight{1.25}\smash{\begin{tabular}[t]{l}$\phi(x)^{-1}$\end{tabular}}}}%
    \put(0.57616017,0.26937094){\makebox(0,0)[lt]{\lineheight{1.25}\smash{\begin{tabular}[t]{l}$\phi(x)$\end{tabular}}}}%
  \end{picture}%
\endgroup%

  \caption{The {\color{accent} gold} segment colored by $x \in X$ contributes $\phi(x)^{-1}$ and $\phi(x)$ to $d\phi(L)$.
  Here the segment is under-under, but the cancellation happens regardless of the type of the boundary crossings.}
  \label{fig:coboundary-proof}
\end{marginfigure}

\begin{proof}
  It is clear that $\omega \mapsto \omega(L)$ is a homomorphism, so to see that cocycles differing by a coboundary give the same invariant it suffices to show that $d\phi(L) = 1$.
  Pick a diagram $D$ of $L$ and consider a crossing $c$ with incoming colors $(x,y)$.
  If $c$ is positive, it is a map $(x, y) \to (x_y, y^x)$, so its contribution to $d\phi(L)$ is
  \[
    d\phi(x,y) = 
    \frac{\phi(x)\phi(y)}{\phi( x_{y} ) \phi( y^{x} )}
  \]
  If instead it is negative, it is a map $(x,y) \to (x^y, x_y)$ and the contribution is
  \[
    d\phi(B^{-1}(x,y))^{-1} = 
    d\phi(x^y, y_x)^{-1} =
    \left( \frac{\phi( y_x ) \phi( x^y )}{\phi(x)\phi(y)} \right)^{-1}
    =
    \frac{\phi(x)\phi(y)}{\phi( x^t ) \phi( x_y )}.
  \]
  In particular, we see that at a crossing $(a,b) \to (d,c)$, $d\phi$ contributes the ratio
  \[ 
    \frac{\phi(a)\phi(b)}{\phi(c) \phi(d)} 
  \]
  regardless of the sign of the crossing.

  Taking the product over all crossings, we see that each segment appears twice: once as in incoming color in the numerator, and once as an outgoing color in the denominator.
  These contribute $\phi(x)$ and $\phi(x)^{-1}$, which cancel, as shown in \cref{fig:coboundary-proof}.
  It follows that $d\phi(L) = 1$.
\end{proof}

Quandle and biquandle cohomology have been used to produce state-sum invariants of links; the idea is to take a finite biquandle $X$ and a link $L$ and sum over $\omega(L)$ for every $X$-coloring of $L$.
Here we take a contrasting perspective: for us, the data of an $X$-coloring of $L$ is inherently interesting, so we think of biquandle invariants as a function on the space of $X$-colorings of $L$.

We conclude with an example:
\begin{ex}
  The function
  \[
    \varkappa(\chi_1, \chi_2) = \frac{a_1}{a_{1'}}
  \]
  where $(\chi_{2'}, \chi_{1'}) = B(\chi_1, \chi_2)$ and $\chi_i = (a_i, b_i, \lambda_i)$ is a $2$-cocycle for the shape biquandle with values in $\CC^\times$.
  This can be checked by direct computation.
\end{ex}

In \cite{McPhailSnyder2020} we discussed a scalar holonomy invariant $\mathcal{K}$ related to a change in normalization between the holonomy invariant of \cite{Blanchet2020} and our version.
While the BGPR holonomy invariant $F'$ and the nonabelian quantum dilogarithm $\vecinv$ are defined in very similar ways, even using the same determinant normalization condition, they are not necessarily equal because different families of bases are used to compute the determinants.

As for group representations, if two representations of the same biquandle agree up to scalars, their ratio is a biquandle $2$-cocyle.
This interpretation is discussed in more detail in \cite[Section 6.4]{McPhailSnyder2020}.
We expect (based on some computations with the Burau representation) that for $\nr = 2$ the cocycle $\varkappa$ is closely related to the ratio between $\vecinv$ and $F'$, which in turn is connected with $\mathcal{K}$.

\chapter{Constructing the nonabelian quantum dilogarithm}
\label{ch:algebras}
\section*{Overview}
In the last chapter we reduced the problem of constructing invariants of $\slg$-links to constructing a representation of the (extended) shape biquandle in a pivotal category.
In this chapter, we construct a family $\vecfunc_{\nr}, \nr \ge 2$ of such representations taking values in the category $\modc{\qgrp}/\Gamma_{2\nr}$.
(We explain the notation $/\units{2\nr}$ below.)

Via the construction of \cref{ch:functors} the functor $\vecfunc$ produces an invariant $\vecinv_\nr(L)$ of extended $\slg$-links $L$ which depends only on the gauge class of the holonomy.
$\vecinv_\nr(L)$ is a complex number well-defined up to multiplication by a $2\nr$th root of unity.
We call $\vecinv_\nr(L)$ the \defemph{nonabelian quantum dilogarithm} of $L$, and we justify the name in \cref{sec:rel-to-other-inv}.

The quotient corresponds to the scalar indeterminacy of $\vecinv_\nr$.
We write $\units n$ for the group $\{z : z^n = 1\}$ of complex $n$th roots of $1$.
For a $\CC$-linear category $\catl C$ and a subgroup $\Gamma \subseteq \CC^\times$, the category $\catl C/\Gamma$ has the same objects and morphisms as $\catl C$, except that two morphisms are considered equivalent if they are equal after multiplication by an element of $\Gamma$.

\subsection{Statement of results}

\begin{thm}
  \label{thm:standard-model}
  For each nonsingular extended shape $\chi$ there is a unique (up to isomorphism) simple $\nr$-dimensional $\qgrp$-module denoted $\irrmod{\chi}$, the \defemph{standard module} for $\chi$.
  Furthermore, for every set of characters $\chi_i$ at a crossing satisfying $B(\chi_1, \chi_2) = (\chi_{2'}, \chi_{1'})$, there are linear maps
  \[
    S_{\chi_1, \chi_2} : \irrmod{\chi_1} \otimes \irrmod{\chi_2} \to \irrmod{\chi_{2'}} \otimes \irrmod{\chi_{1'}}
  \]
  which satisfy the colored \reidthree{} relation up to a $2\nr$th root of unity.
  That is, the pair $(V, S)$ gives a model of the extended shape biquandle in $\modc{\qgrp}/\Gamma_{2\nr}$.
\end{thm}
We can show that this model is a representation, so via the construction of \cref{ch:functors} it gives link invariants:
\begin{thm}
  \label{thm:vecfunc-exists}
  The model constructed in \cref{thm:standard-model} is sideways invertible and induces a twist, hence gives a functor $\vecfunc_{\nr} : \tangshe \to \modc{\qgrp}/\Gamma_{2\nr}$.
  Via the construction of \cref{thm:invariant-construction} and the modified traces on $\modc \qgrp$ we obtain a family of invariants $\vecinv_{\nr}(L)$ defined up to multiplication by a $2\nr$th root of unity, where $L$ is an extended $\slg$-link.
  The invariants $\vecinv_{\nr}(L)$ do not depend on the framing of $L$, up to a $2\nr$th root of unity.
  The $\vecinv_{\nr}(L)$ are gauge-invariant: they depend only on the conjugacy class of the holonomy of $L$.
  We call $\vecinv(L) = \vecinv_{\nr}(L)$ the $\nr$th \defemph{nonabelian quantum dilogarithm} of $L$.
\end{thm}
This theorem is essentially \cite[Proposition 6.9]{Blanchet2020} stated for the shape biquandle instead of the closely related factorization biquandle of $\slg$.
Our main improvement is that we can explicitly compute the braiding maps, at the cost of choosing extra roots of unity.
We are also able to compute the framing dependence explicitly in \cref{sec:twist-computation}: it is a power of $\xi$, which we can ignore because $\vecinv_{\nr}$ is only defined up to a power of $\xi$.

\begin{defn}
  Recall that an extended shape is a triple $\chi = (a, b, \mu)$.
  A \emph{radical} of $\chi$ is a choice of roots $\alpha^\nr = a$ and $\beta^\nr = b$.
\end{defn}
Unlike the global choice of fractional eigenvalue $\mu^\nr = \lambda$, which only needs to be made once for each connected component of a tangle diagram, the choice of radical is local.
In particular, Reidemeister moves on a diagram do not preserve radicals.
This ambiguity leads to the root-of-unity ambiguity in the definition of the braiding.
\begin{thm}
  \label{thm:braiding-factorization}
  Let $\chi_i$ be extended shapes as in \cref{thm:vecfunc-exists}, and assume that the relevant crossing is not pinched.
  Then:
  \begin{enumerate}
    \item
      Choosing radicals $(\alpha_i, \beta_i)$ of the shapes determines bases $\{v_m\}_{m \in \ZZ/\nr}$ of the modules $\irrmod{\chi_i}$.
    If the radicals are compatible in the sense that $\alpha_1 \alpha_2 = \alpha_{1'} \alpha_{2'}$, then with respect to these bases the matrix coefficients
    \[
      S(v_{m_1 }\otimes v_{m_2}) = \sum_{m_1' m_2'} S_{m_1 m_2}^{m_1' m_2'} v_{m_1'} \otimes v_{m_2'}
    \]
    of the braiding are given by
    \begin{equation}
      \label{eq:S-matrix-coeffs}
      S_{m_1 m_2}^{m_1' m_2'} = {\Theta} \frac{ \Lambda_f(m_1' - m_2') \Lambda_b(m_2 - m_1) }{ \Lambda_l(m_2 - m_2') \Lambda_r(m_1' - m_1) }
    \end{equation}
    for functions $\Lambda_i : \ZZ/\nr \to \CC$ and a scalar $\Theta$.
    These matrix coefficients do not depend on the choice of radicals, up to a power of $\xi$.
  \item
    The braiding factors as a product of four linear maps:
    \[
      S = 
      S_f (S_r \otimes S_l) S_b
    \]
    where each map acts by
    \begin{align*}
      S_f (v_{m_1 m_2})
        &=
        {\Lambda_f(m_1 - m_2)} v_{m_1 m_2}
        &
        S_r (v_m)
        &=
        \sum_{m} \frac{\Theta_r}{\Lambda_r(m' - m)} v_{m'}
        \\
        S_b (v_{m_1 m_2})
        &=
        \Lambda_b(m_2 - m_1) v_{m_1 m_2}
        &
        S_l (v_m)
        &=
        \sum_{m} \frac{\Theta_l}{\Lambda_l(m - m')} v_{m'}
    \end{align*}
    Here $v_{m_1 m_2} \defeq v_{m_1} \otimes v_{m_2}$.
    Each map $S_i$ is associated to one of the four tetrahedra $\tau_f, \tau_l, \tau_r, \tau_b$ at the crossing of the diagram corresponding to $S$.
\end{enumerate}
\end{thm}
The functions $\Lambda_i(m) = \qlogn{B_i, A_i}{m}$ above are \defemph{normalized cyclic quantum dilogarithms}.
$\qlogn{B,A}{m}$ is a $\nr$-periodic function of $m$ depending on two parameters $A, B$ with $A^\nr + B^\nr = 1$.
In order for this formula to make sense we must assume that $A_i, B_i \ne 0, 1$, which is generically (but not always) true; see \cref{sec:pinched-crossing-braiding-overview}.

The scalar factors $\Theta_i$ (which also depend on the $A_i$ and $B_i$) appearing above are related to the discrete Fourier transform of $1/\Lambda_i$.
These functions are discussed in more detail in \cref{ch:quantum-dilogarithm}.
The parameters $A_i$ and $B_i$ are $\nr$th roots of the complex dihedral angles of the tetrahedra at the crossing, and are given explicitly in \cref{thm:R-solution-unnorm}.
\note{
  In particular, the $A_i$ and $B_i$ are ratios of the roots $\alpha_i$ and $\beta_i$ used to fix bases.
  Technically we have to make one more choice of root, which introduces some ambiguity into the phases of the factors $S_i$.
  It turns out that this does not affect the values of the overall matrix coefficients $S_{m_1 m_2}^{m_1' m_2'}$.
}

The rest of this chapter is devoted to the proofs of these three theorems.
In \cref{sec:weyl-braiding,sec:constructing-braiding} we explain how to compute the coefficients $S_{m_1 m_2}^{m_1' m_2'}$ and prove \cref{thm:braiding-factorization} and then \cref{thm:standard-model}.
Afterward in \cref{sec:twists-sideways} we discuss twists and sideways invertibility for the braidings and prove \cref{thm:vecfunc-exists}.

\subsection{Relationship to other invariants}
\label{sec:rel-to-other-inv}

Via our association of the terms in the braiding with tetrahedra of the octahedral decomposition, our invariant is closely related to the quantum hyperbolic invariants of \citeauthor{Baseilhac2004} \cite{Baseilhac2004}, which are a generalization of \citeauthor{Kashaev1995}'s quantum dilogarithm invariant \cite{Kashaev1995}.
The $\hat\Psi$-systems of \cite{Geer2012} are also closely related.

As discussed in \cref{sec:abelian-limit} we think of the usual quantum dilogarithm invariant as corresponding to an abelian holonomy representation sending every meridian to (plus or minus) the identity matrix.
For this reason we call the general case where the holonomy can be a nonabelian representation of $\pi(L)$ the \emph{nonabelian} dilogarithm.
In the abelian limit, we recover known invariants of links:
\begin{thm}
  \label{thm:abelian-limit}
  Let $\rho$ be the representation with abelian image sending every meridian of $L$ to the matrix
  \[
    \begin{bmatrix}
      \mu^{\nr} & 0 \\
      0 & \mu^{-\nr}
    \end{bmatrix}
  \]
  and assign every component of $L$ the fractional eigenvalue $\mu$.
  Up to a power of $\xi$, the invariant $\vecinv_{\nr}(L, \rho, \mu)$ is equal to
  \begin{itemize}
    \item the quantum dilogarithm invariant \cite{Kashaev1995} of $L$ when $\mu = \xi^{\nr -1}$, or when $\mu^{\nr}$ is not a root of unity,
    \item the ADO invariant \cite{Akutsu1992} of $L$.
  \end{itemize}
\end{thm}
We prove this theorem in \cref{sec:pinched-crossings-braiding}.

\subsection{Future directions}
The phase ambiguity in the nonabelian quantum dilogarithm is a deficiency in the theory.
We expect it should be fixable.

\begin{conj}
  \label{conj:phase}
  By choosing some extra structure on links we can remove the phase ambiguity in $\vecfunc$.
\end{conj}
The matrices $S_{m_1 m_2}^{m_1' m_2'}$ of the braiding depend on the extra choice of radicals for each shape, leading to the phase ambiguity.
We expect that there is a way to coherently choose radicals that eliminates this ambiguity.
The extra structure is likely related to a choice of flattening \cite[Definition 1.3]{Zickert2009} of the triangulation associated to the shaped tangle, which is roughly a compatible choice of logarithms of the dihedral angles.
In the language of biquandles, \citeauthor{Inoue2014quandle} \cite{Inoue2014quandle} show that such a flattening is related to a \defemph{shadow coloring} of the link diagram.

The connection with flattenings also gives a geometric motivation for \cref{conj:phase}.
The hyperbolic volume $\operatorname{Vol}(K) = \operatorname{Vol}(S^3 \setminus K)$ of a hyperbolic knot complement can be seen as the real part of a  \defemph{complex volume}
\[
  \operatorname{Vol}(K) + i \operatorname{CS}(K) \in \CC/i \pi^2 \ZZ
\]
where $\operatorname{CS}(K)$ is the \defemph{Chern-Simons invariant} \cite{Zickert2009} of $K$.
Via the complexified volume conjecture \cite[Section 5.1]{Murakami2010} the phase of $\vecinv_N(K)$ should be related to the Chern-Simons invariant of $K$.

\subsection{The braiding for pinched crossings}
\label{sec:pinched-crossing-braiding-overview}
In \cref{sec:pinched-crossings} we discussed the singular case of pinched crossings, which occur when the vertices of the ideal octahedra are mapped to the same point of $\hat \CC$.
At such crossings the dihedral angles are $0$ or $1$.
As a consequence, the formula given in \cref{thm:braiding-factorization} is not defined at a pinched crossing because the arguments $A_i$ and $B_i$ are $\nr$th roots of the dihedral angles.

Even though we do not have explicit matrix coefficients in this case, the functor $\vecfunc$ is still well-defined.
We consider two cases.
At a \defemph{good} pinched crossing with modules $V_1 \otimes V_2$, the element $E \otimes F$ acts nilpotently, so we can determine the braiding using the universal $R$-matrix $\mathbf{R}$.
This case is related to the abelian limit in \cref{thm:abelian-limit}.

While in principle this makes sense, in practice the formula for the braiding given by $\mathbf R$ is very different from the formula in \cref{thm:braiding-factorization}.
It would be more satisfying to instead take the $A_i \to 0, B_i \to 1$ limit of \eqref{eq:S-matrix-coeffs}.
This should give an expression for the braiding similar to that of \cite[(2.12)]{Kashaev1995} and \cite[Section 4]{Murakami2001}.

At a \defemph{bad} pinched crossing, $E \otimes F$ does not act nilpotently, nor does the formula of \cref{thm:braiding-factorization} work.
Instead we use an abstract argument to characterize the braiding matrix.
In practice, this is not a barrier to computation: bad pinched crossings are generic%
\note{
  One place they show up is in twists: in general, the crossing occurring in $\theta_\chi^R$ will be a bad pinched crossing.
  However, we can reduce the computation of twists to the good case, as in \cref{sec:twist-computation}.
}
  and can be avoided.

\section{More on the Weyl algebra}
\label{sec:weyl-braiding}
The key ingredient in both determining the matrix coefficients \eqref{eq:S-matrix-coeffs} and in showing that they satisfy the colored braid relation is by presenting $\qgrp$ in terms of a Weyl algebra $\weyl$, as mentioned in \cref{sec:weyl-alg}.
By doing so, we can derive a system (\ref{eq:R-recurrence-i}--\ref{eq:R-recurrence-iv}) of $q$-difference equations%
\note{
  Perhaps $\xi^2$-difference equations would be a better term, since they are very special to the root-of-unity case.
}
that determine the braiding up to an overall scalar.

Our goal in this section is to derive these equations and make a first attempt at solving them: in particular, we do not yet worry about the overall scalar.
It turns out to be a significant problem, which we deal with in the later sections.

Before starting the computation we reminder the reader of the distinction between the $R$-matrix and the $S$-matrix.
The $S$-matrix is a braiding which satisfies the (colored) braid relation
\[
  (S \otimes \id)(\id \otimes S)(S \otimes \id) = (\id \otimes S)(S \otimes \id)(\id \otimes S).
\]
The matrix given in \cref{thm:braiding-factorization} is an $S$-matrix.%
\note{
  That's why we denoted it $S$!
}
An $R$-matrix instead satisfies the \defemph{Yang-Baxter relation}
\[
  R_{12} R_{13} R_{23} = R_{23} R_{13} R_{12}.
\]
These are equivalent via $S = \tau R$, where $\tau$ is the flip.
Similarly, an $S$-matrix intertwines the outer $S$-matrix $\Smat$, while an $R$-matrix intertwines the outer $R$-matrix $\Rmat$, and these automorphisms are related by $\Smat = \tau \Rmat$.

It is more convenient to use $S$ for topological purposes and to use $R$ for algebraic purposes.
For this reason we do some switching back and forth.
In particular, we will switch to focusing on $R$ for a bit, before returning to $S$.

\subsection{The outer \texorpdfstring{$R$}{R}-matrix}
\label{sec:weyl-outer-R-matrix}
Recall that in \cref{sec:weyl-alg} we expressed the quantum group $\qgrp$ in terms of the Weyl algebra $\weyl$ by giving a homomorphism $\phi : \qgrp \to \weyl$.
We can pull back the outer $R$-matrix $\Rmat$ of $\qgrp$ along $\phi$ to obtain an outer $R$-matrix for the Weyl algebra.
\begin{prop}
  \label{thm:R-action-on-weyl}
  There exists a unique algebra automorphism $\Rmat$ of the division algebra $\divalg{\weyl^{\otimes 2}}$ of $\weyl^{\otimes 2}$ which acts on the generators as 
\begin{align*}
  \Rmat^W(x_1)&=x_1g, \\
  \Rmat^W(x_2)&=g^{-1}x_2, \\
  \Rmat^W(y_1^{-1})&=y_2^{-1}+(y_1^{-1}-z_2^{-1}y_2^{-1})x_2^{-1}, \\
  \Rmat^W(y_2)&=\frac{z_1}{z_2}y_1+(y_2-{z_2}^{-1}y_1)x_1, \\
  \intertext{where}
  g&=1-x_1^{-1}y_1(z_1-x_1)y_2^{-1}(x_2-z_2^{-1}),
\end{align*}
such that the following diagram commutes:
\begin{equation}
  \label{eq:automorphism-pullback}
  \begin{tikzcd}
    \divalg{\qgrp^{\otimes 2}} \arrow[r, "\Rmat"] \arrow[d, "\phi \otimes \phi"] & \divalg{\qgrp^{\otimes 2}} \arrow[d, "\phi \otimes \phi"] \\
    \divalg{\weyl^{\otimes 2}} \arrow[r, "\Rmat^W"] & \divalg{\weyl^{\otimes 2}}
  \end{tikzcd}
\end{equation}
The inverse of $\Rmat^W$ is given by
\begin{align*}
  (\Rmat^W)^{-1}(x_1)&=x_1\widetilde{g}^{-1}, \\
  (\Rmat^W)^{-1}(x_2)&=\widetilde{g}x_2, \\
  (\Rmat^W)^{-1}(y_1^{-1})&=\frac{z_1}{z_2}y_2^{-1} + (y_1^{-1}-{z_1}y_2^{-1})x_2,\\
  (\Rmat^W)^{-1}(y_2)&=y_1 + (y_2 - {z_1}y_1)x_1^{-1},
\end{align*}
where 
\[
  \widetilde{g}= 1 - y_1 (z_1 - x_1) y_2^{-1} (1 - z_2^{-1} x_2^{-1}).
\]
\end{prop}
\begin{proof}
  The rules for $x_1$ and $x_2$ follow directly from those for $K_1$ and $K_2$.
  We give $y_1$ as an example, and $y_2$ can be computed similarly.
  Since $\Rmat(E_1) = E_1 K_2$, we have $\Rmat^W(\xi y_1(z_1 - x_1)) = \xi y_1(z_1-x_1) x_2$, or
  \[
    \Rmat^W(y_1)
    \left[
      z_1 - x_1 + y_1(z_1 - x_1) y_2^{-1} (x_2 - z_2^{-1})
    \right]
    =
    y_1(z_1 - x_1)x_2.
  \]
  We can multiply by  $y_1^{-1}(z_1 - x_1)^{-1}$ on the right to get
  \[
    \Rmat^W(y_1) \left[ y_1^{-1} + y_2^{-1} (x_2 - z_2^{-1}) \right] = x_2
  \]
  from which it is easy to derive
  \[
    \Rmat^W(y_1^{-1}) = y_2^{-1}+(y_1^{-1}-z_2^{-1}y_2^{-1})x_2^{-1}.\qedhere
  \]
\end{proof}

Usually we write $\Rmat$ for $\Rmat^W$.
We are particularly interested in the action of $\Rmat$ on the center of $\weyl^{\otimes 2}$.

\begin{prop}
  \label{thm:R-action-on-weyl-center}
  The action of $\Rmat^W$ on the center of $\weyl^{\otimes 2}$ is given by
  \begin{align*}
    \Rmat^W(x_1^r)&=x_1^rG , \\
    \Rmat^W(x_2^r)&=x_2^rG^{-1}, \\
    \Rmat^W(y_1^{-r})&= y_2^{-r}+\left(y_1^{-r} - \frac{y_2^{-r}}{z_2^{r}}\right)x_2^{-r}, \\
    \Rmat^W(y_2^r)&=\frac{z_1^r}{z_2^{r}}y_1^{r}+ \left(y_2^r - \frac{y_1^r}{z_2^r}\right)x_1^r, \\
    \intertext{where}
    G &= 1 + x_1^{-r} \frac{y_1^r}{y_2^r}(x_1^r - z_1^r)(x_2^r - z_2^{-r})
  \end{align*}
  with inverse action
  \begin{align*}
    (\Rmat^W)^{-1}(x_1^r)&= x_1^r \widetilde G^{-1}, \\
    (\Rmat^W)^{-1}(x_2^r)&= x_2^r \widetilde G, \\
    (\Rmat^W)^{-1}(y_1^{-r})&= \frac{z_1^r}{z_2^r}y_2^{-r}+(y_1^{-r} - z_1^r y_2^{-r})x_2^{r} ,  \\
    (\Rmat^W)^{-1}(y_2^{r})&=y_1^{r}+(y_2^r - z_1^r y_1^r)x_1^{-r}, \\
    \intertext{where}
    \widetilde{G} &= 1 + x_2^{-r} \frac{y_1^r}{y_2^r} (x_1^r - z_1^r)(x_2^r - z_2^{-r}).
  \end{align*}
\end{prop}

Because $\Rmat$ acts nontrivially on the center of $\weyl^{\otimes 2}$, it acts nontrivially on Weyl characters, and because $\Rmat$ has braiding properties this defines a biquandle on the characters.
In fact, this biquandle is exactly the shape biquandle.

\begin{prop}
  \label{thm:weyl-biquandle} 
  Let $\chi_1$ and $\chi_2$ be shapes such that $B(\chi_1, \chi_2) = (\chi_{2'}, \chi_{1'})$ is well-defined, as in \cref{fig:shaped-positive-crossing}.
  Then the action of $\Rmat$ on the center of $\weyl^{\otimes 2}$ is compatible with the shape biquandle in the sense that%
  \note{
    The order of the tensor factors on the left is different than for the biquandle because we are working with $\Rmat$ instead of~$\Smat$.
  }
  \[
    \chi_1 \otimes \chi_2
    =
    (\chi_1' \otimes \chi_2') \Rmat
  \]
\end{prop}
\begin{proof}
  Write $\chi_i(x^r) = a_i$ and $\chi_i(x^r) = a_i'$, and similarly for $b$ and $\lambda$.
  Then since
  \[
    \chi_1' \otimes \chi_2'  = (\chi_1 \otimes \chi_2) \Rmat^{-1}
  \]
  we can compute that
  \[
    \begin{aligned}
      b_2' = \chi_2'(y_2^r)
      =
      (\chi_1 \otimes \chi_2) \Rmat^{-1}(y_2^{-r})
      &=
      (\chi_1 \otimes \chi_2)
      \left(
        y_1^{r}+(y_2^r - z_1^r y_1^r)x_1^{-r}
      \right)
      \\
      &=
      b_1 + (b_2 - \lambda_1 b_1)a_1^{-1}
      \\
      &=
      b_1 \left( 1 - a_1^{-1}\left(\lambda_1 - \frac{b_2}{b_1}\right)\right).
    \end{aligned}
  \]
  We have derived one of the transformation rules for the shape biquandle given in \cref{def:shape-biquandle}.
  We can derive the other rules by checking the other generators.
\end{proof}

We defined the shape biquandle in terms of octahedral decompositions and their shape parameters before showing that this definition was related to quantum groups.
This is the opposite of how it was discovered: the material of \cref{sec:octahedral-decompositions} was developed to try to understand the algebraic rules determined by $\Rmat$.

\subsection{Standard cyclic modules}
\label{sec:standard-cyclic-modules}
In the usual construction of tangle invariants from quantum groups we assign $\qgrp[q]$-modules to the strands of a tangle diagram.
For example, the $\nr$th colored Jones polynomial comes from assigning an $\nr$-dimensional $\qgrp[q]$-module $\jkmod$ to each strand.%
\note{There are actually two such modules \cite[Theorem VI.3.5]{Kassel1995} but they determine the same invariant up to some signs.}
This is a highest-weight module with highest weight $\nr-1$: it is generated by a vector $v$ with
\[
  K \cdot v = q^{\nr -1} \text{ and } E \cdot v = 0.
\]
When $q = \nr$, the module $\jkmod$ is associated to the $\cent_0$-character
\[
  \chi(K^\nr) = (-1)^{\nr+1}, \quad \chi(E^\nr) = \chi(F^\nr) = 0
\]
at plus or minus the identity.
We are interested in both $\jkmod$ and ``deformations'' of $\jkmod$ to other, nonidentity $\cent_0$-characters.

On the other hand, we can obtain $\qgrp$-modules by taking $\weyl$-modules and pulling back along the embedding $\phi$ of \cref{thm:weyl-embedding}.
At $q = \xi$ an $\nr$th root of unity, these modules are $\nr$-dimensional, and they correspond exactly to the $\nr$-dimensional $\qgrp$-modules that generalize the colored Jones polynomial.
We describe these modules concretely, then discuss some of their properties.

\begin{defn}
  \label{def:weyl-irrmod}
  Let $\chi$ be an extended shape, that is a central character $\chi : Z(\weyl) \to \CC$ of the (extended) Weyl algebra $\weyl$.
  It is determined by its values on the central generators
  \[
    \chi(x^\nr) = a, \quad
    \chi(y^\nr) = b, \quad
    \chi(z) = \mu, \quad
  \]
  which we assume are nonzero.
  Choose a radical of $\chi$, that is choose roots $\alpha^\nr = a, \beta^\nr = b$.
  We can define a $\weyl$-module with basis $v_m, m = 0, \dots, \nr-1$ by
  \begin{align*}
    x \cdot v_m &= \alpha v_{m-1} 
                &
    y \cdot v_m &= \beta \xi^{2m} v_{m}
                &
    z \cdot v_m &= \mu v_m.
  \end{align*}
  We think of the index as lying in $\ZZ/\nr$, so that $x \cdot v_0 = \alpha v_{\nr-1}$.
  We call this module the \defemph{standard cyclic module} and denote it $\irrmod{\chi}$ or $\irrmod{\alpha, \beta, \mu}$ when we want to specify the choice of roots.

  We sometimes use a Fourier dual basis $\hat v_m$ of $\irrmod{\alpha, \beta, \mu}$ given by
  \[
    \hat v_m \defeq \sum^{r-1}_{k=0} \xi^{2mk} v_k
    \text{ and }
    v_m = \frac{1}{r} \sum_{k=0}^{r-1} \xi^{-2 m k} \hat v_k.
  \]
  With respect to this basis the action of $\weyl$ is given by
  \begin{align*}
    x \cdot \hat v_m &= \alpha \xi^{2m} \hat v_{m} 
                &
    y \cdot \hat v_m &= \beta \hat v_{m+1}
                &
    z \cdot \hat v_m &= \mu \hat v_m.
  \end{align*}
\end{defn}

We can think of the basis $\{\hat v_m\}$ as a highest-weight basis of $\irrmod{\alpha, \beta, \mu}$ with highest weight $w$ satisfying $\alpha = \xi^{w}$.
Pulling back along $\phi$ recovers the standard formulas for the action of $\qgrp$, as in \cite[Section 2]{Murakami2001}.
We have chosen to emphasise $v_m$ over $\hat v_m$ because it is more useful for our computation of the braiding.
\begin{prop}
  The isomorphism class of $\irrmod{\alpha, \beta, \mu}$ is determined by $\alpha^\nr$, $\beta^\nr$, and $\mu$, i.e.\@ by the associated extended shape.
\end{prop}
\begin{proof}
In the proof of \cref{thm:rooting-doesnt-matter} we will give an explicit family of isomorphisms between them.
\end{proof}

One reason to use the standard modules (instead of some other $\qgrp$-modules) is that they correspond to $\weyl$-modules.
However, they are essentially the only reasonable choice for a holonomy invariant defined in terms of $\qgrp$.
For most eigenvalues, the standard modules are the only possible choice of simple module:
\begin{thm}
  \label{thm:module-classification}
  Let $\chi$ be a shape, which determines a $\cent_0$ character $\chi = (\phi^*\chi)|_{\cent_0} : \cent_0 \to \CC$ by pulling back along $\phi : \qgrp \to \weyl$.
  \begin{enumerate}
    \item When $\chi(z^\nr) \ne \pm 1$, equivalently when $\chi(\operatorname{Cb}_\nr(\Omega)) \ne \pm 2$, every simple $\qgrp$-module with character $\chi$ is $\nr$-dimensional and projective.%
  \note{Here $\operatorname{Cb}_\nr$ is the $\nr$th Chebyshev polynomial defined in \cref{sec:quantum-group}.}
    \item There are $\nr$ isomorphism classes of such modules, determined by the $\nr$ distinct values by which $z$ acts.
  \end{enumerate}
\end{thm}
\begin{proof}
  This is a version of \cite[Lemma 6.3]{Blanchet2020} and \cite[Lemma 3, Lemma 24]{Geer2018trace}.
  In our notation, the action of $\Omega$ is determined by $\chi(\Omega) = \chi(\xi z + (\xi z)^{-1}) = \xi \mu + (\xi \mu)^{-1}$.

  The idea is to reduce to the case where $\chi$ has upper-triangular holonomy, where we can check this directly.
  To obtain this reduction, we use the \defemph{quantum coajoint action} of \citeauthor{DeConcini1991} \cite{DeConcini1991,DeConcini1992}.
  By considering certain derivations of $\qgrp$ we obtain a group $\mathfrak{G}$ of automorphisms acting on $\qgrp$.
  In particular, $\mathfrak{G}$ acts on the center $\cent$, and two $\cent$-characters $\chi$ and $\chi'$ are in the same orbit when their holonomy matrices $\psi(\chi), \psi(\chi')$ are conjugate in $\slg$.
  This allows us to reduce to the case $\chi(E^\nr) = \chi(F^\nr) = 0$.
  (In general, to do this we may have to consider  $\cent$-characters that do not correspond to central characters of $\weyl$.)
  We refer to \cite[Lemma 24]{Geer2018trace} for details.
\end{proof}

The condition on $\mu^\nr$ in this theorem excludes boundary-parabolic holonomy, which in some sense is the most interesting case because (as discussed in \cref{def:boundary-parabolic}) the holonomy of the complete hyperbolic structure on a hyperbolic link is always boundary-parabolic.
To include this case we need to be somewhat careful.
\note{
  As a particular case of boundary-parabolic holonomy we have modules where $\chi(x^\nr) = \chi(z^\nr) = \pm 1$.
  These are exactly modules for the small quantum group 
  \[
    \overline{\qgrp} \defeq \qgrp\left/\middle\langle K^{2\nr} - 1, E^\nr, F^\nr \right\rangle.
  \]
  The representation theory of $\overline{\qgrp}$ is much more complicated, and in particular is not semisimple.
  In fact, this is the source of the vanishing quantum dimensions that require our use of modified traces.
}

In order to use modified traces we need to make sure our modules are projective, and in order to prove gauge invariance and the uniqueness of the braiding we need them to be simple.
Fortunately as long as we use the correct root of unity things work out fine.
Let $\chi$ be an extended shape with eigenvalue $\lambda = \pm 1$.
Recall that $\chi$ is admissible if its fractional eigenvalue $\mu$ is equal to $\xi^{\nr -1}$.
This requires $\lambda = (\xi^{\nr-1})^{\nr} = (-1)^{\nr + 1}$.
\begin{thm}
  \label{thm:module-classification-parabolic}
  The standard module $\irrmod{\chi}$ corresponding to a admissible extended character $\chi$ is a simple, projective $\qgrp$-module.
\end{thm}
\begin{proof}
  We proved the case where $\lambda = \chi(z^\nr) \ne \pm 1$ above, and is not hard to extend to this case.
  The standard module associated to the character $\chi(K^\nr) = (-1)^{\nr}$, $\chi(E^\nr) = \chi(F^\nr) = 0$ is exactly the Steinberg module $V_\nr$, which is well-known to be simple and projective.
  If $\irrmod{\chi}$ is an admissible standard module with $\lambda = \pm 1$ our requirement that $\lambda = (-1)^{\nr + 1}$ ensures $\irrmod \chi$ lies in the $\mathfrak{G}$-orbit of $V_\nr$, so $\irrmod{\chi}$ is simple and projective as in the proof of \cref{thm:module-classification}.
\end{proof}

\section{Constructing the model}
\label{sec:constructing-braiding}
In this section we compute the matrix coefficients of the holonomy braiding and prove \cref{thm:braiding-factorization}.
\subsection{The inner \texorpdfstring{$R$ }{R}-matrix}
In order to define our quantum invariants we need to compute the $R$-matrix not just on the algebra $\weyl$, but on \emph{modules} for $\weyl$.
Let $\chi_i, i =1,2,1',2'$ be extended shapes at a positive crossing as in \cref{fig:shaped-positive-crossing}, let $V_i$ be the standard module associated to $\chi_i$, and write $\pi_i : \weyl \to \End_\CC(V_i)$ for the structure maps of the modules.

As first discussed in \cref{sec:holonomy-braidings}, a holonomy $R$-matrix for these modules is a linear operator $R : V_1 \otimes V_2 \to V_{2'} \otimes V_{1'}$ such that the diagram commutes:%
\note{
  The braiding $S$ and the $R$-matrix $R$ are related by $S = \tau R$, where $\tau$ is the flip map.
  It is more convenient to derive the relations for the $R$-matrix.
}%
\begin{equation}
  \label{eq:abstract-module-braiding-diagram}
  \begin{tikzcd}[row sep = large, column sep = large]
    \divalg{\weyl^{\otimes 2}} \arrow[r, "\Rmat"] \arrow[d, swap, "\pi_{1} \otimes \pi_{2}"] & \divalg{\weyl^{\otimes 2}} \arrow[d, "\pi_{2'} \otimes \pi_{1'}"] \\
    \End_\CC(V_1 \otimes V_2) \arrow[r, "a \mapsto R a R^{-1}"] & \End_\CC(V_{2'} \otimes V_{1'})
  \end{tikzcd}
\end{equation}

In order to actually compute things we need to choose bases of the modules $V_i$.
As mentioned in \cref{def:weyl-irrmod}, this requires a choice of radical for each shape $\chi_i$.
We will analyze later how these choices affect our computation; for now we pick them arbitrarily, with the single requirement that
\begin{equation}
  \label{eq:alpha-compatibility}
  \alpha_1 \alpha_2 = \alpha_{1'} \alpha_{2'}.
\end{equation}
This makes sense because $\Rmat(x^\nr \otimes x^\nr) = x^\nr \otimes x^\nr$, so $a_1 a_2 = a_{1'} a_{2'}$.

Given the choice of radical, we get structure maps
\[
  \pi_{i} : \weyl \to \End_\CC(\irrmodname_i) \to \End_\CC(\CC^\nr)
\]
For each $u \in \weyl$, $\pi_{i}(u)$ is not just an abstract endomorphism, but an explicit matrix.
Similarly, we can now think of $R$ as a concrete $\nr^2 \times \nr^2$-matrix such that the diagram
\begin{equation}
  \label{eq:weyl-module-braiding-diagram}
  \begin{tikzcd}[row sep = large, column sep = large]
    \divalg{\weyl^{\otimes 2}} \arrow[r, "\Rmat"] \arrow[d, swap, "\pi_{1} \otimes \pi_{2}"] & \divalg{\weyl^{\otimes 2}} \arrow[d, "\pi_{1'} \otimes \pi_{2'}"] \\
    \End_\CC(\CC^{\nr} \otimes \CC^{\nr}) \arrow[r, "a \mapsto R a R^{-1}"] & \End_\CC(\CC^{\nr} \otimes \CC^{\nr})
  \end{tikzcd}
\end{equation}
commutes.

\begin{thm}
  \label{thm:R-matrix-exists}
  An invertible matrix $R$ satisfying \cref{eq:weyl-module-braiding-diagram} exists and is unique up to an overall scalar.
\end{thm}
\begin{proof}
  The module $\irrmod{\chi_1}$ is simple, so its $\weyl$-endomorphism algebra is isomorphic to the algebra $\weyl/\ker \chi_1$.
  Similar arguments hold for the other representations, and we see that $\Rmat$ induces an automorphism
  \[
    \weyl/\ker {\chi_1} \otimes \weyl/\ker {\chi_2} \to
    \weyl/\ker {\chi_1'} \otimes \weyl/\ker {\chi_2'}.
  \]
  Pulling back along the structure maps $\pi_i$ gives an automorphism
  \[
    \End( (\CC^\nr)^{\otimes 2}) \to \End( (\CC^\nr)^{\otimes 2})
  \]
  of matrix algebras.
  Any such automorphism is inner and given by conjugation by some invertible matrix $R$, unique up to an overall scalar.
\end{proof}

We are now ready to compute recurrences for the matrix coefficients of a solution $R$ to \cref{eq:weyl-module-braiding-diagram}.
Abbreviate
\[
  \pi = \pi_1 \otimes \pi_2, \ \ \pi' = \pi_{1'} \otimes \pi_{2'}.
\]
Commutativity of the diagram in \cref{eq:weyl-module-braiding-diagram} implies matrix equations
\begin{align}
  \label{eq:intertwiner-rel}
  R \pi(x) &=  \pi'\left(\Rmat(x)\right) R, \ \ x \in \weyl \\
  \intertext{and}
  \label{eq:intertwiner-rel-inv}
  \pi'(x) R &= R \pi\left(\Rmat^{-1}(x)\right). \ \ x \in \weyl
\end{align}

Setting $x = y_1^{-1}$ in \cref{eq:intertwiner-rel} gives the relation
\begin{equation}
  \label{eq:yinv-rel}
  \begin{aligned}
    R \pi(y_1^{-1}) &= \pi'\left( y_2^{-1} + (y_1^{-1} - z_2^{-1} y_2^{-1})x_2^{-1}\right) R \\
                    &= \pi'\left( y_2^{-1} + x_2^{-1}(y_1^{-1} - z_2^{-1} y_2^{-1})\right) R
  \end{aligned}
\end{equation}
We want to understand this in terms of the matrix coefficients of $R$.
Recall the basis $v_m$ of $\irrmod{\alpha, \beta, \mu}$ indexed by $\ZZ/\nr$.
Abbreviating $v_{m_1 m_2} \defeq v_{m_1} \otimes v_{m_2}$, $R$ has matrix coefficients
\[
  R \cdot v_{m_1 m_2} = \sum_{m_1' m_2'} R_{m_1 m_2}^{m_1' m_2'} v_{m_1' m_2'}.
\]
Here and in the rest of this chaper sums without explicit limits are over $\{0, \dots, \nr -1\}$.%
\note{
  Usually our functions will be $\nr$-periodic, so we can think of the sums as being over $\ZZ/\nr$.
}
In terms of these matrix coefficients, the relation \eqref{eq:yinv-rel} from $y_1^{-1}$ becomes
\begin{align*}
  &  \sum_{m_1' m_2'} R_{m_1 m_2}^{m_1' m_2'} ( \beta_1 \xi^{2m_1} )^{-1} v_{m_1' m_2'} \\
  &= \sum_{m_1' m_2'} R_{m_1 m_2}^{m_1' m_2'} (\beta_{2'} \xi^{2m_2'} )^{-1} v_{m_1' m_2'} \\
  &\phantom{=} + \sum_{m_1' m_2'} R_{m_1 m_2}^{m_1' m_2'}\left[ (\beta_{1'} \xi^{2m_1'})^{-1} - \mu_2^{-1} (\beta_{2'} \xi^{2m_2'})^{-1} \right] \alpha_{2'}^{-1} v_{m_1', m_2'+1}
\end{align*}
which we can rewrite as the recursion
\[
  R_{m_1 m_2}^{m_1' m_2'} \left[ \beta_1^{-1} \xi^{-2m_1} - \beta_{2'}^{-1} \xi^{-2m_2'} \right]
  =
  R_{m_1 m_2}^{m_1' m_2'-1} \alpha_{2'}^{-1} \left[ \beta_{1'}^{-1} \xi^{-2m_1'} - \xi^2 \mu_2^{-1} \beta_{2'}^{-1} \xi^{-2m_2'} \right]
\]
or
\[
  R_{m_1 m_2}^{m_1' m_2' } =  R_{m_1 m_2}^{m_1', m_2' - 1} {\alpha_{2'}}^{-1} \frac{\beta_{1'}^{-1} \xi^{-2 m_1'} - \xi^2 \mu_2^{-1} \beta_{2'}^{-1} \xi^{-2m_2' } }{\beta_1^{-1} \xi^{-2m_1} - \beta_{2'}^{-1} \xi^{-2m_2' }} 
\]
In a slightly more convenient form, this is
\begin{align*}
  R_{m_1 m_2}^{m_1' m_2' }
  &=
  R_{m_1 m_2}^{m_1', m_2' -1}
  \frac{1}{\alpha_{2'} \mu_2 }
  \frac
  {1 - (\mu_2 \beta_{2'}/\xi^{2} \beta_{1'}) \xi^{2(m_2' - m_1')}}
  {1 - (\beta_{2'}/\beta_1) \xi^{2(m_2' - m_1)}} 
\end{align*}
By using \cref{eq:intertwiner-rel} with $y_2$ and \cref{eq:intertwiner-rel-inv} with  $y_1^{-1}$ and $y_2$ we get a family of recurrences that determine $R$.
\begin{prop}
  \label{thm:R-recurrences}
  The matrix coefficients satisfy recurrence relations
  \begin{align}
    \label{eq:R-recurrence-i}
    R_{m_1 m_2}^{m_1' m_2' }
    &=
    R_{m_1 m_2}^{m_1', m_2' -1}
    \frac{1}{\alpha_{2'} \mu_2 }
    \frac
    {1 - (\mu_2 \beta_{2'}/\xi^2 \beta_{1'}) \xi^{2(m_2' - m_1')}}
    {1 - (\beta_{2'}/\beta_1) \xi^{2(m_2' - m_1)}} 
    \\
    \label{eq:R-recurrence-ii}
    R_{m_1 m_2}^{m_1', m_2'}
    &=
    R_{m_1 m_2}^{m_1'-1, m_2'}
    \frac{\mu_1}{\alpha_{1'}}
    \frac
    {1 - (\mu_2 \beta_2/ \mu_1 \beta_{1'}) \xi^{2(m_2 - m_1') + 2} }
    {1 - (\mu_2 \beta_{2'}/ \xi^2 \beta_{1'}) \xi^{2(m_2' - m_1') + 2} }
    \\
    \label{eq:R-recurrence-iii}
    R_{m_1, m_2}^{m_1' m_2'}
    &=
    R_{m_1, m_2 -1}^{m_1' m_2'}
    {\alpha_2 \mu_2 }
    \frac
    {1 - (\beta_2/\xi^2 \mu_1 \beta_1) \xi^{2(m_2 - m_1)} }
    {1 - (\mu_2 \beta_2/ \mu_1 \beta_{1'}) \xi^{2(m_2 - m_1')}}
    \\ 
    \label{eq:R-recurrence-iv}
    R_{m_1 m_2}^{m_1' m_2'}
    &=
    R_{m_1 - 1, m_2}^{m_1' m_2'}
    \frac{\alpha_1}{\mu_1}
    \frac
    {1 - (\beta_{2'}/\beta_1) \xi^{2(m_2' - m_1) +2 }}
    {1 - (\beta_2/\xi^2 \mu_1 \beta_1) \xi^{2(m_2 - m_1) +2} } 
  \end{align}
  which determine $R$ uniquely up to an overall scalar.
\end{prop}
\begin{proof}
  The above relations are clearly sufficient to determine $R$ up to a scalar, and then existence and uniqueness follow from \cref{thm:R-matrix-exists}.
\end{proof}

It is not immediately clear that the recurrences in \cref{thm:R-recurrences} give solutions that are periodic mod~$\nr$; we will show this by solving them in terms of the quantum dilogarithm.

\begin{defn}
  \label{def:cyclic-dilog}
  Let $A, B \in \CC \setminus \{0,1\}$ satisfy $A^\nr + B^\nr = 1$.
  The \defemph{cyclic quantum dilogarithm}
  \note{
    It is not particularly clear at first glance what this function has to do dilogarithms.
    See \cref{ch:quantum-dilogarithm} for an explanation.
  }
  is the function defined by%
  \[
    \qlog{B,A}{n} = \prod_{k=1}^{n} A^{-1} (1 - \xi^{2k} B)
  \]
  for $n = 0, \dots, \nr-1$.
\end{defn}
We are actually more interested in a particular normalization $\qlogn{B,A}{n}$ of the quantum dilogarithm, which satisfies $\qlogn{B,A}{n}/\qlogn{B,A}{0} = \qlog{B,A}{n}$.
We discuss this in the next section and in \cref{ch:quantum-dilogarithm}.
For now, it is sufficient to check some recurrence properties of  the unnormalized version:
\begin{prop}
  \label{thm:dilog-recurrences}
  $\qlog{A, B}{-}$ is well-defined on $\ZZ/\nr$:
  \[
    \qlog{A, B}{m + \nr} = \qlog{A, B}{m}
  \]
  and satsifies the recursive relations
  \begin{align*}
    \qlog{B, A}{m} &= B^{-1} (1 - \xi^{2m} A) \qlog{B, A}{m-1} \\
    \qlog{B, A}{m} &= B (1 - \xi^{2m+2} A)^{-1} \qlog{B, A}{m+1}
  \end{align*} 
\end{prop}
\begin{proof}
  The relations are obvious for $m = 0, \dots, \nr-1$ and follow for all $m$ from the first claim.
  It therefore suffices to show
  \[
    1 = \qlog{B, A}{\nr} = \frac{1}{A^\nr} \prod_{k=1}^{\nr} (1 - \xi^{2k} B),
  \]
  and the necessary relation
  \[
    \prod_{k=1}^{\nr} (1 - \xi^{2k} B) = 1 - B^\nr
  \]
  is the generalization of $(1 - B)(1 +  B) = 1 - B^2$ to higher-order roots of unity.
  To confirm the sign of $B^r$, notice that its coefficient is
  \[
    (-1)^\nr \xi^{\sum_{k = 1}^\nr 2k} = (-1)^\nr \xi^{\nr(\nr+1)} = (-1)^{\nr + \nr + 1} = -1. \qedhere
  \]
\end{proof}

\begin{thm}
  \label{thm:R-solution-unnorm}
  The coefficients of an $R$-matrix $R = R_{\chi_1, \chi_2}$ satisfying \cref{eq:weyl-module-braiding-diagram} are periodic mod $\nr$ and given by
  \[
    R_{m_1 m_2}^{m_1' m_2'} = \frac{ L_f(m_2' - m_1') L_b(m_2 - m_1) }{ L_l(m_2 - m_1') L_r(m_2' - m_1) }.
  \]
  We call this matrix the \defemph{unnormalized $R$-matrix}.
  The functions $L_i$ are quantum dilogarithms
  \begin{align*}
    L_f(m)
    &=
    \qlog{B_f, A_f}{m}
    =
    \qlog{\frac{\mu_2 \beta_2'}{\xi^{2} \beta_1'}, A_f}{m}
    \\
    L_b(m)
    &=
    \qlog{B_b, A_b}{m}
    =
    \qlog{ \frac{\beta_2}{\xi^{2} \mu_1 \beta_1}, \frac{\alpha_1}{\mu_1 \mu_2 \alpha_2'} A_f}{m}
    \\
    L_l(m)
    &=
    \qlog{B_l, A_l}{m}
    =
    \qlog{\frac{\mu_2 \beta_2}{\mu_1 \beta_1'}, \frac{\alpha_1'}{\mu_1} A_f}{m}
    \\
    L_r(m)
    &=
    \qlog{B_r, A_r}{m}
    =
    \qlog{\frac{\beta_2'}{\beta_1}, \frac{1}{\mu_2 \alpha_2'} A_f}{m}
  \end{align*}
  and $A_f$ is a scalar such that
  \[
    A_f^\nr = 1 - \left( \frac{\mu_2 \beta_2'}{\xi^2 \beta_1'} \right)^\nr = 1 - \lambda_2 \frac{b_2'}{b_1'}.
  \]
  The values of the coefficients $R_{m_1 m_2}^{m_1' m_2'}$ do not depend on the choice of $A_f$.
\end{thm}

\begin{remark}
  The four quantum dilogarithms appearing in the statement of the theorem can be associated to the four tetrahedra at a crossing in \cref{fig:labeled-octahedron}.
  For example, by consulting \cref{table:positive-crossing-data} we see that their arguments are $\nr$th roots of the complex dihedral angles:
  \begin{align*}
    B_f^{\nr} &= z_f^{\circ}
              &
    B_b^{\nr} &= z_b^{\circ}
              &
    B_r^{\nr} &= \frac{1}{z_r^{\circ}}
              &
    B_l^{\nr} &= \frac{1}{z_l^{\circ}}
    \\
    A_f^{\nr} &= \frac{1}{z_f^{\circ\circ}}
              &
    A_b^{\nr} &= \frac{1}{z_b^{\circ\circ}}
              &
    A_r^{\nr} &= {z_r}
              &
    A_l^{\nr} &= {z_l}
  \end{align*}
  As mentioned in \cref{thm:braiding-factorization} this corresponds to the factorization of the braiding into four terms, one for each tetrahedron.
\end{remark}

\begin{proof}
  It is immediate from \cref{thm:R-recurrences,thm:dilog-recurrences} that there is a solution in terms of quantum dilogarithms
  \begin{align*}
    L_f(m)
    &=
    \qlog{\frac{\mu_2 \beta_2'}{\xi^2 \beta_1'}, A_f}{m}
    &
    L_b(m)
    &=
    \qlog{ \frac{\beta_2}{\xi^2 \mu_1 \beta_1}, A_b}{m}
    \\
    L_l(m)
    &=
    \qlog{\frac{\mu_2 \beta_2}{\mu_1 \beta_1'}, A_l}{m}
    &
    L_r(m)
    &=
    \qlog{\frac{\beta_2'}{\beta_1},  A_r}{m}
  \end{align*}
  where
  \begin{align*}
    \frac{A_r}{A_f} &= \frac{1}{\alpha_{2'} \mu_2}
                    &
    \frac{A_f}{A_l} &= \frac{\mu_1}{ \alpha_{1'}}
                    &
    \frac{A_l}{A_b} &= \mu_2 \alpha_2
                    &
    \frac{A_b}{A_r} &= \frac{\alpha_1}{\mu_1}
  \end{align*}
  We can solve for the parameters $A_l, A_r$, and $A_b$ in terms of $A_f$:
  \begin{align*}
    A_l &= \frac{\alpha_{1'}}{\mu_1} A_f
        &
    A_r &= \frac{1}{\alpha_{2'} \mu_2} A_f
        &
    A_b &= \frac{\alpha_1}{\mu_1} A_r =  \frac{\alpha_1}{\alpha_{2'} \mu_1 \mu_2} A_f
  \end{align*}
  To see that these choices are consistent, we must check that
  \[
    \frac{A_l}{A_b} = \mu_2 \alpha_2
  \]
  equivalently that
  \[
    \frac{\alpha_{1'}}{\mu_1} \frac{\mu_1 \mu_2 \alpha_{2'}}{\alpha_1} = \mu_2 \alpha_2,
  \]
  which follows from the choice that $\alpha_1 \alpha_2 = \alpha_{1'} \alpha_{2'}$ made in \cref{eq:alpha-compatibility}.

  We chose $A_f$ so that $B_f^\nr + A_f^\nr = 1$.
  We can show that $B_i^\nr + A_i^\nr = 1$ for the other parameters by using the interpretation of the $a$-paramters in \cref{eq:vertical-angles-positive}.
  For example,
  \[
    A_r^\nr = \frac{A_f^\nr}{a_{2'} \lambda_2} = \frac{z_r z_f^{\circ\circ}}{z_f^{\circ \circ}} = 1 - B_r^{\nr}
  \]
  as required.

  Finally we show that the matrix coefficients $R_{m_1 m_2}^{m_1' m_2'}$ do not depend on the choice of root $A_f$.
  Recall that
  \[
    \qlog{B,A}{m} = A^{-m} \prod_{k=1}^{m} (1 -  \xi^{2k} B)
  \]
  so $R_{m_1 m_2}^{m_1' - m_2'}$ is the product of terms depending only on the $B_i$ and the expression
  \[
  \frac{ A_f^{m_1' - m_2'} A_b^{m_1 - m_2} }{ A_l^{m_1' - m_2} A_r^{m_1 - m_2'} }
  \]
  in which all factors of $A_f$ cancel.
\end{proof}

\subsection{The braiding for modules}
\label{sec:normalized-braiding}
We are now ready to present our computation in the form given in \cref{thm:standard-model}.
In addition to adding a flip (since $S = \tau R$) we also want to add some scalar factors.

\begin{defn}
  \label{def:qlogn}
  The \defemph{normalized cyclic quantum dilogarithm} is the function
  \[
    \qlogn{B,A}{n} \defeq g(B,A) \qlog{B,A}{n} = g(B,A) \prod_{k=1}^{n} (1 - \xi^{2k} B)/A
  \]
  where $g(B,A)$ is the scalar defined in \cref{thm:normalization-properties}.
\end{defn}

\begin{defn}
  \label{def:normalized-braiding}
  The \defemph{normalized holonomy braiding} is the linear map $S_{\chi_1, \chi_2}$ with matrix coefficients
  \begin{equation}
    \label{eq:normalized-braiding}
    S_{m_1 m_2}^{m_1' m_2'} =
    {\Theta_l \Theta_r}
    \frac{ \Lambda_f(m_1' - m_2') \Lambda_b(m_2 - m_1) }{ \Lambda_l(m_2 - m_2') \Lambda_r(m_1' - m_1) }.
  \end{equation}
  Here $\Lambda_i(m) = \qlogn{B_i, A_i}{m}$ are \emph{normalized} quantum dilogarithms
  \begin{align*}
     \Lambda_f(m)
    &=
    \qlogn{ B_f,  A_f}{m}
    =
    \qlogn{\frac{\mu_2 \beta_2'}{\xi^{2} \beta_1'}, A_f}{m}
    \\
     \Lambda_b(m)
    &=
    \qlogn{ B_b,  A_b}{m}
    =
    \qlogn{ \frac{\beta_2}{\xi^{2} \mu_1 \beta_1}, \frac{\alpha_1}{\mu_1 \mu_2 \alpha_2'} A_f}{m}
    \\
     \Lambda_l(m)
    &=
    \qlogn{ B_l,  A_l}{m}
    =
    \qlogn{\frac{\mu_2 \beta_2}{\mu_1 \beta_1'}, \frac{\alpha_1'}{\mu_1} A_f}{m}
    \\
     \Lambda_r(m)
    &=
    \qlogn{ B_r,  A_r}{m}
    =
    \qlog{\frac{\beta_2'}{\beta_1}, \frac{1}{\mu_2 \alpha_2'} A_f}{m}
  \end{align*}
  with the same parameters as \cref{thm:R-solution-unnorm}.
  The normalization factors $\Theta_l$ and $\Theta_r$ are given by
   \[
     \Theta_i = \frac{\qlogsum{\xi^{-2} A_i, B_i}}{N}
  \]
  where $\qlogsumname$ is the function%
  \note{
    $\qlogsum{A,B}$ is related to the sum $\sum_{k=0}^{\nr -1} \qlogn{B,A}{k}$ of values of the quantum dilogarithm.
    We give a relatively simple expression for it in \cref{thm:qlogsum-formula}.
  }
  given in \cref{thm:qlog-fourier}.
\end{defn}

\begin{proof}[Proof of \cref{thm:braiding-factorization} (1)]
  We claimed that $S = S_f(S_r \otimes S_l) S_b$, where
  \begin{align*}
    S_f (v_{m_1 m_2})
    &=
    {\Lambda_f(m_1 - m_2)} v_{m_1 m_2}
    &
    S_r (v_m)
    &=
    \sum_{m} \frac{\Theta_r}{\Lambda_r(m' - m)} v_{m'}
    \\
    S_b (v_{m_1 m_2})
    &=
    \Lambda_b(m_2 - m_1) v_{m_1 m_2}
    &
    S_l (v_m)
    &=
    \sum_{m} \frac{\Theta_l}{\Lambda_l(m - m')} v_{m'}
  \end{align*}
  Given our formula for $S$, this is an easy check:
  \[
    \begin{split}
      S_f(S_r \otimes S_l) S_b (v_{m_1 m_2})
      &=
      S_f(S_r \otimes S_l) \left( \Lambda_b(m_2 - m_1) v_{m_1 m_2} \right)
      \\
      &=
      S_f \left( \sum_{m_1' m_2'} \Theta_r \Theta_l \frac{\Lambda_b(m_2 - m_1)}{\Lambda_r(m_1' - m_1) \Lambda_l(m_2' - m_2)} v_{m_1' m_2'} \right)
      \\
      &=
      \sum_{m_1' m_2'} \Theta_r \Theta_l \frac{\Lambda_f(m_1' - m_2') \Lambda_b(m_2 - m_1)}{\Lambda_r(m_1' - m_1) \Lambda_l(m_2' - m_2)} v_{m_1' m_2'} 
      \\
      &=
      \sum_{m_1' m_2'} S_{m_1 m_2}^{m_2' m_1'} v_{m_1 m_2}
      \\
      &=
      S(v_{m_1 m_2}). \qedhere
    \end{split}
  \]
\end{proof}

It remains to prove part (2) of \cref{thm:braiding-factorization}.
When we defined the holonomy braidings
\[
  S_{\chi_1, \chi_2} :
  \irrmod{\chi_1} \otimes \irrmod{\chi_2}
  \to
  \irrmod{\chi_{2'}} \otimes \irrmod{\chi_{1'}}
\]
we made two extra choices:
\begin{enumerate}
  \item A choice of root $A_f^\nr = 1 - \lambda_2 \frac{b_{2'}}{b_{1'}}$.
  \item A choice radicals of the shapes, i.e.\@ of roots $\alpha_i^\nr = a_i$ and $\beta_i^\nr = b_i$ with $\alpha_1 \alpha_2 = \alpha_{1'} \alpha_{2'}$, which corresponds to a choice of basis for the modules $\irrmod{\chi_i}$.
\end{enumerate}
We can now check that these extra choices only affect our answer up to some powers of $\xi$.
The first is easy:
\begin{lem}
  The coefficients $S_{m_1 m_2}^{m_1' m_2'}$ depend on the choice of $A_f$ up to an overall power of $\xi^2$.
\end{lem}
\begin{proof}
  Any other choice would be of the form $\xi^{2k}A_f$ for some $m$.
  Recall the matrix coefficients $S_{m_1m_2}^{m_1' m_2'}$ in \cref{eq:normalized-braiding} are a product of two normalization terms $\Theta_r, \Theta_l$ and four normalized dilogarithms $\Lambda_i$.
  Set $\eta = \xi^{-(\nr-1)^2}$.

  Since the $A$-argument of each normalized dilogarithm in \cref{eq:normalized-braiding} is a multiple of $A_f$, we can use \cref{eq:qlog-shift-A} to see that changing $A_f \to \xi^{2m} A_f$ changes the dilogarithms by a factor of
  \[
    \left(\eta\xi^2\right)^{-k(m_1' - m_2' + m_2 - m_1 - m_2 + m_2' - m_1' + m_2)} 
    =
    1.
  \]
  On the other hand, the $\Theta_i$ transform as
  \[
    \Theta_r \Theta_l \to \eta^{2k} \Theta_r \Theta_l = \xi^{-2k(\nr-1)^2}.\qedhere
  \]
\end{proof}

Analyzing the dependence of $S$ on the second choice is more involved because it involves a change of basis.
Suppose we make two different choices of radical for the same shape, so we have two distinct isomorphic modules $\irrmod{\alpha, \beta, \mu}$ and $\irrmod{\xi^{2k} \alpha, \xi^{2l} \beta, \mu}$.
We pick a preferred family of isomorphisms between them:
\begin{equation}
  \label{eq:rooting-isormorphisms}
  f_{k,l}:
  \begin{cases}
    \irrmod{\alpha, \beta, \mu} \to \irrmod{\xi^{2k} \alpha, \xi^{2l} \beta, \mu}
    \\
    v_m \mapsto \xi^{-2(m-l)k}v_{m-l}.
  \end{cases}
\end{equation}
We want to show that changing basis using the $f_{k,l}$ is equivalent to computing $S$ using the other radical, up to some powers of $\xi$.

In more detail, suppose $B(\chi_1, \chi_2) = (\chi_{2'}, \chi_{1'})$ and $S = S_{\chi_1, \chi_2}$ is the corresponding holonomy braiding computed using roots $\alpha_i, \beta_i$ given in \cref{def:normalized-braiding}.
We could have made another choice of roots $\mathring{\alpha}_i = \xi^{2k_i} \alpha_i$ and $\mathring{\beta}_i = \xi^{2l_i} \beta_i$ and computed a matrix $\mathring{S}$ as in the same way as $S$ but with the $\mathring{\alpha}_i$ and $\mathring{\beta}_i$.
Write
\begin{align*}
  f
    &=
    f_{k_1, l_1} \otimes f_{k_2, l_2} :
    \irrmod{\alpha_1, \beta_1, \mu_1} \otimes \irrmod{\alpha_2, \beta_2, \mu_2}
    \\
    &\to
    \irrmod{\xi^{2k_1}\alpha_1, \xi^{2 l_1}\beta_1, \mu_1} \otimes \irrmod{\xi^{2k_2}\alpha_2, \xi^{2l_2}\beta_2, \mu_2}
    \\
    \intertext{and}
    g
    &= 
    f_{k_{2'}, l_2} \otimes f_{k_{1'}, l_{1'}} :
    \irrmod{\alpha_{2'}, \beta_{2'}, \mu_{2}} \otimes \irrmod{\alpha_{1'}, \beta_{1'}, \mu_{1}} \\
    &\to
    \irrmod{\xi^{2k_{2'}}\alpha_{2'}, \xi^{2l_{2'}}\beta_{2'}, \mu_{2}} \otimes \irrmod{\xi^{2k_{1'}}\alpha_{1'}, \xi^{2l_{1'}}\beta_{1'}, \mu_{1}}
\end{align*}
for the isomorphisms changing bases between those given by the two different choices of roots.

\begin{lem}
  \label{thm:rooting-doesnt-matter}
  The matrices $S$ and $\mathring{S}$ are compatible in the sense that
  \[
    S = g^{-1} \mathring{S} f
  \]
  up to multiplication by a power of $\xi$.
\end{lem}
\begin{proof}
  Abbreviate $\xi^{2} = \zeta$, as in \cref{ch:quantum-dilogarithm}.
  Recall that $S$ has coefficients
  \[
    S_{m_1 m_2}^{m_1' m_2'} =
    \frac{\qlogsum{\zeta^{-1}A_r, B_r} \qlogsum{\zeta^{-1}A_l, B_l}}{\nr^2}
    \frac{ \qlogn{B_f, A_f}{m_1' - m_2'} \qlogn{B_b, A_b}{m_2 - m_1} }{ \qlogn{B_l, A_l}{m_2 - m_2'} \qlogn{B_r, A_r}{m_1' - m_1} }
  \]
  where the $B_i$ and $A_i$ are certain ratios of the $\alpha_i, \beta_i$, and $\mu_i$ given in \cref{thm:R-solution-unnorm}.
  Because it uses different choices of roots, the matrix $\mathring{S}$ has coefficients
  \begin{align*}
    \left(\mathring{S}\right)_{m_1 m_2}^{m_1' m_2'} 
    &=
    \frac{ \qlogsum{\zeta^{-1 + k_{1'}} A_l, \zeta^{l_2 - l_{1'} } B_l} \qlogsum{\zeta^{-1 - k_{2'}}A_r, \zeta^{l_{2'}-l_1} B_r} }{\nr^2}
    \\
    &\times\frac{ \qlogn{ \zeta^{l_{2'} - l_{1'} }B_f, A_f}{m_1' - m_2'} \qlogn{ \zeta^{l_2 - l_1} B_b, \zeta^{k_1 - k_{2'}} A_b}{m_2 - m_1} }{ \qlogn{\zeta^{l_2 - l_{1'} }B_l, \zeta^{k_{1'}} A_l}{m_2 - m_2'} \qlogn{\zeta^{l_{2'}-k_1}B_r, \zeta^{-k_{2'}} A_r}{m_1' - m_1} }.
  \end{align*}
  Set $\eta = \xi^{-(\nr -1)^2} = \zeta^{-(\nr - 1)^2/2}$.
  By applying \cref{eq:qlog-shift-B,eq:qlog-shift-A,eq:qlog-sum-shift-B,eq:qlog-sum-shift} we can simplify $\left(\mathring{S}\right)_{m_1 m_2}^{m_1' m_2'}$ to
  \begin{equation}
    \label{eq:rooting-proof-eq-i}
    \begin{aligned}
      &(\zeta \eta)^{l_{2'} + l_2 - l_{1'} - l_1} \eta^{k_{1'} - k_{2'}} &&\text{ (terms from $\qlogsumname$)}
      \\
      &\times
      \eta^{k_1 - k_{2'} - k_{1'} + k_{2'}}
      \zeta^{-(m_2 - m_1)(k_1 - k_{2'}) + (m_2 - m_2')(k_{1'}) - (m_{1}' - m_1)(k_{2'}) } &&\text{ (terms from the $A_i$ ) }
      \\
      &\times
      \left( S \right)_{m_1 + l_1, m_2 + l_2}^{m_1' + l_{1'}, m_2' + l_{2'}} &&\text{ (index shifts from the $B_i$) }
    \end{aligned}
  \end{equation}
  Now we need to compose with $f$ and $g^{-1}$, which act by 
  \begin{equation}
    \label{eq:rooting-proof-eq-ii}
    \begin{aligned}
      f(v_{m_1 m_2}) &= \zeta^{-(m_1 - l_1) k_1 - (m_2 - l_2) k_2} v_{m_1 - l_1, m_2 - l_2}.
      \\
      g^{-1}(v_{m_1' m_2'}) &= \zeta^{(m_1' + l_{2'}) k_{2'} + (m_2' + l_{1'}) k_{1'}} v_{m_1' + l_{2'}, m_2' + l_{1'}}.
    \end{aligned}
  \end{equation}

By combining \cref{eq:rooting-proof-eq-i,eq:rooting-proof-eq-ii} we see that
\begin{align*}
  \left(g^{-1} \mathring{S} f\right)_{m_1 m_2}^{m_1' m_2'}
  &=
  (\zeta \eta)^{l_{2'} + l_2 - l_{1'} - l_1} \eta^{k_{1} - k_{2'}}
  \\
  &\phantom{=}\times
  \zeta^{-(m_2 - m_1)(k_1 - k_{2'}) + (m_2 - m_2')(k_{1'}) - (m_{1}' - m_1)(k_{2'}) } 
  \\
  &\phantom{=}\times
  \zeta^{ -m_1 k_1 - m_2 k_2 + m_1' k_{2'} + m_2' k_{1'} } 
  \\
  &\phantom{=}\times
  S_{m_1 m_2}^{m_1' m_2'}
  \\
  &=(\zeta \eta)^{l_{2'} + l_2 - l_{1'} - l_1} \eta^{k_{1} - k_{2'}}
  \zeta^{m_2(k_{1'} + k_{2'} -k_1 - k_2)}
  S_{m_1 m_2}^{m_1' m_2'}
\end{align*}
Recall the compatibility condition
\[
  \alpha_1 \alpha_2 = \alpha_{1'} \alpha_{2'}
\]
for the choices of $\alpha_i$.
Imposing the same condition on the $\mathring{\alpha}_i$ means that
\[
  \zeta^{k_1 + k_2} \alpha_1 \alpha_2 = \zeta^{k_{1'} + k_{2'}}\alpha_{1'} \alpha_{2'}
  \text{, so } \zeta^{k_{1'} + k_{2'} - k_1 - k_2} = 1.
\]
Applying this, we see that $g^{-1} \mathring S f$ and $S$ are equal up to a scalar:
\begin{align*}
  \left(g^{-1} \mathring{S} f\right)_{m_1 m_2}^{m_1' m_2'}
  &=\zeta^{l_{2'} + l_2 - l_{1'} - l_1} \eta^{l_{2'} + l_2 - l_{1'} - l_1 + k_{1} - k_{2'}}
  S_{m_1 m_2}^{m_1' m_2'}.
\end{align*}
$\eta$ and $\zeta$ are powers of $\xi$ so this scalar is a power of $\xi$ as claimed.
\end{proof}

We have now completed the proof of the second claim in \cref{thm:braiding-factorization}.
Our proof suggests that it should be possible to eliminate the ambiguous phase in $S$ by fixing some extra structure on the shaped tangle.
Since the parameters $a_i$ and $b_i$ we are taking roots of are related to the complex dihedral angles of the octahedral decomposition, we expect this structure to be related to \defemph{flattenings} \cite{Zickert2009} of the octahedral decomposition.

\section{Wrapping up the construction}
\label{sec:twists-sideways}

Now that we have computed the braiding, it remains to show how to define it at pinched crossing and to prove that it satisfies the \reidthree{} relation.
We check both of these via a continuity argument by considering the space of radicals as a covering of the set $\shapeext$ of extended shapes.
Afterwards we conclude the section with some miscellaneous computations: checking sideways invertibility, computing the inverse of \eqref{eq:S-matrix-coeffs}, and computing the twist.

\subsection{The braiding for pinched crossings}
\label{sec:pinched-crossings-braiding}
\begin{defn}
  Write $\shaperad$ for the set of triples $(\alpha, \beta, \mu)$ of complex numbers.
  We think of these as radicals of extended shapes via the $\nr^2$-fold covering map
  \[
    \pi_{\shape} : \shaperad \to \shapeext , \quad (\alpha, \beta, \mu) \mapsto (\alpha^{\nr}, \beta^{\nr}, \mu).
  \]
\end{defn}
Strictly speaking, we defined the braiding $S_{-,-}$ as a representation of $\shaperad$, not of $\shapeext$.
However, $\shaperad$ is not a biquandle because given $\beta_1$ and $\beta_2$ there is in general no canonical way to choose $\beta_{1'}$ and $\beta_{2'}$, and similarly for the $\alpha_i$.
To define a representation of $\shapeext$, we pick a section of the cover $\pi : \shaperad \to \shapeext$ arbitrarily, and in \cref{thm:rooting-doesnt-matter} we showed that the resulting braiding is defined up to a power of $\xi$.

Now suppose $\chi_1, \chi_2, \chi_{1'}, \chi_{2'}$ are extended shapes at a positive crossing of a diagram, as in \cref{fig:biquandle-crossing-labels}.
We say (\cref{sec:pinched-crossings}) the crossing is pinched when the dihedral angles of \cref{table:positive-crossing-data} lie in $0,1$.
In terms of the parameters $A_i$ and $B_i$ of the quantum dilogarithms defining the braiding, pinched crossings are exactly the case where $A_i^\nr \to 0$ and $B_i^\nr \to 1$.

Understanding the behavior of the formula \eqref{eq:S-matrix-coeffs} in this limit would give a concrete expression for the braiding at pinched crossings, but this is difficult.%
\note{
  However, it should be possible to take this limit.
  It appears to be how \citeauthor{Kashaev1997} derived the ``mysterious formula'' of \cite[Section 4]{Murakami2001}.
}
For now, we use an abstract argument to show that the limit exists.

\begin{defn}
  We write $\pi : \mathcal{V} \to \shaperad $ for the bundle of vector spaces over $\shaperad$ which assigns the point $(\alpha, \beta, \mu)$ the vector space $\irrmod{\alpha, \beta, \mu}$.
\end{defn}

The braidings $S$ can be thought of as bundle isomorphisms in the following sense:
\begin{prop}
  Let $\widetilde \chi_1, \widetilde \chi_2, \widetilde \chi_{1'}, \widetilde \chi_{2'} \in \shaperad$ be radicals and write $\chi_i = \pi(\widetilde \chi_i)$ for the corresponding extended shapes.
  Then whenever $B(\chi_1, \chi_2) = (\chi_{2'}, \chi_{1'})$ and the corresponding crossing is not pinched, there are open neighborhoods $N_i$ of the $\widetilde \chi_i$ such that the braidings $S_{\chi_1, \chi_2}$ induce continuous isomorphisms
  \[
    \mathcal{V}|_{N_1} \otimes \mathcal{V}|_{N_2} \to \mathcal{V}|_{N_{2'}} \otimes \mathcal{V}|_{N_{1'}}.
  \]
\end{prop}
\begin{proof}
  Away from pinched crossings the matrix coefficients \eqref{eq:S-matrix-coeffs} are continuous functions of the radicals.
\end{proof}

\begin{thm}
  \label{thm:pinched-braidings}
  At each \emph{pinched} crossing with a choice of radicals there is a unique braiding induced by the braidings $S_{-,-}$ at non-pinched crossings.
\end{thm}
\begin{proof}
  Write $A_2$ for the subspace of $\shaperad \times \shaperad$ where the braiding is defined (that is, throw out inadmissible crossings) and consider the bundle $\mathcal{V}_2 \to A_2$ induced by $\mathcal{V} \otimes \mathcal{V} \to \shaperad \times \shaperad$.
  The $R$-matrices, hence the braidings, are solutions to the equations \eqref{eq:weyl-module-braiding-diagram}.
  These equations are \emph{continuous} on $\mathcal{V}_2$, with a one-dimensional space of solutions, so the set of solutions to the braiding equations defines a fiber bundle $\mathcal{V}'_2 \to A_2$ with fibers $\CC^\times$.
  Our braiding matrices $S_{-,-}$ are a continuous section of the bundle $\mathcal{V}'_2$, but are only defined away from the points of $A_2$ corresponding to pinched crossings.
  Since pinched crossings are isolated points of $A_2$, the section $S_{-,-}$ extends to all of $A_2$.%
\end{proof}
Note that the uniqueness here depends on the choice of radicals. 

As mentioned in the introduction to the chapter, at good pinched crossings another approach is possible.
At a good pinched crossing with modules $V_1 \otimes V_2 \to V_{1'} \otimes V_{2'}$, the operator $E \otimes F$ acts nilpotently, so the action of the universal $R$-matrix
\[
  \mathbf{R} = 
  \operatorname{HH} \;  \operatorname{exp}_q(E \otimes F)
  =
  q^{H \otimes H/2} \sum_{n = 0}^{\infty} \frac{q^{n(n-1)/2}}{\{n\}!} (E \otimes F)^{n}
\]
converges.
The linear operator $R : V_1 \otimes V_2 \to V_{1'} \otimes V_{2'}$ induced by the action of $\mathbf R$ is a solution to the equations \eqref{eq:inner-R-matrix-cond}, so it gives a braiding.
As mentioned in \cref{thm:abelian-limit} this braiding corresponds to known invariants of links.

\begin{proof}[Proof of \cref{thm:abelian-limit}]
  In the proof of this theorem we assume that our braidings satisfy the colored \reidthree{} relation, as proved in the next section.
  (Actually, in the present case, a much easier proof is possible by using the universal $R$-matrix.)

  Consider a shaped link diagram in which every segment is assigned the extended shape $\chi = (\mu^\nr, \beta^\nr, \mu)$ for some $\beta$.
  This diagram has a canonical choice of radical $(\mu, \beta, \mu)$ for every segment.
  When $\mu = \xi^{\nr -1}$, the module $\irrmod{\xi^{\nr -1}, \beta, \xi^{\nr -1}}$ is exactly the module defining the $\nr$th colored Jones polynomial \cite{Murakami2001} at a $\nr$th root of unity.
  The choice of module determines the braiding is up to a scalar, so the braiding defining the colored Jones polynomial and the braiding defining $\vecinv_{\nr}$ can differ only by a scalar $\theta \in \CC^\times$.
  Because this scalar is the same for every crossing (it is inverted for negative crossings) it only affects the framing dependence of the invariant.
  We show in \cref{sec:twist-computation} that $\theta$ is a power of $\xi$.

  When $\mu^\nr$ is not a root of unity, essentially the same argument works with a small complication.
  The central character corresponding to $(\mu, \beta, \mu)$ is
  \[
    \chi(K^\nr) = \mu^\nr, \quad \chi(E^\nr) = 0, \quad \chi(F^\nr) = b^{-1}(1 - \mu^{-2\nr})
  \]
  so $F^\nr$ does not act by $0$, as it does for the ADO invariant.
  Instead we obtain the \emph{semi-cyclic} invariant of \citeauthor{Geer2013} \cite{Geer2013} corresponding to the holonomy representation $\rho'$ sending every meridian $x$ to the matrix
  \[
    \rho'(x) =
    \begin{bmatrix}
      \mu^{\nr} & 0 \\
      b^{-1}(\mu^{\nr} - \mu^{-\nr}) & \mu^{-\nr}
    \end{bmatrix}.
  \]
  (\citeauthor{Geer2013} use upper-triangular instead of lower-triangular matrices, corresponding to $F^\nr = 0$ instead of $E^\nr = 0$.)
  The semi-cyclic invariant is gauge-invariant%
  \note{
    This is not discussed in \cite{Geer2013}, but we can use the arguments of \cref{ch:functors} to prove it.
  }
  and as discussed in the proof of \cref{thm:links-presentable} the representation $\rho'$ is conjugate to the diagonal representation
  \[
    \rho(x) =
    \begin{bmatrix}
      \mu^{\nr} & 0 \\
      0 & \mu^{-\nr}
    \end{bmatrix}.
  \]
  In this case the holonomy invariant is known \cite{Costantino2014invariants} to recover the ADO invariant.
\end{proof}

\subsection{Obtaining a model}
\label{sec:proof-of-model}
We can now begin the proof of \cref{thm:standard-model}, which says that $S$ gives a holonomy braiding (that is, a model of the extended shape biquandle) up to an $2\nr$th root of unity.
It is immediate that $S$ gives a holonomy braiding up to an overall scalar.
\begin{prop}
  \label{thm:projective-model}
  The matrices $S = S_{\chi_1, \chi_2}$ give a model of the extended shape biquandle in $\qgrp/\CC^\times$.
\end{prop}
\begin{proof}
  We need to check the colored \reidthree{} relation, which for ordinary braids reads
  \[
    \sigma_1 \sigma_2 \sigma_1 = \sigma_2 \sigma_1 \sigma_2.
  \]
  Suppose we start with shapes $\chi_1, \chi_2, \chi_3$.
  The left-hand side is a linear map
  \[
    \vecfunc(\sigma_1 \sigma_2 \sigma_1)
    =
    (S_{-,-} \otimes \id_-) (\id_- \otimes S_{-,-}) (S_{\chi_1, \chi_2} \otimes \id_{\irrmod{\chi_3}})
  \]
  where we use $-$ to represent a variable completed by the biquandle.%
  \note{
    The notation here is a bit unfortunate: we are composing braids left-to-right but linear maps right-to-left.
    Despite our earlier militancy about the order of composition for braids, reversing matrix composition seems like a step too far.
  }
  This is a linear map intertwining the homomorphism
  \[
    (\Smat \otimes \id) (\id \otimes \Smat) (\Smat \otimes \id)
  \]
  of (appropriate quotients of) $\qgrp \otimes \qgrp \otimes \qgrp$.
  Similarly, the right-hand side is a linear map 
  \[
    \vecfunc(\sigma_2 \sigma_1 \sigma_2)
    (\id_- \otimes S_{-,-}) (S_{-,-} \otimes \id_-) (\id_{\irrmod{\chi_1}} \otimes S_{\chi_2, \chi_3})
  \]
  intertwining the homomorphism
  \[
     (\id \otimes \Smat) (\Smat \otimes \id) (\id \otimes \Smat).
  \]
  However, $\Smat$ satisfies the \reidthree{} relation:
  \[
    (\Smat \otimes \id) (\id \otimes \Smat) (\Smat \otimes \id)
    =
     (\id \otimes \Smat) (\Smat \otimes \id) (\id \otimes \Smat).
  \]

  Because an intertwiner of either map is determined up to an overall scalar, we see that
  \[
    \vecfunc(\sigma_1 \sigma_2 \sigma_1)
    = \theta
    \vecfunc(\sigma_2 \sigma_1 \sigma_2)
  \]
  for some scalar $\theta \in \CC^\times$.
\end{proof}
Because $\theta$ could be any element of $\CC^\times$, we have only proved that this gives a model in $\modc \qgrp/\CC^\times$.
This would produce useless link invariants: they would be elements of $\CC$ defined up to multiplication by an element of $\CC^\times$.
To make them useful we need to reduce the scalar indeterminacy, which we do by computing the determinants of the matrices $S_{\chi_1, \chi_2}$.

The idea is to pick some extra structure on the modules $\irrmod{\chi}$ that is preserved by both sides of the \reidthree{} relation.
In the holonomy case, this is tricky, because $\vecfunc(\sigma_1 \sigma_2 \sigma_1)$ is no longer an endomorphism.%
\note{
  Later, in \cref{ch:doubles} we are able to strengthen this argument by finding a family of vectors preserved by the double $\doubfunc$ of $\vecfunc$.
  Constructing such a family for $\vecfunc$ seems much more difficult.
}
We follow the approach of \cite[Appendix D]{Blanchet2020}.

Write $A_3$ for the subset of $\shapeext \times \shapeext \times \shapeext$ such that the colored braids
\begin{align*}
  &\sigma_1 \sigma_2 \sigma_1: (\chi_1, \chi_2, \chi_3) \to (\chi_1', \chi_2', \chi_3')
  \\
  &\sigma_2 \sigma_1 \sigma_2 : (\chi_1, \chi_2, \chi_3) \to (\chi_1', \chi_2', \chi_3') 
\end{align*}
are defined, and write $\widetilde A_3$ for the $\nr^6$-fold covering of $A_3$ induced by the covering $\shaperad \to \shapeext$ of extended shapes by radicals.
As discussed in the previous section, the braidings are really defined on $\widetilde A_3$, and picking an arbitrary lift from $A_3$ introduces the phase ambiguity of a power of $\xi$.

To show the colored \reidthree{} relation holds up to a power of $\xi$, it suffices to show that the endomorphism
\[
  \vecfunc(\sigma_1 \sigma_2 \sigma_1 \sigma_2^{-1} \sigma_1^{-1} \sigma_2^{-1})
\in \End_{\qgrp}( \irrmod{\chi_1} \otimes \irrmod{\chi_2} \otimes \irrmod{\chi_3})
\]
is equal to the identity map for every $(\widetilde \chi_1, \widetilde \chi_2, \widetilde \chi_3) \in B_3$, where $\widetilde \chi_i$ is a radical for $\chi_i$.
By \cref{thm:projective-model}, there is a scalar function $\theta : \widetilde A_3 \to \CC^\times$  with 
\[
  \vecfunc(\sigma_1 \sigma_2 \sigma_1 \sigma_2^{-1} \sigma_1^{-1} \sigma_2^{-1})
  = \theta(\chi_1, \chi_2, \chi_3) \id_{\irrmod{\chi_1} \otimes \irrmod{\chi_2} \otimes \irrmod{\chi_3}}.
\]
\begin{lem}
  \label{thm:det-is-1}
  The determinant of the normalized holonomy braiding $S = S_{\chi_1, \chi_2}$ given in \cref{eq:S-matrix-coeffs} is $1$.
\end{lem}
We delay the proof of the lemma until the end of the section.
It implies that the map $\vecfunc(\sigma_1 \sigma_2 \sigma_1 \sigma_2^{-1} \sigma_1^{-1} \sigma_2^{-1})$ has determinant $1$.
Since the identity map also has determinant $1$ and both are endomorphisms of a $\nr^3$-dimensional vector space we conclude that the scalar function $\theta$ satisfies $\theta^{\nr^3} = 1$.
In particular, $\theta$ takes \emph{discrete} values in the group $\units{\nr^3}$ of $\nr^3$th roots of unity.

The set $A_3$ is a $9$-dimensional complex variety in $\CC^9$ obtained by removing complex hypersurfaces, so it is path-connected, so the finite cover $\widetilde A_3$ is path-connected as well.
Because the braidings are continuous, the function $\theta$ is continuous, hence locally constant on $\widetilde A_3$.
Because it takes values in the \emph{discrete} group $\units{\nr^3}$ and $\widetilde A_3$ is connected, $\theta$ must be constant, and we can take the constant to be $1$.
To see this, consider the radical $\widetilde \chi_0 = (\xi^{\nr -1}, 1, \xi^{\nr -1})$ corresponding to Kashaev invariant.%
\note{
  Technically speaking this is wrong when $\nr$ is even: as in the proof of \cref{thm:links-presentable} we need to use a slightly more complicated coloring where the signs of the $b$-parameters alternate.
  However, this does not affect the $\qgrp$-isomorphism class of the modules, so the argument here still works.
}
Because $B(\chi_0, \chi_0) = (\chi_0, \chi_0)$ the six braidings appearing in the \reidthree{} relation can be given by the \emph{same} matrix, so in this case $\theta = 1$.

We conclude that the braiding has no scalar ambiguity when defined on the set of radicals.
However, $\vecfunc$ is a representation of $\shapeext$, and we defined the braiding on $\shapeext$ by arbitrarily picking lifts to $P$.
By \cref{thm:rooting-doesnt-matter}, a different choice of lift would change the braiding by a power of $\xi$, a $2\nr$th root of unity.
We conclude that $\vecfunc$ satisfies the \reidthree{} relation up to a power of $\xi$, i.e.\@ up to an element of $\units{2\nr}$.
 
\begin{proof}[Proof of \cref{thm:det-is-1}]
  We use the factorization of \cref{thm:braiding-factorization}.
  The matrices $S_f$ and $S_b$ are diagonal, so we can apply \cref{eq:qlog-product}:
  \begin{equation*}
    \det S_f = \prod_{m_1 m_2} \Lambda_f(m_1' - m_2') = \prod_{m} \Lambda_f(m)^\nr = 1
  \end{equation*}
  and similarly $\det S_b = 1$.

  To compute the determinants of $S_r$ and $S_l$ we want to diagonalize them via Fourier transform.
  We have
  \begin{align*}
    S_r \cdot \hat v_m
    &=
    \Theta_r \sum_{k} \xi^{2mk} S_r \cdot v_k
    \\
    &=
    \Theta_r \sum_{k k'} \xi^{2mk} \Lambda_r(k' - k)^{-1} v_{k'}
    \\
    &=
    \frac{\Theta_r}{\nr} \sum_{k k' m'} \xi^{2mk} \Lambda_r(k' - k)^{-1} {\xi^{-2m' k'}} \hat v_{m'}
    \\
    &=
    \frac{\Theta_r}{\nr} \sum_{k d m'} \xi^{2mk - 2m'(k + d)} \Lambda_r(d)^{-1} \hat v_{m'}
    \\
    &=
    \Theta_r \sum_{d m'} \delta_{m m'} \xi^{-2m' d} \Lambda_r(d)^{-1} \hat v_{m'}
    \\
    &=
    \Theta_r \left( \sum_{d} \xi^{-2md} \Lambda_r(d)^{-1} \right) \hat v_m
    \\
    &=
    \qlogn{\xi^{-2} A_r, B_r}{-m} \hat v_m
  \end{align*}
  where in the last line we applied \cref{eq:qlog-recip-fourier}.
  By another application of \cref{eq:qlog-product}.
  \begin{align*}
    \det S_r
    =
    \prod_{m} 
    \qlogn{\xi^{-2} A_r, B_r}{-m}
    =
    1.
  \end{align*}
  The same argument gives $\det S_l = 1$, and we conclude that $\det S = 1$.
\end{proof}

This concludes the proof of \cref{thm:standard-model}.

\subsection{The negative braiding}
Before checking that the sideways braidings \cref{fig:sideways-braidings} induced by the model are inverses we need to compute the matrix coefficients of the usual inverse braiding.
These will also be useful in \cref{ch:doubles} when we define the mirror of the braiding.
\begin{thm}
  \label{thm:negative-braiding-matrix}
  Let $\chi_i$ be extended shapes at a negative crossing, as in \cref{fig:shaped-negative-crossing}.
  The matrix coefficients of the braiding assigned to this crossing are given by
  \begin{equation}
    \label{eq:negative-braiding-matrix}
    \widetilde S_{m_1 m_2}^{m_1' m_2'} = 
    \frac{1}{\widetilde \Theta_l \widetilde \Theta_r}
    \frac{\widetilde \Lambda_l(m_2' - m_2) \widetilde \Lambda_r (m_1 - m_1')}{ \widetilde \Lambda_f(m_2' - m_1') \widetilde \Lambda_b(m_1 - m_2)}
  \end{equation}
  where
  \begin{align*}
    \widetilde \Lambda_f(m)
    &=
    \qlogn{\widetilde B_f, \widetilde A_f}{m}
    =
    \qlogn{ \frac{\beta_{1'}}{ \xi^2 \beta_{2'} \mu_2} , \frac{\alpha_{2'}}{\mu_1 \mu_2 \alpha_1} \widetilde A_b}{m}
    \\
    \widetilde \Lambda_b(m)
    &=
    \qlogn{\widetilde B_b, \widetilde A_b}{m}
    =
    \qlogn{ \frac{ \mu_1 \beta_1 }{ \xi^2 \beta_2 }, \widetilde A_b}{m}
    \\
    \widetilde \Lambda_l(m)
    &=
    \qlogn{\widetilde B_l, \widetilde A_l}{m}
    =
    \qlogn{ \frac{ \mu_1 \beta_{1'} }{ \xi^2 \mu_2 \beta_2 }, \frac{\alpha_2}{ \xi^2 \mu_2 } \widetilde A_b }{m}
    \\
    \widetilde \Lambda_r(m)
    &=
    \qlogn{\widetilde B_r, \widetilde A_r}{m}
    =
    \qlogn{ \frac{ \beta_1}{ \xi^2  \beta_{2'} } , \frac{1}{ \xi^2 \mu_1 \alpha_1 } \widetilde A_b}{m}
  \end{align*}
  and
  \[
    \widetilde A_b^\nr
    = 1 - \left( \frac{\mu_1 \beta_1}{\xi^2 \beta_{2}}\right)^\nr
    = 1 - \lambda_1 \frac{b_1}{b_{2}}.
  \]
  Here the normalization factors are
   \[
     \widetilde \Theta_i = {\qlogsum{ \widetilde{B}_i, \widetilde{A}_i}}
  \]
  where $\qlogsumname$ is given in \cref{thm:qlog-fourier}.
  The negative braiding factors into four terms as
  \[
    \widetilde S = \widetilde{S}_f(\widetilde S_r \otimes \widetilde S_l) \widetilde S_b.
  \]
\end{thm}
\begin{proof}
  The computation goes as for $S$, but replacing $\Smat$ with $\Smat^{-1}$, although translating the  matrix computed this way to the form given requires some applications of the identities of \cref{ch:quantum-dilogarithm}.

  We instead directly check that $\widetilde S$ is the inverse of $S$.
  Because the matrices have parameters, we need to be somewhat careful when interpreting this claim: since the braiding is a map
  \[
    S :
    \irrmod{\chi_1} \otimes \irrmod{\chi_2}
    \to
    \irrmod{\chi_{2'}} \otimes \irrmod{\chi_{1'}}
  \]
  its inverse is the negative braiding with shapes
  \[
    \widetilde S :
    \irrmod{\chi_{2'}} \otimes \irrmod{\chi_{1'}}
    \to
    \irrmod{\chi_1} \otimes \irrmod{\chi_2},
  \]
  so when we compute the matrix product
  $\widetilde S \cdot S$ we want to swap $1$ and $2$ and primed and unprimed indices in the parameters of $\widetilde S$.

  We need to compute the product
  \[
    \widetilde{S}_f(\widetilde S_r \otimes \widetilde S_l) \widetilde S_b
    S_f( S_r \otimes S_l) S_b.
  \]
  Notice that
  \[
    A_f^\nr = 1 - \lambda_2 \frac{b_{2'}}{b_{1'}} = \widetilde{A}_b^{\nr}
  \]
  so we can choose $A_f = \widetilde A_b$.
  Then
  \[
    \widetilde S_b (S_f(v_{m_1 m_2}))
    =
    \frac
    { \qlogn{ \frac{\mu_2 \beta_{2'}}{\xi^2 \beta_{1'}}, A_f }{m_1-m_2} }
    { \qlogn{ \frac{\mu_2 \beta_{2'}}{\xi^2 \beta_{1'}}, {A}_f }{m_1-m_2} }
    v_{m_1 m_2}
    =
    v_{m_1 m_2}.
  \]
  Now we turn to the products $\widetilde S_r S_r$ and $\widetilde S_l S_l$.
  We want to work in the Fourier dual basis $\hat v_m$, where (as discussed in the proof of \cref{thm:det-is-1})
  \[
    S_r \cdot \hat v_m
    = \qlogn{\xi^{-2} A_r, B_r}{-m} \hat v_m
    = \qlogn{\frac{A_f}{\xi^{2}\mu_2 \alpha_{2'}}, \frac{\beta_2'}{\beta_1}}{-m} \hat v_m.
  \]
  By a similar computation with \cref{eq:qlog-fourier} we get
  \[
    \widetilde S_r \cdot \hat v_m
    =
    \frac{1}{\qlogsum{\widetilde B_r, \widetilde A_r}}
    \frac{\qlogsum{\widetilde B_r, \widetilde A_r}}{ \qlogn{\widetilde A_r, \xi^2 \widetilde B_r}{-m}}
    =
    {\qlogn{ \frac{\widetilde{A}_b}{\xi^2 \mu_2 \alpha_{2'}} , \frac{\beta_{2'}}{\beta_1} }{-m}}^{-1} \hat v_m.
  \]
  Because $\widetilde A_b = A_f$ we see that $\widetilde S_r(S_r(\hat v_m)) = v_m$.
  A similar argument works for $\widetilde S_l S_l$, so
  \[
    (\widetilde S_r \otimes \widetilde S_l) (S_r \otimes S_l) = 1.
  \]
  A computation analogous to the one given for $\widetilde S_b S_f$ shows that $\widetilde S_f S_b = 1$.
\end{proof}

\subsection{Sideways invertibility}

\begin{lem}
  Let $V = \irrmod{\alpha, \beta, \mu}$ be a standard module.
  Then the evaluation and coevaluation maps act by
  \begin{align*}
    \coevup V &: 1 \mapsto \sum_{m} v_m \otimes v^m
    \\
    \coevdown V &: 1 \mapsto \alpha^{\nr -1} \sum_{m} v^m \otimes v_{m+1}
    \\
    \evup V &: v^m \otimes v_n \mapsto \delta_{nm}
    \\
    \evdown V &: v_m \otimes v^n \mapsto \alpha^{1-\nr} \delta_{n, m-1}
  \end{align*}
\end{lem}

\begin{lem}
  The image of the sideways braiding $s_L^+(\chi_{2'}, \chi_1)$ under the quantum dilogarithm model of the extended shape biquandle has matrix coefficients
  \[
  v^{m_1} \otimes v_{m_2} \mapsto 
    \sum_{m_1' m_2'} S_{m_2 m_2'}^{m_1 m_1'} v_{m_1'} \otimes v^{m_2'}
  \]
  so it factors as the product $C_l(C_f \otimes C_b)C_r$, where
 \begin{align*}
   C_r(\tensor{v}{^{m_1}_{m_2}})
   &=
   \frac{\Theta_r}{\Lambda_r(m_1 - m_2)} \tensor{v}{^{m_1}_{m_2}}
   &
   C_f(v^{m_1})
   &=
   \sum_{m_1'} \Lambda_f(m_1 - m_1')v_{m_1'}
   \\
   C_l(\tensor{v}{_{m_1}^{m_2}})
   &=
   \frac{\Theta_l}{\Lambda_l(m_2 - m_1)} \tensor{v}{_{m_1}^{m_2}}
   &
   C_b(v_{m_2})
   &=
   \sum_{m_2'} \Lambda_b(m_2' - m_1)v^{m_2'}
 \end{align*}
\end{lem}
\begin{proof}
  By putting the component morphisms of $s_L^+(\chi_{2'}, \chi_1)$ together, we get
  \begin{align*}
  v^{m_1} \otimes v_{m_2}
    &\mapsto \sum_{m_3} v^{m_1} v_{m_2 m_3} \otimes v^{m_3}
    \\
    &\mapsto \sum_{m_3, m_2', m_3'} S_{m_2 m_3}^{m_2' m_3'} v^{m_1} \otimes v_{m_2' m_3'} \otimes v^{m_3}
    \\
    &\mapsto \sum_{m_3, m_2', m_3'} S_{m_2 m_3}^{m_2' m_3'} \delta_{m_1 m_2'}v_{m_3'} \otimes v^{m_3}
    \\
    &= \sum_{m_3, m_3'} S_{m_2 m_3}^{m_1 m_3'} v_{m_3'} \otimes v^{m_3}
  \end{align*}
  after which we can relabel some indices.
  The factorization is analogous to the factorization of $S$.
\end{proof}

\begin{lem}
  The image of the sideways braiding $s_R^-(\chi_{2'}, \chi_1)$ under the quantum dilogarithm model of the extended shape biquandle has matrix coefficients
  \[
    \tensor{v}{_{m_1}^{m_2}} \mapsto
    \left(\frac{\alpha_{2'}}{\alpha_{2}}\right)^{\nr -1} \sum_{m_1', m_2'} \widetilde{S}_{m_1'+1, m_1}^{m_2', m_2+1} v^{m_1'} \otimes v_{m_2'}
  \]
  so it factors as the product $(\alpha_{2'}/\alpha_{2})^{\nr-1} \widetilde C_r(\widetilde C_b \otimes \widetilde C_f)\widetilde C_l$, where
 \begin{align*}
   \widetilde C_r(\tensor{v}{^{m_1}_{m_2}})
   &= 
   \frac{\widetilde \Lambda_r(m_1 - m_2 + 1)}{\widetilde \Theta_r} \tensor{v}{^{m_1}_{m_2}}
   &
   \widetilde C_f(v^{m_2})
   &=
   \sum_{m_2'} \frac{1}{\widetilde \Lambda_f(m_2 - m_2' + 1)}v_{m_2'}
   \\
   \widetilde C_l(\tensor{v}{_{m_1}^{m_2}})
   &=
   \frac{\widetilde \Lambda_l(m_2 - m_1 + 1)}{\widetilde \Theta_l} \tensor{v}{_{m_1}^{m_2}}
   &
   \widetilde C_b(v_{m_1})
   &=
   \sum_{m_2'} \frac{1}{\widetilde \Lambda_b(m_1' - m_1 + 1)}v^{m_1'}
 \end{align*}
\end{lem}
\begin{proof}
  Again we can directly compute
  \begin{align*}
    v_{m_2} \otimes v^{m_3}
    &\mapsto
    \alpha_{2'}^{\nr -1} \sum_{m_1} v^{m_1-1} \otimes v_{m_1 m_2} \otimes v^{m_3}
    \allowdisplaybreaks
    \\
    &\mapsto
    \alpha_{2'}^{\nr -1} \sum_{m_1, m_1', m_2'} \widetilde{S}_{m_1 m_2}^{m_1' m_2'} v^{m_1-1} \otimes v_{m_1' m_2'} \otimes v^{m_3}
    \allowdisplaybreaks
    \\
    &\mapsto
    \alpha_{2'}^{\nr -1} \alpha_{2}^{1 - \nr} \sum_{m_1, m_1', m_2'} \widetilde{S}_{m_1 m_2}^{m_1' m_2'} \delta_{m_3, m_2' - 1} v^{m_1-1} \otimes v_{m_1'}
    \allowdisplaybreaks
    \\
    &=
    \left(\frac{\alpha_{2'}}{\alpha_{2}}\right)^{\nr -1} \sum_{m_1, m_1'} \widetilde{S}_{m_1+1, m_2}^{m_1', m_3+1} v^{m_1} \otimes v_{m_1'} \qedhere
  \end{align*}
\end{proof}

\begin{prop}
  \label{thm:standard-is-sideways-inv}
  The model $(V,S)$ of \cref{thm:standard-model} is sideways invertible.
\end{prop}
\begin{proof}
  Now that we have the matrix coefficients and the quantum dilogarithm identities of \cref{ch:quantum-dilogarithm} this is an explicit computation, which works just as in the proof of \cref{thm:negative-braiding-matrix}.
  We show enough of the computation to give an idea of how it works.
  We need to compute the product
  \[
    C_l(C_f \otimes C_b)C_r
    \widetilde C_r(\widetilde C_b \otimes \widetilde C_f)\widetilde C_l.
  \]
  Here
  \begin{align*}
    C_r(\tensor{v}{^{m_1}_{m_2}})
    &=
  \nr^{-1}{\qlogsum{\frac{1}{\xi^2 \mu_2 \alpha_{2'}} A_f, \frac{\beta_{2'}}{\beta_1}}} {\qlogn{ \frac{\beta_{2'} }{\beta_1}, \frac{1}{\mu_2 \alpha_{2'}}A_f}{m_1-m_2}}^{-1}
    \\
    \widetilde C_r(\tensor{v}{^{m_1}_{m_2}})
    &=
    \qlogsum{\frac{\beta_{2'}}{\xi^2\beta_1}, \frac{1}{\xi^2 \mu_2 \alpha_{2'}} A_f  }^{-1}
    \qlogn{ \frac{\beta_{2'}}{\xi^2\beta_1}, \frac{1}{\xi^2 \mu_2 \alpha_{2'}} A_f }{m_1 - m_2 + 1}
  \end{align*}
  where as before the second braiding acts after a transformation of the shapes, so it has indices and primes swapped, and we have again used $\widetilde{A}_b = A_f$.
  We can simplify
  \begin{align*}
    \qlogsum{\frac{\beta_{2'}}{\xi^2\beta_1}, \frac{1}{\xi^2 \mu_2 \alpha_{2'}} A_f  }
    &=
    \left( \frac{\mu_2 \alpha_{2'} \beta_{2'} }{\beta_1} \right)^{\nr-1}
    {\qlogsum{\frac{1}{\xi^2 \mu_2 \alpha_{2'}} A_f, \frac{\beta_{2'}}{\beta_1}}}
  \end{align*}
  and
  \begin{align*}
    &\qlogn{ \frac{\beta_{2'}}{\xi^2\beta_1}, \frac{1}{\xi^2 \mu_2 \alpha_{2'}} A_f }{m_1 - m_2 + 1}
    \\
    &=
    \qlogn{ \frac{\beta_{2'}}{\beta_1}, \frac{1}{\xi^2 \mu_2 \alpha_{2'}} A_f }{m_1 - m_2}
    \\
    &=
    \xi^{-2(m_1 - m_2)} \xi^{-(\nr-1)^2}
    \qlogn{ \frac{\beta_{2'}}{\beta_1}, \frac{1}{\xi^2 \mu_2 \alpha_{2'}} A_f }{m_1 - m_2}
  \end{align*}
  so that
  \[
    C_r \widetilde C_r(\tensor{v}{^{m_1}_{m_2}})
    =
    \xi^{-2(m_1 - m_2)} \xi^{-(\nr-1)^2} 
    \left( \frac{\mu_2 \alpha_{2'} \beta_{2'} }{\beta_1} \right)^{\nr-1}
    \tensor{v}{^{m_1}_{m_2}}.
  \]
  Unlike in the proof of \cref{thm:negative-braiding-matrix} they do not cancel.
  However, the action $v^{m_1} \mapsto \xi^{-2m_1} v^{m_1}$ works out just right to compensate for the index shift in $\widetilde C_b$, and similarly for $v_{m_2}$ and $\widetilde C_f$.
  When computing the products of the other factors we also get terms of the form $(A_i/B_i)^{\pm(\nr-1)}$ and $\xi^{\pm(\nr-1)^2}$, and together with the factor $(\alpha_{1}/\alpha_{1'})^{\nr -1}$ these all cancel out.
\end{proof}

\subsection{Determining the twist}
\label{sec:twist-computation}
By \cref{thm:twist-induced}, we know that our model induces a twist $\theta^R_\chi = \theta^L_\chi$.
Furthermore, via the gauge transformations of \cref{sec:internal-gauge-transf}, the scalar $\left\langle \theta_{\chi} \right\rangle$ depends only on the gauge class of the extended shape $\chi$, i.e.\@ the fractional eigenvalue $\mu$ of $\chi$.
We can now show that the twist is a power of $\xi$ for every shape, so that (up to a power of $\xi$) our invariant $\vecinv_{\nr}$ does not depend on the framing of the link $L$.
\begin{thm}
\label{thm:twist-computation}
  The model $\vecfunc$ induces a twist
  \[
    \theta_{\chi} = \xi^{\nr(\nr-1)/2}
  \]
  for \emph{every} admissible extended shape $\chi$.
  In particular, the framing dependence is a power of $\xi$.
\end{thm}
\begin{proof}
  Via gauge transformations, we may assume without loss of generality that $\chi = (\mu^\nr, b, \mu)$ for some $\mu$, so that $E^\nr$ acts by $0$ on $\irrmod \chi$.
  As in \cref{sec:mod-dim-weight-mods}, $\irrmod \chi = V_{\gamma}$ is a \defemph{highest-weight module}: it is generated by a vector $v_{\gamma}$ with
  \[
    E \cdot v_{\gamma} = 0 \text{ and } \quad K \cdot v_{\gamma} = \xi^{\gamma} v_{\gamma} \text{, where } \mu = \xi^{\gamma}.
  \]
  We can choose a weight basis%
  \note{
    In the notation of \cref{def:weyl-irrmod} $v_{\gamma -2k} = \hat v_{-k}$.
  }
  $\{v_{\gamma - 2k}, k = 0, \dots, \nr-1\}$ for $V_{\gamma}$ with
  \[
    K \cdot v_{\gamma - 2k} = \xi^{\gamma - 2k} v_{\gamma - 2k}.
  \]
  The dual basis of $V_\gamma^*$ will be denoted $v^{\gamma - 2k}$.

  In this case $\alpha(\chi) = \chi$, so we can compute the twist using the braiding for $V_{\gamma} \otimes V_{\gamma}$.
  As discussed in \cref{sec:pinched-crossings-braiding,sec:mod-dim-weight-mods} we can compute this braiding using the universal $R$-matrix
  \[
    \mathbf R = 
    \operatorname{HH} \;  \operatorname{exp}_q(E \otimes F)
    =
    q^{H \otimes H/2} \sum_{n = 0}^{\infty} \frac{q^{n(n-1)/2}}{\{n\}!} (E \otimes F)^{n}
  \]
  By \cref{thm:nilpotent-braiding}, the action of $\tau \mathbf R$ defines a braiding $V_{\gamma} \otimes V_{\gamma} \to V_{\gamma} \otimes V_{\gamma}$.
  However, to compute the twist correctly we need to normalize this braiding to match ours, that is to have determinant $1$.

  Because $E \otimes F$ acts nilpotently on $V_{\gamma} \otimes V_{\gamma}$, the $\exp_\xi(E \otimes F)$ term has determinant $1$.
  In more detail, since $(E \otimes F)^\nr$ acts by $0$, we can replace $\exp_q(E \otimes F)$ with a truncation
  \[
    \sum_{n =0}^{\nr} \frac{\xi^{n(n-1)/2}}{\{n\}!} (E \otimes F)^n = 1 + p(E \otimes F)
  \]
  where $p$ is some polynomial with no constant term.
  Since $E \otimes F$ is nilpotent, there is a basis of $V_{\gamma} \otimes V_{\gamma}$ where its matrix is strictly upper-triangular, so the matrix of  $p(E \otimes F)$ is as well, so the determinant of $\exp_q(E \otimes F)$ is  $1$ as claimed.

  It is straightforward to compute the determinant of $\xi^{H \otimes H/2}$ using the weight basis: because
  \[
    (H \otimes H) \cdot (v_{\lambda - 2k} \otimes v_{\lambda - 2l}) = (\lambda - 2k)(\lambda-2l)(v_{\lambda - 2k} \otimes v_{\lambda - 2l})
  \]
  and
  \[
    \sum^{\nr}_{k, l = 0} (\lambda - 2k)(\lambda-2l) = (\nr\lambda - \nr(\nr-1))^2 = \nr^2(\lambda - (\nr -1))^2
  \]
  the determinant of the action of $\xi^{H \otimes H/2}$ on $V_{\gamma} \otimes V_{\gamma}$ is
  \[
    \xi^{ \nr^2(\lambda - (\nr -1))^2/2}.
  \]

  Because the vector space $V_{\gamma}$ is $\nr$-dimensional, the flip $\tau$ has determinant
  \[
    (-1)^{\nr(\nr+1)/2} = (-1)^{\nr(\nr-1)/2} = \xi^{\nr^2(\nr-1)/2}.
  \]
  Set
  \[
    \Xi(\lambda) = \xi^{-(\gamma - (\nr -1))^2/2 - (\nr -1)/2 } = \xi^{-\gamma^2/2 + \gamma(\nr -1) - \nr(\nr+1)/2}.
  \]
  Then, writing $R_{\gamma, \gamma}$ for the action of $\mathbf R$ on $V_{\gamma} \otimes V_{\gamma}$, we conclude that
  \[
    S_{\gamma, \gamma} =  \Xi(\lambda) \tau R_{\gamma, \gamma}
  \]
  is a braiding with determinant $1$.

  We can now compute the twist using $S_{\gamma, \gamma}$.
  It suffices to understand the action of
  \[
    \theta_\chi^R =
    (\id_{V_{\gamma}} \otimes \coevup{V_{\gamma}})
    (S_{\gamma, \gamma} \otimes V_{\gamma}^* )
    (\id_{V_{\gamma}} \otimes \evdown{V_{\gamma}})
  \]
  on the highest-weight vector $v_{\gamma}$.
  (Here for clarity we write function composition left-to-right.)
  Then
  \begin{align*}
    \theta_\chi^R : v_{\gamma}
    &\mapsto
    \sum_{k = 0}^{\nr -1} v_{\gamma} \otimes v_{\gamma - 2k} \otimes v^{\gamma - 2k}
    \\
    &\mapsto
    \Xi(\gamma) \sum_{k = 0}^{\nr -1} \xi^{\gamma(\gamma - 2k)/2} v_{\gamma - 2k} \otimes v_{\gamma} \otimes v^{\gamma - 2k}
    \\
    &\mapsto
    \Xi(\gamma) \sum_{k = 0}^{\nr -1} \xi^{\gamma(\gamma - 2k)/2} v_{\gamma - 2k} \otimes v^{\gamma - 2k}( K^{1 - \nr} \cdot v_{\gamma})
    \\
    &=
    \xi^{\gamma^2/2} \xi^{(1 - \nr)\gamma} \Xi(\gamma) v_{\gamma}  .
  \end{align*}
  The twist simplifies as
  \[
    \xi^{\gamma^2/2} \xi^{(1 - \nr)\gamma} \Xi(\gamma)
    = \xi^{-\nr(\nr+1)/2}.
  \]
  In particular, this is a power of $\xi$.
\end{proof}

\chapter{The graded quantum double}
\label{ch:doubles}
\section*{Overview}
In the previous three chapters we constructed a (family of) functors
\[
  \vecfunc : \tangshe \to \modc{\qgrp}/\units{\nr^2}
\]
that compute the deformed Kashaev invariants.
They are interesting, but are in some ways quite difficult to handle.
We point out two issues in particular:
\begin{enumerate}
  \item The phase of $\vecfunc$ is only defined up to a $2\nr$th root of unity.
  \item The explicit matrices computed in \cref{ch:algebras} are rather complicated,%
    \note{For example, the braiding matrices are very nonsparse: every entry is nonzero!}
    which makes understanding the corresponding invariants quite difficult.
\end{enumerate}

In this chapter, we solve these problems by taking the \defemph{quantum double} of $\vecfunc$.
Specifically, we define a functor
\[
  \doubfunc : \tangshe \to \modc{\qgrp \otimes \qgrp^{\cop}}
\]
which we can think of as the external tensor product of $\vecfunc$ and its \defemph{mirror image} $\overline{\vecfunc}$:
\[
  \doubfunc = \vecfunc \boxtimes \overline{\vecfunc}.
\]
There are two equivalent ways to get $\overline{\vecfunc}$ from $\vecfunc$.
One is to define a functor $\mirrorfunc : \tangshe \to \tangshe$ that corresponds to taking the mirror image and set $\overline{\vecfunc} = \vecfunc \mirrorfunc$ after reversing the order of tensor factors.
Alternately, we can algebraically describe $\overline{\vecfunc}$ in using the opposite braiding between dual modules (compared to $\vecfunc$).
As we show later these are equivalent.

Because $\vecfunc$ comes from a representation of the extended shape biquandle, $\doubfunc$ does as well.
The modified trace on $\qgrp$ compatible with $\vecfunc$ also works for $\overline{\vecfunc}$, and we show in \cref{sec:double-traces} how to construct a trace on $\modc {\qgrp \otimes \qgrp^{\cop}}$ from them.
Via the construction of \cref{ch:functors}, the functor $\doubfunc$ yields an invariant $\doubinv$ of enhanced $\slg$-links.
The proof of the following theorem is essentially automatic from the construction of $\doubfunc$ and the corresponding modified trace.
\begin{thm}
  \label{thm:quantum-double-inv}
  The functor $\doubfunc$ defines an invariant $\doubinv(L)$ of extended $\slg$-links which satisfies
  \[
    \doubinv(L) = 
    \doubinv(\overline{L}) 
  \]
  and
  \[
    \doubinv(L) = \vecinv(L) \vecinv(\overline{L}).
  \]
\end{thm}

Just like passing from $z$ to $|z|^2$ loses information about a complex number $z$, the invariant $\doubinv$ is weaker than the invariant $\vecinv$.
Despite this, by using $\doubfunc$ we gain two advantages:
\begin{enumerate}
  \item We can prove directly that $\doubfunc$ satisfies the colored braid relation with no phase ambiguity.%
    \note{In fact, the structure of $\doubfunc$ is such that we can have this hold by definiton.}
  \item The braid action defined by $\doubfunc$ is very closely related to the action of the outer $S$-matrix $\Smat$ on $\qgrp^{\otimes 2}$.
\end{enumerate}
The second point is key to our proof of Schur-Weyl duality for $\nr = 2$ in \cref{ch:torions}, and thus to the proof that $\doubinv$ gives the torsion of the link complement when $\nr = 2$.

\subsection{Anomaly cancellation}
We can think of the indeterminacy of the phase of $\vecfunc$ as a \defemph{anomaly}.
This phrase comes from physics, so its exact mathematical meaning is somewhat unclear.
For our purposes, an anomaly is a failure of a function to be well-defined, which can be fixed by some choice of extra structure.
In this case, the anomaly is the phase indeterminacy of $\vecfunc$, which arises from making arbitrary choices of roots $\alpha_i, \beta_i$.
The (conjectural) extra structure required to fix this was discussed at the end of \cref{sec:normalized-braiding}.
When we take the norm-square $\doubfunc$ of $\vecfunc$ the anomalies of $\vecfunc$ and $\overline{\vecfunc}$ cancel, so we do not need to add any extra structure.

The relationship between $\doubfunc$ and $\vecfunc$ is a familiar one in quantum topology: it is analagous to the relationship between surgery (Reshetikhin-Turaev) and state-sum (Turaev-Viro) theories.
In particular, the cancellation of the anomaly is to be expected.

Chern-Simons theory with gauge group $\operatorname{SU}(2)$ \cite{Witten1989} can be rigorously realized as the \defemph{Reshetikhin-Turaev} or \defemph{surgery} topologial quantum field theory \cite{Reshetikhin1991,Bakalov2000}.
As a special case, the values of this theory on links in $S^3$ are colored Jones polynomials at certain roots of unity.
This theory has an anomaly: if one tries to define it as a functor from a cobordism category, it is not well-defined.
To fix this, extra data (a sort of framing) is required.

There is a different theory called the \defemph{Turaev-Viro} or \defemph{state-sum} TQFT \cite{Turaev1992,Balsam1} that can be constructed from similar algebraic data as the surgery theory.
The definition of this theory is (apparently) quite different, but its values are closely related: the value of the state-sum theory on $M$ is the value of the surgery theory on $M$ times its value on $\overline{M}$ ($M$ with the opposite orientation).
In addition, there is no anomaly in the TV theory; we can think of this as the anomalies of $M$ and $\overline{M}$ cancelling.
For details, see \cite{Turaev2016} and \cite{Balsam1,Balsam2,Balsam3}.

The anomaly of $\vecfunc$ is of a different type: among other things, it appears at the level of links, not just manifolds.
In addition, we do not yet understand exactly what extra structure is required on shaped tangles to remove the anomaly.
Nonetheless, it still seems reasonable to interpret $\doubfunc$ as a version of $\vecfunc$ where the anomalies cancel.

\subsection{Surgery and state-sum theories}
While we can identify $\doubfunc$ as a state-sum theory because it is the double of a surgery theory, we still compute it as a surgery theory.
Concretely, we mean that we define $\doubfunc$ in terms of braiding matrices, not state-sums over triangulations.
It would be interesting to give a definition of $\doubfunc$ in terms of state-sums.
We mention some constructions related to such a description.

Let $\mathcal{C}$ be a $G$-graded fusion category.
\citeauthor{Turaev2019}  \cite{Turaev2019} define notions of state-sum and surgery \defemph{homotopy quantum field theory} (a.k.a.~$G$-graded TQFT) and show that the state-sum theory from $\mathcal{C}$ is equivalent to the surgery theory from $\mathcal{Z}_G(\mathcal{C})$, where $\mathcal{Z}_G(\mathcal{C})$ is a graded version of the Drinfeld center of $\mathcal{C}$.

In our context, $\mathcal C$ is the category of $\qgrp$-weight modules, by which we mean finite-dimensional $\qgrp$-modules on which $\cent_0$ acts diagonalizably.
$\mathcal{Z}_G(\mathcal{C})$ would be something like $\mathcal{C} \boxtimes \overline{\mathcal{C}}$, where $\overline{\mathcal{C}}$ is the category of weight modules for $\qgrp^{\cop}$ and the braiding is given by an inverted, flipped version of $\vecfunc$.

To get a state-sum definition of $\doubfunc$ it would be useful to directly relate our construction of $\mathcal{C} \boxtimes \overline{\mathcal{C}}$ to the more abstract construction of the $G$-center $\mathcal{Z}_G(\mathcal{C})$.
Objects of $\mathcal{C} \boxtimes \overline{\mathcal{C}}$ are of the form $V \boxtimes V^*$ for $V$ an object of $\mathcal{C}$, while objects of $\mathcal{Z}_G(\mathcal{C})$ are pairs $(V, \sigma_V)$ with $\sigma_V$ a half-braiding relative to the identity-graded component $\mathcal{C}_1$ of $\mathcal{C}$.

One difficulty in understanding this relationship is that the category $\mathcal{C}$ is not semisimple, and the non-semisimplicity is concentrated in $\mathcal{C}_1$.
We expect that an appropriate semisimplification of $\mathcal{C}_1$ will allow an application of the theory of \citeauthor{Turaev2019} to the construction of $\mathcal{C} \boxtimes \overline{\mathcal{C}}$.
A less serious issue is that the braid action on the gradings of $\mathcal{C}$ is not simply conjugation, as it is in~\cite{Turaev2019}.

We mention a few state-sum holonomy invariants that are \emph{not} constructions of $\doubfunc$.
\citeauthor{Baseilhac2004} \cite{Baseilhac2004} constructed holonomy invariants (in their language, quantum hyperbolic invariants) using state-sums of quantum dilogarithms over decorated triangulations.
Their invariants should correspond to to $\vecfunc$ and not to $\doubfunc$ because they use $6j$-symbols for the \emph{Borel} subalgebra of $\qgrp$, not the whole thing.
A similar story holds for the $\widehat \Psi$-system invariants defined in \cite[Section 12]{Geer2012}, although it is possible that a different $\widehat \Psi$-system could give a construction of $\doubfunc$.

\section{The mirror of \texorpdfstring{$\vecfunc$}{J}}
Before we can discuss what we mean by $\vecfunc \boxtimes \overline{\vecfunc}$, we must first explain what $\overline{\vecfunc}$ is.
The basic idea is to take $\vecfunc$ and then:
\begin{itemize}
  \item invert the holonomy,
  \item take the opposite coproduct, and
  \item take the inverse braiding.
\end{itemize}
Together all these inversions give a representation $\overline{\vecfunc}$ of $\tangshe$ in $\modc{\qgrp^\cop}$ that we can think of as the mirror image of $\vecfunc$.
In particular, the holonomy invariant $\overline{\vecinv}$ defined by $\overline{\vecfunc}$ will satisfy
\[
  \overline{\vecinv}(L) = \vecinv(\overline{L})
\]
where $\overline{L}$ is the mirror image of $L$ in the usual topological sense.

\subsection{Dual modules, the antipode, and mirrors of shapes}

\begin{defn}
  \label{def:mirror-image}
  Let $\chi = (a, b, \lambda)$ be a shape.
  The \defemph{mirror image} of $\chi$ is the shape
  \[
    \overline{\chi} \defeq (a^{-1}, b \lambda, \lambda^{-1}).
  \]
  For extended shapes we replace the fractional eigenvalue $\mu$ with $\xi^{-2}\mu^{-1}$.
\end{defn}
To understand this definition, observe that the mirror image of a shape corresponds to taking the inverse holonomy:
\begin{align*}
  \chi
  &\leftrightarrow
  \left(
    \begin{bmatrix}
      a & 0 \\
      (a - 1/\lambda)/b & 1
    \end{bmatrix}
    ,
    \begin{bmatrix}
      1 & (a - \lambda)b \\
      0 & a
    \end{bmatrix}
  \right)
  \\
  \overline{\chi}
  &\leftrightarrow
  \left(
    \begin{bmatrix}
      1/a & 0 \\
      -(1 - 1/a\lambda)/b & 1
    \end{bmatrix}
    ,
    \begin{bmatrix}
      1 & -(1 - \lambda/a)b \\
      0 & 1/a
    \end{bmatrix}
  \right)
\end{align*}
This relationship is how to derive the mirror of a shape.
Recall that central characters of $\qgrp$ form a group, with multiplication $\chi_1 \chi_2 \defeq (\chi_1 \otimes \chi_2) \Delta$ given by the coproduct and inverse $\chi^{-1} \defeq \chi S$ given by the antipode.
We can pull back the antipode along $\phi$ to get an inverse on central $\weyl$-characters, i.e.~shapes.

\begin{lem}
  The antiautomorphism $S : \weyl[q] \to \weyl[q]$ given by
  \[
    S(x) = x^{-1},
    \quad
    S(y) = q^{2} zy,
    \quad
    S(z) = q^{-2} z^{-1}
  \]
  pulls back the antipode  $\qgrp[q]$ along the embedding $\phi : \qgrp[q] \to \weyl[q]$.
\end{lem}
\begin{proof}
  Since $S(K) = K^{-1}$, we must have $S(x) = x^{-1}$.
  Guessing $S(z) = q^{-2} z^{-1}$, we see that $S(E) = - EK^{-1}$ becomes
  \[
    S(qy(z - x)) = -qy(z - x)x^{-1}
  \]
  or
  \[
    (q^{-2} z^{-1} - x^{-1})S(y) = y (1 - zx^{-1}) = (1 - q^2 z x^{-1})y
  \]
  which matches $S(y) = q^{2} z y$.
  A similar check works for $S(F)$.
\end{proof}

As an immediate corollary, we see that mirror of a shape is given by composition with the antipode:
\[
  \overline{\chi} = \chi S.
\]
Since the antipode makes the dual space of a module into a module, we see that mirror images are related to taking duals.
For any $\weyl$-module $N$, the dual space $N^* \defeq \hom_\CC(N, \CC)$ is a $\weyl$-module via $u \cdot f : n \mapsto f(S(u) \cdot n)$.
This is exactly the standard definition for Hopf algebras, except that $\weyl$ isn't a Hopf algebra.

\begin{prop}
  Let $\chi = (a, b, \mu)$ be an extended character and pick roots $\alpha^{\nr} = a$, $\beta^{\nr} = b$.
  As in \cref{def:weyl-irrmod} this determines a basis $\{v_{m}\}$ of the standard module $\irrmod \chi$.
  If $\{v^m\}$ is the basis of $\irrmod{\alpha, \beta, \mu}^*$ dual to $\{v_{m}\}$, then the action of $\weyl$ on $\irrmod{\chi}^*$ is given by
  \begin{align*}
    x \cdot v^m &= \alpha^{-1} v^{m-1} 
                &
    y \cdot v^m &= \xi^{2} \mu \beta \xi^{2m} v^{m}
                &
    z \cdot v_m &= \xi^{-2} \mu^{-1} v^m.
  \end{align*}
  In particular, the dual of a standard module is isomorphic to the standard module corresponding to the mirrored shape:
  \[
    \irrmod{\alpha, \beta, \mu}^* \iso \irrmod{\alpha^{-1}, \xi^2 \mu \beta, \xi^{-2} \mu^{-1}}.
  \]
\end{prop}

\subsection{Mirror images of shaped tangles}
Now that we know how to take mirror images of shapes, we can discuss how to take mirror images of links.
\begin{defn}
  Let $L$ be an $\slg$-link in $S^3$.
  Recall that this means that $L$ is a link in $S^3$ along with a representation $\rho : \pi_1(S^3 \setminus L) \to \slg$.
  Let $r : S^3 \to S^3$ be an orientation-reversing homeomorphism.
  The \defemph{mirror image} $\overline{L}$ of $L$ is the link $r(L)$ with representation $\rho \circ r$.
  Similarly, the mirror image of an extended $\slg$-link is the mirror image of the underlying link with the inverse fractional eigenvalues.%
  \note{This is a different convention on the fractional eigenvalues than in \cite{McPhailSnyder2020}.}
\end{defn}

We are interested in a particular way of taking mirror images of shaped tangle diagrams, such that the mirror image of a shaped diagram of $L$ is a shaped diagram of $\overline{L}$.

\begin{marginfigure}
  \centering
  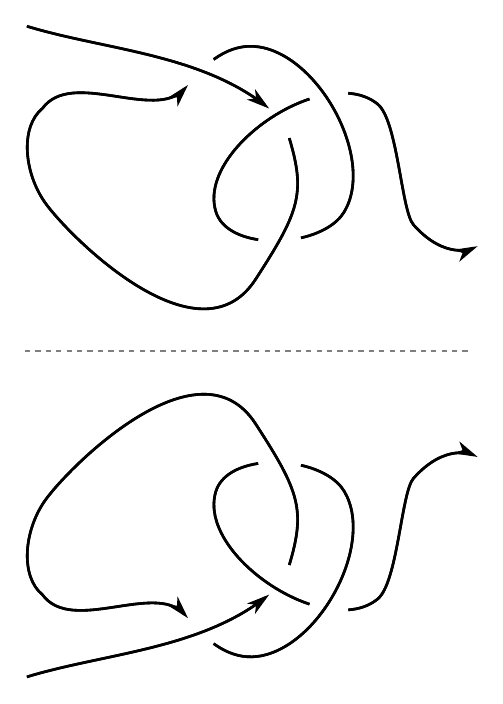
  \caption{The mirror image of a tangle.}
  \label{fig:mirror-example}
\end{marginfigure}

\begin{defn}
  The \defemph{mirror image} of a tangle diagram is the reflection across the horizontal axis, as in \cref{fig:mirror-example}.
  To take the mirror image of a shaped tangle diagram, take its mirror image as a diagram and then take the mirror image of every shape.
\end{defn}

\begin{prop}
  The process of taking the mirror image is functorial.
  We write
  \[
    \mirrorfunc : \tangshe \to \tangshe
  \]
  for the functor taking an extended shaped tangle to its mirror image.
  (There are similar functors for $\tang$ and $\tangsh$, which we also denote by $\mirrorfunc$.)
\end{prop}
\begin{proof}
  We need to check that the gluing equations work.
  This is tedious but elementary.
  Alternately, this proposition is a consequence of \cref{thm:mirror-braiding-exists}; we have chosen to talk about $\mirrorfunc$ and first in order to motivate $\overline{\vecfunc}$, but the logical order is the opposite.
\end{proof}

We can \emph{almost} define $\overline{\vecfunc}$ as $\vecfunc \mirrorfunc$, but this will cause problems later.
To explain this, let's first work out some of the properties of $\vecfunc \mirrorfunc$.
Suppose we have (extended) shapes with $B(\chi_1, \chi_2) = (\chi_{2'}, \chi_{1'})$.
The braiding defining $\vecfunc$ is a map
\[
  \vecfunc(\sigma) :
  \irrmod{\chi_1} \otimes \irrmod{\chi_2}
  \to 
  \irrmod{\chi_{2'}} \otimes \irrmod{\chi_{1'}}
\]
and it follows that the braiding for $\vecfunc \mirrorfunc$ is 
\[
  \vecfunc \mirrorfunc (\sigma) :
  \irrmod{\overline{\chi}_2} \otimes \irrmod{\overline{\chi}_1}
  \to 
  \irrmod{\overline{\chi}_{1'}} \otimes \irrmod{\overline{\chi}_{2'}}
\]
equivalently
\[
  \vecfunc \mirrorfunc (\sigma) :
  \irrmod{\chi_2}^* \otimes \irrmod{\chi_1}^*
  \to 
  \irrmod{\chi_{1'}}^* \otimes \irrmod{\chi_{2'}}^*
\]
To understand $\doubfunc$ we want the tensor factors to line up: the first tensor factor on the left-hand side should be $\irrmod{{\chi}_1}^*$, not $\irrmod{{\chi}_2}^*$.
This will be very important for the computations in \cref{sec:quantum-double,ch:torions}.

To get this to work, we want to switch tensor factors and define  $\overline{\vecfunc}$ to be the representation of the extended shape biquandle that sends $\chi$ to $\irrmod{\overline{\chi}}$ and whose braiding intertwines 
\begin{equation}
  \label{eq:mirror-outer-braiding-def}
  \overline{\Smat} \defeq \tau \Smat^{-1} \tau = \tau \Rmat^{-1}.
\end{equation}
We will see later in \cref{sec:mirror-functor-props} that $\overline{\vecfunc}$ defines the same invariant on links as $\vecfunc \mirrorfunc$.
To match the automorphism defined in \cref{eq:mirror-outer-braiding-def} we need to use the \emph{opposite} comultiplication for $\qgrp$ :
\[
  \Delta^{\op}(E) = E \otimes 1 + K \otimes E,
  \quad
  \Delta^{\op}(F) = F \otimes K^{-1} + 1 \otimes F.
\]
We write $\qgrp^{\cop}$ for $\qgrp$ with the opposite comultiplication.

Recall that the holonomy braidings constructed in \cref{ch:algebras} that define $\vecfunc$ are maps intertwining $\Smat = \tau \Smat$.
This leads to a braiding because
\[
  \Smat \Delta = \Delta,
  \text{ equivalently }
  \Rmat \Delta = \Delta^\op.
\]
Dually, can think of $\Rmat^{-1}$ as intertwining $\Delta^{\op}$ and $(\Delta^{\op})^{\op} = \Delta$,
\[
  \Rmat^{-1} \Delta^{\op} = \Delta,
\]
so to get maps commuting with $\Delta^\op$ we look for those that intertwine the automorphism
\[
  \overline{\Smat} = \tau \Rmat^{-1} = \tau \Smat^{-1} \tau
\]
of $\qgrp^{\cop} \otimes \qgrp^{\cop}$.
More generally, a holonomy braiding for $\overline{\vecfunc}$ should be a family of linear maps intertwining $\overline{\Smat}$.

\begin{prop}
  \label{thm:mirror-braiding-exists}
  Suppose we have (extended) shapes $\chi_i$ related by the shape biquandle as $B(\chi_1, \chi_2) = (\chi_{2'}, \chi_{1'})$, so that $\Smat$ gives an algebra automorphism
  \[
    \Smat :
    \qgrp/\ker \chi_1 \otimes \qgrp/\ker \chi_2
    \to
    \qgrp/\ker \chi_{2'} \otimes \qgrp/\ker \chi_{1'}
  \]
  which acts by the identity on the image of $\Delta(\qgrp)$.
  Then the mirror braiding $\overline{\Smat}$ gives an algebra automorphism
  \[
    \overline{\Smat} :
    \qgrp/\ker \overline{\chi}_1 \otimes \qgrp/\ker \overline{\chi}_2
    \to
    \qgrp/\ker \overline{\chi}_{2'} \otimes \qgrp/\ker \overline{\chi}_{1'}
  \]
  which acts by the identity on the image of $\Delta(\qgrp^{\cop}) = \Delta^{\op}(\qgrp)$.
\end{prop}
\begin{proof}
  This is not hard to compute directly by using the defining relations of $\Rmat$ given in \cref{thm:outer-R-mat-def} and the defining relations
  \[
    S(E) = - EK^{-1}, \quad S(F) = - KF, \quad S(K) = K^{-1}
  \]
  of the antipode.
  Because
  \[
    \Smat(\Omega \otimes 1) = 1 \otimes \Omega
    \text{ and }
    \Smat(1 \otimes \Omega) = \Omega \otimes 1
  \]
  there are no issues with passing to extended shapes.
\end{proof}

\subsection{The mirror representation}
\begin{cor}
  \label{thm:mirror-braiding-exists-modules}
  For extended shapes $\chi_i$ as in \cref{thm:mirror-braiding-exists}, there is an invertible holonomy braiding
  \[
    \overline{S}_{\chi_1, \chi_2} :
    \irrmod{\overline{{\chi}}_1} \otimes \irrmod{\overline{{\chi}}_2} \to 
      \irrmod{\overline{\chi}_{2'}} \otimes \irrmod{\overline{\chi}_{1'}} 
  \]
  equivalently a map
  \[
    \overline{S}_{\chi_1, \chi_2} :
    \irrmod{{{\chi_1}}}^* \otimes \irrmod{{{\chi_2}}}^* \to 
      \irrmod{{\chi_{2'}}}^* \otimes \irrmod{{\chi_{1'}}}^* 
  \]
\end{cor}
\begin{proof}
  The argument is exactly as for \cref{thm:R-matrix-exists}.
\end{proof}

Just as before, it is straightforward to abstractly characterize the braiding but computing the explicit matrix coefficients is more difficult.
We can use the same techniques as in \cref{ch:algebras} to find them in terms of quantum dilogarithms.

\begin{thm}
  \label{thm:mirror-braiding-coefficients}
  Let $\chi_i$ be extended shapes with $B(\chi_1, \chi_2) = (\chi_{2'}, \chi_{1'})$ and components $\chi_i = (a_i, b_i, \mu_i)$, and pick $\nr$th roots $\alpha_i^\nr = a_i$, $\beta_i^{\nr} = b_i$.
  The matrix coefficients of the mirror holonomy braiding  are
  \begin{equation}
    \label{eq:mirror-braiding-coeffs}
    \overline{S}_{m_1 m_2}^{m_1' m_2'} =
    {\Theta_l \Theta_r}
    \frac{ \overline{\Lambda}_l(m_2 - m_2') \overline{\Lambda}_r(m_1' - m_1) }{ \overline{\Lambda}_f(m_1' - m_2') \overline{\Lambda}_b(m_2 - m_1) }
  \end{equation}
  where we use the bases $\{v^m\}$ of the modules  $\irrmod{\chi_i}^{^*}$ induced by the choice of roots $\alpha_i, \beta_i$.
  The factors are each normalized quantum dilogarithms
  \begin{align*}
    \overline \Lambda_f(m)
    &=
    \qlogn{\overline B_f, \overline A_f}{m}
    =
    \qlogn{ \frac{\mu_2 \beta_{2'}}{ \xi^{2} \beta_{1'} } , A_f }{m}
    \\
    \overline \Lambda_b(m)
    &=
    \qlogn{\overline B_b, \overline A_b}{m}
    =
    \qlogn{ \frac{ \beta_2 }{ \xi^{2} \mu_1 \beta_1 }, \frac{\alpha_1}{\mu_1 \mu_2 \alpha_{2'}} A_b }{m}
    \\
    \overline \Lambda_l(m)
    &=
    \qlogn{\overline B_l, \overline A_l}{m}
    =
    \qlogn{ \frac{ \beta_{2'} }{ \xi^2 \beta_1 }, \frac{\alpha_{1'}}{\xi^2 \mu_1} A_f }{m}
    \\
    \overline \Lambda_r(m)
    &=
    \qlogn{\overline B_r, \overline A_r}{m}
    =
    \qlogn{ \frac{ \mu_2 \beta_2}{ \xi^2  \mu_1 \beta_{1'} } , \frac{1}{\xi^{2} \mu_2 \alpha_{2'}} A_f }{m}
  \end{align*}
  and
   \[
     \Theta_i = {\qlogsum{B_i, A_i}}
  \]
  where $\mathfrak{S}$ is the function given in \cref{thm:qlog-fourier}.
  Notice that the parameters for  $\overline{\Lambda}_f$ and $\overline{\Lambda}_b$ are the same as those for $\Lambda_f$ and $\Lambda_b$, while those for $\overline{\Lambda}_{r}$ and $\overline{\Lambda}_l$ differ by a factor of $\xi^{2}$.
\end{thm}

\begin{thm}
  \label{thm:mirror-representation}
  Assigning the modules $\irrmod{\chi}^*$ to extended shapes $\chi$ and crossings to the holonomy braidings given in \cref{thm:mirror-braiding-coefficients} gives a representation of the extended shape biquandle in $\modc{\qgrp^{\cop}}/\units{\nr^2}$.
  This representation is compatible with the modified trace $\modtr$ of \cref{thm:modified-trace-exists}.
\end{thm}
\begin{proof}
  The proof goes exactly as the proof of \cref{thm:standard-model} given in \cref{sec:proof-of-model}.
  We know that the matrix $\overline{S}$ given in \cref{eq:mirror-braiding-coeffs} intertwines $\overline{\Smat}$, so it must give a model up to an overall scalar.
  Since the determinant of $\overline{S}$ is $1$ (by using the same results on quantum dilogarithms as in \cref{thm:det-is-1}) we can show that it actually gives a model up to an $\nr^2$th root of unity.
  This model is sideways invertible and induces a twist for exactly the same reasons as $\vecfunc$, so it gives a representation, and it is is similarly compatible with $\modtr$ so it gives an invariant of extended $\slg$-links.
\end{proof}

\subsection{Properties of \texorpdfstring{$\overline{\vecfunc}$}{J bar}}
\label{sec:mirror-functor-props}

As an immediate corollary of \cref{thm:mirror-representation} we get an invariant $\overline{\vecinv}(L)$ of extended shaped links $L$ defined via the functor $\overline{\vecfunc}$.
This invariant is not anything new, but is just the value of $\vecinv(L)$ on the mirror image of $L$.
\begin{thm}
  \label{thm:mirrors-are-mirrors}
  Let $L$ be an extended shaped link.
  Then
  \[
    \overline{\vecinv}(L) = \vecinv(\overline{L})
  \]
  where $\overline{L}$ is the mirror image of $L$.
\end{thm}
\begin{proof}
  Consider a crossing of a diagram $L$ with extended shapes $\chi_i$ as in \cref{fig:shaped-positive-crossing}.
  The functor $\overline{\vecfunc}$ assigns this crossing the matrix $\overline{S}$ given in \cref{eq:mirror-braiding-coeffs}.
  On the other hand, the mirror image of this crossing is assigned the matrix $\widetilde{S}$ given in \cref{eq:negative-braiding-matrix} but with mirrored coefficients:
  \[
    \alpha \mapsto \alpha^{-1}, 
    \quad
    \beta \mapsto \mu \beta,
    \quad
    \mu \mapsto \xi^{-2} \mu^{-1}.
  \]
  We can check directly that the matrices $\overline{S}_{m_1 m_2}^{m_1' m_2'}$ and $\widetilde{S}_{m_1 m_2}^{m_1' m_2'}$ are identical after swapping $m_1$ and $m_2$ and $m_1'$ and $m_2'$ (that is, after conjugating by the flip).

  Now suppose $L$ is represented as the closure of an extended shaped braid $\beta$ on $n$ strands.
  (See \cref{sec:braids} for more on braids.)
  The functor $\overline{\vecfunc}$ assigns $\beta$ some $\nr^n \times \nr^n$ matrix
  \[
    \overline{M}_{m_1 m_2 \cdots m_n}^{m_1' m_2' \cdots m_n'}
  \]
  and similarly $\vecfunc \mirrorfunc$ assigns $\beta$ a matrix
  \[
    \widetilde{M}_{m_1 m_2 \cdots m_n}^{m_1' m_2' \cdots m_n'}.
  \]
  By the argument in the previous paragraph,
  \[
    \overline{M}_{m_1 m_2 \cdots m_n}^{m_1' m_2' \cdots m_n'}
    =
    \widetilde{M}_{m_n m_{n-1} \cdots m_1}^{m_n' m_{n-1}' \cdots m_1'}
  \]
  so these matrices are equal after conjugating by the linear isomorphism
  \[
    F(v^{m_1 m_2 \cdots m_n})
    =
    F(v^{m_n m_{n-1} \cdots m_1}).
  \]
  It follows that their modified traces agree and thus that $\overline{\vecinv}(L) = \vecinv(\overline{L})$.
\end{proof}

\section{The doubled braiding}
\subsection{External tensor products}
\label{sec:tensor-products}
The category $\modc{\qgrp}$ is a monoidal category: given two objects $V$ and $W$, we can take their tensor product $V \otimes W$, which becomes a $\qgrp$-module via the coproduct of $\qgrp$.
Because it stays inside $\modc{\qgrp}$ we call this an \defemph{internal} tensor product and denote it by $\otimes$.

To construct $\doubfunc$, we want to consider an \emph{external} tensor product that takes two categories or functors and produces another, larger category or functor.
To distinguish this from the internal tensor product, we denote it by $\boxtimes$.

\begin{defn}
  Let $H_1$ and $H_2$ be Hopf algebras over $\CC$, and write $\modc{H_i}$ for the category of (finite-dimensional) $H$-modules.
  The \defemph{tensor product} of $\modc{H_1}$ and $\modc{H_2}$ is the category
  \[
    \modc{H_1} \boxtimes \modc{H_2} \defeq \modc{(H_1 \otimes_{\CC} H_2)}
  \]
  of finite-dimensional $H_1 \otimes_\CC H_2$-modules.
  If $V_1$ and  $V_2$ are modules for $H_1$ and $H_2$ respectively, then we write
  \[
    V_1 \boxtimes V_2 \defeq V_1 \otimes_\CC V_2
  \]
  for the corresponding $(H_1 \otimes_\CC H_2)$-module, an object of $\modc{H_1} \boxtimes \modc{H_2}$.

  Now suppose $f_1 : V_1 \to W_1$ is a morphism of $\modc{H_1}$, and similarly for $f_2: V_2 \to W_2$.
  Then we define their external tensor product $f_1 \boxtimes f_2 : V_1 \boxtimes V_2 \to W_1 \boxtimes W_2$ by
  \[
    (f_1 \boxtimes f_2)(v_1 \boxtimes v_2) = f_1(v_1) \boxtimes f_2(v_2)
  \]
  for all $v_1 \in V_1$, $v_2 \in V_2$.
\end{defn}
This is a special case of the Deligne tensor product \cite[\S1.11]{EGNO2015} of categories, which is one reason we use the notation~$\boxtimes$.
Since we only are interested in (subcategories of) $\modc{\qgrp}$, we do not need the construction in full generality.

\begin{remark}
  \label{rem:internal-external-commute}
  Notice that $\otimes$ and $\boxtimes$ commute%
  \note{Formally speaking, this equality should be a natural isomorphism.}
  in the sense that
  \[
    (V_1 \otimes W_1) \boxtimes (V_2 \otimes W_2) = (V_1 \boxtimes V_2) \otimes (W_1 \boxtimes W_2).
  \]
  Here the tensor products $\otimes$ on the left are the internal tensor products of $\modc{H_1}$ and $\modc{H_2}$, respectively, while on the right $\otimes$ is the internal tensor product of $\modc{(H_1 \otimes_\CC H_2)}$.

  Depending on the context, both sides of the above equation are useful.
  For example, it is much easier to describe an external tensor product of two maps using the left-hand side.
\end{remark}

\begin{defn}
  Let $H_1$ and $H_2$ be Hopf algebras as above.
  Suppose we have functors
  \[
    \mathcal{F}_i : \tang[X] \to \modc{H_i}, \quad i = 1, 2
  \]
  from a colored tangle category to the category modules of a Hopf algebra.%
  \note{$\vecfunc$ and $\overline{\vecfunc}$ are two examples of such functors.}
  The \defemph{external tensor product} of $\mathcal{F}_1$ and $\mathcal{F}_2$ is the functor
  \[
    \mathcal{F}_1 \boxtimes \mathcal{F}_2 : \tang[X] \to \modc{(H_1 \otimes_\CC H_2)}
  \]
  defined by
  \[
    (\mathcal{F}_1 \boxtimes \mathcal{F}_2)(\chi_1, \cdots, \chi_n)
    =
    (\mathcal{F}_1(\chi_1, \dots, \chi_n)) \boxtimes (\mathcal{F}_2(\chi_1, \dots, \chi_n))
  \]
  on objects and by
  \[
    (\mathcal{F}_1 \boxtimes \mathcal{F}_2)(T)
    =
    \mathcal{F}_1(T) \boxtimes \mathcal{F}_2(T)
  \]
  on morphisms.%
  \note{We can think of this as a ``grouplike'' coproduct, especially when we restrict to shaped braid groupoids in \cref{ch:torions}.}
\end{defn}

\begin{remark}
  \label{rem:internal-external-commute-ii}
  When the functors $\mathcal{F}_i$ above come from a representation of a biquandle $X$, they are monoidal in the sense that
  \[
    \mathcal{F}_i(\chi_1, \dots, \chi_n) = \mathcal{F}_i(\chi_1) \otimes \cdots \otimes \mathcal{F}_i(\chi_n).
  \]
  for $\chi_j \in X$.
  For example, $\vecfunc$ and $\overline{\vecfunc}$ are of this type.
  As noted in \cref{rem:internal-external-commute}, we can think of the image of $(\chi_1, \dots, \chi_n)$ under their tensor product in two equivalent ways:
  \[
    \bigotimes_{j=1}^{n} \mathcal{F}_1(\chi_j) \boxtimes \mathcal{F}_2(\chi_j)
    =
    (\mathcal{F}_1 \boxtimes \mathcal{F}_2)(\chi_1, \dots, \chi_n)
    =
    \left(\bigotimes_{j=1}^{n} \mathcal{F}_1(\chi_j)\right)
    \boxtimes
    \left(\bigotimes_{j=1}^{n} \mathcal{F}_2(\chi_j)\right).
  \]
  The left-hand side is more useful when trying to understand tensor products in $\modc{H_1} \boxtimes \modc{H_2}$, while the right-hand side is more useful when defining the image of morphisms under $\mathcal{F}_1 \boxtimes \mathcal{F}_2$.
\end{remark}

\section{The double of the quantum holonomy invariant}
\label{sec:quantum-double}

\begin{defn}
  The \defemph{double} of functor $\vecfunc$  is the external tensor product
  \[
    \doubfunc \defeq \vecfunc \boxtimes \overline{\vecfunc} :
    \tangshe \to \modc{\qgrp \otimes \qgrp^{\cop}}/\units{\nr^2}
  \]
  of $\vecfunc$ and its mirror image.
  Equivalently, it is the functor corresponding to the representation of the extended shape biquadle in $\modc{\qgrp \otimes \qgrp^{\cop}}$ given by $(\irrmod{-} \boxtimes \irrmod{-}^*, S_{-,-} \boxtimes \overline{S}_{-,-})$.
\end{defn}

\subsection{Determining the phase}
We can now show that the phase of $\doubfunc$ can be defined unambiguously.
\begin{thm}
  \label{thm:double-lifts}
  The functor $\doubfunc$ lifts to a functor
  \[
    \doubfunc :
    \tangshe \to \modc{\qgrp \otimes \qgrp^{\cop}}
  \]
  with no scalar ambiguity.
\end{thm}
\begin{proof}
  The key idea is to use the fact that a strand with shape $\chi$ is mapped to the module
  \[
    \doubfunc(\chi) = \irrmod{\chi} \boxtimes \irrmod{\chi}^*.
  \]
  As a consequence, we get a \emph{basis-independent} family of vectors%
  \[
    w(\chi) \defeq \sum_{m} v_m \boxtimes v^m \in \doubfunc(\chi).
  \]
  We check in the lemma below that $\doubfunc$ preserves these in the sense that
  \[
    \doubfunc(\sigma)(w(\chi_1) \otimes w(\chi_2)) = w(\chi_{2'}) \otimes w(\chi_{1'}).
  \]
  This condition that $\doubfunc$ preserves the $w(\chi)$ is enough to show that it satisfies the colored braid relation exactly.
  The two homomorphisms
  \[
    \doubfunc(\sigma_1 \sigma_2 \sigma_1),\doubfunc(\sigma_2 \sigma_1 \sigma_2) : 
    \doubfunc(\chi_1, \chi_2, \chi_3)
    \to
    \doubfunc(\chi_{3'''}, \chi_{2'''}, \chi_{1'''})
  \]
  appearing on each side of the \reidthree{} relation are determined up to a scalar.
  Since
  \begin{align*}
    &\doubfunc(\sigma_1 \sigma_2 \sigma_1)(w(\chi_1) \otimes w(\chi_2) \otimes w(\chi_3))
    \\
    &=
    w(\chi_{3'''}) \otimes w(\chi_{2'''}) \otimes w(\chi_{1'''})
    \\
    &=
    \doubfunc(\sigma_2 \sigma_1 \sigma_2) (w(\chi_1) \otimes w(\chi_2) \otimes w(\chi_3))
  \end{align*}
  we conclude that
  \[
    \doubfunc(\sigma_1 \sigma_2 \sigma_1)
    =
    \doubfunc(\sigma_2 \sigma_1 \sigma_2).\qedhere
  \]
\end{proof}

\begin{lem}
  \label{thm:invariant-vector}
  The functor $\doubfunc$ preserves the vectors $w(\chi)$.  
\end{lem}
\begin{proof}
  First consider a non-pinched crossing; to extend to the pinched case, we can use the same continuity argument as in \cref{thm:pinched-braidings}.
  Now that we know the matrix coefficients and factorizations%
  \note{
    Computing these was quite hard, but it had other benefits.
    In any case, the proof here is more enlightening than the complicated proof in \cite[Appendix C]{McPhailSnyder2020}, which only works in the case $\nr = 2$.
  }
  of $S$ and $\overline{S}$ this is straightforward.
  For example,
  \[
    S_b \boxtimes \overline{S}_b \cdot w \otimes w
    =
    \sum_{m_1 m_2} \frac{\Lambda_b(m_2 - m_1) }{\overline{\Lambda}_b(m_2 -m_1)} v_{m_1} \boxtimes v^{m_1} \otimes v_{m_2} \boxtimes v^{m_2}
    =
    w \otimes w
  \]
  because $\overline{\Lambda}_b(m) = \Lambda_b(m)$.
  To check invariance under $S_r \boxtimes \overline{S_r}$ we (as before) want to work in the Fourier dual basis as in the proof of \cref{thm:det-is-1}.
  Because $w$ is basis-independent,
  \[
    \sum_{m} v_m \boxtimes v^m = \sum_{m} \hat v_m \boxtimes \hat v^m,
  \]
  and since (by \cref{eq:qlog-fourier,eq:qlog-recip-fourier})
  \begin{align*}
    S_r \cdot \hat v_m
    &=
    \qlogn{\xi^{-2} A_r, B_r}{-m} \hat v_m
    \\
    \overline{S}_r \cdot \hat v^m
    &=
    \qlogn{\xi^{-2} A_r, B_r}{-m} \hat v^m
  \end{align*}
  we also have
  \[
    S_r \boxtimes \overline{S}_r \cdot w \otimes w = w \otimes w.
  \]
  Similar checks work for $l$ and $f$.
\end{proof}

\subsection{Proof of theorem \ref{thm:quantum-double-inv}}

Because $\doubfunc$ satisfies the colored Reidemeister relations, it gives a model of the extended shape biquandle in $\modc \qgrp \boxtimes \modc{\qgrp^{\cop}}$.
This model is absolutely simple and regular, so it induces a twist, and it is sideways invertible because $\vecfunc$ and $\overline{\vecfunc}$ are.
By \cref{thm:double-traces} there is a modified trace on the ideal of projective objects of $\modc \qgrp \boxtimes \modc{\qgrp^{\cop}}$ that is compatible with those for $\modc \qgrp$ and $\modc{\qgrp^{\cop}}$ in the sense that for any $f : X \to X$, $g : Y \to Y$,
\[
  \modtr_{X \boxtimes Y}(f \boxtimes g) = \modtr_X(f) \modtr_Y(g).
\]
In particular this modified trace is gauge invariant.
We can therefore apply \cref{thm:invariant-construction} to conclude that $\doubinv$ is well-defined and gauge invariant.

The compatibility property of the trace and the fact that (up to $\nr^2$th roots of unity) $\doubfunc = \vecfunc \boxtimes \overline{\vecfunc}$ shows that
\[
  \doubinv(L) = \vecinv(L) \overline{\vecinv}(L)
\]
up to $\nr^2$th roots of unity, and we can apply \cref{thm:mirrors-are-mirrors} to conclude that
\[
  \doubinv(L) = \vecinv(L) \vecinv(\overline{L}).
\]

\chapter{Torsions}
\label{ch:torions}
\section*{Overview}
Set $\nr = 2$, so that $\xi = i$.
Our goal in this chapter is to prove:
\begin{thm}
  \label{thm:T-is-torsion}  
  Let $L$ be an admissible enhanced $\slg$-link.
  Then
  \[
    \doubinv_2(L) = \tau(L)
  \]
  where $\tau(L)$ is the Reidemeister torsion of $S^3 \setminus L$ twisted by the holonomy representation $\rho : \pi(L) \to \slg$.
\end{thm}
Recall that for $\nr = 2$, a link is admissible when no meridian has $1$ as an eigenvalue.
As discussed in \cref{sec:torsion-def} this condition is also required for the torsion to be well-defined.

The torsion is a well-understood invariant of topological spaces, and twisting by a nonabelian representation $\rho$ incorporates geometric information.
We can think of the torsion as a classical%
\note{
  ``Classical'' has two interpretations here, both of which apply:
  the physical ``not quantum'', or the mathematical ``understood before the author started their PhD''.
}
holonomy invariant, as mentioned in the introduction.
As such, this result is a first step towards understanding better what the nonabelian quantum dilogarithm says about geometry.
It would be very interesting to understand how $\doubinv_\nr$ for $\nr > 2$ relates to the torsion, especially in the context of the semiclassical limit $\nr \to \infty$ appearing in the Volume Conjecture.

\subsection{Strategy of the proof}
We first need to define the torsion in the context of shaped tangles.
When we think of our link $L$ as the closure of a shaped braid $\beta$ we can think of $\beta$ as acting on a punctured disc.
This gives a braid group action on the homology of the punctured disc (twisted by the holonomy of the shapes) called the \defemph{Burau representation} which can be used to compute the torsion.
We discuss this in more detail in \cref{sec:burau}.

To relate the Burau representation to the braid group representation coming from $\doubfunc$ we prove a \defemph{Schur-Weyl duality} between them in \cref{sec:schur-weyl}.
As discussed in \cref{sec:schur-weyl-motiv}, the idea is to find a subalgebra of $\qgrp[i]^{\otimes n}$ corresponding to the homology that (super)commutes with the image of $\qgrp[i]$ under the coproduct and use this to understand the decomposition of tensor products of $\qgrp[i]$-modules and the braid group action on them.

Actually applying Schur-Weyl duality to the tensor decomposition of the image of $\doubfunc$ requires a bit of a detour; we need to slightly weaken the type of multiplicity space we consider.
In \cref{sec:torsion-proof} we work out these issues and prove a version (\cref{thm:schur-weyl-modules}) of Schur-Weyl duality for modules.
Once this is done \cref{thm:T-is-torsion} is an immediate corollary.

\section{The shaped braid groupoid and Burau representations}
\label{sec:burau}
To define the Burau representation and the torsion we want to work with with diagrams more regular than tangles.%
\note{
  From another perspective, our functor $\doubfunc_2$ can be seen as a generalization of the exterior algebra on the twisted reduced Burau representation to tangles.
}

\subsection{Braids and shaped braids}
\label{sec:braids}
\begin{defn}
  \label{def:braids}
  A \defemph{braid} is a tangle that only goes left-to-right.
  More formally, it is a tangle in which the $x$-components of the tangent vectors to the strands have constant sign.
  A braid $\beta$ is necessarily a morphism $n \to n$ for some $n = 0, 1, \dots$, in which case we call $\beta$ a \defemph{braid on $n$ strands}.
  An example is given in \cref{fig:braid-example}.

  We assume for simplicity that the strands of a braid are oriented left-to-right, and write $\brd$ for the subcategory of $\tang$ of oriented braids.
  $\brd$ is a groupoid, and its connected components are the usual groups $\brd[][n]$ of braids on $n$ strands.
\end{defn}

\begin{marginfigure}
  \centering
  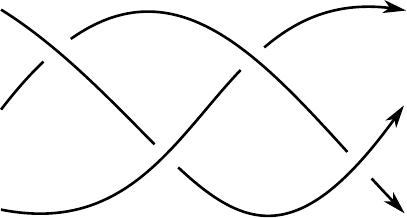
  \caption{The braid $\sigma_1 \sigma_2^{-1} \sigma_1 \sigma_2^{-1}$ on $3$ strands, whose closure is the figure-eight knot.}
  \label{fig:braid-example}
\end{marginfigure}

To use braids to compute holonomy invariants we color their segments just as for tangles.

\begin{defn}
  \label{def:colored-braids}
  If $X$ is a biquandle, we write $\brd[X]$ for the subcategory of $\tang[X]$ of braids, which we call the \defemph{$X$-colored braid groupoid}.
  As before we write $\brd[X][n]$ for the subgroupoid of $X$-colored braids on $n$ strands.

  This definition still works when $X$ is partially defined; we simply have to disallow Reidemeister moves and compositions that would result in undefined colors.
  In particular, we have the \defemph{shaped braid groupoid} $\brdsh$ and \defemph{extended} shaped braid groupoid $\brdshe$ of braids colored by the shape biquandle and extended shape biquandle, respectively.
\end{defn}

As before, the simplest way to describe $\slg$-links is as closures of the braids colored by (the conjugation quandle of) $G$, as discussed below.
However, in the context of $\qgrp$ we want to use shaped braids instead, and as before this works because the holonomy of the shapes gives an appropraite representation into $\slg$.
In the context of torsions, it is also more convenient to derive the Burau representation in terms of $G$-colored braids and then pull back this description to shaped braids.

\begin{marginfigure}
  \centering
\begingroup%
  \makeatletter%
  \providecommand\color[2][]{%
    \errmessage{(Inkscape) Color is used for the text in Inkscape, but the package 'color.sty' is not loaded}%
    \renewcommand\color[2][]{}%
  }%
  \providecommand\transparent[1]{%
    \errmessage{(Inkscape) Transparency is used (non-zero) for the text in Inkscape, but the package 'transparent.sty' is not loaded}%
    \renewcommand\transparent[1]{}%
  }%
  \providecommand\rotatebox[2]{#2}%
  \newcommand*\fsize{\dimexpr\f@size pt\relax}%
  \newcommand*\lineheight[1]{\fontsize{\fsize}{#1\fsize}\selectfont}%
  \ifx\svgwidth\undefined%
    \setlength{\unitlength}{94.50118446bp}%
    \ifx\svgscale\undefined%
      \relax%
    \else%
      \setlength{\unitlength}{\unitlength * \real{\svgscale}}%
    \fi%
  \else%
    \setlength{\unitlength}{\svgwidth}%
  \fi%
  \global\let\svgwidth\undefined%
  \global\let\svgscale\undefined%
  \makeatother%
  \begin{picture}(1,1)%
    \lineheight{1}%
    \setlength\tabcolsep{0pt}%
    \put(0,0){\includegraphics[width=\unitlength,page=1]{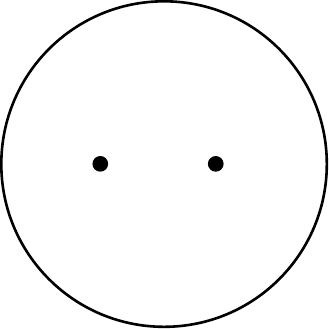}}%
    \put(0.20238464,0.2817488){\makebox(0,0)[lt]{\lineheight{1.25}\smash{\begin{tabular}[t]{l}$w_2$\end{tabular}}}}%
    \put(0.6388871,0.2817488){\makebox(0,0)[lt]{\lineheight{1.25}\smash{\begin{tabular}[t]{l}$w_1$\end{tabular}}}}%
    \put(0,0){\includegraphics[width=\unitlength,page=2]{disc-generators.pdf}}%
  \end{picture}%
\endgroup%

  \caption{Generators of the fundamental group  $\pi_1(D_2)$ of the twice-punctured disc.
  This picture can be obtained from \cref{fig:wirtinger-presentation} by slicing a vertical plane perpendicular to the page through the left-hand side of the diagram.}
  \label{fig:disc-generators}
\end{marginfigure}
\begin{ex}
  Let $G$ be a group.
  The conjugation quandle (recall \cref{def:conjugation-quandle}) is the biquandle given by
  \[
    B(g_1, g_2) = (g_1^{-1} g_2 g_1, g_1).
  \]
  We call the category $\brd[G]$ of braids colored by the conjugation quandle of $G$ the category of \defemph{$G$-colored braids}.
  In more detail, $\brd[G]$ is the category with:
  \begin{description}
    \item[objects]  are tuples $(g_1, \dots, g_n)$ of elements of $G$.
    \item[morphisms] $(g_1, \dots, g_n) \to (g_1', \dots, g_n')$ are braids on $n$ strands, which act on the group elements by the Wirtinger rule $\sigma_1 : (g_1, g_2) \to (g_1^{-1} g_2 g_1, g_1)$ given above and in \cref{fig:wirtinger-presentation}.
  \end{description}
\end{ex}
Topologically, morphisms of $\brd[G]$ are braids with representations of their complements into $G$, just as for $G$-colored tangles in \cref{sec:wirtinger-presentation}.
Because the fundamental group of $D_n$ is a free group $F_n$ on $n$ generators $w_1, \dots, w_n$ (see \cref{fig:disc-generators}), we can associate an object $(g_1, \dots, g_n)$ $\brd[G][n]$ to the representation $\rho : \pi_1(D_n) \to G$ given by $\rho(w_i) = g_i$.

To understand this definition topologically, recall that $\brd[][n]$ is the mapping class group of an $n$-punctured disc.
\note{
  The mapping class group of a space $X$ is the quotient of the automorphism group of $X$ where we identify isotopic maps.
  In this case, we want to think of $D_n$ as a closed disc with punctures and consider automorphisms fixing the boundary.
}%
In particular, the braid group $\brd[][n]$ on $n$ strands acts on $D_n$ by homemorphisms fixing the boundary, hence acts on $\pi_1(D_n)$.
We can explicitly give this action as
\begin{equation}
  \label{eq:braid-action-free-group}
  w_j \cdot \sigma_i 
  =
  \begin{cases}
    w_i^{-1} w_{i+1} w_{i} & j = i,
    \\
    w_i & j = i+1,
    \\
    w_j & \text{otherwise.}
  \end{cases}
\end{equation}
This is a \emph{right} action because braids are composed left-to-right.
\note{
  I feel strongly that topological morphisms like braids and tangles should be written left-to-right, since that is the natural order of composition in a language like English that is written left-to-right.
  The required right actions look strange, but I think this is worth the cost.
}

Braids act on these representations via $\beta : \rho \mapsto \rho\beta^{-1}$ and the conjugation biquandle is just an algebraic description of this action.
The $G$-colored braid groupoid is a version of the mapping class group of $D_n$ where we keep track of the additional data of the representations $\rho : \pi_1(D_n) \to G$ and the corresponding action on them by the braids.

\begin{ex}
  The (extended) shaped braid groupoid has
  \begin{description}
    \item[objects]
      tuples $(\chi_1, \dots, \chi_n)$ of shapes.
      Recall that a shape is a tuple $\chi = (a, b, \lambda)$ of nonzero complex numbers, equivalently a character of the central subalgebra $\CC[x^{\pm \nr}, y^{\pm \nr}, z^{\pm \nr}]$ of the algebra $\weyl$.
      An extended shape replaces $\lambda$ with a $\nr$th root $\mu$.
    \item[morphisms]
      $(\chi_1, \dots, \chi_n) \to (\chi_1', \dots, \chi_n')$ braids on $n$ strands acting on the shape according to the rules in \cref{def:shape-biquandle}. 
  \end{description}
\end{ex}

The same topological interpretation works for the shaped braid groupoid: we can think of it as a mapping class groupoid for $D_n$ that also keeps track of geometric data via the shapes.
The shapes are an alternative coordinate system on the variety of representations $\pi_1(D_n) \to \slg$ that is more convenient for quantum invariants and hyperbolic geometry the.
We can obtain a $\slg$-colored braid from a shaped braid via the functor $\Psi$ (\cref{thm:links-presentable}), just as for tangles.

In the case of braids, we can be more explicit about the action of $\Psi$ on objects: a tuple $(\chi_1, \dots, \chi_n)$ has $\Psi(\chi_1, \dots, \chi_n) = (g_1, \dots, g_n)$, where
\[
  g_k = g^+(\chi_1) \cdots g^+(\chi_{k-1}) g^+(\chi_k) g^-(\chi_k) g^+(\chi_{k-1})^{-1} \cdots g^+(\chi_{1})^{-1}.
\]
In terms of the generators $y_k = w_k \cdots w_1$ introduced later, the formula is a bit simpler:
\[
  \rho(y_k) = g_k \cdots g_1 = g^+(\chi_1) \cdots g^+(\chi_{k}) \left( g^-(\chi_1) \cdots g^-(\chi_k) \right)^{-1}.
\]

\subsection{Twisted homology and Burau representations}
Because braids act on the punctured disc they also act on topological invariants like homology.
We are interested in the Burau representations, which come from the action of the braid group on \emph{twisted} homology.
For simplicity, we first describe this for $\gl(V)$-colored braids, then explain how to pull this description back to the shaped braid groupoid.

\begin{defn}
  Let $X$ be a finite CW complex with fundamental group $\pi = \pi_1(X)$, and let $\rho : \pi \to \gl(V)$ be a representation, where $V$ is a vector space.%
\note{More generally this works for a module over any commutative ring; this perspective is important when defining the twisted Alexander polynomial.}
We think of this as a right representation acting on row vectors, so that $V$ is a right $\ZZ[\pi]$-module.

Let $\widetilde X$ be the universal cover of $X$.
The group $\pi = \pi_1(X)$ acts on the cells of the universal cover, and this action commutes with the differentials.
We take this to be a left action, so that the cellular chain complex $C_*(\widetilde X)$ of the universal cover becomes a complex of left $\ZZ[\pi]$-modules.

  The \defemph{$\rho$-twisted homology} $\homol[*]{X;\rho}$ of $X$ is the homology of the \defemph{$\rho$-twisted chain complex}
  \[
    C_*(X; \rho) \defeq V \otimes_{\ZZ[\pi]} C_*(\widetilde X).
  \]
\end{defn}
We have given this definition in terms of a CW complex for $X$ and a choice of lifts, but it can be shown to not depend on the choice of lifts.
In fact, the $\rho$-twisted homology also does not depend on the CW structure.
One way to see this is to give a definition in terms of $\gl(V)$-local systems.

The Burau representation is given by the action of braids on the twisted homology groups.
Because a braid $\beta$ acts nontrivially on the representations, it should be understood as a representation of the groupoid $\brd[\gl(V)]$.
\begin{defn}
  The \defemph{Burau representation} is the functor $\widetilde{\burau} : \brd[\operatorname{GL}(V)] \to \vect{\CC}$ sending an object $\rho$ to the vector space $\homol{D_n;\rho}$ and a $\gl(V)$-colored braid $\beta : \rho \to \rho'$ to the linear map
  \[
    \widetilde{\burau}(\beta) : \homol{D_n;\rho} \to \homol{D_n;\rho'}
  \]
  corresponding to the action of $\beta$ on $D_n$.
  Here $\vect{\CC}$ is the category of $\CC$-vector spaces and linear maps.

  Any braid $\beta$ fixes the boundary of $D_n$, so we can define the \defemph{boundary-reduced Burau representation} as the action on homology relative to the boundary:
  \[
    \burau^{\partial}(\beta) : \homol{D_n, \partial D_n; \rho} \to \homol{D_n, \partial D_n; \rho'}.
  \]
\end{defn}

In connection with the reduced Burau representation, it is helpful to use a slighly different presentation of $\pi_1(D_n)$.
Recall the generators $w_1, \dots, w_n$ of $\pi_1(D_n)$ shown in \cref{fig:disc-generators}.
There is another set of generators
\[
  y_i \defeq w_i \cdots w_1, \text{ for } i = 1, \dots, n.
\]
on which the braid group acts by
\begin{equation}
  y_j \cdot \sigma_i
  =
  \begin{cases}
    y_{i-1} y_{i}^{-1} y_{i+1} & i = j,
    \\
    y_j & i \ne j
  \end{cases}
\end{equation}
The generators for $D_3$ are shown in in \cref{fig:disc-generators-ii}.
\begin{marginfigure}
  \centering
\begingroup%
  \makeatletter%
  \providecommand\color[2][]{%
    \errmessage{(Inkscape) Color is used for the text in Inkscape, but the package 'color.sty' is not loaded}%
    \renewcommand\color[2][]{}%
  }%
  \providecommand\transparent[1]{%
    \errmessage{(Inkscape) Transparency is used (non-zero) for the text in Inkscape, but the package 'transparent.sty' is not loaded}%
    \renewcommand\transparent[1]{}%
  }%
  \providecommand\rotatebox[2]{#2}%
  \newcommand*\fsize{\dimexpr\f@size pt\relax}%
  \newcommand*\lineheight[1]{\fontsize{\fsize}{#1\fsize}\selectfont}%
  \ifx\svgwidth\undefined%
    \setlength{\unitlength}{94.50118446bp}%
    \ifx\svgscale\undefined%
      \relax%
    \else%
      \setlength{\unitlength}{\unitlength * \real{\svgscale}}%
    \fi%
  \else%
    \setlength{\unitlength}{\svgwidth}%
  \fi%
  \global\let\svgwidth\undefined%
  \global\let\svgscale\undefined%
  \makeatother%
  \begin{picture}(1,1)%
    \lineheight{1}%
    \setlength\tabcolsep{0pt}%
    \put(0,0){\includegraphics[width=\unitlength,page=1]{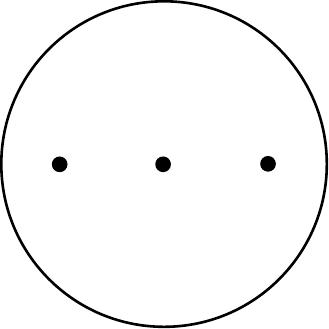}}%
    \put(0.20914557,0.81970731){\makebox(0,0)[lt]{\lineheight{1.25}\smash{\begin{tabular}[t]{l}$y_3$\end{tabular}}}}%
    \put(0.2902318,0.65162551){\makebox(0,0)[lt]{\lineheight{1.25}\smash{\begin{tabular}[t]{l}$y_2$\end{tabular}}}}%
    \put(0.61896796,0.45328038){\makebox(0,0)[lt]{\lineheight{1.25}\smash{\begin{tabular}[t]{l}$y_1$\end{tabular}}}}%
    \put(0,0){\includegraphics[width=\unitlength,page=2]{disc-generators-ii.pdf}}%
  \end{picture}%
\endgroup%

  \caption{Alternative generators for $\pi_1(D_3)$.}
  \label{fig:disc-generators-ii}
\end{marginfigure}

The images of the $y_i$ give a basis for the twisted homology.
More formally, if $e_1, \dots, e_k$ is a basis of the vector space $V$, then
\[
  \{y_i \otimes e_j : i = 1, \dots n, j = 1, \dots, k\}
\]
is a basis of $\homol{D_n; \rho}$; here we abuse notation and write $y_i$ for its image.
The final generator $y_n$ corresponds to the boundary of the disc, so we similarly have a basis
\[
  \{y_i \otimes e_j : i = 1, \dots n -1, j = 1, \dots, k\}
\]
of $\homol{D_n, \partial D_n, \rho}$.

\begin{ex}
  Let $V = \mathbb{Q}(t)$ be the one-dimensional vector space over the field%
  \note{
    We want to work over fields for simplicity, but we could equally well define this over $\ZZ[t, t^{-1}]$, equivalently the group ring of $\ZZ$.
    Then the map $\alpha$ is roughly the abelianization map of $\brd[][n]$ extended to the group ring.
  }
  $\mathbb{Q}(t)$ of rational functions in $t$, and let $\alpha$ be the representation sending every meridian to $t^{-1}$.
  (We choose $t^{-1}$ instead of $t$ to match later conventions.)
  Then the groupoid $\brd[\gl(V)][n]$ preserves $\alpha$, so the group of endomorphisms of $\alpha$ is equivalent to the ordinary braid group $\brd[][n]$.
  By choosing an appropriate basis we get a Burau representation
  \[
    \burau_\alpha : \brd[][n] \to \gl_{n}(\mathbb{Q}(t))
  \]
  and a reduced Burau representation
  \[
    \burau^{\partial}_\alpha : \brd[][n] \to \gl_{n-1}(\mathbb{Q}(t))
  \]
  We call these the \defemph{abelian} Burau representations; usually \cite[Section 3.3]{Birman1974} the term ``Burau representation'' refers to them.

  The reduced abelian Burau representation sends the generator $\sigma_i$ braiding strand $i$ over strand $i-1$ to the matrix
  \[
    I_{i-2}
    \oplus
    \begin{bmatrix}
      1 & 0 & 0 \\
      1 & -t & t \\
      0 & 0 & 1
    \end{bmatrix}
    \oplus
    I_{n-i-2}
  \]
  with $1, -t , t$ appearing in the $i$th row.
  \note{
    For sufficiently large or small $i$ the identity summands are zero and we may have to truncate the $3\times 3$ matrix in the middle.
  }
  $V$ is one-dimensional, and we are using the basis $\{y_i, i =1, \dots, n-1\}$ of $\homol{D_n, \partial D_n; \alpha}$.

  To match the left-to-right composition of braids, we want to think of this matrix as acting on row vectors.
  By doing so, we get a representation
\[
  \burau^{\partial}(\beta_1 \beta_2) = \burau^\partial(\beta_1) \burau^\partial(\beta_2)
\]
instead of an anti-representation.
\end{ex}

This computation is a special case of the general formula:
\begin{prop}
  \label{thm:boundary-burau-matrix}
  Choose a basis $e_1, \dots, e_k$ of $V$, so that we can identify $\gl(V)$ with $\gl(\CC^k)$.
  With respect to the basis
  \[
    \{y_i \otimes e_j : i = 1, \dots n-1, j = 1, \dots, k\}
  \]
  of $\homol{D_n, \partial D_n; \rho}$,
  the matrices of the boundary-reduced twisted Burau representation are given on braid generators $\sigma_i : \rho_0 \to \rho_1$ by
  \begin{align*}
    \left[\burau^{\partial}(\sigma_i)\right]
    &=
    I_{(i-2)k} \oplus
    \begin{bmatrix}
      I_k & 0 & 0 \\
      I_k & -\rho_0(y_{i-1} y_{i}^{-1}) & \rho_0(y_{i-1} y_{i}^{-1}) \\
      0 & 0 & I_k
    \end{bmatrix}
    \oplus I_{(n-i-2)k}
    \\
    &=
    I_{(i-2)k} \oplus
    \begin{bmatrix}
      I_k & 0 & 0 \\
      I_k & -\rho_1(y_{i} y_{i+1}^{-1}) & \rho_1(y_{i} y_{i+1}^{-1}) \\
      0 & 0 & I_k
    \end{bmatrix}
    \oplus I_{(n-i-2)k}
  \end{align*}
  where the matrices act on row vectors from the right.
  \note{
    We have switched notation from $\rho \to \rho'$ to $\rho_0 \to \rho_1$ as to avoid conflict with the notation $\rho^\vee$ appearing below.
  }
\end{prop}
\begin{proof}
  This is a standard result, which can be computed by identifying the action of the braid group on the twisted chain groups with the action of the \defemph{Fox derivatives} on the free group $F_n = \pi_1(D_n)$.
  For more details, see \cite{Conway2017}, in particular \cite[Example 11.3.7]{Conway2017}.
  Our matrices differ slightly from those of \citeauthor{Conway2017} because we have picked a different convention for the action of $\brd[][n]$ on $F_n$.

  To see that $\rho_0(y_{i-1}y_i^{-1}) = \rho_1(y_i y_{i+1}^{-1})$, recall that by definition $\rho_1 = \rho_0 \beta^{-1}$.
\end{proof}

To match the braid action on the quantum group we want the dual of this representation.
The most convenient way to do this is to consider \defemph{locally-finite} or \defemph{Borel-Moore} homology $\homol{D_n; \rho}[\lf]$.

The untwisted form of this homology has a basis spanned by arcs between the punctures of $D_n$, and it is dual to $\homol{D_n, \partial D_n}$ via the obvious intersection pairing.
For example, \cref{fig:borel-moore-example} shows the basis $y_1, y_2$ of $\homol{D_3, \partial D_3; \CC}$ associated to the generators  $y_1, y_2 \in \pi_1(D_3)$ and the dual basis $z_1, z_2$ of $\homol{D_3; \CC}[\lf]$.

\begin{marginfigure}
  \centering
\begingroup%
  \makeatletter%
  \providecommand\color[2][]{%
    \errmessage{(Inkscape) Color is used for the text in Inkscape, but the package 'color.sty' is not loaded}%
    \renewcommand\color[2][]{}%
  }%
  \providecommand\transparent[1]{%
    \errmessage{(Inkscape) Transparency is used (non-zero) for the text in Inkscape, but the package 'transparent.sty' is not loaded}%
    \renewcommand\transparent[1]{}%
  }%
  \providecommand\rotatebox[2]{#2}%
  \newcommand*\fsize{\dimexpr\f@size pt\relax}%
  \newcommand*\lineheight[1]{\fontsize{\fsize}{#1\fsize}\selectfont}%
  \ifx\svgwidth\undefined%
    \setlength{\unitlength}{94.50118446bp}%
    \ifx\svgscale\undefined%
      \relax%
    \else%
      \setlength{\unitlength}{\unitlength * \real{\svgscale}}%
    \fi%
  \else%
    \setlength{\unitlength}{\svgwidth}%
  \fi%
  \global\let\svgwidth\undefined%
  \global\let\svgscale\undefined%
  \makeatother%
  \begin{picture}(1,1)%
    \lineheight{1}%
    \setlength\tabcolsep{0pt}%
    \put(0.26762058,0.64172161){\color[rgb]{0.18823529,0.64313725,1}\makebox(0,0)[lt]{\lineheight{1.25}\smash{\begin{tabular}[t]{l}$y_2$\end{tabular}}}}%
    \put(0.77433647,0.64302314){\color[rgb]{0.18823529,0.64313725,1}\makebox(0,0)[lt]{\lineheight{1.25}\smash{\begin{tabular}[t]{l}$y_1$\end{tabular}}}}%
    \put(0,0){\includegraphics[width=\unitlength,page=1]{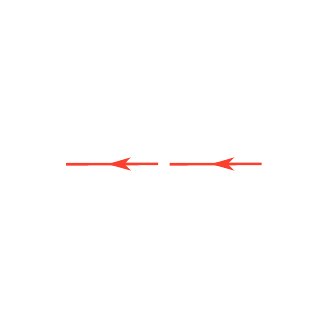}}%
    \put(0.28174872,0.40079493){\color[rgb]{0.98039216,0.24705882,0.17647059}\makebox(0,0)[lt]{\lineheight{1.25}\smash{\begin{tabular}[t]{l}$z_1$\end{tabular}}}}%
    \put(0.60714137,0.40079493){\color[rgb]{0.98039216,0.24705882,0.17647059}\makebox(0,0)[lt]{\lineheight{1.25}\smash{\begin{tabular}[t]{l}$z_2$\end{tabular}}}}%
    \put(0,0){\includegraphics[width=\unitlength,page=2]{borel-moore-example.pdf}}%
  \end{picture}%
\endgroup%

  \caption{A basis {\color{myred} $z_1$, $z_2$} of the locally-finite homology $\homol{D_3}[\lf]$ and the dual basis  {\color{myblue} $y_1, y_2$} of the homology rel boundary $\homol{D_3, \partial D_3}$.}
  \label{fig:borel-moore-example}
\end{marginfigure}

To extend this to the twisted case, we need to obtain a right $\ZZ[\pi]$-module dual to $V$.
The dual space $V^* \defeq \hom_\CC(V, \CC)$ is a right $\ZZ[\pi]$-module via
\[
  x \cdot f = v \mapsto f(v \rho(x^{-1}) )
\]
We write $(V^\vee, \rho^\vee)$ for this representation.

\begin{prop}
  There is a $\pi$-equivariant nondegenerate pairing
  \[
    \homol{D_n; \rho^{\vee}}[\lf] \otimes \homol{D_n, \partial D_n; \rho} \to \CC.
  \]
\end{prop}
\begin{proof}
  This is an easy extension of the result for untwisted homology, using the $\pi$-equivariant pairing between $V^\vee$ and $V$ given by
  \[
    x\cdot (f \otimes v) \mapsto f(\rho(x^{-1}) \rho(x) v) = f(v). \qedhere
  \]
\end{proof}

\begin{defn}
  The \defemph{reduced twisted Burau representation} (from here on the \defemph{Burau representation}) is the functor 
  sending a braid $\beta : \rho_0 \to \rho_1$ to the map
  \[
    \burau(\beta) : \homol{D_n; \rho_0^{\vee}}[\lf] \to \homol{D_n; \rho_1^{\vee}}[\lf]
  \]
\end{defn}

\begin{cor}
  \label{thm:burau-coords}
  Let $e^1, \dots, e^k$ be the basis of $V^*$ dual to the basis chosen in \cref{thm:boundary-burau-matrix}, and similarly let $z_1, \dots, z_{n-1}$ be the basis dual to $y_1, \dots, y_{n-1}$.
  Then with respect to the basis $\{z_i \otimes e^j : i = 1, \dots, n-1, j = 1, \dots k\}$ of $\homol{D_n; \rho}[\lf]$, the matrices of the Burau representation $\burau$ are given on braid generators by
  \[
    \left[\burau(\sigma_i)\right]
    =
    I_{(i-2)k} \oplus
    \begin{bmatrix}
      I_k & I_k & 0 \\
      0 & - \rho_1^{\vee}(y_i y_{i+1}^{-1}) & 0 \\
      0 & \rho_1^{\vee}(y_i y_{i+1}^{-1}) & I_k
    \end{bmatrix}
    \oplus I_{(n-i-2)k}
  \]
  where the matrices $\rho^{\vee}$ are the inverse transposes of those of $\rho$.
\end{cor}

\subsection{Burau representations for shaped braids}
To define the Burau representation, it is most natural to use braid groupoids colored by the conjugation quandle of a matrix group.
However, because $\brd[\slg]$ is equivalent to $\brdsh$ we can similarly define a Burau representation for the shaped braid groupoid by pulling back along the functor $\Psi$ of \cref{thm:links-presentable}.
By choosing appropriate bases we can considerably simplify this matrix.

\begin{thm}
    \label{thm:burau-coords-nice}
  Let $\sigma : \rho_0 \to \rho_1$ be a shaped braid generator, with target shapes 
  \[
    \rho_1 = (\chi_1, \dots, \chi_n) \text{ and } \chi_i = (a_i, b_i, \lambda_i).
  \]
  There exists a family of bases of the homology $\homol{D_n;\rho}[\lf]$ such that the matrix of $\burau(\sigma_i)$ is given by
  \begin{equation}
    \label{eq:burau-coords-nice}
    I_{2(i-2)} \oplus
    \begin{bmatrix}
      1 & 0 & a_i^{-1} & -\phi_i a_i^{-1} & &  &  \\
      0 & 1 & 0 & 1  \\
        &  & -a_i^{-1} & \phi_ia_i^{-1} &  &  \\
        &  & -\epsilon_{i+1} & -a_{i+1} &  &  \\
        &  & 1 & 0 & 1 & 0 \\
        &  &  \epsilon_{i+i} & a_{i+1} & 0 & 1
    \end{bmatrix}
    \oplus I_{2(n-i-2)}
  \end{equation}
  where $ \epsilon_i = (a_i - \lambda_i) b_i$ and $ \phi_i = (a_i - \lambda_i^{-1}) b_i^{-1}$.
\end{thm}
\begin{proof}
  For a groupoid representation $\mathcal{F} : \catl{G} \to \vect{\CC}$, choosing bases means choosing a basis of the vector space $\mathcal{F}(\rho)$ for each object $\rho$ of $\catl G$, so that we obtain matrices $[\mathcal{F}(g)]$ for each morphism $g: \rho_0 \to \rho_1$ of $C$.
  Changing the bases transforms the matrix of $g$ as
  \[
    [\mathcal{F}(g)]
    \mapsto
    Q_{\rho_0} [\mathcal{F}(g)] Q_{\rho_1}^{-1}.
  \]
  Because our matrices are acting on row vectors the domain $\rho_0$ goes on the left.
  The theorem follows from making the right choice of $Q_\rho$.

  We need some notation: we write $s_i^{\pm}$ and $t_i^{\pm}$ for the transposes of the holonomy of the components of $\rho_0$ and $\rho_1$.
  Specifically, we have
  \[
    t_i^+
    =
    \begin{bmatrix}
      a_i & \phi_i \\
      0 & 1
    \end{bmatrix}
    \quad
    t_i^-
    =
    \begin{bmatrix}
      1 & 0 \\
      \epsilon_i & a_i
    \end{bmatrix}
  \]
  where $a_i$, $\epsilon_i$, and $\phi_i$ are determined by the components $\chi_i$ of $\rho_1$.
  We similarly write $s_i^{\pm}$ for the same matrices, but determined by the component shapes $\tilde \chi_i$ of $\rho_0$.
  Setting
  \begin{align*}
    p_j(\rho_1)
    &\defeq
    t_j^+ \cdots t_1^+
    &
    m_j(\rho_1)
    &\defeq
    t_j^- \cdots t_1^-
    \\
    p_j(\rho_0)
    &\defeq
    s_j^+ \cdots s_1^+
    &
    m_j(\rho_0)
    &\defeq
    s_j^- \cdots s_1^-
  \end{align*}
  we have
  \begin{align*}
    \rho_1^\vee(y_j) &= p_j(\rho_1)^{-1} m_j(\rho_1)
    \\
    \rho_0^\vee(y_j) &= p_j(\rho_0)^{-1} m_j(\rho_0)
  \end{align*}
  and in particular
  \begin{align*}
    \rho_1^\vee(y_{i+1} y_i^{-1}) &= p_{i+1}(\rho_1)^{-1} t^{-}_{i+1} p_{i}(\rho_1)
    \\
    \rho_0^\vee(y_{i} y_{i-1}^{-1}) &= p_{i}(\rho_0)^{-1} s^{-}_{i} p_{i-1}(\rho_0)
  \end{align*}
  so the non-identity block of the matrix of \cref{thm:burau-coords} is
  \[
    \begin{bmatrix}
      I_2 & I_2 & 0 \\
      0 & - \rho_1^{\vee}(y_i y_{i+1}^{-1}) & 0 \\
      0 & \rho_1^{\vee}(y_i y_{i+1}^{-1}) & I_2
    \end{bmatrix}
    =
    \begin{bmatrix}
      I_2 & I_2 & 0 \\
      0 & - p_{i+1}(\rho_1)^{-1} t^{-}_{i+1} p_{i}(\rho_1) & 0 \\
      0 & p_{i}(\rho_0)^{-1} s^{-}_{i} p_{i-1}(\rho_0) & I_2
    \end{bmatrix}.
  \]

  The right change-of-basis matrices are
  \[
    Q_{\rho} =
    \begin{bmatrix}
      p_1(\rho) \\
      & \ddots \\
      & & p_1(\rho)
    \end{bmatrix}
  \]
  Because $p_j(\rho_0) = p_j(\rho_1)$ for all $j \ne i$, we see the identity blocks of the matrix of \cref{thm:burau-coords} are unchanged, while the nontrivial block becomes
  \begin{gather*}
    \begin{bmatrix}
      p_{i-1}(\rho_0) & 0 & 0 \\
      0 & p_i(\rho_0) & 0 \\
      0 & 0 & p_{i+1}(\rho_0)
    \end{bmatrix}
    \begin{bmatrix}
      I_2 & I_2 & 0 \\
      0 & -p_{i}(\rho_0)^{-1} s_i^- p_{i-1}(\rho_0) & 0 \\
      0 & p_{i+1}(\rho_1)^{-1} t_{i+1}^- p_i(\rho_1) & I_2
    \end{bmatrix}
    \begin{bmatrix}
      p_{i-1}(\rho_1)^{-1} & 0 & 0 \\
      0 & p_i(\rho_1)^{-1} & 0 \\
      0 & 0 & p_{i+1}(\rho_1)^{-1}
    \end{bmatrix}
    \\
    =
    \begin{bmatrix}
      p_{i-1}(\rho_0) & p_{i-1}(\rho_0) & 0 \\
      0 & -s_i^- p_{i-1}(\rho_0) & 0 \\
      0 & p_{i+1}(\rho_0) p_{i+1}(\rho_1)^{-1} t_{i+1}^- p_i(\rho_1) & p_{i+1}(\rho_0)
    \end{bmatrix}
    \begin{bmatrix}
      p_{i-1}(\rho_1)^{-1} & 0 & 0 \\
      0 & p_i(\rho_1)^{-1} & 0 \\
      0 & 0 & p_{i+1}(\rho_1)^{-1}
    \end{bmatrix}
    \\
    =
    \begin{bmatrix}
      p_{i-1}(\rho_0) p_{i-1}(\rho_1)^{-1} & p_{i-1}(\rho_0) p_i(\rho_1)^{-1} & 0 \\
      0 & - s_i^- p_{i-1}(\rho_0) p_i(\rho_1)^{-1} & 0 \\
      0 & p_{i+1}(\rho_0) p_{i+1}(\rho_1)^{-1} b_{i+1}^- p_i(\rho_1) p_i(\rho_1)^{-1} & p_{i+1}(\rho_0)p_{i+1}(\rho_1)^{-1}
    \end{bmatrix}
    \\
    =
    \begin{bmatrix}
      p_{i-1}(\rho_1) p_{i-1}(\rho_1)^{-1} & p_{i-1}(\rho_1) p_{i-1}(\rho_1)^{-1} (t_i^+)^{-1} & 0 \\
      0 & - s_i^- p_{i-1}(\rho_0) p_{i-1}(\rho_0)^{-1}(t_i^+)^{-1} & 0 \\
      0 & p_{i+1}(\rho_1) p_{i+1}(\rho_1)^{-1} t_{i+1}^- p_i(\rho_1) p_i(\rho_1)^{-1} & p_{i+1}(\rho_1)p_{i+1}(\rho_1)^{-1}
    \end{bmatrix}
    \\
    =
    \begin{bmatrix}
      I_2 & (t_i^+)^{-1} & 0 \\
      0 & - s_i^- (t_i^+)^{-1} & 0 \\
      0 & t_{i+1}^- & I_2
    \end{bmatrix}
  \end{gather*}
  Again the cancellations follow from the fact that $p_j(a) = p_j(b)$ for all $j \ne i$.

  We have immediately that
  \[
    (t_i^+)^{-1}
    =
    \begin{bmatrix}
      a_i^{-1} & - \phi_i a_i^{-1} \\
      0 & 1
    \end{bmatrix}
    \text{ and }
    t_{i+1}^-
    =
    \begin{bmatrix}
      1 & 0 \\
      \epsilon_{i+1} & a_{i+1}
    \end{bmatrix}
  \]
  so it remains only to check that $-s_i^-(t_i^+)^{-1}$ gives the correct $2 \times 2$ matrix.
  Writing
  \[
    s_i^-
    =
    \begin{bmatrix}
      1 & 0 \\
      \tilde \epsilon_{i} & \tilde a_i
    \end{bmatrix}
    \text{ and }
    (t_i^+)^{-1}
    =
    \begin{bmatrix}
      a_i^{-1} & - \phi_i a_i^{-1} \\
      0 & 1
    \end{bmatrix}
  \]
  we have
  \[
    s_i^-(t_i^+)^{-1}
    =
    \begin{bmatrix}
      a_i^{-1} & - \phi_i a_{i}^{-1} \\
      \tilde \epsilon_i a_i^{-1} & \tilde a_i - \tilde \epsilon_i \phi_i a_{i}^{-1}
    \end{bmatrix}
  \]
  To simplify the bottom row, we need the identities
  \begin{align*}
    \tilde \epsilon_i &= a_i \epsilon_i \\
    \tilde \kappa_i &= \kappa_{i+1} + \phi_i \epsilon_{i+1} \\
    \intertext{equivalently, the identities}
    \Smat(E^2 \otimes 1) &= K^{2} \otimes E^{2} \\
    \Smat(K^{2} \otimes 1) &= 1 \otimes K^{2} + K^{2} F^{2} \otimes E^{2}
  \end{align*}
  which can be derived from \cref{thm:outer-R-mat-def}.
  Then we see that
  \[
    s_i^-(t_i^+)^{-1}
    =
    \begin{bmatrix}
      a_i^{-1} & - \phi_i a_{i}^{-1} \\
      \tilde \epsilon_i a_i^{-1} & \tilde a_i - \tilde \epsilon_i \phi_i a_{i}^{-1}
    \end{bmatrix}
    =
    \begin{bmatrix}
      a_i^{-1} & - \phi_i a_{i}^{-1} \\
      \epsilon_{i+1} & \tilde a_i - \epsilon_{i+1} \phi_i
    \end{bmatrix}
    =
    \begin{bmatrix}
      a_i^{-1} & - \phi_i a_{i}^{-1} \\
      \epsilon_{i+1} & a_{i+1}
    \end{bmatrix}
  \]
  as claimed.
\end{proof}

\subsection{Torsions}
\label{sec:torsion-def}
The \defemph{Reidemeister torsion} is a well-known invariant of links defined using the twisted Burau represenations.
It is closely related to the Alexander polynomial.
The classical untwisted/abelian case is discussed in \cite{Turaev2001}, while twisted torsions (in the form of the closely related twisted Alexander polynomials) are discussed in \cite{Conway2015,Conway2017,Friedl2009}.

We sketch the definition.
Let $L$ be a $\slg$-link (with representation $\rho : \pi_1(S^3 \setminus L) \to \slg$) obtained as the closure of a $\slg$-colored braid $\beta$.
In general, the complex $C_*(S^3 \setminus L; \rho)$ is acyclic: all the twisted homology groups vanish.
However, we can still obtain a nontrivial invariant via the torsion.

Acyclicity is equivalent to exactness of the sequence
\[
  \cdots C_i \xrightarrow{\partial_i} C_{i-1} \cdots
\]
in which case we get isomorphisms $\ker \partial_i = \im \partial_{i-1}$.
If we choose a basis of each $C_i$, we can use the above isomorphisms to change these bases.
The alternating product of determinants of the basis-change matrices gives an invariant of the acyclic complex $C_*$.
In general this torsion can depend on the choice of basis for each chain space, but for link complements it does not.

Given a presentation of $L$ as the closure of a braid $\beta$ we get a presentation of $\pi_L = \left\langle y_1, \dots, y_n | y_i = y_i \cdot \beta\right\rangle$, which in turn gives a CW structure on $S^3 \setminus L$; the $2$-cells are obtained by the relations $y_i = y_i \cdot \beta$.
Link complements are aspherical, so we do not need to add any higher-dimensional cells.

\begin{defn}
  Let $\rho : \pi_1(S^3 \setminus L) \to \gl(V)$ be a representation such that the $\rho$-twisted chain complex $C^*(S^3 \setminus L; \rho)$ is acyclic, in which case we say the $\gl(V)$-link $L$ is \defemph{acyclic.}
  Then the \defemph{$\rho$-twisted torsion} $\tau(L, \rho)$  is the torsion of the $\rho$-twisted homology $C^*(S^3 \setminus L; \rho)$.
  Usually when $\rho$ has abelian image this is called the \defemph{Reidemeister torsion}.
  When the image of $\rho$ is nonabelian it is called the \defemph{twisted torsion.}
  We prefer to instead refer to these cases as abelian and nonabelian torsions.
\end{defn}
\begin{thm}
  \label{thm:torsion-computation}
  Let $(L, \rho)$ be a $\slg$-link, and let $\beta$ be a braid whose closure is $L$.
  View $\beta : (g_1, \dots, g_n) \to (g_1, \dots g_n)$ as a morphism of the colored braid groupoid $\brd[\slg]$, and suppose that
  \[
    \det(1 - g_n \cdots g_1) \ne 0,
  \]
  that is, that the holonomy $g_n \cdots g_1$ of the path $y_n$ around all the punctures of $D_n$ does not have $1$ as an eigenvalue.
  Then
  \begin{enumerate}
    \item The twisted homology $\homol[*]{S^3 \setminus L, \rho}$ is acyclic, so the torsion $\tau(L, \rho)$ is a complex number defined up to $\pm \det \rho = \pm 1$, 
    \item we can compute the torsion as
      \[
        \tau(L, \rho)
        = \frac{\det(1 - \burau^{\partial}(\beta))}{\det(1 - g_n \cdots g_1)}
        = \frac{\det(1 - \burau(\beta))}{\det(1 - (g_n \cdots g_1)^{-1})},
      \]
    \item and if $(L, \rho)$ is an $\slg$-link such that $\det(1 - \rho(w_i)) \ne 0$ for every meridian $w_i$, then such a braid $\beta$ always exists.
  \end{enumerate}
\end{thm}
\begin{proof}
  (1) and (2) are standard results in the theory of torsions.
The idea is that we use the basis corresponding to $y_1, \dots, y_{n-1}$ for $\homol{D_n; \rho}$ and to $y_n$ for $\homol[0]{D_n; \rho}$, and these bases give nondegenerate matrix $\tau$-chains \cite[Section 2.1]{Turaev2001} for the complex, so they compute the torsion.
  More details can be found in \cite[Theorem 3.15]{Conway2015}; that paper discusses twisted Alexander polynomials, which correspond with the torsion when the variables $t_i$ are all $1$.
  Finally, we can use $\burau$ and locally-finite homology instead of $\burau^\partial$ and ordinary homology to compute the torsion because these are dual.

  The only novel claim (to our knowlege) is (3).
  The proof is due to Bertrand Patureau-Mirand (via private communication).
  Represent $(L, \rho)$ as the closure of a $\slg$-braid $\beta$ on $n$ strands which is an endomorphism of the color tuple $(g_1, \dots, g_n)$, and write $h_n  = g_n \cdots g_1$ for the total holonomy.
  Consider the colored braids
  \begin{align*}
    \beta &: (g_1, \cdots, g_n) \to (g_1, \cdots, g_n)\\
    \beta\sigma_n &: (g_1, \cdots, g_n, g_n ) \to (g_1, \cdots, g_n, g_n)\\
    \beta\sigma_n \sigma_{n+1} &: (g_1, \cdots, g_n, g_n, g_n) \to (g_1, \cdots, g_n, g_n, g_n)
  \end{align*}
  Their closures are all $(L, \rho)$, and they have total holonomies
  \[
    h_n, \ \ g_n h_n, \ \ g_n^2 h_n 
  \]
  respectively.
  Because these matrices all lie in $\slg$, we have
  \[
    \tr(g_n^2 h_n) + \tr(h_n) = \tr(g_n)\tr(g_n h_n)
  \]
  Recall that an element $g \in \slg$ has $1$ as an eigenvalue if and only if $\tr g = 2$.
  Since $\tr g_n \ne 2$, at least one of $\tr(g_n^2 h_n),$ $\tr(g_n h_n),$ or $\tr h_n$ has trace not equal to $2$.
  We conclude that at least one braid with closure $(L, \rho)$ has nontrivial total holonomy.
\end{proof}
Taking the closure of a braid relates the complex $C^*(D_n; \rho)$ to $C^*(S^3 \setminus L; \rho)$ by adding a term in dimension $2$, so it is reasonable to expect a relationship between the torsion and the Burau representation.
As a special case, we have a formula for shaped braids.
\begin{cor}
  \label{thm:torsion-computation-cor}
  Let $L$ be an admissible $\slg$-link (that is, one without $1$ as the eigenvalue of a meridian) expressed as the closure of a shaped braid $\beta : (\chi_1, \dots, \chi_n) \to (\chi_1, \dots, \chi_n)$, and let
  \[
    h_n = g^+(\chi_1) \cdots g^+(\chi_n) g^-(\chi_n) \cdots g^-(\chi_1)
  \]
  be the holonomy around all the strands of the braid.
  Let $\lambda^{\pm 1}$ be the eigenvalues of $h_n$.

  Every $\slg$-link can conjugated so that it admits presentation as the closure of a shaped braid with $\lambda \ne \pm 1$, in which case the torsion of $L$ can be computed as
  \[
    \frac{\det(1 - \burau(\beta))}{2 - \lambda + \lambda^{-1}}
  \]
  where $\burau$ is given by \cref{eq:burau-coords-nice}.
\end{cor}
\begin{proof}
  By \cref{thm:torsion-computation}, whenever $L$ is admissible it can be expressed as the closure of a $\slg$-colored braid such that the total holonomy does not have $1$ as an eigenvalue.
  Such a braid may not be expressible in terms of shapes, but we can always conjugate it so that it is.

  The denominator can be computed as
  \[
    \det\left( 1- (g_n \cdots g_1)^{-1} \right)
    =
    \det\left( 1- 
      \begin{bmatrix}
        \lambda^{-1} & 0 \\
        0 & \lambda
      \end{bmatrix}
    \right)
    =
    2 - \lambda - \lambda^{-1}.
    \qedhere
  \]
\end{proof}

\section{Schur-Weyl duality}
\label{sec:schur-weyl}
In this section we prove \cref{thm:schur-weyl}, which gives a Schur-Weyl duality between the (reduced twisted) Burau representation $\burau$ and the algebra $\qgrp[i] = \qgrplong{i}$.
To apply this duality to the proof of \cref{thm:T-is-torsion} we need to also give a mirrored version, \cref{thm:schur-weyl-mirror};.
We apply these theorems in the next section.

\subsection{Motivation and statement}
\label{sec:schur-weyl-motiv}
Before stating the theorem, we explain what we mean by ``Schur-Weyl duality.''
Consider a Hopf algebra $H$ and a simple $H$-module~$V$ with structure map $\pi : H \to \End_\CC(V)$.
The algebra $H$ acts on $V^{\otimes n}$ via the map $\pi^{\otimes n } \circ \Delta^n : H \to H^{\otimes n} \to \End_\CC(V^{\otimes n})$.

We want to understand the decomposition of the tensor product module $V^{\otimes n}$ into simple factors.
One way is to find a subalgebra $B \subseteq H^{\otimes n}$ that commutes with $\Delta^n(H)$, the image of $H$ under the iterated coproduct.
If $B$ is large enough, then we can use the double centralizer theorem to understand the decomposition of $V^{\otimes n}$.
In this section, we address this problem in the case $H = \qgrp[i]$, with a few modifications.

To get a satisfactory answer, we want think of $\qgrp[i]$ as a superalgebra and find a subalgebra $\clifford n$ (which turns out to be a Clifford algebra generated by a space $\opspace n$) that supercommutes with $\Delta^n (\qgrp[i])$.
In addition, to match the shaped braid groupoid and its Burau representation, we consider tensor products of the form
\[
  \qgrp[i]/\ker \chi_1
  \otimes \cdots \otimes
  \qgrp[i]/\ker \chi_n
\]
where the functions $\chi_i : \cent_0 \to \CC$ are the $\cent_0$-characters induced by the shapes.
Since the Burau representation is a braid group representation, we also describe the braiding on $\qgrp[i]$ and its action on our subalgebra.

Recall the outer $S$-matrix $\Smat = \tau \Rmat$ of $\qgrp$; it is the algebra automorphism that the braidings intertwine.
We can think of $\Smat$ as giving a braiding directly on the algebra $\qgrp$ in the following sense.
Let $\chi_1, \chi_2$ be shapes, which we can think of as characters $\chi : \cent_0 \to \CC$.
Then by assigning  $\chi$ to $\qgrp/\ker \chi$ and a braid generator $\sigma : (\chi_1, \chi_2) \to (\chi_{2'}, \chi_{1'})$ to
\[
  \Smat : (\qgrp \otimes \qgrp)/\ker(\chi_1 \otimes \chi_2) \to (\qgrp \otimes \qgrp)/\ker (\chi_{2'} \otimes \chi_{1'})
\]
we get a model%
\note{
  Strictly speaking we would need to make the target category pivotal for this to match our earlier definition of the model of a biquandle.
}
of the shape biquandle in the category of algebras and algebra homomorphisms.
We are interested in the corresponding functor for braids.
\begin{defn}
  \label{def:alg-functor}
  The outer $S$-matrix $\Smat$ gives a functor from the shaped braid groupoid to the category of algebras and algebra homomorphisms.
  We also denote this functor by $\Smat$,
  \note{
    In \cite{McPhailSnyder2020} we wrote $\mathcal{A}$ (for ``algebra'') for  the functor $\Smat$ and $\check{\Rmat}$ for the automorphism $\Smat$.
  }
    and it is defined on
  \begin{description}
    \item[objects] by $\Smat(\chi_1, \dots, \chi_n) \defeq \bigotimes_{i=1}^{n} \qgrp/\ker \chi_i = \qgrp^{\otimes n} / \ker(\chi_1 \otimes \cdots \chi_n)$
    \item[braid generators] $\sigma_i$ by acting by $\Smat$ on the corresponding tensor factors:
      \[
        \Smat(\sigma_i) \defeq \Smat_{i, i+1}
      \]
  \end{description}
  The functor $\Smat$ extends to all braids in the obvious way.
\end{defn}

In the future, we hope to find a Schur-Weyl duality between $\Smat$ and an appropriate Burau representation for $\nr > 2$.
For now, we limit ourselves to the case $\nr = 2$, $\xi = i$.

\begin{thm}[Schur-Weyl duality for the Burau representation]
  \label{thm:schur-weyl}
  Let 
  \[
    \beta : (\chi_1, \dots, \chi_n) \to (\chi_1', \dots, \chi_n')
  \]
  be an admissible shaped braid on $n$ strands, and let $\rho : \pi_1(D_n) \to \slg$ be the representation given by the holonomy of the shapes $(\chi_1, \dots, \chi_n)$, and similarly for $\rho'$ and $(\chi_1', \dots, \chi_n')$.
  Then for each $n \ge 2$ :
  \begin{enumerate}
    \item There exists a subspace
      \note{
        We need to invert the Casimir $\Omega$ for the computations to work out.
        The condition that our braid be admissible (i.e.\@ not have $1$ as the eigenvalue of a meridian) is exactly  the same as $\chi_i(\Omega)$ acting invertibly for all $i$.
      }
      $\opspace n \subseteq \qgrp[i][\Omega^{-1}]^{\otimes n}$ and a family of injective linear maps $\phi_\rho$ such that the diagram commutes:
      \[
        \begin{tikzcd}
          \homol{D_n; \rho}[\lf] \arrow[d, "\phi_\rho"] \arrow[r, "\burau(\beta)"] & \homol{D_n; \rho'}[\lf] \arrow[d, "\phi_{\rho'}"]\\
          \opspace n / \ker(\chi_1 \otimes \cdots \otimes \chi_n) \arrow[r, "\Smat(\beta)"] & \opspace n / \ker(\chi_1 \otimes \cdots \otimes \chi_n) 
        \end{tikzcd}
      \]
    \item The subspace $\opspace n$ generates a Clifford algebra $\clifford n$ inside $\qgrp[i]^{\otimes n}$ which supercommutes with $\Delta(\qgrp[i])$, the image of $\qgrp[i]$ in $\qgrp[i]^{\otimes n}$ under the coproduct.
  \end{enumerate}
\end{thm}
The proof is a computation that is entirely striaghtforward once $\opspace n$ is given and we have the version of the Burau representation given in \cref{eq:burau-coords-nice}.
First, we need to discuss superalgebras.

\subsection{Superalgebras}
\begin{defn}
  A \defemph{superalgebra} is a $\ZZ/2$-graded algebra.
  We call the degree $0$ and $1$ the even and odd parts, respectively, and write $|x|$ for the degree of $x$.
  We say that $x$ and $y$ \defemph{supercommute} if
  \[
    xy - (-1)^{|x| |y|} yx = 0.
  \]
\end{defn}

\begin{ex}
  Let $V$ be a module over a commutative ring $Z$ and $\eta$ a symmetric $Z$-valued bilinear form on $V$.
  The \defemph{Clifford algebra} generated by $V$ is the quotient of the tensor algebra on $V$ by the relations
  \[
    vw + wv = \eta(v, w)
  \]
  for $v,w \in V$.
  By considering the image of $V$ to be odd the Clifford algebra becomes a superalgebra.
\end{ex}

Write $\F = i K F$.
We can regard $\qgrp[i]$ as the algebra generated by $K^{\pm 1}, E, \F$ with relations
\[
  \{E, K\} = \{\F,  K \} = 0, \quad [E, \F] = -2 i K \Omega,
\]
where $\{A, B\} \defeq AB + BA$ is the anticommutator and
\[
  \Omega = FE + i K - i K^{-1}.
\]
Using these generators, we see that $\qgrp[i]$ is a superalgebra with grading
\[
  |E| = |\F| = 0, \quad |K| = |\Omega| = 1.
\]
The choice that $E$ and $iKF$ (instead of $iKE$ and $F$) are even is for compatibility with the outer $S$-matrix $\Smat$.
More generally, our choice of grading here is motivated by Schur-Weyl duality, is rather ad hoc, and seems very special to the case $\nr = 2$, $q = i$.
At $q$ a $4m$th root of unity we expect a $\ZZ/m$-grading instead.
\note{
  Generalizing this relationship to $\nr \ge 2$ will probably involve more use of the Weyl algebra, especially once we no longer have $\xi^2 = \xi^{-2}$.
}

\subsection{Proof of theorem \ref{thm:schur-weyl}}

\begin{defn}
  For $j = 1, \dots, n$, consider the elements%
  \note{
    Here the upper index ranges over the space $\CC^2$ where $\slg$ acts, and the lower index ranges over the homology.
    Below we will see that $\beta_j^\nu$  corresponds to the basis element $v_j^\nu$  of \cref{thm:burau-coords-nice}.
  }
  \begin{align*}
    \alpha_j^1 & \defeq K_1 \cdots K_{j-1} E_j \Omega_j^{-1}  \\
    \alpha_j^2 & \defeq K_1 \cdots K_{j-1} \F_j \Omega_j^{-1}
  \end{align*}
  of $\qgrp[i][\Omega^{-1}]$, where $\F = iKF$, and set
  \[
    \beta_j^\nu = \alpha_{j}^\nu - \alpha_{j+1}^{\nu}.
  \]
  We write $\opspace n \subseteq (\qgrp[i][\Omega^{-1}])^{\otimes n}$ for the $(\cent_0[\Omega^{-2}])^{\otimes n}$-span of the $\beta_j^\nu$.
  Similarly, we write $\clifford n$ for the subalgebra generated by $\opspace n$.
\end{defn}

\begin{lem}
  \label{thm:Cn-is-clifford}
  $\clifford n$ is a Clifford algebra over the ring $(\cent_0[\Omega^{-2}])^{\otimes n}$.
\end{lem}
\begin{proof}
  The $\alpha_j^\nu$ satisfy anticommutation relations
  \begin{align*}
    \{\alpha_j^1, \alpha_k^1\} & = 2 \delta_{jk} K_1^2 \cdots K_{j-1}^2 E_j^2 \Omega_j^{-2} \\
    \{\alpha_j^2, \alpha_k^2\} & = 2 \delta_{jk} K_1^2 \cdots K_{j-1}^2 K_j^2 F_j^2 \Omega_j^{-2} \\
    \{\alpha_j^1, \alpha_k^2\} & = 2i \delta_{jk} K_1^2 \cdots K_{j-1}^2 (1 - K_j^{-2}) \Omega_j^{-2}.
  \end{align*}
  In particular, their anticommutators lie in $(\cent_0[\Omega^{-2}])^{\otimes n}$, so the same holds for anticommutators of elements of $\opspace n$.
\end{proof}

\begin{lem}
  \label{thm:raising-op-computation}
  The braiding automorphism $\Smat$ acts by
  \begin{align*}
    \Smat(\alpha_1^1) &= \alpha_2^1 \\
    \Smat(\alpha_2^1) &= (K_2)^2 \alpha_1^1 + (1 - (K_2)^2) \alpha_2^1 - (E_2)^2 (\alpha_1^2 - \alpha_2^2) \\
    \Smat(\alpha_1^2) &= (1- (K_1)^{-2}) \alpha_2^1 + (K_1)^{-2} \alpha_2^2 + (F_1)^2 (\alpha_1^1 - \alpha_2^1) \\
    \Smat(\alpha_2^2) &= \alpha_2^1
  \end{align*}
  so that the matrix of $\Smat_{i, i+1}$ acting on $\opspace n$ is given by
  \begin{equation}
    \label{eq:quantum-burau-action}
    I_{2(i-2)} \oplus
    \begin{bmatrix}
      1 & 0 & K_i^{-2} & - F_i^2 \\
      0 & 1 & 0 & 1 \\
        & & - K_{i}^{-2} & F_i^2 \\
        & & - E_{i+1}^2 & - K_{i+1}^{2} \\
        & & 1 & 0 & 1 & 0 \\
        & & E_{i+1}^2  & K_{i+1}^2  & 0 & 1
    \end{bmatrix}
    \oplus I_{2(n- 1) - 2(i + 1)}
  \end{equation}
  with the matrix action given by right multiplication on row vectors with respect to the basis $\{\beta_1^2, \beta_1^1, \cdots, \beta_{n-1}^2, \beta_{n-1}^1\}$ of $\opspace n$.
\end{lem}
\begin{proof}
  This is straightforward to verify.
\end{proof}

We can now explicitly give the embeddings $\phi_{\rho}$ and prove the duality.
\begin{defn}
  Recall the basis $v_j^\nu$ of  $\homol{D_n;\rho}[\lf]$ constructed in \cref{thm:burau-coords-nice}.
  For each nonsingular object%
  \note{
    A nonsingular object is just a tuple of shapes whose total holonomy does not have $1$ as an eigenvalue.
  }
  $\rho = (\chi_1, \dots, \chi_n)$ of $\brdsh$, define a linear map
   \[
    \phi_\rho :
    \left\{
    \begin{aligned}
      \homol{D_n;\rho}[\lf]
      &\to
      \opspace n/\ker{\chi_1 \otimes \cdots \otimes \chi_n}\\
      v_j^\nu &\mapsto \beta_{j}^{\nu}
    \end{aligned}
  \right.
  \]
\end{defn}

\begin{proof}[Proof of \cref{thm:schur-weyl}]
  The proof of the first part (commutativity of the diagram) is essentially done: the last ingredient is the observation that the image of the matrix in \cref{eq:quantum-burau-action} under $(\chi_1', \dots, \chi_n')$ is exactly the matrix in \cref{eq:burau-coords-nice}.

  It remains to prove the second statement about supercommutativity.
  We showed in \cref{thm:Cn-is-clifford} that the image of $\opspace n$ generates a Clifford algebra.
  We therefore think of the elements $\beta_j^\nu$ as being odd, so to check that they supercommute we must show that
  \begin{align*}
    \{\Delta K, \beta^\nu_k\} &= 0
                              &
    [\Delta E, \beta^\nu_k] &= 0
    \\
    \{\Delta \Omega, \beta^\nu_k\}  &=  0 
                                    &
    [\Delta \F, \beta^\nu_k] &= 0
  \end{align*}
  where $\{A, B\} \defeq AB + BA$ and  $[A,B] \defeq AB - BA$.
  To check this, we can use the anticommutation relations
  \begin{align*}
    \{\alpha_j^1, \alpha_k^1\} & = -2 \delta_{jk} K_1^2 \cdots K_{j-1}^2 E_j^2 \Omega_j^{-2} \\
    \{\alpha_j^2, \alpha_k^2\} & = -2 \delta_{jk} K_1^2 \cdots K_{j-1}^2 K_j^2 F_j^2 \Omega_j^{-2} \\
    \{\alpha_j^1, \alpha_k^2\} & = -2i \delta_{jk} K_1^2 \cdots K_{j-1}^2 (1 - K_j^{-2}) \Omega_j^{-2},
  \end{align*}
  the fact that $\beta_j^\nu \defeq \alpha_j^{\nu} - \alpha_{j+1}^{\nu}$, and the identity
  \[
    \Omega = iK - iK^{-1}(1 + E \F). \qedhere
  \]
\end{proof}

\subsection{Mirrored Schur-Weyl duality}
Recall that the functor $\doubfunc_2$ is construced by using $\qgrp[i] \boxtimes \qgrp[i]^{\cop}$-modules and a braiding intertwining $\Smat \boxtimes \overline{\Smat}$.
We gave Schur-Weyl duality for $\Smat$, so now we give a version for $\overline{\Smat}$.

\begin{defn}
  For $j = 1, \dots, n$, set  
  \begin{align*}
    \overline{\alpha}_j^1 & \defeq E_j \Omega_j^{-1} K_{j+1} \cdots K_n  \\
    \overline{\alpha}_j^2 & \defeq \F_j \Omega_j^{-1} K_{j+1} \cdots K_n \\
    \intertext{and}
    \overline{\beta}_j^\nu & \defeq \alpha_j^{\nu} - \alpha_{j+1}^{\nu}.
  \end{align*}
  We write $\opspacebar n$ for the $(\cent_0[\Omega^{-2}])$-span of the $\beta_j^\nu$ and $\cliffordbar n$ for the algebra generated by $\opspacebar n$.
\end{defn}
\begin{lem}
  \label{thm:supercommute-bar}
  $\cliffordbar n$ is the Clifford algebra on $\opspacebar n$ and it supercommutes with the image of $\qgrp[i]$ under the opposite coproduct:
  \begin{align*}
    \{\Delta^{\op} K, \overline{\beta}^\nu_k\} &= 0 &
    [\Delta^{\op} E, \overline{ \beta }^\nu_k] &= 0 \\
    [\Delta^{\op} \F, \overline{ \beta }^\nu_k] &= 0 &
    \{\Delta^{\op} \Omega, \overline{ \beta }^\nu_k\}  &=  0 
  \end{align*}
\end{lem}
\begin{proof}
  This follows from the relations
  \begin{align*}
    \{\overline{ \alpha }_j^1, \overline{ \alpha }_k^1\} & = -2 \delta_{jk} E_j^2 \Omega_j^{-2} K_{j+1}^2 \cdots K_{n}^2 \\
    \{\overline{ \alpha }_j^2, \overline{ \alpha }_k^2\} & = -2 \delta_{jk} K_j^2 F_j^2 \Omega_j^{-2} K_{j+1}^2 \cdots K_{n}^2\\
    \{\overline{ \alpha }_j^1, \overline{ \alpha }_k^2\} & = -2i \delta_{jk} (1 - K_j^{-2}) \Omega_j^{-2} K_{j+1}^2 \cdots K_{n}^2 \qedhere
  \end{align*}
  and then from the same argument as in the proof of \cref{thm:schur-weyl}.
\end{proof}

\begin{lem}
  The braiding automorphism $\overline{\Smat}$ for $\qgrp[i]^{\cop}$ acts by
  \begin{align*}
    \overline{\Smat}(\overline{\alpha}_1^1) &= \overline{\alpha}_2^1 \\
    \overline{\Smat}(\overline{\alpha}_2^1) &= (K_2)^{-2} \overline{\alpha}_1^1 + (1 - (K_2)^{-2}) \overline{\alpha}_2^1 + (K_2)^{-2} (E_2)^2 (\overline{\alpha}_1^2 - \overline{\alpha}_2^2) \\
    \overline{\Smat}(\overline{\alpha}_1^2) &= ((K_1)^2 - 1) \overline{\alpha}_1^1 + (K_1)^2 \overline{\alpha}_2^2 + (K_1)^2 (F_1)^2 (\overline{\alpha}_1^1 - \overline{\alpha}_2^1) \\
    \overline{\Smat}(\overline{\alpha}_2^2) &= \overline{\alpha}_1^2
  \end{align*}
  so the matrix of $\overline{\Smat}_{i,i+1}$ acting on $\opspacebar n$ is
  \begin{equation}
    \label{eq:quantum-burau-action-opinv}
    I_{2(i-2)} \oplus
    \begin{bmatrix}
      1 & 0 & K_i^{2} & - K_i^2 F_i^2 \\
      0 & 1 & 0 & 1 \\
        & & - K_{i}^{2} & K_i^2 F_i^2 \\
        & & - K_{i+1}^{-2} E_{i+1}^2 & - K_{i+1}^{-2} \\
        & & 1 & 0 & 1 & 0 \\
        & & K_{i+1}^{-2} E_{i+1}^2  & K_{i+1}^{-2}  & 0 & 1
    \end{bmatrix}
    \oplus I_{2(n- 1) - 2(i + 1)}
  \end{equation}
\end{lem}

Recall the functor $\Smat$ of \cref{def:alg-functor} constructed from the automorphism $\Smat$.
There is a mirror version $\overline{\Smat}$ is constructed in the same way from $\overline{\Smat}$, and we have shown that it also satisfies a Schur-Weyl duality:
\begin{lem}
  \label{thm:schur-weyl-mirror}
  Let $\beta : \rho \to \rho'$ be an admissible braid in $\slg$, and let $\rho$ and $\rho'$ correspond to characters $\chi_i$ and $\chi_i'$, respectively.
  Define linear maps $\overline{\phi}_\rho : \homol{D_n; \rho} \to \opspacebar n$ by
  \[
    \overline{\phi}_{\rho}(v_j^{\nu}) = \overline{\beta}_{j}^{\nu}.
  \]
  Then the diagram commutes:
  \[
    \begin{tikzcd}
      \homol{D_n; \rho}[\lf] \arrow[r, "\burau(\beta)"] \arrow[d, "\phi_\rho"] & \homol{D_n; \rho}[\lf] \arrow[d, "\phi_{\rho'}"] \\
      \opspacebar n/\ker(\chi_1^{-1} \otimes \cdots \otimes \chi_n^{-1}) \arrow[r, "\overline{\Smat}(\beta)"] & \opspacebar n/\ker((\chi_1')^{-1} \otimes \cdots \otimes (\chi_n')^{-1})
    \end{tikzcd}
  \]
\end{lem}
\begin{proof}
  This is proved exactly the same way as \cref{thm:schur-weyl} but using the computations above.
  A key observation is that the matrix \eqref{eq:quantum-burau-action-opinv} is sent to the matrix \eqref{eq:burau-coords-nice} under the map $(\chi_1')^{-1} \otimes \cdots (\chi_n')^{-1} = (\chi_1' \circ S) \otimes \cdots \otimes (\chi_n' \circ S)$, where $S$ is the antipode.%
  \note{
    The conflict of $S$ for the antipode and $\Smat$ for the outer $S$-matrix is unfortunate.
  }
\end{proof}

\section{The torsion as a quantum invariant}
\label{sec:torsion-proof}
We can now prove \cref{thm:T-is-torsion}.

\subsection{Graded multiplicity spaces for \texorpdfstring{$\doubfunc$}{T}}
Represent the link $L$ of \cref{thm:T-is-torsion} as the closure of a shaped braid $\beta : (\chi_1, \dots, \chi_n) \to (\chi_1, \dots, \chi_n)$.
To define $\doubfunc_2(\beta)$, we need the characters $\chi_i$ to be extended, so pick square roots $\mu_i$ of the eigenvalues $\lambda_i$.%
\note{
  Since there is only one eigenvalue for each connected component of $L$, we only need to pick one square root for each component.
}
Then $\doubfunc_2(\beta)$ is an endomorphism of
\begin{equation}
  \label{eq:doubled-braid-modules}
  \doubmod{\beta} \defeq \doubmod{\chi_1, \mu_1} \otimes \cdots \otimes \doubmod{\chi_n, \mu_n}
\end{equation}
and to compute the trace of $\doubfunc_2(\beta)$ we want to decompose the $\qgrp[i] \boxtimes \qgrp[i]^{\cop}$-module in \cref{eq:doubled-braid-modules} into simple direct summands.
Write $\chi$ for the $\cent_0$-character of this module, which is the product of the characters for the factors:
\[
  \chi \defeq \chi_1 \cdots \chi_n = (\chi_1 \otimes \cdots \chi_n)\Delta^n
\]
and let $\lambda$ be an eigenvalue for $\chi$.
Choose a square root $\mu$ of $\lambda$.
By \cref{thm:torsion-computation-cor}, we may assume that $\lambda \ne 1$.

By making this assumption, we can guarantee that the module $\doubfunc_2(\chi_1, \dots, \chi_n)$ of \cref{eq:doubled-braid-modules} decomposes into a direct sum of simple $\qgrp[i]\boxtimes\qgrp[i]^{\cop}$-modules.
In the case $\nr = 2$, these are fairly easy to compute.
\begin{lem}
  \label{thm:tensor-decomp-2}
  Let $\chi_1, \chi_2$ be shapes with fractional eigenvalues $\mu_i$.
  Thinking of the $\chi_i$ as $\cent_0$-characters, their product is $\chi = (\chi_1 \otimes \chi_2)\Delta$.
  If $\chi$ is nonsingular (i.e.\@ does not have $1$ as an eigenvalue), then
  \[
      \irrmod{\chi_1, \mu_1} \otimes \irrmod{\chi_2, \mu_2} \iso \irrmod{\chi, \mu} \oplus \irrmod{\chi, - \mu}
  \]
  where $\irrmodname$ is the standard $2$-dimensional representation of $\qgrp[i]$.
  Similarly we have
  \[
      \irrmod{\chi_1, \mu_1}^* \otimes^{\op} \irrmod{\chi_2, \mu_2}^* \iso \irrmod{\chi, \mu}^* \oplus \irrmod{\chi, - \mu}^*.
  \]
\end{lem}
\begin{proof}
    We know that 
    \[
      X = V(\chi_1, \mu_1) \otimes V(\chi_2, \mu_2)
    \]
    is a $4$-dimensional module with $\cent_0$-character $\chi$, so by \cref{thm:module-classification} it must be a direct sum of modules of the form $\irrmod{\chi, \pm \mu}$.
    We can then check the action of $\Delta \Omega$ to see that there is one summand for each sign.

    The second relation follows similarly; there it is important that we took the \emph{opposite} coproduct for the mirrored $\qgrp[i]$, so that the product
    \[
      \chi_1^{-1} \cdot^{\op} \chi_2^{-1} = \chi_2^{-1} \cdot \chi_1^{-1} = (\chi_1 \chi_2)^{-1} = \chi^{-1}
    \]
    works out.
\end{proof}
As a consequence, tensor products like $\doubmod{\chi_1, \mu_1} \otimes \doubmod{\chi_2, \mu_2}$ will decompose into a direct sum of modules of the form
\[
    \doubmod{\chi, \mu}[\epsilon_1 \epsilon_2]
    \defeq 
    \irrmod{\chi, (-1)^{\epsilon_1} \mu}
    \boxtimes
    \irrmod{\chi, (-1)^{\epsilon_2} \mu}
\]
where $\epsilon_1, \epsilon_2 \in \ZZ/2$.
In our previous notation, $\doubmod{\chi, \mu} = \doubmod{\chi, \mu}[00]$.

We can iterate the decomposition in \cref{thm:tensor-decomp-2} to compute
\begin{equation*}
    \begin{split}
        \bigotimes_{j=1}^{n} \irrmod{\chi_j, \mu_j}
            &\iso
            \irrmod{\chi, \mu}^{\oplus 2^{n-2}} 
            \oplus
            \irrmod{\chi, -\mu}^{\oplus 2^{n-2}} 
            \\
            &=
            (X_0 \otimes_\CC \irrmod{\chi, \mu})
            \oplus
            (X_1 \otimes_\CC \irrmod{\chi, -\mu})
    \end{split}
\end{equation*}
where $X_0$ and $X_1$ are \defemph{multiplicity spaces}, vector spaces which index the repeated copies of each simple direct summand.
Similarly in the mirrored case we have 
\begin{equation*}
  \bigotimes_{j=1}^{n} \irrmod{\chi_j, \mu_j}^*
  \iso
  (\overline{X}_0 \otimes_\CC \irrmod{\chi, \mu}^*)
  \oplus
  (\overline{X}_1 \otimes_\CC \irrmod{\chi, -\mu}^*)
\end{equation*}
were as before we think of the left-hand side as a product of $\qgrp[i]^{\cop}$-modules.
We can then understand the decomposition of the module \eqref{eq:doubled-braid-modules} in terms of the multiplicity spaces $X_\epsilon$ and $\overline{X}_\epsilon$.

\begin{thm}
  \label{thm:tensor-decomp-D}
  \begin{enumerate}
    \item 
        The module \eqref{eq:doubled-braid-modules} decomposes as
      \[
        \bigotimes_{i=1}^n \doubmod{\chi_i, \mu_i}
        \iso
        \bigoplus_{\epsilon_1, \epsilon_2 \in \ZZ/2}
        (X_{\epsilon_1} \otimes \overline{X}_{\epsilon_2}) \otimes \doubmod{\chi, \mu}[\epsilon_1 \epsilon_2].
      \]
    \item
      The modified dimension of $\doubmod{\chi, \mu}[\epsilon_1 \epsilon_2]$ is 
      \[
        \moddim{\doubmod{\chi, \mu}[\epsilon_1 \epsilon_2]}
        =
         -\frac{(-1)^{\epsilon_1 + \epsilon_2}}{(\mu - \mu^{-1})^2}
         =
         \frac{(-1)^{\epsilon_1 + \epsilon_2}}{2- \lambda - \lambda^{-1}}.
      \]
    \item Let $f \in \End_{\qgrp[i]\boxtimes \qgrp[i]^{\cop}}(\doubmod{\beta})$ be an endomorphism.
      Then there are linear maps $g_{\epsilon_1 \epsilon_2} \in \End_\CC(X_{\epsilon_1} \otimes \overline{X}_{\epsilon_2})$ with
      \[
        f = \bigoplus_{\epsilon_1, \epsilon_2 \in \ZZ/2} g_{\epsilon_1 \epsilon_2} \otimes \id_{\doubmod{\chi, \mu}[\epsilon_1\epsilon_2]},
      \]
      and the modified trace of $f$ is given by
      \[
        \modtr(f)
         = \sum_{\epsilon_1, \epsilon_2 \in \ZZ/2} \frac{(-1)^{\epsilon_1 + \epsilon_2}}{(\mu - \mu^{-1})^2} \tr g_{\epsilon_1 \epsilon_2}.
      \]
  \end{enumerate}
\end{thm}
\begin{proof}
  The first statement is immediate from \cref{thm:tensor-decomp-2}, and the third follows from the first two.
  To prove (2), first note that by \cref{thm:modified-trace-exists} we have
  \[
    \moddim{\irrmod{\chi, \mu}} = \frac{(i\mu) - (i\mu)^{-1}}{(i\mu)^2 - (i\mu)^{-2}}
    = \frac{1}{i\mu + (i \mu)^{-1}} = \frac{-i}{\mu - \mu^{-1}}
  \]
  so that by \cref{thm:double-traces}
  \[
    \moddim{\doubmod{\chi, \mu}[\epsilon_1 \epsilon_2]}
    = - \frac{(-1)^{\epsilon_1 + \epsilon_2}}{(\mu - \mu^{-1)}}
    = \frac{(-1)^{\epsilon_1 + \epsilon_2}}{2- \lambda - \lambda^{-1}}.
  \]
\end{proof}
Part (3) says that the modified trace of an endomorphism of $\doubmod \beta$ (in particular the trace of $\doubfunc_2(\beta)$) can be computed as a $\ZZ/2$-graded trace on the multiplicity space.

\begin{defn}
  A \defemph{super vector space} is a $\ZZ/2$-graded vector space $Y = Y_0 \oplus Y_1$.
  We call $Y_0$ and $Y_1$ the \defemph{even} and \defemph{odd} parts, respectively.
  A morphism $f = f_0 \oplus f_1$ of super vector spaces preserves the grading, and we define the \defemph{supertrace} by
  \[
    \str f \defeq \tr f_0 - \tr f_1.
  \]
  
\end{defn}

\begin{ex}
  \label{ex:extp}
  If $W$ is an ordinary vector space, then the exterior algebra $\extp W$ becomes a super vector space by setting the image of $W$ in $\extp W$ to be odd.
  A Clifford algebra on $W$ becomes a super vector space in the same way.
\end{ex}

\begin{defn}
  The \defemph{multiplicity superspace} of $\doubmod{\beta}$ is the super vector space
  \[
    Y = Y((\chi_1, \mu_1), \dots, (\chi_n, \mu_n))
  \]
  with even part
  \begin{align*}
    Y_0 &\defeq X_0 \otimes \overline{X}_0 \oplus X_1 \otimes \overline{X}_1
    \\
    \intertext{and odd part}
    Y_1 &\defeq X_0 \otimes \overline{X}_1 \oplus X_1 \otimes \overline{X}_0.
  \end{align*}
\end{defn}

We see that the problem of computing the modified trace $\modtr(\doubfunc_2(\beta))$ of a braid can be reduced to understanding the action of $\doubfunc_2(\beta)$ on the multiplicity superspace~$Y$.
To solve this problem, we identify $Y$ with the exterior algebra $\extp \homol{D_n, \rho}[\lf]$ of the twisted homology, then apply \cref{thm:schur-weyl,thm:schur-weyl-mirror} to compute the braid action on $Y$ in terms of the Burau representation $\burau$.

\subsection{Computing the multiplicity spaces}
To compute the space $Y$ we use a sort of tensor product of the algebras $\opspace n$ and $\opspacebar n$.
\begin{defn}
  Set
  \note{
    Since the operators $\alpha_j^\nu$ are sort of like $E$, it makes sense that we embed them using the $E$-like coproduct $X \mapsto X \boxtimes 1 + K \boxtimes X $.
  }
  \begin{align*}
    \gamma^\nu_j &\defeq \alpha^\nu_j \boxtimes 1 + \Delta K \boxtimes \overline{\alpha}^\nu_j \\
    \theta^\nu_j &\defeq \beta^\nu_j \boxtimes 1 + \Delta K \boxtimes \overline{\beta}^\nu_j = \gamma_j^{\nu} - \gamma_{j+1}^{\nu}
  \end{align*}
  Write $\opspaced n$ for the $(\cent_0[\Omega^{-2}] \boxtimes \cent_0[\Omega^{-2}])^{\otimes n} $-span of the $\theta_j^\nu$ and $\cliffordd n$ for the algebra generated by $\opspaced n$.
  As before, $\cliffordd n$ is a Clifford algebra.
\end{defn}

Unfortunately, we still are not quite ready to use \cref{thm:schur-weyl} to compute the action on the multiplicity superspaces.
The problem is that $\cliffordd n$ does not quite supercommute with the superalgebra $\qgrp[i] \otimes \qgrp[i]^\cop$, regardless of whether we take the ordinary or super tensor product.
Fortunately, this is not necessary.

According to part 3 of \cref{thm:tensor-decomp-D}, to compute the modified traces we do not need the detailed multiplicity spaces $X_{\epsilon_1} \otimes \overline{X}_{\epsilon_2}$, only the spaces $Y_0 = X_{00} \oplus X_{11}$ and $Y_1 = X_{01} \oplus X_{10}$.%
\note{
  We can think of this as flattening the $\ZZ/2 \times \ZZ/2$-graded space $\oplus_{\epsilon_1 \epsilon_2} X_{\epsilon_1 \epsilon_2}$ to the $\ZZ/2$-graded space $Y_0 \oplus Y_1$ along the homomorphism $(x_1, x_2) \mapsto x_1 + x_2$.
}
The simpler problem of computing $Y$ has a very nice answer in terms of $\cliffordd n$, given below.
After proving it we will be able to immediately apply Schur-Weyl duality.

\begin{thm}
  \label{thm:mult-space-calc}
  The $\ZZ/2$-graded multiplicity space $Y(\rho)$ of $\doubfunc(\rho)$ is isomorphic as a super vector space to $\cliffordd n$.
  Explicitly, a basis is given by
  \[
    (\theta_{k_1^1}^1 \cdots \theta_{k_{s_1}^1}^1 \theta_{k_1^2}^2 \cdots \theta_{k_{s_2}^2}^2) \cdot w(\rho)
  \]
  where $w(\rho)$ is the invariant vector of \cref{thm:invariant-vector}, $1 < k_\nu \cdots < k_{s_\nu} \le n -1$, and $s_\nu = 0, \dots, n-1$ for $\nu = 1, 2$.
  Furthermore, the actions of the $\theta_k^\nu$ on $w(\rho)$ anticommute, so we can identify $Y(\rho)$ as an exterior algebra with one generator for each $\theta_k^\nu$.
\end{thm}
We give the proof in the remainder of this subsection; the theorem is a corollary of \cref{thm:mult-space-calc-lem}.

Because we only need to understand the space $Y$ it suffices to consider a weaker sort of supercommutativity.
Let $W$ be a $\qgrp[i] \otimes \qgrp[i]^{\cop}$-module.
It becomes a $\qgrp[i]$-module via the action ($w \in W$)
\[
  \begin{aligned}
    K \cdot w &= ( K \boxtimes K) \cdot w,
    \\
    E \cdot w &= (E \boxtimes K + 1 \boxtimes E) \cdot w,
    \\
    F \cdot w &= (F \boxtimes 1 + K^{-1} \boxtimes F) \cdot w,
  \end{aligned}
\]
that is by embedding $\qgrp[i]$ into $\qgrp[i] \otimes \qgrp[i]^{\cop}$ via the coproduct.
In each product $\irrmod{\chi} \boxtimes \irrmod{\chi}^*$ the factors have characters $\chi$ and $\chi^{-1}$, so their product has character
\[
  \epsilon = \chi \cdot \chi^{-1}, \quad
  \epsilon(K^2) = 1, \quad
  \epsilon(E^2) = \epsilon(F^2) = 0.
\]
The character $\epsilon$ is the identity%
\note{
  We can obtain $\epsilon$ by restricting the counit of $\qgrp[i]$ to $\cent_0$.
  While $\epsilon(\Omega) = 0$, there are modules like $P_0$ with $\cent_0$ character on which $\Omega$ does not act by zero but $\Omega^2$ does.
}
of the group $\spec \cent_0$, and we need to introduce some modules with character $\epsilon$.
\begin{defn}
  \label{def:proj-cover-mod}
  $P_0$ is the $4$-dimensional $\qgrp[i]$-module with character $\epsilon$ described by the diagram
  \[
    \begin{tikzcd}
        & t & & 1\\
      Et && Ft & -1\\
         & b & & 1
         \arrow["E"', color={myblue}, from=1-2, to=2-1]
         \arrow["F", color={myred}, from=1-2, to=2-3]
         \arrow["F"', color={myred}, from=2-1, to=3-2]
         \arrow["E", color={myblue}, from=2-3, to=3-2]
    \end{tikzcd}
  \]
  The vectors are eigenvectors of $K$, whose action is given by the numbers in the right column.
  The {\color{myblue} blue} arrows give the action of $E$, the {\color{myred} red} arros give the action of $F$, and the absence of an arrow means the corresponding generator acts by $0$.

  There is a similar module $P_1$ given by the same diagram, but with the action of $K$ negated.
  It still has character $\epsilon$ because $K^2$ acts by $1$.
\end{defn}

We will use a family of modules generalizing $P_0$ (there denoted $P$) to define modified traces in \cref{sec:traces-for-qgrp}.

\begin{prop}
  \label{thm:proj-cover-iso}
  As a $\qgrp[i]$-module, $\doubmod{\chi, \mu}[\epsilon_1 \epsilon_2] \iso P_{\epsilon_1 + \epsilon_2}$.
\end{prop}
\begin{proof}
  It is not hard to see that the linear map $f : P_0 \to \irrmod{\chi, \mu} \otimes \irrmod{\chi, \mu}^*$ given by
  \begin{align*}
    f(t) &= \hat v_0 \otimes \hat v^0 - \hat v_1 \otimes \hat v^1
    \\
    f(b) &= \hat v_0 \otimes \hat v^0 + \hat v_1 \otimes \hat v^1
  \end{align*}
  extends to an isomorphism of $\qgrp[i]$-modules.
  A similar computation works for $P_1$.
\end{proof}

\begin{lem}
  \label{thm:mult-space-calc-lem}
  Let $\rho = ((\chi_1, \mu_1) \cdots, (\chi_n, \mu_n))$ be a tuple of extended shapes with nonsingular total holonomy.
  Write $\pi_\rho$ for the structure map
  \[
    \pi_\rho : (\qgrp[i] \boxtimes \qgrp[i])^{\otimes n}  \to \End_{\CC}(\doubfunc(\rho))
  \] 
  Then
  \begin{enumerate}
    \item $\pi_\rho(\cliffordd n)$ is an exterior algebra on $2(n-1)$ generators,
    \item $\cliffordd n$ acts faithfully on $\doubfunc(\rho)$, and
    \item thinking of $\doubfunc(\rho)$ as a $\qgrp[i]$-module, $\pi_\rho(\cliffordd n)$ super-commutes with $\qgrp[i]$.
  \end{enumerate}
\end{lem}
\begin{proof}
  (1) We show that the anticommutators
  \[
    \pi_\rho(\{ \gamma_j^\mu, \gamma_k^\nu \})
  \]
  vanish, so that the image is an exterior algebra on the $2n$ independent generators $\pi_\rho(\gamma_k^\nu)$, $k = 1, \dots, n$, $\nu = 1, 2$.
  Since $\cliffordd n$ is generated by the $\theta_k^{\nu} = \gamma_k^{\nu} - \gamma_{k+1}^{\nu}$ we get the desired result.

  To start, use $\{\alpha_k^\nu, \Delta K\} = 0$ to compute that
  \begin{align*}
    \{ \gamma_j^\mu, \gamma_k^\nu \} &= \{ \alpha_j^\mu, \alpha_k^\nu \} \boxtimes 1 + \Delta K^2 \boxtimes \{ \overline{\alpha}_j^\mu, \overline{\alpha}_k^\nu \}
  \end{align*}
  By using the anticommutator computations of \cref{thm:supercommute-bar,thm:Cn-is-clifford} we can show directly that these vanish.
  For example, the above expression vanishes unless $j = k$.
  We give the case $\mu = \nu = 1$ in detail; the remaining others follow similarly.

  Observe that
  \begin{align*}
    &\{ \alpha_j^1, \alpha_j^1 \} \boxtimes 1 + \Delta K^2 \boxtimes \{ \overline{\alpha}_j^1, \overline{\alpha}_j^1 \} \\
    &=  -2 K_1^2 \cdots K_{j-1}^2 E_j^2 \Omega_j^{-2} \boxtimes 1 - 2 K_1 \cdots K_n^2 \boxtimes E_j^2 \Omega_j^{-2} K_{j+1}^2 \cdots K_n^2
  \end{align*}
  Write $\chi_j(K^2) = \kappa_j$, $\chi_j(E^2) = \epsilon_j$, $\chi_j(\Omega^2) = \omega_j^2$, so that
  \begin{align*}
    \pi_\rho(K_j^2 \boxtimes 1) &= \kappa_j & \pi_\rho(1 \boxtimes K_j^2) &= \kappa_j^{-1} \\
    \pi_\rho(E_j^2 \boxtimes 1) &= \epsilon_j & \pi_\rho(1 \boxtimes E_j^2) &= - \epsilon_j \kappa_j^{-1} \\
    \pi_\rho(\Omega_j^2 \boxtimes 1) &= \omega_j^2 & \pi_\rho(1 \boxtimes \Omega_j^2) &= \omega_j^2 
  \end{align*}
  using the fact that the representations in the second half of the $\boxtimes$ product (corresponding to $\overline{\vecfunc}$) use the inverse characters.
  Hence
  \begin{align*}
    & \pi_\rho( 2 K_1^2 \cdots K_{j-1}^2 E_j^2 \Omega_j^{-2} \boxtimes 1 + 2 K_1 \cdots K_n^2 \boxtimes E_j^2 \Omega_j^{-2} K_{j+1}^2 \cdots K_n^2) \\
    &= \frac{2}{\omega_j^2} \left( \kappa_1 \cdots \kappa_{j-1} \epsilon_j + \kappa_1 \cdots \kappa_n (-\epsilon_j \kappa_j^{-1}) \kappa_{j+1}^{-1} \cdots \kappa_n^{-1} \right) = 0
  \end{align*}
  as claimed.

  (2) It is enough to show that the operators $\pi_\rho(\gamma_k^\nu)$ all act independently.
  Since up to a scalar $\gamma_k^1, \gamma_k^2$ only act on the $k$th $\otimes$-factor of the product
  \[
    \doubfunc(\rho) = \bigotimes_{j = 1}^n \irrmod{\chi_j, \mu_j} \boxtimes \irrmod{\hat \chi_j, \mu_j}^*
  \]
  it is enough to check that $\gamma_k^1$ and $\gamma_k^2$ act independently.
  It is not hard to compute explicitly that the vectors
  \[
    \pi_\rho(\gamma_k^1) \cdot w(\rho)
    \text{ and }
    \pi_\rho(\gamma_k^2) \cdot w(\rho)
  \]
  are independent, where $w(\rho)$ is the invariant vector of \cref{thm:invariant-vector}, and (2) follows.

  (3) We can check directly that
  \begin{align*}
    &[ \Delta E \boxtimes \Delta^{\op} K + 1 \boxtimes \Delta^{\op} E, \theta_k^\nu] \\
    &= [ \Delta E \boxtimes \Delta^{\op} K + 1 \boxtimes \Delta^{\op} E, \beta_k^\nu \boxtimes 1 + \Delta K \boxtimes \overline{\beta}_k^\nu] \\
    &=  [ \Delta E, \beta_k^\nu ] \boxtimes 1 + \Delta K \boxtimes [ \Delta^{\op} E, \overline{\beta}_k^\nu ] \\
    &= 0
  \end{align*}
  The other generators $K, \F$ follow similarly.
\end{proof}

\begin{proof}[Proof of \cref{thm:mult-space-calc}]
  Apply a (super version) of the double centralizer theorem.
  By (3) of \cref{thm:mult-space-calc-lem},  $\pi_\rho$ induces an inclusion $\cliffordd n \to Y(\rho)$, and by (2) this inclusion is injective.
  Both spaces have dimension $2^{2n-2}$ over $\CC$, so it is an isomorphism.
\end{proof}

\subsection{Schur-Weyl duality for modules and the proof}
By using \cref{thm:schur-weyl,thm:mult-space-calc} we can compute the action of $\doubfunc_2(\beta)$ on the multiplicity space $Y(\rho)$ of $W(\rho)$ and prove \cref{thm:T-is-torsion}.

\begin{thm}[Schur-Weyl duality for modules]
  \label{thm:schur-weyl-modules}
  Let 
  \[
    \beta : (\chi_1, \dots, \chi_n) \to (\chi_1', \dots, \chi_n')
  \]
  be an admissible shaped braid on $n$ strands, and let $\rho : \pi_1(D_n) \to \slg$ be the representation given by the holonomy of the shapes $(\chi_1, \dots, \chi_n)$, and similarly for $\rho'$ and $(\chi_1', \dots, \chi_n')$.
  Write $\Phi_\rho$ for the linear map sending the basis vector $v_j^\nu$ of $\homol{D_n;\rho}$ to the element $\theta_j^\nu \cdot w(\rho)$, as in \cref{thm:mult-space-calc}.
  \begin{enumerate}
    \item $\Phi_\rho$ induces an isomorphism of super vector spaces $\extp \Phi_\rho : \extp \homol{D_n;\rho} \to Y(\rho)$
    \item such the diagram commutes
      \[
        \begin{tikzcd}
          \extp \homol{D_n; \rho}[\lf] \arrow[d, "\extp \Phi_\rho"] \arrow[r, "\extp \burau(\beta)"] & \extp \homol{D_n; \rho'}[\lf] \arrow[d, "\extp \Phi_{\rho'}"]\\
          Y(\rho) \arrow[r, "\doubfunc_2(\beta)"] & Y(\rho) 
        \end{tikzcd}
      \]
  \end{enumerate}
\end{thm}
\begin{proof}
  (1) is an immediate corollary of \cref{thm:mult-space-calc}.
  To prove (2) it is enough to check commutativity on generating vectors $v_j^\nu$ and their images $\theta_j^\nu \cdot w(\rho)$ under $\Phi_\rho$.
  Explicitly, for a braid generator $\sigma_i$ we have
  \[
    \begin{split}
      \doubfunc_2(\sigma_i)(\theta_j^\nu \cdot w(\rho))
      &=
      (\Smat_{i,i+1} \boxtimes \overline{\Smat}_{i,i+1})(\theta_j^\nu) \cdot \doubfunc_2(\sigma_i)(w(\rho))
      \\
      &=
      (\Smat_{i,i+1} \boxtimes \overline{\Smat}_{i,i+1})(\theta_j^\nu) \cdot w(\rho').
    \end{split}
  \]
  Since
  \[
    \begin{split}
      (\Smat_{i,i+1} \boxtimes \overline{\Smat}_{i,i+1})(\theta_j^\nu)
      &=
      \Smat_{i,i+1}(\beta_j^\nu) \boxtimes \overline{\Smat}_{i,i+1}(1)
      +
      \Smat_{i,i+1}(\Delta K) \boxtimes \overline{\Smat}_{i,i+1}(\overline{\beta}_{k}^{\nu})
      \\
      &=
      \Smat_{i,i+1}(\beta_j^\nu) \boxtimes 1
      +
      \Delta K \boxtimes \overline{\Smat}_{i,i+1}(\overline{\beta}_{k}^{\nu})
    \end{split}
  \]
  the result follows from \cref{thm:schur-weyl,thm:schur-weyl-mirror}.
\end{proof}

Once we recall a fact about exterior powers we can prove $\doubfunc_2$ computes the torsion.
\begin{lem}
  \label{thm:str-to-det}
  Let $W$ be a vector space of dimension $m$ and $A : W \to W$ a linear map.
  Write $\extp A$ for the induced map $\extp W \to \extp W$ on the exterior algebra of $W$.
  Then
  \[
    \str\left( \extp A \right) = (-1)^m \det(1 - A).
  \]
\end{lem}
\begin{proof}
  Recall that
  \[
    \det (\lambda - A) = \sum_{k = 0}^m \lambda^{m-k}(-1)^{m-k} \tr \left( \extp^k A \right)
  \]
  so in particular
  \[
    \det(1 - A) = (-1)^m \sum_{k=0}^m (-1)^k \tr \left( \extp^kA \right) = (-1)^m \operatorname{str} \left( \extp{}A \right).
  \]
\end{proof}

\begin{proof}[Proof of \cref{thm:T-is-torsion}]
  Let $L$ be an extended admissible $\slg$-link.
  By \cref{thm:torsion-computation-cor} we can represent $L$ as the closure of an extended shaped braid $\beta : \rho \to \rho$ with admissible total holonomy.
  We want to compute
  \[
    \doubfunc_2(L) \defeq \modtr \doubfunc_2(\beta).
  \]
  Choose a fractional eigenvalue $\mu$ for the total holonomy of $\beta$.
  By \cref{thm:schur-weyl-modules} the intertwiner $\doubfunc(\beta) = \doubfunc_2(\beta)$ factors through the multiplicity superspace $Y(\rho)$ as $\extp \burau(\beta)$, so by \cref{thm:tensor-decomp-D,thm:str-to-det} we have
   \[
     \modtr \doubfunc(\beta) =
     \frac{\str \extp \burau(\beta)}{(\mu - \mu^{-1})^2}
     =
     \frac{\det(1 - \burau(\beta))}{2 - \lambda + \lambda^{-1}}.
  \]
  By \cref{thm:torsion-computation-cor} this is exactly the torsion $\tau(L,\rho)$ of $L$.
\end{proof}

\appendix
\chapter{Modified traces}
\label{ch:modified-traces}
\section*{Overview}
We apply the methods of Geer, Kujawa, and Patureau-Mirand \cite{Geer2018} to construct the modified traces used in \cref{ch:functors}.
The approach of \cite{Geer2018} is rather abstract, and few concrete examples have appeared in the literature,%
\note{
  One discussion is \cite[Appendix B]{McPhailSnyder2020}, which is what this section of the thesis is based on.
}
so this appendix may also be helpful as a guide to applying their techniques to quantum topology.
In particular, in \cref{sec:traces-for-qgrp} we sketch the relationship between the abstract methods of \cref{sec:modified-trace-construction} and \cite{Geer2018} and the more concrete Hopf link methods of \cite{Geer2009,Geer2018trace}.

In this appendix we frequently state results for a pivotal $\CC$-linear category $\catl C$, by which mean a pivotal category whose $\hom$-spaces are vector spaces over $\CC$ and whose tensor product is $\CC$-bilinear.
$\modc \qgrp$ and more generally the category of modules of a pivotal Hopf $\CC$-algebra are examples of such categories.

More specific results of \cite{Geer2018} place extra conditions on $\catl C$ (local finiteness) and on certain distinguished objects (absolute decomposability, end-nilpotence, etc.) which are satisfied for finite-dimensional representations of an algebra over an algebraically closed field, perhaps with some diagonalizability assumptions.
All our examples satisfy these hypotheses.

\section{Construction of modified traces}
\label{sec:modified-trace-construction}
\begin{defn}
  \label{def:pivotal-ideal}
  Let $\catl C$ be a pivotal $\CC$-category.
  A \defemph{right (left) ideal} $I$ is a full subcategory of $\catl C$ that is:
  \begin{enumerate}
    \item \defemph{closed under right (left) tensor products:} If $V$ is an object of $I$ and $W$ is any object of $\catl C$, then $V \otimes W$ ($W \otimes V$) is an object of $I$.
    \item \defemph{closed under retracts:} If $V$ is an object of $I$, $W$ is any object of $\catl C$, and there are morphisms $f, g$ with
      \[
        \begin{tikzcd}
          W \arrow[r, "f"] \arrow[rr, bend right=40, swap, "\id_{W}"]  & V \arrow[r, "g"] & W
        \end{tikzcd}
      \]
      commuting, then $W$ is an object of $I$.
  \end{enumerate}
  An \defemph{ideal} of $\catl C$ is a full subcategory which is both a left and right ideal.
\end{defn}

\begin{defn}
  Let $\catl C$ be a category.
  We say an object $P$ of $\catl C$ is \defemph{projective} if for any epimorphism $p : X \to Y$ and any map $f : P \to Y$, there is a lift $g : P \to X$ such that the diagram commutes:
  \[
    \begin{tikzcd}
      & X \arrow[d, "p"] \\
      P \arrow[ur, dashed, "g"] \arrow[r, "f"] & Y
    \end{tikzcd}
  \]
  We say $I$ is \defemph{injective} if $I$ is a projective object in $\catl{C}^{\op}$, i.e.\@ if $I$ satisfies the opposite of the above diagram.
  We write $\proj(\catl C)$ for the class of projective objects of $\catl C$.
\end{defn}

\begin{prop}
  Let $\catl C$ be a  pivotal category.
  Then the projective and injective objects coincide and $\proj(\catl C)$ is an ideal.
\end{prop}
\begin{proof}
  See \cite[Lemma 17]{Geer2013a}.
\end{proof}

Let $\catl C$ be a pivotal $\CC$-category with tensor unit $\unit$.
Consider the projective cover $P \to \unit$,%
\note{
  If $\catl C$ is semisimple, then $\unit$ is projective and we recover the usual trace in a pivotal category.
} and assume that $P$ is finite-dimensional.
Then $P$ is indecomposable and projective and the space $\hom_\catl{C}(P, \unit)$ is $1$-dimensional over $\CC$.
Because $\catl{C}$ is pivotal, $P$ is also injective and $\hom_\catl{C}(\unit, P)$ is similarly $1$-dimensional.

The choice of $P$ and a basis of each space are the data necessary to define a modified trace on $\proj(\catl{C})$, which we call a trace tuple.
Our definition is a special case (\cite[\S5.3]{Geer2018}) of the more general trace tuples of \cite{Geer2018}, setting ${\color{RubineRed} \alpha} = {\color{RubineRed} \beta} = \unit$.
These more general traces can be defined for larger ideals than $\proj(\catl{C})$.
\begin{defn}
  Let $\catl{C}$ be a pivotal $\CC$-category with tensor unit $\unit$, and let $P \to \unit$ be a finite-dimensional cover.
  $(P, \iota, \pi)$ is a \defemph{trace tuple} if $P$ is indecomposable and projective, $\iota$ is a basis of $\hom_\catl{C}(\unit, P)$, and $\pi$ is a basis of $\hom_\catl{C}(P, \unit)$.
\end{defn}
\begin{ex}
  \label{eq:weight-modules-2}
  A finite-dimensional $\qgrp$-module $W$ is a \defemph{weight module} if $\cent_0$ acts diagonalizably on $W$.
  Consider the case $\nr = 2$, $\xi = i$.
  Let $\catl{W}$ be the category of finite-dimensional $\qgrp[i]$-weight modules, and let $P_0$ be the module defined in \cref{def:proj-cover-mod}, which can be obtained as $V \otimes V^*$, where $V$ is any irreducible $\nr$-dimensional modules of \cref{def:weyl-irrmod}.
  As discussed in \cref{sec:traces-for-qgrp}, $P_0$ is the projective cover of the tensor unit with covering map $\evdown V : P_0 \iso V \otimes V^* \to \unit$ and $(P_0, \coevup{V}, {\evdown{V}})$ is a trace tuple for $\catl W$.
\end{ex}
Because $P$ is indecomposable, projective, and finite-dimensional, any endomorphism $f \in \End_\catl{C}(P)$ decomposes $f = a + n$ as an automorphism plus a nilpotent part.
Because $\CC$ is algebraically closed, $a$ is a scalar, and we write $\langle f \rangle = a \in \CC$.

If $g \in \hom_\catl{C}(\unit, P), h \in \hom_\catl{C}(P, \unit)$ are any morphisms, we can similarly define $\langle g \rangle_\iota, \langle h \rangle_\pi \in \CC$ by
\[
  g = \langle g \rangle_\iota \iota, \ \ h = \langle h \rangle_\pi \pi
\]

\begin{lem}
  \label{thm:brackets-agree}
  Let $(P, \iota, \pi)$ be a trace tuple.
  Then for any $f \in \End_\catl{C}(P)$,
  \begin{enumerate}
    \item $\pi f = \langle f \rangle_\pi \pi$
    \item $f \iota = \langle f \rangle_\iota \iota$
    \item  $\langle f \rangle = \langle f \iota \rangle_\iota = \langle \pi f \rangle_\pi$
  \end{enumerate}
\end{lem}
\begin{proof}
  We have $f = \langle f \rangle \id_P + n$ for some nilpotent $n$.
  The first statement follows from $\pi n = 0$.
  Since $\pi$ is a basis for $\hom_\catl{C}(P, \unit)$, we have $\pi n = \lambda \pi$ for some $\lambda \in \CC$.
  But $n^k = 0$ for some $k$, so $\lambda^k = 0 \Rightarrow \lambda = 0$ because $\CC$ is an integral domain.
  The second statement follows from a similar argument, and the third from the first two.
\end{proof}

\begin{lem}
  \label{thm:lifts-exist}
  Let $(P, \iota, \pi)$ be a trace tuple for $\catl{C}$ and $V$ a projective object.
  Then there are maps $\sigma_V : P \otimes V \to V$, $\tau_V : V \to P \otimes V$ such that the diagrams commute:
  \[
    \begin{tikzcd}
      & P \otimes V \arrow[dl, swap, "\sigma_V"] \\
      V & V \iso \unit \otimes V \arrow[l, "\iota \otimes \id_V"] \arrow[u, swap, "\iota \otimes \id_{V}"] 
    \end{tikzcd} \quad \quad
    \begin{tikzcd}
      & P \otimes V \arrow[d, "\pi \otimes \id_V"] \\
      V \arrow[ur, "\tau_V"] \arrow[r, swap, "\id_V"] & V \iso \unit \otimes V
    \end{tikzcd}
  \]
\end{lem}
\begin{proof}
  $V$ is projective and  $\pi \otimes \id_V: P \otimes V \to \unit \otimes V \to V$ is an epimorphism, so a lift $\tau_V$ exits.
  The dual argument works for $\sigma_V$.
\end{proof}

\begin{thm}
  Let $(P, \iota, \pi)$ be a trace tuple for $\catl{C}$ and choose maps as in \cref{thm:lifts-exist}.
  Then there exits a right modified trace on $\proj(\catl C)$ defined for $f \in \hom_\catl{C}(V, V)$ by 
  \[
    \modtr_V(f) = \langle \ptr_V^r(\tau_V f) \rangle_\iota = \langle \ptr_V^r(\sigma_V f) \rangle_\pi
  \]
\end{thm}
This is a special case of \cite[Theorem 4.4]{Geer2018}.
\begin{proof}
  In the diagrams in this proof, we identify
  \[
    \End_\catl{C}(P)/J \iso \hom_\catl{C}(\unit, P) \iso \hom_\catl{C}(P, \unit) \iso \CC
  \]
  via the maps $\langle - \rangle$, $\langle - \rangle_\iota$, and $\langle - \rangle_\pi$.
  Here $J$ is the ideal of nilpotent elements of $\End_\catl{C}(P)$, so when we draw a diagram representing a morphism $P \to P$ we really mean its image in this quotient.

  $\tau_V$ and $\sigma_V$ exist by \cref{thm:lifts-exist}, but are not unique.
  We show that the trace does not depend on the choice of either. 
  In graphical notation, $\ptr_V^r(\tau_V f)$ can be written as
  \begin{center}
    \includegraphics{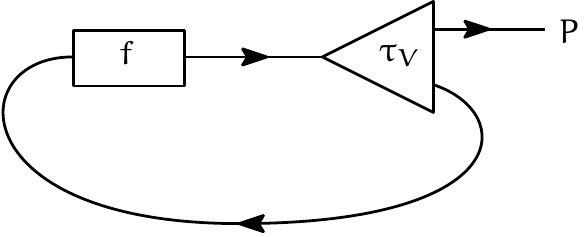}
  \end{center}
  Since $\sigma_V(\iota \otimes \id_V) = \id_V$, we can rewrite this morphism as
  \begin{center}
    \includegraphics{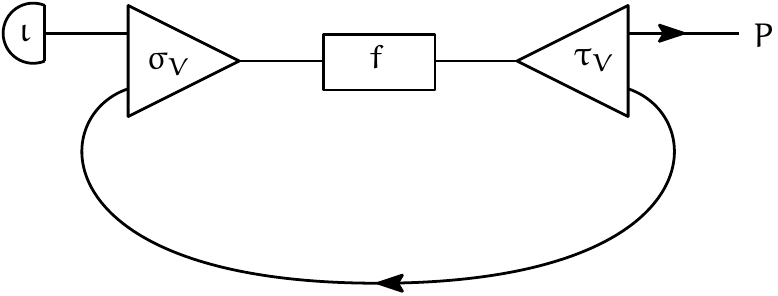}
  \end{center}
  where $\iota$ has no left-hand arrows because it is a map $\unit \to P$.
  By Lemma \ref{thm:brackets-agree}, the above diagram is equal to
  \begin{center}
    \includegraphics{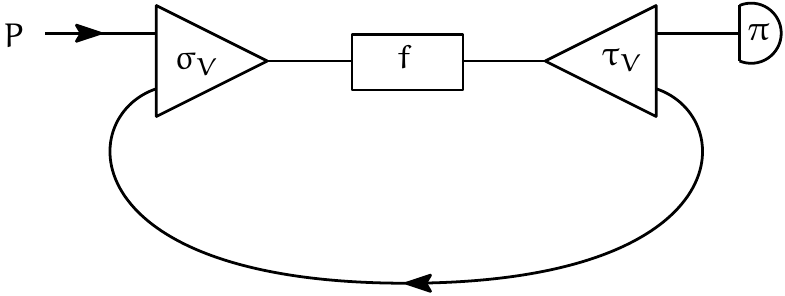}
  \end{center}
  But since $(\pi \otimes \id_V) \tau_V = \id_V$, this is equal to $\ptr_V^r(f \sigma_V)$:
  \begin{center}
    \includegraphics{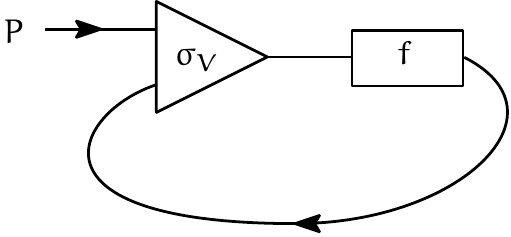}
  \end{center}
  It follows that 
  \[
    \langle \ptr_V^r(\tau_V f) \rangle_\iota = \langle \ptr_V^r(\sigma_V f) \rangle_\pi
  \]
  as claimed.

  To check the compatibility with the partial trace, let $f : V \otimes W \to V \otimes W$.
  Choose $\tau_V$ with $( \pi \otimes \id_V)\tau_V = \id_V$, and notice that we can set $\tau_{V \otimes W} = \tau_V \otimes \id_W$.
  Then $\modtr_{V \otimes W}(f)$ is
  \begin{center}
    \includegraphics{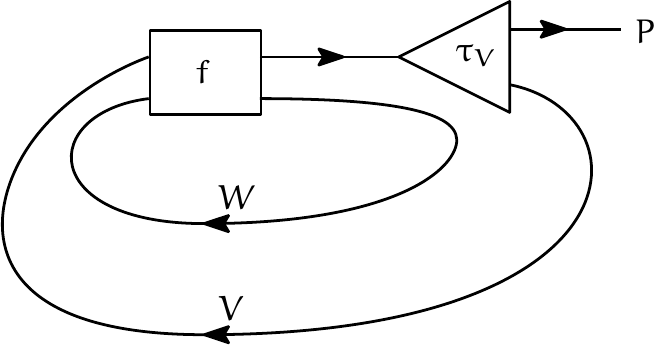}
  \end{center}
  which is clearly equal to $\modtr_V(\tr_W^r(f))$.

  Finally, we show cyclicity.
  Suppose $f : V \to W$ and $g : W \to V$.
  Then $\modtr_V(gf)$ is equal to
  \begin{center}
    \includegraphics{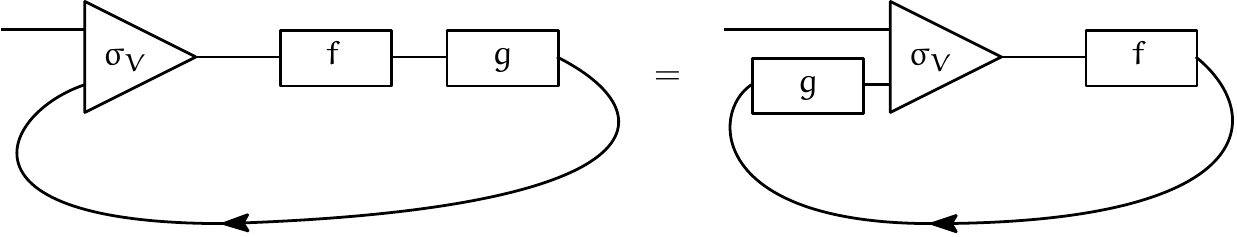}
  \end{center}
  by the cyclicity of the usual trace.
  But by inserting $(\pi \otimes \id_W) \tau_W = \id_W$ and then applying \cref{thm:brackets-agree} as before, we can rewrite this as
  \begin{center}
    \includegraphics{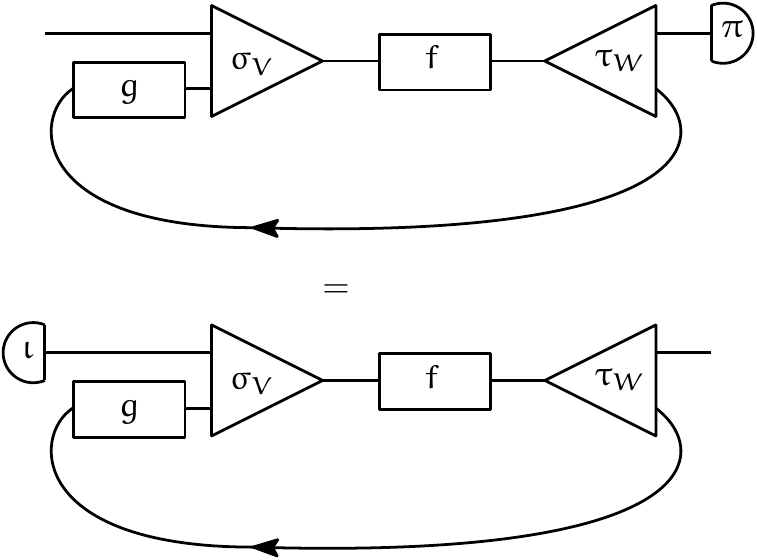}
  \end{center}
  By absorbing $\iota$ into $\sigma_V$, we see that this is equal to
  \begin{center}
    \includegraphics{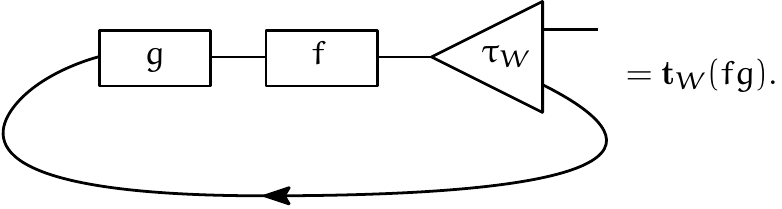}
  \end{center}
\end{proof}

It can be shown that the modified trace on $\proj(\catl{C})$ is essentially unique; choosing different $\iota$ or $\pi$ will simply change $\modtr$ by an overall scalar.
The paper \cite{Geer2018} proves this and a number of other useful results about these modified traces, such as non-degeneracy and compatibility with the left-hand version of the construction.

Concretely, if $\modtr'$ is defined using the tuple $(P, \alpha \iota, \beta \pi)$ and $\modtr$ using $(P, \iota, \pi)$, then
\[
  \modtr' = \frac{1}{\alpha \beta} \modtr.
\]

\section{Application to the quantum function algebra}
\label{sec:traces-for-qgrp}

\begin{defn}
  \label{def:wtmod-cat}
  We say a finite-dimensional $\qgrp$-module $W$ is a \defemph{weight module} if the central subalgebra $\cent_0$ acts diagonalizably on $W$.
  We write $\wtmodc$ for the category%
  \note{
    In \cite{McPhailSnyder2020} we wrote $\mathcal{C}$ for $\wtmodc$.
  }
  of finite-dimensional weight modules, and $\overline{\wtmodc}$ for weight modules of $\qgrp^{\cop}$.
\end{defn}

The image of the functors $\vecfunc$ and $\overline{\vecfunc}$ lie in $\wtmodc$ and $\overline{\wtmodc}$, specifically in their ideals of projective objects.
Our goal in this section is to construct a modified trace on $\wtmodc$ (from which $\overline{\wtmodc}$ is an easy corollary) giving the dimensions of \cref{thm:modified-trace-exists}.
We can think of this as an expansion of Section 5.4 of \cite{Geer2018} that makes the connection to the open Hopf links of \cite{Geer2009,Geer2013topological,Geer2018trace} explicit.

\subsection{Constructing projective lifts}
We begin by explaining how to construct the lifts $\tau_V$ of $\id_V$ for objects of $\wtmodc$ by using the braiding of $\wtmodc$.
By picking well-behaved objects whenever possible, we can express the braiding using the $\hbar$-adic $R$-matrix and make explicit computations.
\begin{defn}
  \label{def:jk-mod}
  Recall the integer $\nr \ge 2$ and $2\nr$th root of unity $\xi$.
  The \defemph{Jones-Kashaev module} $\jkmod$ is generated by a highest-weight vector $v$ with
  \[
    K \cdot v = \xi^{\nr -1} v = -\xi^{-1} v \text{ and } E \cdot v = 0.
  \]
  It is $\nr$-dimensional, with basis $\{v, Ev, \dots, E^{\nr -1} v\}$.
  We can alternatively describe $\jkmod$ as the standard $\weyl$-module $\irrmod{\xi^{\nr-1}, \beta, \xi^{\nr-1}}$ for any $\beta \in \CC^\times$.
\end{defn}
The module $\jkmod$ has $\cent_0$-character $\eta_\nr(K^\nr) = (-1)^{\nr}, \eta_\nr(E^\nr) = \eta_\nr(F^\nr) = 0$, which is $(-1)^{\nr + 1}$ times the identity element of $\spec \cent_0 = \slg^*$.
As mentioned in \cref{sec:standard-cyclic-modules}, the quantum invariant corresponding to the module $\jkmod$ is Kashaev's quantum dilogarithm \cite{Kashaev1995}, equivalently \cite{Murakami2001} the $\nr$th colored Jones polynomial at a $\nr$th root of unity.
Alternately, $\jkmod$ is one of the Steinberg modules%
\note{
  Specifically, it is $\operatorname{St}$ when $\nr$ is odd, so that $\nr+1$ is even, and similarly $\overline{\operatorname{St}}$ when $\nr$ is even.
}
for $\qgrp$ defined in \cite[Corollary 3.9]{Suter1994modules}.
In particular, it is projective, simple, and has $\jkmod \iso \jkmod^*$.

Because $\jkmod^* \otimes \jkmod$ is a projective module with a map
\[
  \evdown{\jkmod} : \jkmod^* \otimes \jkmod \to \unit
\]
it contains the projective cover $P \to \unit$ of the tensor unit, and similarly $\coevdown \jkmod$ contains the injective hull of $\unit$.
More formally:
\begin{prop}
  The projective cover $P \to \unit$ is a direct summand of $\jkmod^* \otimes \jkmod$ and the restrictions
  \begin{align*}
    \pi
    &\defeq
    \left. \evdown \jkmod \right|_{P} : P \to \unit
    \\
    \iota
    &\defeq
    \coevdown \jkmod : \unit \to P \subseteq \jkmod^* \otimes \jkmod
  \end{align*}
  give a trace tuple for $\wtmodc$.
\end{prop}
\begin{marginfigure}
\[\begin{tikzcd}[column sep=small]
	& \bullet && {\xi^{2(N-1)}} \\
	& \vdots &&& {} \\
	& \bullet && {\xi^2} \\
	t && b & 1 \\
	& \bullet && {\xi^{-2}} \\
	& \vdots \\
	& \bullet && {\xi^{-2(N-1)}}
	\arrow[shift right=1, draw=myred, from=3-2, to=2-2]
	\arrow[draw=myred, from=4-1, to=3-2]
	\arrow[draw=myblue, from=3-2, to=4-3]
	\arrow[draw=myred, from=5-2, to=4-3]
	\arrow[draw=myblue, from=4-1, to=5-2]
	\arrow[shift right=1, draw=myred, from=6-2, to=5-2]
	\arrow[shift right=1, draw=myred, from=2-2, to=1-2]
	\arrow[shift right=1, color=myred, from=7-2, to=6-2]
	\arrow[shift right=1, color=myblue, from=6-2, to=7-2]
	\arrow[shift right=1, color=myblue, from=2-2, to=3-2]
	\arrow[shift right=1, color=myblue, from=1-2, to=2-2]
	\arrow[shift right=1, color=myblue, from=5-2, to=6-2]
\end{tikzcd}\]
  \caption{The weight spaces of the cover $P \to \unit$.
  Here $F$ acts by {\color{myred} red} arrows, $E$ by {\color{myblue} blue} arrows, and $K$ by the scalars in the right column.}
  \label{fig:proj-cover-diagram}
\end{marginfigure}
\begin{proof}
  By \cite[Theorem 3.7]{Suter1994modules} the module $P$ (in \citeauthor{Suter1994modules}'s notation, $\mathbf{P}_0$) is $2\nr$-dimensional and can be described by the diagram in \cref{fig:proj-cover-diagram}.
  Here the {\color{myred} red} upward arrows give the action of $F$, the {\color{myblue} blue} downward arrows give the action of $E$, and the right column describes the action of $K$.
When an arrow is missing the corresponding generator acts by $0$.

From the diagram, we see that the map $1 \mapsto b$ gives an inclusion $\unit \to P$, and similarly the quotient sending all vectors other than $t$ to zero gives a cover $P \to \unit$.
It follows that $P$ is a submodule of $\jkmod^* \otimes \jkmod$, and as discussed in \cite[Section 4]{Suter1994modules} (in particular, by Fact 1) it is a direct summand.
\end{proof}

\begin{remark}
  \label{rem:trace-norm}
  The trace tuple corresponding to the modified dimensions in \cref{thm:modified-trace-exists} is a slightly different normalization, using $\nr (-1)^{\nr+1} \pi$ instead of $\pi$.
\end{remark}

The point of expressing $P$ this way is that $\jkmod$ behaves well with respect to the braiding.
If we identify the $\cent_0$-character of $\jkmod$ with the shape%
\note{
  Here we choose $1$ as the value of $b$ for simplicity: any nonzero complex number would work.
}
\[
  \eta = ((-1)^{\nr +1}, 1, (-1)^{\nr+1}),
\]
then for any other shape $\chi = (a, b, \lambda)$ the braiding acts as
\begin{align*}
  B(\eta, \chi) &= \left( (a, (-1)^{\nr +1} b, \lambda), ((-1)^{\nr + 1}, (-1)^{\nr+1}, (-1)^{\nr +1}) \right)
  \\
  B(B(\eta,\chi)) &= (\eta,\chi).
\end{align*}
In particular, the \emph{squared} braiding $\sigma_{\jkmod, V} \sigma_{V, \jkmod} $ shown in \cref{fig:squared-braiding-jk} is an endomorphism of $\jkmod \otimes V$ for any object $V$ of $\wtmodc$.
We can use these endomorphisms to construct the lifts required to define the modified trace.
\begin{marginfigure}
\begingroup%
  \makeatletter%
  \providecommand\color[2][]{%
    \errmessage{(Inkscape) Color is used for the text in Inkscape, but the package 'color.sty' is not loaded}%
    \renewcommand\color[2][]{}%
  }%
  \providecommand\transparent[1]{%
    \errmessage{(Inkscape) Transparency is used (non-zero) for the text in Inkscape, but the package 'transparent.sty' is not loaded}%
    \renewcommand\transparent[1]{}%
  }%
  \providecommand\rotatebox[2]{#2}%
  \newcommand*\fsize{\dimexpr\f@size pt\relax}%
  \newcommand*\lineheight[1]{\fontsize{\fsize}{#1\fsize}\selectfont}%
  \ifx\svgwidth\undefined%
    \setlength{\unitlength}{53.62492418bp}%
    \ifx\svgscale\undefined%
      \relax%
    \else%
      \setlength{\unitlength}{\unitlength * \real{\svgscale}}%
    \fi%
  \else%
    \setlength{\unitlength}{\svgwidth}%
  \fi%
  \global\let\svgwidth\undefined%
  \global\let\svgscale\undefined%
  \makeatother%
  \begin{picture}(1,0.46574126)%
    \lineheight{1}%
    \setlength\tabcolsep{0pt}%
    \put(0,0){\includegraphics[width=\unitlength,page=1]{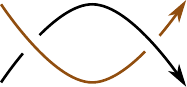}}%
    \put(1.05481302,-0.04683184){\color[rgb]{0,0,0}\makebox(0,0)[lt]{\lineheight{1.25}\smash{\begin{tabular}[t]{l}$V$\end{tabular}}}}%
    \put(1.05481302,0.37274912){\color[rgb]{0.56862745,0.30588235,0.05882353}\makebox(0,0)[lt]{\lineheight{1.25}\smash{\begin{tabular}[t]{l}$\jkmod$\end{tabular}}}}%
  \end{picture}%
\endgroup%

  \caption{The squared braiding with the Jones-Kashaev module.}
  \label{fig:squared-braiding-jk}
\end{marginfigure}

\begin{marginfigure}
\begingroup%
  \makeatletter%
  \providecommand\color[2][]{%
    \errmessage{(Inkscape) Color is used for the text in Inkscape, but the package 'color.sty' is not loaded}%
    \renewcommand\color[2][]{}%
  }%
  \providecommand\transparent[1]{%
    \errmessage{(Inkscape) Transparency is used (non-zero) for the text in Inkscape, but the package 'transparent.sty' is not loaded}%
    \renewcommand\transparent[1]{}%
  }%
  \providecommand\rotatebox[2]{#2}%
  \newcommand*\fsize{\dimexpr\f@size pt\relax}%
  \newcommand*\lineheight[1]{\fontsize{\fsize}{#1\fsize}\selectfont}%
  \ifx\svgwidth\undefined%
    \setlength{\unitlength}{83.39274502bp}%
    \ifx\svgscale\undefined%
      \relax%
    \else%
      \setlength{\unitlength}{\unitlength * \real{\svgscale}}%
    \fi%
  \else%
    \setlength{\unitlength}{\svgwidth}%
  \fi%
  \global\let\svgwidth\undefined%
  \global\let\svgscale\undefined%
  \makeatother%
  \begin{picture}(1,0.64066753)%
    \lineheight{1}%
    \setlength\tabcolsep{0pt}%
    \put(0.76820541,0.31839182){\color[rgb]{0.56862745,0.30588235,0.05882353}\makebox(0,0)[lt]{\lineheight{1.25}\smash{\begin{tabular}[t]{l}$\jkmod$\end{tabular}}}}%
    \put(0.76820541,0.03980146){\makebox(0,0)[lt]{\lineheight{1.25}\smash{\begin{tabular}[t]{l}$V$\end{tabular}}}}%
    \put(0,0){\includegraphics[width=\unitlength,page=1]{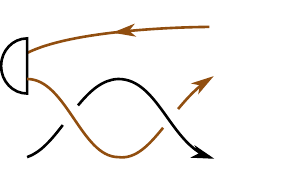}}%
    \put(0.76819465,0.50244654){\color[rgb]{0.56862745,0.30588235,0.05882353}\makebox(0,0)[lt]{\lineheight{1.25}\smash{\begin{tabular}[t]{l}$\jkmod$\end{tabular}}}}%
  \end{picture}%
\endgroup%

  \caption{The lifting map $\widetilde \tau_V : V \to \jkmod^* \otimes \jkmod \otimes V$.
  The unlabeled semicircle is the inclusion $\iota : \unit \to P$.}
  \label{fig:braiding-lift-constr}
\end{marginfigure}
Specifically, we consider the map $\widetilde \tau_V$ defined in \cref{fig:braiding-lift-constr} for any object $V$ of $\wtmodc$. 
(There is a similar map $\widetilde \sigma_V$, but we only need one of them to compute the modified traces.)
This is not quite the map $\tau_V$ of \cref{thm:lifts-exist}, for two reasons.
The first is that it has codomain $\jkmod^* \otimes \jkmod \otimes V$, not $P \otimes V$.
However, because the braiding is a $\qgrp$-module morphism, the image of $\widetilde \tau_V$ lies in  $P \otimes V$, so we can think of $\widetilde \tau_V$ as having codomain $P \otimes V$, although this is awkward to indicate diagrammatically.

The more significant problem is that $\tau_V$ might be a lift of some scalar multiple of $\id_V$, instead of $\id_V$.
Abstractly, this is not a problem, but it is if we want to be able to explicitly compute our modified traces.
In the next section, we show how to understand this scalar in terms of open Hopf links.

\subsection{Modified dimensions of weight modules}
\label{sec:mod-dim-weight-mods}
Because our functors $\vecfunc$ and $\doubfunc$ assign strands of a diagram to simple modules, to compute their modified diagram invariants (defined \cref{thm:cutting-indep-diagram}) it suffices to compute the modified dimensions
\[
  \moddim{\irrmod \chi} = \modtr \id_{\irrmod \chi}
\]
where $\irrmod \chi$ is the simple $\qgrp$-module corresponding to the extended shape $\chi$.
In this section, we prove the formula%
\note{
  This formula for the modified dimensions appears as early as \cite[Proposition 5.3]{Akutsu1992}.
}
\begin{equation}
  \label{eq:mod-dim-formula}
  \moddim{\irrmod{\chi}} =
  \begin{cases}
    \frac{\xi\mu - (\xi\mu)^{-1}}{(\xi\mu)^{\nr} - (\xi\mu)^{-\nr}}
 & \text{$\mu$ not a root of unity}
    \\
    1 & \text{$\mu$ a root of unity}
  \end{cases}
\end{equation}
given in \cref{thm:modified-trace-exists}, where $\mu = \chi(z)$ is the fractional eigenvalue of $\chi$.

\begin{defn}
  \label{defn:nilpotent-weight-mods}
  Let $\alpha \in \CC$.
  If $\alpha \in \CC \setminus \ZZ \cup (\nr - 1) \ZZ$, then we say $\wtmod \alpha$ is \defemph{admissible}.
  We write $\wtmod \alpha$ for the standard module $\irrmod{\xi^{\alpha}, 1, \xi^{\alpha}}$ and call $\wtmod \alpha$ a \defemph{semi-cyclic highest-weight module}.%
  \note{
    Because $\cent_0$ acts by scalars on $\wtmod \alpha$, highest-weight modules are weight modules in the sense of \cref{def:wtmod-cat}.
    However, for general cyclic modules $\irrmod \chi$ it doesn't make sense to consider a \emph{highest} weight because the action of $E$ has no kernel.
  }
  $\wtmod \alpha$ is a $\nr$-dimensional $\qgrp$-module $\wtmod \alpha$ generated by a \defemph{highest-weight vector} $w_\alpha$ with
  \[
    K \cdot w_{\alpha} = \xi^{\alpha} w_\alpha  \text{ and } E \cdot w_{\alpha} = 0.
  \]
  When $\alpha = \nr -1$ we recover the Jones-Kashaev module $\jkmod$.
\end{defn}
In our previous notation%
\note{
  We have a conflict between the multiplicative notation $K \cdot w_{\alpha} = \alpha w_{\alpha}$ and the additive notation $K \cdot w_{\alpha} = \xi^{\alpha} w_{\alpha}$, but we only use the second one here.
}
the highest-weight vector $w_\alpha$ is $\hat v_0$.
The isomorphism class of $\wtmod \alpha $ depends only on the value of $\alpha$ modulo $2\nr\ZZ$, and $\wtmod \alpha^* \iso \wtmod{2(\nr -1) -\alpha}$.
In general $F$ does \emph{not} act nilpotently on $\wtmod \alpha$, although $E$ does, which is why we call them \emph{semi}-cyclic modules.

Because $E$ acts nilpotently on $\wtmod \alpha$ it is straightforward to compute the formula \eqref{eq:mod-dim-formula} for highest-weight modules; we obtain the general case by another application of the quantum coadjoint action used in the proof of \cref{thm:module-classification}.
\begin{prop}
  \label{thm:nilpotent-braiding}
  For any admissible weights $\alpha$ and $\beta$ the action of the universal $R$-matrix
  \[
    \mathbf R = 
    \operatorname{HH} \;  \operatorname{exp}_q(E \otimes F)
    =
    q^{H \otimes H/2} \sum_{n = 0}^{\infty} \frac{q^{n(n-1)/2}}{\{n\}!} (E \otimes F)^{n}
  \]
  on $\wtmod \alpha \otimes \wtmod \beta$ is well-defined, and the action of $\tau \mathbf R$ defines a braiding
  \[
    S_{\alpha, \beta} :
    \wtmod \alpha \otimes \wtmod \beta \to 
    \wtmod \beta \otimes \wtmod \alpha.
  \]
  Here
  \[
    \{n\} = {q^n - q^{-n}} \text{ and } \{n\}! = \{n\} \{n-1\} \cdots \{1\}.
  \]
\end{prop}
\begin{proof}
  The formal power series $\mathbf R$ is the usual universal $R$-matrix for quantum $\lie{sl}_2$ as in \cite[Theorem XVII.4.2]{Kassel1995} and \cite{Ohtsuki2002book} using our normalization of $\qgrp$.
  We discuss the factors $\operatorname{HH}$ and $\exp_q(E \otimes F)$ separately.

  Because we are acting on highest-weight modules, we have
  \begin{align*}
    \operatorname{HH} \cdot ( w_{\alpha} \otimes w_\beta) &= \xi^{\alpha \beta/2}
    \intertext{and more generally}
    \operatorname{HH} \cdot ( E^iw_\alpha \otimes E^jw_\beta) &= \xi^{(\alpha - 2i) (\beta - 2j)/2}.
  \end{align*}
  As discussed in \cref{ch:prelim} the $q$-exponential $\exp_q$ can fail to converge when $q = \xi$ is a root of unity.
  However, because $E^\nr$ acts by $0$ on any highest-weight module, we can replace it with the truncation
  \[
    \exp_q^{<\nr}(E \otimes F) = \sum_{n = 0}^{\nr -1} \frac{q^{n(n-1)/2}}{\{n\}!} (E \otimes F)^{n}.
  \]
  The formal power series computations that give $\mathbf R \Delta = \Delta^{\op} \mathbf R$ and the Yang-Baxter relation still work when we replace $\exp_q$ with $\exp_q^{<\nr}$, so we still get a braiding.
  For details, see \cite{Ohtsuki2002book}.

  We have previously emphasized that braidings for $q = \xi$ can fail to preserve isomorphism classes, but at a crossing between weight modules this is not the case and the action of $\mathbf R$ gives a map
  \[
    \wtmod \alpha \otimes \wtmod \beta
    \to
    \wtmod \alpha \otimes \wtmod \beta
  \]
  as required.
  Because the braidings $S_{\alpha, \beta}$ are defined by the action of a universal $R$-matrix they satisfy the \reidthree{} relation exactly.
\end{proof}

To match the braidings constructed in \cref{ch:algebras}, we should strictly speaking normalize $S_{\alpha,\beta}$ by $\det S_{\alpha, \beta} = 1$.
However, we will see below that this does not affect the computation of the modified dimensions.
 
\begin{marginfigure}
\begingroup%
  \makeatletter%
  \providecommand\color[2][]{%
    \errmessage{(Inkscape) Color is used for the text in Inkscape, but the package 'color.sty' is not loaded}%
    \renewcommand\color[2][]{}%
  }%
  \providecommand\transparent[1]{%
    \errmessage{(Inkscape) Transparency is used (non-zero) for the text in Inkscape, but the package 'transparent.sty' is not loaded}%
    \renewcommand\transparent[1]{}%
  }%
  \providecommand\rotatebox[2]{#2}%
  \newcommand*\fsize{\dimexpr\f@size pt\relax}%
  \newcommand*\lineheight[1]{\fontsize{\fsize}{#1\fsize}\selectfont}%
  \ifx\svgwidth\undefined%
    \setlength{\unitlength}{54.68550396bp}%
    \ifx\svgscale\undefined%
      \relax%
    \else%
      \setlength{\unitlength}{\unitlength * \real{\svgscale}}%
    \fi%
  \else%
    \setlength{\unitlength}{\svgwidth}%
  \fi%
  \global\let\svgwidth\undefined%
  \global\let\svgscale\undefined%
  \makeatother%
  \begin{picture}(1,0.86457977)%
    \lineheight{1}%
    \setlength\tabcolsep{0pt}%
    \put(0,0){\includegraphics[width=\unitlength,page=1]{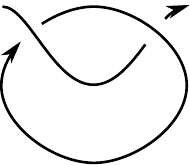}}%
    \put(1.04154654,0.34973881){\color[rgb]{0,0,0}\makebox(0,0)[lt]{\lineheight{1.25}\smash{\begin{tabular}[t]{l}$\wtmod \alpha$\end{tabular}}}}%
    \put(1.04154654,0.76118235){\color[rgb]{0,0,0}\makebox(0,0)[lt]{\lineheight{1.25}\smash{\begin{tabular}[t]{l}$\wtmod \beta$\end{tabular}}}}%
  \end{picture}%
\endgroup%

  \caption{The open Hopf link $\mathcal{H}(\alpha, \beta)$.}
  \label{fig:open-hopf-weight}
\end{marginfigure}
\begin{defn}
  For two weights $\alpha, \beta$, the associated \defemph{open Hopf link} $\mathcal{H}(\alpha, \beta)$ is given by \cref{fig:open-hopf-weight}.
  The crossings are defined using the braiding $S_{\alpha,\beta}$ of \cref{thm:nilpotent-braiding}, or in algebraic notation, $\mathcal{H}(\alpha, \beta) = S_{\alpha, \beta} S_{\beta, \alpha}$.
  Because $\wtmod \beta$ is simple, we can also consider the scalar $H(\alpha, \beta) \defeq \left\langle \mathcal{H}(\alpha, \beta)\right\rangle$, i.e.\@ by $\mathcal{H}(\alpha, \beta) = H(\alpha, \beta) \id_{\wtmod \beta}$.
\end{defn}

\begin{prop}
  \label{thm:mod-dim-ratio}
  The modified dimension of $\wtmod \alpha$ is given by
  \[
    \moddim{\wtmod \alpha} = \frac{H(\alpha, \nr-1)}{H(\nr-1, \alpha)}
  \]
  and does not depend the scalar normalization of $S_{\alpha,\beta}$.
\end{prop}
\begin{proof}
  The independence is an immediate consequence of the claimed equation because the numerator and denominator both contain $S_{\alpha, \beta}$ and $S_{\beta, \alpha}$.
  It remains to prove the first claim.
  Recall our construction of the map $\widetilde \tau = \widetilde \tau_{\wtmod \alpha} : \wtmod{\alpha} \to P \otimes \wtmod \alpha$ by using the braiding with $\jkmod$.
  To obtain a lift of the \emph{identity} map of ${\wtmod \alpha}$ and not some multiple of it, we must to consider the map $\widetilde \tau (\pi \otimes \id_{\wtmod \alpha}) : \wtmod \alpha \to \wtmod \alpha$, or diagrammatically
  \begin{center}
\begingroup%
  \makeatletter%
  \providecommand\color[2][]{%
    \errmessage{(Inkscape) Color is used for the text in Inkscape, but the package 'color.sty' is not loaded}%
    \renewcommand\color[2][]{}%
  }%
  \providecommand\transparent[1]{%
    \errmessage{(Inkscape) Transparency is used (non-zero) for the text in Inkscape, but the package 'transparent.sty' is not loaded}%
    \renewcommand\transparent[1]{}%
  }%
  \providecommand\rotatebox[2]{#2}%
  \newcommand*\fsize{\dimexpr\f@size pt\relax}%
  \newcommand*\lineheight[1]{\fontsize{\fsize}{#1\fsize}\selectfont}%
  \ifx\svgwidth\undefined%
    \setlength{\unitlength}{186.36211395bp}%
    \ifx\svgscale\undefined%
      \relax%
    \else%
      \setlength{\unitlength}{\unitlength * \real{\svgscale}}%
    \fi%
  \else%
    \setlength{\unitlength}{\svgwidth}%
  \fi%
  \global\let\svgwidth\undefined%
  \global\let\svgscale\undefined%
  \makeatother%
  \begin{picture}(1,0.28325763)%
    \lineheight{1}%
    \setlength\tabcolsep{0pt}%
    \put(0,0){\includegraphics[width=\unitlength,page=1]{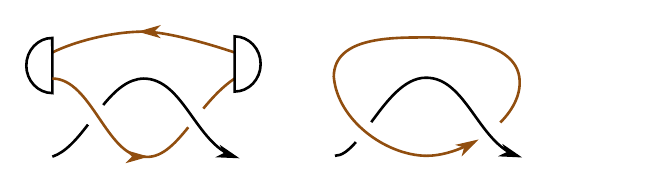}}%
    \put(0.42699448,0.12163484){\makebox(0,0)[lt]{\lineheight{1.25}\smash{\begin{tabular}[t]{l}$=$\end{tabular}}}}%
    \put(0.18754505,0.26138755){\color[rgb]{0.56862745,0.30588235,0.05882353}\makebox(0,0)[lt]{\lineheight{1.25}\smash{\begin{tabular}[t]{l}$\jkmod$\\\end{tabular}}}}%
    \put(0.37065253,0.04114637){\makebox(0,0)[lt]{\lineheight{1.25}\smash{\begin{tabular}[t]{l}$\wtmod \alpha$\end{tabular}}}}%
    \put(0.83346117,0.12163484){\makebox(0,0)[lt]{\lineheight{1.25}\smash{\begin{tabular}[t]{l}$ = \mathcal{H}(N-1, \alpha)$.\end{tabular}}}}%
  \end{picture}%
\endgroup%

  \end{center}
  Here the left semicircle represents $\iota : \unit \to P$ and the right semicircle $\pi : P \to \unit$; because $P$ is a proper submodule of $\jkmod^* \otimes \jkmod$ this is a slight abuse of notation.
  However, because $\iota$ and $\pi$ are defined in terms of the coevaluation and evaluation, we see that $\widetilde{\tau} (\pi \otimes \id_{\wtmod \alpha})$ is simply the open Hopf link $\mathcal{H}(N-1, \alpha)$, as shown above.%
  \note{
    Strictly speaking this is false: the diagram above has the opposite orientation on the strand colored by $\jkmod$ as in \cref{fig:open-hopf-weight}, so we should replace $\jkmod$ with $\jkmod^*$.
    However, these are isomorphic, so it doesn't affect the argument.
  }

  We conclude that 
  \[
    \tau_{\wtmod \alpha} \defeq \frac {\widetilde \tau} {H(N-1, \alpha)} : \wtmod \alpha \to P \otimes \wtmod \alpha
  \]
  is the required lift of $\id_{\wtmod \alpha}$, so that
  \[
    \moddim{\wtmod \alpha} = \left\langle \ptrr{\wtmod \alpha} (\tau_{\wtmod \alpha}) \right\rangle_{\iota}
    =
    \frac{\left\langle \ptrr{\wtmod \alpha} (\widetilde \tau) \right\rangle_{\iota}}{H(N-1,\alpha)}.
  \]
  Because the right partial trace of $\widetilde \tau$ is given by
  \begin{center}
\begingroup%
  \makeatletter%
  \providecommand\color[2][]{%
    \errmessage{(Inkscape) Color is used for the text in Inkscape, but the package 'color.sty' is not loaded}%
    \renewcommand\color[2][]{}%
  }%
  \providecommand\transparent[1]{%
    \errmessage{(Inkscape) Transparency is used (non-zero) for the text in Inkscape, but the package 'transparent.sty' is not loaded}%
    \renewcommand\transparent[1]{}%
  }%
  \providecommand\rotatebox[2]{#2}%
  \newcommand*\fsize{\dimexpr\f@size pt\relax}%
  \newcommand*\lineheight[1]{\fontsize{\fsize}{#1\fsize}\selectfont}%
  \ifx\svgwidth\undefined%
    \setlength{\unitlength}{83.39274502bp}%
    \ifx\svgscale\undefined%
      \relax%
    \else%
      \setlength{\unitlength}{\unitlength * \real{\svgscale}}%
    \fi%
  \else%
    \setlength{\unitlength}{\svgwidth}%
  \fi%
  \global\let\svgwidth\undefined%
  \global\let\svgscale\undefined%
  \makeatother%
  \begin{picture}(1,0.64066753)%
    \lineheight{1}%
    \setlength\tabcolsep{0pt}%
    \put(0.69625663,0.03980146){\makebox(0,0)[lt]{\lineheight{1.25}\smash{\begin{tabular}[t]{l}$\wtmod \alpha$\end{tabular}}}}%
    \put(0,0){\includegraphics[width=\unitlength,page=1]{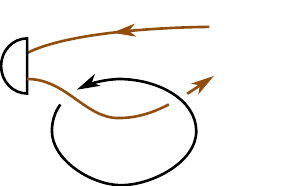}}%
    \put(0.31625788,0.59311605){\color[rgb]{0.56862745,0.30588235,0.05882353}\makebox(0,0)[lt]{\lineheight{1.25}\smash{\begin{tabular}[t]{l}$\jkmod$\end{tabular}}}}%
  \end{picture}%
\endgroup%

  \end{center}
  we similarly have
  \[
    \left\langle \ptrr{\wtmod \alpha} (\widetilde \tau) \right\rangle_{\iota} = H(\alpha, N-1).\qedhere.
  \]
\end{proof}

\begin{lem}
  $H(\alpha, \nr -1) = \nr$ and
  \begin{align*}
    H(N-1, \alpha)
    &= 
    (-1)^{\nr +1}
    \frac{\xi^{\nr(\alpha + 1)} - \xi^{-\nr(\alpha + 1)}}{\xi^{(\alpha + 1)} - \xi^{-(\alpha + 1)}}
    \\
    &=
    (-1)^{\nr +1} \left[
      \xi^{(\nr -1)(\alpha + 1)}
      +
      \xi^{(\nr -3)(\alpha + 1)}
      +
      \cdots
      +
      \xi^{-(\nr -1)(\alpha + 1)}
    \right]
  \end{align*}
\end{lem}
\begin{proof}
  By computing the action of $\mathbf R$ on highest-weight vectors we can mostly ignore the $\exp_q^{< \nr}$ factor.
  As a first example, we compute $H(\nr-1, \alpha)$.
  Let $v_{\nr -1}$ be the highest weight vector of $\jkmod$; because $\jkmod$ is simple we have
  \[
    \mathcal{H}(\nr -1, \alpha)(v_{\nr -1}) = H(\nr -1, \alpha) v_{\nr -1}.
  \]
  Consider the weight basis $w_{\alpha - 2k} = E^k w_\alpha$ for $k \in \{ 0, \dots, \nr -1\}$ of $\wtmod \alpha$ and write $w^{\alpha -2k}$ for the dual basis.
Since $E \cdot w_{\nr -1} = 0$, for any $w \in \wtmod \alpha$ we have
\[
  \mathbf{R} \cdot (w_{\nr -1} \otimes w_{\alpha -2k}) = \operatorname{HH} \cdot (w_{\nr -1} \otimes w_{\alpha -2k})
  = 
  \xi^{(\nr -1)(\alpha - 2k)/2} w_{\nr -1} \otimes w_{\alpha -2k}
\]
so the open Hopf link (writing function composition left-to-right)
\[
  \mathcal{H}(\alpha, \nr-1) =
  (\id_{\jkmod} \otimes \coevup{\wtmod \alpha}) 
  (S_{\nr -1, \alpha} \otimes \id_{\wtmod \alpha^*})
  (S_{\alpha, \nr -1} \otimes \id_{\wtmod \alpha^*})
  (\id_{\jkmod} \otimes \evdown{\wtmod \alpha}) 
\]
acts on $w_{\nr-1}$ by
\begin{align*}
  w_{\nr -1}
  &\mapsto
  \sum_{k = 0}^{\nr -1} w_{\nr-1} \otimes w_{\alpha - 2k} \otimes w^{\alpha - 2k}
  \\
  &\mapsto
\sum_{k = 0}^{\nr -1} 
  \xi^{(\nr -1)(\alpha - 2k)/2} 
  w_{\alpha - 2k} \otimes w_{\nr-1} \otimes w^{\alpha - 2k}
  \\
  &\mapsto
\sum_{k = 0}^{\nr -1} 
  \xi^{(\nr -1)(\alpha - 2k)} 
  w_{\nr-1} \otimes w_{\alpha - 2k} \otimes w^{\alpha - 2k}
  + \left(\text{lower-weight terms}\right)
  \\
  &\mapsto
\sum_{k = 0}^{\nr -1} 
  \xi^{(\nr -1)(\alpha - 2k)} 
  w_{\nr-1} \xi^{(1-\nr)(\alpha - 2k)}w^{\alpha -2k}(w_{\alpha - 2k})
  + 0
  \\
  &=
  w_{\nr-1} \sum_{k=0}^{\nr -1} 1
  =
  \nr w_{\nr -1}
\end{align*}
and $H(\alpha, \nr-1) = \nr$ as claimed.
The key part of this computation is the third step: by ``lower-weight terms'' we mean a sum of vectors of the form
\[
  F^l w_{\nr -1} \otimes E^l w_{\alpha - 2k} \otimes w^{\alpha - 2k}, \quad 0 < l < \nr.
\]
The first tensor factor has a weight strictly lower than $w_{\nr -1}$, so terms of this form can never contribute%
\note{
The reader who is unsatisfied with this justification can work out the details: when we apply the coevaluation map these vectors will contribute terms like $\sum_{n = 0}^{\nr -1} \xi^{2 l n} = 0$ to the final answer.
}
to the scalar by which $\mathcal{H}(\alpha, \nr-1)$ acts on $w_{\nr -1}$.

By a parallel argument, we see that
\begin{align*}
  H(\nr-1, \alpha)
  &=
  \sum_{k = 0}^{\nr -1} \xi^{\alpha(\nr - 1 - 2k)} \xi^{(1 - \nr)(\nr -1 - 2k)}
  \\
  &=
  \sum_{k = 0}^{\nr -1} \xi^{(\alpha + 1 - \nr)(\nr - 1 - 2k)}
  \\
  &=
  \xi^{(\alpha + 1 - \nr)(\nr - 1)} + \xi^{(\alpha + 1 - \nr)(\nr -3)} + \cdots + \xi^{-(\alpha + 1 - \nr)(\nr -1)}
  \\
  &=
  \frac{\xi^{\nr(\alpha + 1 - \nr)} - \xi^{-\nr(\alpha + 1 - \nr)}}{\xi^{(\alpha + 1 - \nr)} - \xi^{-(\alpha + 1 - \nr)}}
  \\
  &=
  (-1)^{\nr +1}
  \frac{\xi^{\nr(\alpha + 1)} - \xi^{-\nr(\alpha + 1)}}{\xi^{(\alpha + 1)} - \xi^{-(\alpha + 1)}}
  \qedhere
\end{align*}
\end{proof}

We can now prove \cref{thm:modified-trace-exists}.
By \cref{thm:mod-dim-ratio} and the previous lemma, the modified dimension of the module $\wtmod \alpha$ associated to the trace tuple $(P, \iota, \pi)$ is
\[
  \frac{H(\alpha,\nr-1)}{H(\nr-1,\alpha)} = 
  \nr (-1)^{\nr +1}
  \frac{\xi^{(\alpha + 1)} - \xi^{-(\alpha + 1)}}{\xi^{\nr(\alpha + 1)} - \xi^{-\nr(\alpha + 1)}}.
  =
  \frac{\xi\mu - (\xi\mu)^{-1}}{(\xi\mu)^{\nr} - (\xi\mu)^{-\nr}}
\]
since $\mu = \xi^{\alpha}$.
This formula differs from \cref{thm:modified-trace-exists,eq:mod-dim-formula} by a factor of $\nr(-1)^{\nr +1}$, so to obtain those we use the trace tuple $(P, \iota, \nr(-1)^{\nr+1})$ as in \cref{rem:trace-norm}.

We have only proved this formula for modules of the form $\wtmod \alpha = \irrmod{\xi^{\alpha}, 1, \xi^{\alpha}}$, so it remains to prove that it works for the general case
\[
  V = \irrmod{\xi^{\alpha'}, \beta, \xi^{\alpha}} 
\]
with $\alpha'$ not necessarily equal to $\alpha$.
One method is to again use the quantum coadjoint action, as in the proof of \cref{thm:module-classification-parabolic}.
This shows that the modified dimension of $\irrmod{\chi, \mu}$ depends only on the action of the Casimir $\Omega$, hence on the value of $\mu = \xi^{\alpha}$.
More details of this argument are given in \cite[proof of Lemma 21]{Geer2018trace}.

Another method is to use diagrammatic arguments as in \cref{sec:internal-gauge-transf} to show that the modified dimension function $\moddim{V}$ must be gauge invariant.
It is not hard to show that the open Hopf links%
\note{
  Here we modify the notation from before, writing $\mathcal{H}(\wtmod{\alpha}, \wtmod{\beta})$ for $\mathcal{H}(\alpha,\beta)$.
}
satisfy $\mathcal{H}(V, \jkmod) = \nr \id_V$ for any module $V$, so the key computation is understanding $\mathcal{H}(\jkmod, V)$.
By manipulations like
\begin{center}
\begingroup%
  \makeatletter%
  \providecommand\color[2][]{%
    \errmessage{(Inkscape) Color is used for the text in Inkscape, but the package 'color.sty' is not loaded}%
    \renewcommand\color[2][]{}%
  }%
  \providecommand\transparent[1]{%
    \errmessage{(Inkscape) Transparency is used (non-zero) for the text in Inkscape, but the package 'transparent.sty' is not loaded}%
    \renewcommand\transparent[1]{}%
  }%
  \providecommand\rotatebox[2]{#2}%
  \newcommand*\fsize{\dimexpr\f@size pt\relax}%
  \newcommand*\lineheight[1]{\fontsize{\fsize}{#1\fsize}\selectfont}%
  \ifx\svgwidth\undefined%
    \setlength{\unitlength}{386.97843933bp}%
    \ifx\svgscale\undefined%
      \relax%
    \else%
      \setlength{\unitlength}{\unitlength * \real{\svgscale}}%
    \fi%
  \else%
    \setlength{\unitlength}{\svgwidth}%
  \fi%
  \global\let\svgwidth\undefined%
  \global\let\svgscale\undefined%
  \makeatother%
  \begin{picture}(1,0.30414446)%
    \lineheight{1}%
    \setlength\tabcolsep{0pt}%
    \put(0,0){\includegraphics[width=\unitlength,page=1]{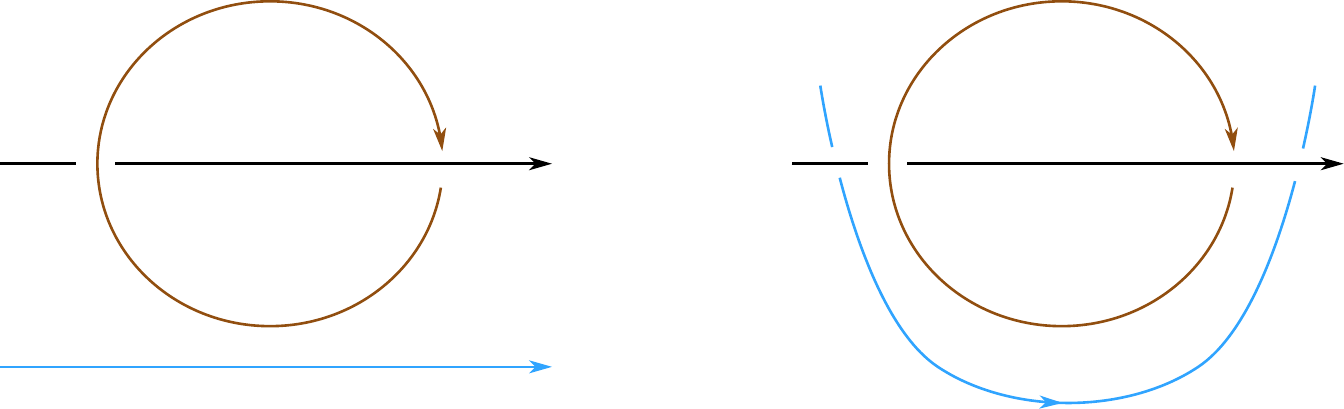}}%
    \put(0.48452314,0.18229326){\makebox(0,0)[lt]{\lineheight{1.25}\smash{\begin{tabular}[t]{l}$\leftrightarrow$\end{tabular}}}}%
    \put(0.31009479,0.25981696){\color[rgb]{0.56862745,0.30588235,0.05882353}\makebox(0,0)[lt]{\lineheight{1.25}\smash{\begin{tabular}[t]{l}$\jkmod$\end{tabular}}}}%
    \put(0.91276864,0.26258567){\color[rgb]{0.56862745,0.30588235,0.05882353}\makebox(0,0)[lt]{\lineheight{1.25}\smash{\begin{tabular}[t]{l}$\jkmod$\end{tabular}}}}%
    \put(0.34885665,0.14353142){\makebox(0,0)[lt]{\lineheight{1.25}\smash{\begin{tabular}[t]{l}$V$\end{tabular}}}}%
    \put(0.77523695,0.2016742){\makebox(0,0)[lt]{\lineheight{1.25}\smash{\begin{tabular}[t]{l}$V_\alpha$\end{tabular}}}}%
    \put(0.18605569,0.2016742){\makebox(0,0)[lt]{\lineheight{1.25}\smash{\begin{tabular}[t]{l}$V$\end{tabular}}}}%
    \put(0.9776527,0.14415271){\makebox(0,0)[lt]{\lineheight{1.25}\smash{\begin{tabular}[t]{l}$V$\end{tabular}}}}%
  \end{picture}%
\endgroup%

\end{center}
we can show that $H(V, \jkmod)$ depends only on the gauge class of $V$, which reduces the computation of the modified dimensions to the case of semi-cyclic highest-weight modules.
(Here the {\color{myblue} blue} strand is colored by some module chosen to gauge-transform $V$ to a highest-weight module.)

\section{Traces for external tensor products}
\label{sec:double-traces}
Using the formalism of \cref{sec:modified-trace-construction}, it is not hard to compute modified traces for $\qgrp \otimes \qgrp^{\cop}$-modules.
\begin{defn}
  A $\qgrp \otimes \qgrp^{\cop}$-module $X$ is a \defemph{weight module} if $\cent_0 \otimes \cent_0$ acts diagonalizably on $X$.
  We write $\dwtmodc$ for the category%
  \note{
    In \cite{McPhailSnyder2020} we wrote $\mathcal{D}$ for $\dwtmodc$.
  }
  of  weight modules that are \defemph{locally homogeneous} in the sense that for any $Z \in \cent_0$,
  \[
    Z \boxtimes 1 \cdot w = 1 \boxtimes S(Z) \cdot w.
  \]
\end{defn}
In particular, the modules $\doubmod{\chi}[\epsilon_1 \epsilon_2]$ are objects of $\dwtmodc$ and the image of $\doubfunc$ lies in $\proj(\dwtmodc)$.
We do not explicitly use the local homogeneity condition, but it is included so that $\dwtmodc$ becomes a $\slg^*$-graded category.

\begin{thm}
  \label{thm:double-traces}
  Let $\operatorname{Proj}(\dwtmodc)$ be the subcategory of projective $\qgrp \otimes \qgrp^{\cop}$-modules in $\catl D$.
  $\operatorname{Proj}(\dwtmodc)$ admits a nontrivial modified trace which is compatible with the trace for $\wtmodc$ in the following sense:
  let $X$ be a projective object of $\wtmodc$  and $\overline{X}$ a projective object of $\overline{\wtmodc}$.
  Then for any endomorphisms $f : V \to V$, $g : \overline{V} \otimes \overline{V}$,
  \[
    \modtr(f \boxtimes g) = \modtr(f) \modtr(g)
  \]
  with $f \boxtimes g$ the obvious endomorphism of $V \boxtimes \overline{V}$.
  In particular, the modified dimensions for $\dwtmodc$ are given by
  \begin{align*}
    \moddim{ \irrmod{\chi, \mu} \boxtimes \irrmod{\chi', \mu'} } = 
    \moddim{ \irrmod{\chi, \mu}} \moddim{\irrmod{\chi', \mu'} }.
  \end{align*}
\end{thm}
\begin{proof}
  Suppose the trace for $\proj(\wtmodc)$ is given by the trace tuple $(P, \iota, \pi)$.
  It is not hard to see that $(P, \iota, \pi)$ is also a trace tuple for $\proj(\overline{\wtmodc})$.
  The modified trace on $\dwtmodc$ is constructed using the trace tuple
  \[
    (P_0 \boxtimes P_0^*, \iota \boxtimes \iota, \pi \boxtimes \pi).
  \]

  We need to show the compatibility of the traces.
  Choose lifts $\tau_V, \tau_{\overline{V}}$ as usual.
  Then the diagram
  \[
    \begin{tikzcd}
      & ( P_0 \boxtimes P_0^* ) \otimes ( V \boxtimes \overline{V} ) \arrow[d, "(\pi \boxtimes \pi) \otimes (\id_V \boxtimes \id_{\overline{V}})"] \\
      V \boxtimes \overline{V} \arrow[ur, "\tau_V \boxtimes \tau_{\overline{V}}"] \arrow[r, swap, "\id_{V \boxtimes \overline{V}}"] & V \boxtimes \overline{V}
    \end{tikzcd}
  \]
  commutes, so $\tau_V \boxtimes \tau_{\overline{V}}$ is a lift for $V \boxtimes \overline{V}$.
  But then we can use the compatibility of the pivotal structures to write
  \begin{align*}
    \modtr(f \boxtimes g) &= \left\langle \ptr_{V \boxtimes \overline{V}}^r( (\tau_V \boxtimes \tau_{\overline{V}}) (f \boxtimes g) ) \right\rangle_{\iota \boxtimes \iota} \\
                          &= \left\langle \ptr_V^r(\tau_V f) \boxtimes \ptr_{\overline{V}}^r(\tau_{\overline{V}} g) \right\rangle_{\iota \boxtimes \iota} \\
                          &= \left\langle \ptr_V^r(\tau_V f) \right\rangle_\iota \left \langle \ptr_{\overline{V}}^r(\tau_{\overline{V}} g) \right\rangle_{\iota} \\
                          &= \modtr(f) \modtr(g).\qedhere
  \end{align*}
\end{proof}

\section{Cutting presentations and invariants of closed diagrams}
We can now prove that modified traces give invariants of closed diagrams.
If $D$ is a closed $X$-colored diagram for some biquandle $X$, the image of $D$ under a model $\mathcal{F}$ of $X$ in $\catl C$ might vanish for nontrivial $D$.
However, we can extract an invariant of $D$ by cutting it open to obtain a $(1,1)$-tangle $T$ (i.e.\@ by choosing a cutting presentation) and computing $\modtr(\mathcal{F}(T))$.
This  is well-defined by  \cite[Theorem 5]{Geer2013a}.
We give a self-contained version of the relevant part of their theorem.

\begin{thm}
  \label{thm:cutting-indep-diagram}
  Let $D$ be a closed $X$-colored tangle diagram and $T$ a cutting presentation of $D$, and let $\mathcal{F} : \tang[X] \to \catl C$ be a representation of $X$ in $\catl C$ compatible with the modified trace $\modtr$.
  Then the scalar $\modtr(\mathcal{F}(T))$ does not depend on the choice of $T$.
  We call $\invl F(D) \defeq \modtr(\mathcal{F}(T))$ the \defemph{modified diagram invariant} associated to $\mathcal{F}$ and $\modtr$.
\end{thm}

Before giving the proof we mention a few technical points about pivotal categories and prove a lemma (which is essentially \cite[Lemma 4 (a)]{Geer2013a}).
The pivotal structure on $\catl C$ gives a monoidal natural isomorphism $\phi$ between the identity and the double dual which acts on objects $X$ of $\catl C$ as
\[
  \phi_X : X \to X^{**} = (\evdown X \otimes \id_{X^{**}}) (\id_{X} \otimes \coevup{X^*})
\]
The string diagram of $\phi_X$ looks like a snake, but is a bit awkward to draw (try it yourself!) because the arrows collide.
This is because the whole point of a pivotal category is that you can straighten out snake diagrams.
As such, we typically do not indicate the morphisms $\phi_{-}$ in our diagrams, but use them implicitly whenever we need to identify $X$ and $X^{**}$.

Second, recall that for any morphism $f : V \to W$ in a pivotal category we have the dual $f^* : W^* \to V^*$ defined by 
\begin{align*}
  f^* &= (\id_{Y^*} \otimes \coevup X ) (\id_{Y^*} \otimes f \otimes \id_{X^*}) (\evup Y \otimes \id_{X^*})
  \\
      &= (\coevdown X \otimes \id_{Y^*})(\id_{X^*} \otimes f \otimes \id_{Y^*})(\id_{X^*} \otimes \evdown Y)
\end{align*}
As usual, a better way to explain what $f^*$ is is to look at the diagram in \cref{fig:pivotal-dual}.

\begin{marginfigure}
  \centering
\begingroup%
  \makeatletter%
  \providecommand\color[2][]{%
    \errmessage{(Inkscape) Color is used for the text in Inkscape, but the package 'color.sty' is not loaded}%
    \renewcommand\color[2][]{}%
  }%
  \providecommand\transparent[1]{%
    \errmessage{(Inkscape) Transparency is used (non-zero) for the text in Inkscape, but the package 'transparent.sty' is not loaded}%
    \renewcommand\transparent[1]{}%
  }%
  \providecommand\rotatebox[2]{#2}%
  \newcommand*\fsize{\dimexpr\f@size pt\relax}%
  \newcommand*\lineheight[1]{\fontsize{\fsize}{#1\fsize}\selectfont}%
  \ifx\svgwidth\undefined%
    \setlength{\unitlength}{90.1547699bp}%
    \ifx\svgscale\undefined%
      \relax%
    \else%
      \setlength{\unitlength}{\unitlength * \real{\svgscale}}%
    \fi%
  \else%
    \setlength{\unitlength}{\svgwidth}%
  \fi%
  \global\let\svgwidth\undefined%
  \global\let\svgscale\undefined%
  \makeatother%
  \begin{picture}(1,0.59718143)%
    \lineheight{1}%
    \setlength\tabcolsep{0pt}%
    \put(0,0){\includegraphics[width=\unitlength,page=1]{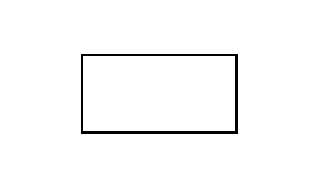}}%
    \put(0.48877259,0.26533803){\makebox(0,0)[lt]{\lineheight{1.25}\smash{\begin{tabular}[t]{l}$f$\end{tabular}}}}%
    \put(0,0){\includegraphics[width=\unitlength,page=2]{pivotal-dual.pdf}}%
    \put(1.0247975,0.48602258){\makebox(0,0)[lt]{\lineheight{1.25}\smash{\begin{tabular}[t]{l}$V$\end{tabular}}}}%
    \put(-0.22737733,0.0457458){\makebox(0,0)[lt]{\lineheight{1.25}\smash{\begin{tabular}[t]{l}$W$\end{tabular}}}}%
  \end{picture}%
\endgroup%

  \caption{The dual of a map $f : V \to W$.}
  \label{fig:pivotal-dual}
\end{marginfigure}

\begin{lem}
  \label{thm:rotation-lemma}
  Let $f \in \End_{\catl C}(X_0 \otimes X_1^{*})$.
  Then for any modified trace on an ideal $I$ containing $X_0$ and $X_1$, we have
  \[
    \modtr_{X_1} \left( \phi_{X_1} \left(\ptr^l_{X_0} (f)\right)^* \phi_{X_1}^{-1}\right)
    =
    \modtr_{X_0} \left( \ptr^r_{X_1^*}(f)\right)
  \]
  or, pictorially,
  \begin{equation*}
\begingroup%
  \makeatletter%
  \providecommand\color[2][]{%
    \errmessage{(Inkscape) Color is used for the text in Inkscape, but the package 'color.sty' is not loaded}%
    \renewcommand\color[2][]{}%
  }%
  \providecommand\transparent[1]{%
    \errmessage{(Inkscape) Transparency is used (non-zero) for the text in Inkscape, but the package 'transparent.sty' is not loaded}%
    \renewcommand\transparent[1]{}%
  }%
  \providecommand\rotatebox[2]{#2}%
  \newcommand*\fsize{\dimexpr\f@size pt\relax}%
  \newcommand*\lineheight[1]{\fontsize{\fsize}{#1\fsize}\selectfont}%
  \ifx\svgwidth\undefined%
    \setlength{\unitlength}{308.37651443bp}%
    \ifx\svgscale\undefined%
      \relax%
    \else%
      \setlength{\unitlength}{\unitlength * \real{\svgscale}}%
    \fi%
  \else%
    \setlength{\unitlength}{\svgwidth}%
  \fi%
  \global\let\svgwidth\undefined%
  \global\let\svgscale\undefined%
  \makeatother%
  \begin{picture}(1,0.28876853)%
    \lineheight{1}%
    \setlength\tabcolsep{0pt}%
    \put(0,0){\includegraphics[width=\unitlength,page=1]{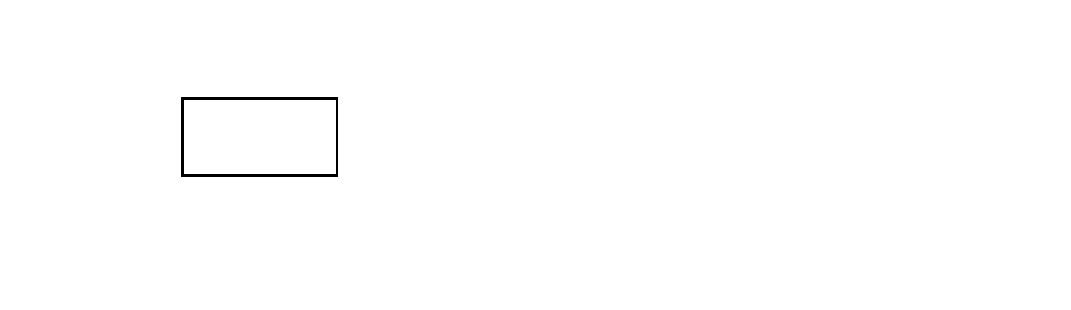}}%
    \put(0.23662088,0.15130268){\makebox(0,0)[lt]{\lineheight{1.25}\smash{\begin{tabular}[t]{l}$f$\end{tabular}}}}%
    \put(0,0){\includegraphics[width=\unitlength,page=2]{rotation-lemma-equality.pdf}}%
    \put(0.37173745,0.19496132){\makebox(0,0)[lt]{\lineheight{1.25}\smash{\begin{tabular}[t]{l}$X_0$\end{tabular}}}}%
    \put(0.40595711,0.08632695){\makebox(0,0)[lt]{\lineheight{1.25}\smash{\begin{tabular}[t]{l}$X_1^*$\end{tabular}}}}%
    \put(0.46450536,0.14065755){\makebox(0,0)[lt]{\lineheight{1.25}\smash{\begin{tabular}[t]{l}$=$\end{tabular}}}}%
    \put(0,0){\includegraphics[width=\unitlength,page=3]{rotation-lemma-equality.pdf}}%
    \put(0.70358268,0.15130271){\makebox(0,0)[lt]{\lineheight{1.25}\smash{\begin{tabular}[t]{l}$f$\end{tabular}}}}%
    \put(0,0){\includegraphics[width=\unitlength,page=4]{rotation-lemma-equality.pdf}}%
    \put(0.84313739,0.18148026){\makebox(0,0)[lt]{\lineheight{1.25}\smash{\begin{tabular}[t]{l}$X_0$\end{tabular}}}}%
    \put(0.83072273,0.10989582){\makebox(0,0)[lt]{\lineheight{1.25}\smash{\begin{tabular}[t]{l}$X_1^*$\end{tabular}}}}%
    \put(0.50318646,0.1406559){\makebox(0,0)[lt]{\lineheight{1.25}\smash{\begin{tabular}[t]{l}$\modtr_{X_0}$\end{tabular}}}}%
    \put(-0.00323012,0.1406559){\makebox(0,0)[lt]{\lineheight{1.25}\smash{\begin{tabular}[t]{l}$\modtr_{X_1}$\end{tabular}}}}%
  \end{picture}%
\endgroup%

  \end{equation*}
\end{lem}
\begin{marginfigure}
  \centering
\begingroup%
  \makeatletter%
  \providecommand\color[2][]{%
    \errmessage{(Inkscape) Color is used for the text in Inkscape, but the package 'color.sty' is not loaded}%
    \renewcommand\color[2][]{}%
  }%
  \providecommand\transparent[1]{%
    \errmessage{(Inkscape) Transparency is used (non-zero) for the text in Inkscape, but the package 'transparent.sty' is not loaded}%
    \renewcommand\transparent[1]{}%
  }%
  \providecommand\rotatebox[2]{#2}%
  \newcommand*\fsize{\dimexpr\f@size pt\relax}%
  \newcommand*\lineheight[1]{\fontsize{\fsize}{#1\fsize}\selectfont}%
  \ifx\svgwidth\undefined%
    \setlength{\unitlength}{90.1547699bp}%
    \ifx\svgscale\undefined%
      \relax%
    \else%
      \setlength{\unitlength}{\unitlength * \real{\svgscale}}%
    \fi%
  \else%
    \setlength{\unitlength}{\svgwidth}%
  \fi%
  \global\let\svgwidth\undefined%
  \global\let\svgscale\undefined%
  \makeatother%
  \begin{picture}(1,0.59718143)%
    \lineheight{1}%
    \setlength\tabcolsep{0pt}%
    \put(0,0){\includegraphics[width=\unitlength,page=1]{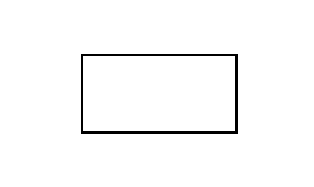}}%
    \put(0.48877259,0.26533803){\makebox(0,0)[lt]{\lineheight{1.25}\smash{\begin{tabular}[t]{l}$f$\end{tabular}}}}%
    \put(0,0){\includegraphics[width=\unitlength,page=2]{rotation-lemma-alpha.pdf}}%
    \put(1.03855828,0.35402119){\makebox(0,0)[lt]{\lineheight{1.25}\smash{\begin{tabular}[t]{l}$X_0$\end{tabular}}}}%
    \put(1.0482985,0.16416703){\makebox(0,0)[lt]{\lineheight{1.25}\smash{\begin{tabular}[t]{l}$X_1^*$\end{tabular}}}}%
    \put(0,0){\includegraphics[width=\unitlength,page=3]{rotation-lemma-alpha.pdf}}%
  \end{picture}%
\endgroup%

  \caption{The map $\alpha : X_1^* \otimes X_1 \to X_0^* \otimes X_0$.}
  \label{fig:rotation-lemma-alpha}
\end{marginfigure}
\begin{marginfigure}
  \centering
\begingroup%
  \makeatletter%
  \providecommand\color[2][]{%
    \errmessage{(Inkscape) Color is used for the text in Inkscape, but the package 'color.sty' is not loaded}%
    \renewcommand\color[2][]{}%
  }%
  \providecommand\transparent[1]{%
    \errmessage{(Inkscape) Transparency is used (non-zero) for the text in Inkscape, but the package 'transparent.sty' is not loaded}%
    \renewcommand\transparent[1]{}%
  }%
  \providecommand\rotatebox[2]{#2}%
  \newcommand*\fsize{\dimexpr\f@size pt\relax}%
  \newcommand*\lineheight[1]{\fontsize{\fsize}{#1\fsize}\selectfont}%
  \ifx\svgwidth\undefined%
    \setlength{\unitlength}{35.15036201bp}%
    \ifx\svgscale\undefined%
      \relax%
    \else%
      \setlength{\unitlength}{\unitlength * \real{\svgscale}}%
    \fi%
  \else%
    \setlength{\unitlength}{\svgwidth}%
  \fi%
  \global\let\svgwidth\undefined%
  \global\let\svgscale\undefined%
  \makeatother%
  \begin{picture}(1,0.91222266)%
    \lineheight{1}%
    \setlength\tabcolsep{0pt}%
    \put(0,0){\includegraphics[width=\unitlength,page=1]{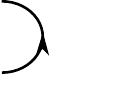}}%
    \put(-0.02833804,0.08051345){\makebox(0,0)[lt]{\lineheight{1.25}\smash{\begin{tabular}[t]{l}$X_0$\end{tabular}}}}%
    \put(0.60233045,0.087293){\makebox(0,0)[lt]{\lineheight{1.25}\smash{\begin{tabular}[t]{l}$X_1^*$\end{tabular}}}}%
    \put(0,0){\includegraphics[width=\unitlength,page=2]{rotation-lemma-beta.pdf}}%
  \end{picture}%
\endgroup%

  \caption{The map $\beta : X_0 \otimes X_0^* \to X_1^* \otimes X_1$.}
  \label{fig:rotation-lemma-beta}
\end{marginfigure}
\begin{proof}
  The trick is to use the maps $\alpha$ and $\beta$ defined in \cref{fig:rotation-lemma-alpha,fig:rotation-lemma-beta}.
  By using cyclicity and compatibility with partial traces, we have
  \[
    \modtr_{X_1}\left(\ptr^l_{X_1^*}(\alpha \beta)\right)
    =
    \modtr_{X_1 \otimes X_1^*}(\alpha \beta)
    =
    \modtr{X_0^* \otimes X_0}(\beta \alpha)
    =
    \modtr_{X_0}\left(\ptr^l_{X_0^*} (\beta \alpha) \right).
  \]
  But $\ptr^l_{X_1^*}(\alpha \beta)$ is equal to 
  \begin{center}
\begingroup%
  \makeatletter%
  \providecommand\color[2][]{%
    \errmessage{(Inkscape) Color is used for the text in Inkscape, but the package 'color.sty' is not loaded}%
    \renewcommand\color[2][]{}%
  }%
  \providecommand\transparent[1]{%
    \errmessage{(Inkscape) Transparency is used (non-zero) for the text in Inkscape, but the package 'transparent.sty' is not loaded}%
    \renewcommand\transparent[1]{}%
  }%
  \providecommand\rotatebox[2]{#2}%
  \newcommand*\fsize{\dimexpr\f@size pt\relax}%
  \newcommand*\lineheight[1]{\fontsize{\fsize}{#1\fsize}\selectfont}%
  \ifx\svgwidth\undefined%
    \setlength{\unitlength}{125.49554729bp}%
    \ifx\svgscale\undefined%
      \relax%
    \else%
      \setlength{\unitlength}{\unitlength * \real{\svgscale}}%
    \fi%
  \else%
    \setlength{\unitlength}{\svgwidth}%
  \fi%
  \global\let\svgwidth\undefined%
  \global\let\svgscale\undefined%
  \makeatother%
  \begin{picture}(1,0.51410703)%
    \lineheight{1}%
    \setlength\tabcolsep{0pt}%
    \put(0,0){\includegraphics[width=\unitlength,page=1]{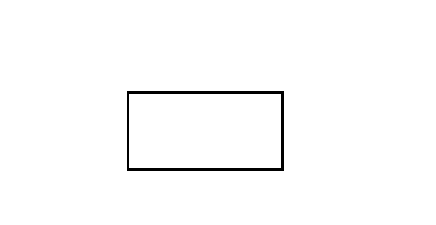}}%
    \put(0.45660301,0.19061626){\makebox(0,0)[lt]{\lineheight{1.25}\smash{\begin{tabular}[t]{l}$f$\end{tabular}}}}%
    \put(0,0){\includegraphics[width=\unitlength,page=2]{rotation-lemma-comp-i.pdf}}%
    \put(0.83015521,0.08751553){\makebox(0,0)[lt]{\lineheight{1.25}\smash{\begin{tabular}[t]{l}$X_1^*$\end{tabular}}}}%
    \put(0.9246684,0.48939046){\makebox(0,0)[lt]{\lineheight{1.25}\smash{\begin{tabular}[t]{l}$X_1^*$\end{tabular}}}}%
    \put(0.44012853,0.32225909){\makebox(0,0)[lt]{\lineheight{1.25}\smash{\begin{tabular}[t]{l}$X_0$\end{tabular}}}}%
  \end{picture}%
\endgroup%

  \end{center}
  which is just $\phi_{X_1} \left(\ptr^l_{X_0} (f)\right)^* \phi_{X_1}^{-1}$.
  Similarly, $\ptr^l_{X_0^*}(\beta \alpha) = \ptr^r_{X_1^*}(f)$, and the lemma follows.
\end{proof}

\begin{proof}[Proof of \cref{thm:cutting-indep-diagram}]
  Suppose we have two cutting presentations $T_i \in \End_{\tang[X]} (x_i, \epsilon_i)$ for $i \in 0, 1$ of the same diagram $D$.
  Without loss of generality, we can draw $D$ as
  \begin{center}
\begingroup%
  \makeatletter%
  \providecommand\color[2][]{%
    \errmessage{(Inkscape) Color is used for the text in Inkscape, but the package 'color.sty' is not loaded}%
    \renewcommand\color[2][]{}%
  }%
  \providecommand\transparent[1]{%
    \errmessage{(Inkscape) Transparency is used (non-zero) for the text in Inkscape, but the package 'transparent.sty' is not loaded}%
    \renewcommand\transparent[1]{}%
  }%
  \providecommand\rotatebox[2]{#2}%
  \newcommand*\fsize{\dimexpr\f@size pt\relax}%
  \newcommand*\lineheight[1]{\fontsize{\fsize}{#1\fsize}\selectfont}%
  \ifx\svgwidth\undefined%
    \setlength{\unitlength}{131.33773327bp}%
    \ifx\svgscale\undefined%
      \relax%
    \else%
      \setlength{\unitlength}{\unitlength * \real{\svgscale}}%
    \fi%
  \else%
    \setlength{\unitlength}{\svgwidth}%
  \fi%
  \global\let\svgwidth\undefined%
  \global\let\svgscale\undefined%
  \makeatother%
  \begin{picture}(1,0.41418832)%
    \lineheight{1}%
    \setlength\tabcolsep{0pt}%
    \put(0,0){\includegraphics[width=\unitlength,page=1]{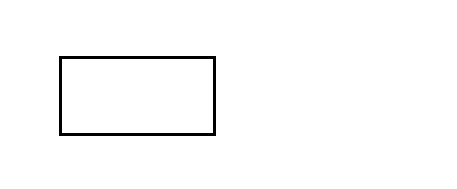}}%
    \put(0.28785193,0.18196405){\makebox(0,0)[lt]{\lineheight{1.25}\smash{\begin{tabular}[t]{l}$E$\end{tabular}}}}%
    \put(0,0){\includegraphics[width=\unitlength,page=2]{cutting-indep-i.pdf}}%
    \put(0.59857535,0.30254118){\makebox(0,0)[lt]{\lineheight{1.25}\smash{\begin{tabular}[t]{l}$x_0$\end{tabular}}}}%
    \put(0.60254277,0.08833793){\makebox(0,0)[lt]{\lineheight{1.25}\smash{\begin{tabular}[t]{l}$x_1^*$\end{tabular}}}}%
  \end{picture}%
\endgroup%

  \end{center}
  for some morphism $E : (x_0, +) \otimes (x_1, -) \to (x_0, +) \otimes (x_1, -)$.
  Cutting along the top edge produces $T_0$ and cutting along the bottom edge produces $T_1$, and our claim is that
  \[
    \modtr_{\mathcal{F}(x_0)} \mathcal{F}(T_0) = \modtr_{\mathcal{F}(x_1)} \mathcal{F}(T_1).
  \]
This follows from \cref{thm:rotation-lemma} applied to $f = \mathcal{F}(E)$.
\end{proof}

\chapter{The cyclic quantum dilogarithm}
\label{ch:quantum-dilogarithm}

Let $A, B \in \CC\setminus\{0,1\}$ satisfy $A^\nr + B^\nr = 1$.
The \defemph{cyclic quantum dilogarithm} is given by
\[
  \qlog{B,A}{m} = A^{-m} \prod_{k=1}^{m} (1 - \xi^{2k} B).
\]
In this appendix we prove some properties of this and a related normalization
\[
  \qlogn{B,A}{m} = g(B,A) \qlog{B,A}{m}.
\]

The name ``dilogarithm'' can appear strange without context.
A variant of the function $m \mapsto \qlog{B,A}{m}$ function was introduced in \cite{Faddeev1994} as a truncated version of the quantum dilogarithm
\[
  \Psi(x) = \prod_{n=1}^{\infty} (1-q^{n} x)
\]
which is interpreted there as an analog of the classical dilogarithm
\[
  \operatorname{Li}_2(x) = - \int_0^{x} \frac{\log(1-z)}{z} \, dz.
\]
The function $\operatorname{Li}_2$ shows up in a number of places in mathematics, including (via the closely related Rogers dilogarithm) the computation of the complex volumes of hyperbolic $3$-manifolds \cite{Zickert2009}.

One perspective on our holonomy invariants is that they are cyclic analogues of the complex volume obtained by replacing the classical dilogarithm with the cyclic quantum dilogarithm.
This idea appears to have motivated \citeauthor{Kashaev1995}'s construction of the dilogarithm invariant \cite{Kashaev1995} and statement of the volume conjecture \cite{Kashaev1997}, and \citeauthor{Baseilhac2004} \cite{Baseilhac2004} explicitly take this perspective on their invariants, which appear to be closely related to $\vecfunc$.

Particularly important are the formulas for the Fourier transforms of the cyclic dilogarithm and its inverse derived in \cref{sec:fourier-transforms}.
These appear to be ``well-known'' but not written down in most sources, so we follow the method given in \cite[Appendix D]{Hikami2014}.

\section{Basic definitions}

\begin{defn}
  The \defemph{shifted quantum factorial}%
  is defined for $n \in \ZZ$ by 
  \[
    \qfac{x}[q]{n} \defeq
    \begin{cases}
      (1 - x)(1 - qx) \cdots (1 - q^{n-1}x) & n > 0, \\
      1 & n = 0, \\
      (1 - q^{-1} x)^{-1}(1 - q^{-2} x)^{-1} \cdots (1 - q^{-n} x)^{-1} & n < 0.
    \end{cases}
  \]
  Recall that $\xi = \exp(\pi i /r)$ is a primitive $2r$th root of $1$, so that $\zeta = \xi^2$ is an $r$th root of $1$.
  Since we usually consider $q = \zeta = \xi^2$, we abbreviate
  \[
    \qfac{x}{n} \defeq \qfac{x}[\zeta]{n} = \qfac{x}[\xi^2]{n}.
  \]
\end{defn}
\begin{remark}
  We can think of $\qfac{q}[q]{n}(1 - q)^{-n}$ as a $q$-analog of $n!$ corresponding to the $q$-analog
  \[
    \lim_{q \to 1} \frac{1 - q^n}{1-q} = n
  \]
  of $n$.
  In the quantum groups literature it is common to use the alternative $q$-analogs 
  \[
    [n]_q = \frac{q^n - q^{-n}}{q - q^{-1}} \text{ and } [n]_{q}! = [n]_{q} [n-1]_{q} \cdots [1]_{q}
  \]
  of $n$ and $n!$.
  They are related via
  \note{
    \citeauthor{Kassel1995} states this formula slightly differently \cite[IV.2, eq. 2.2]{Kassel1995}.
  }
  \[
    [n]_{q}! = \frac{q^{-n(n-1)/2}}{(q - q^{-1})^n} \qfac{q}[q^2]{n}.
  \]
\end{remark}

From \cref{def:cyclic-dilog,thm:dilog-recurrences},
\[
  \qlog{B,A}{n} = \prod_{k=1}^{n} A^{-1} (1 - \zeta^k B) = A^{-n} \qfac{\zeta B}{n}
\]
and $\qlog{B,A}{-}$ is well-defined on $\ZZ/\nr$.
In \cref{def:qlogn} we introduced a normalized version
\[
  \qlogn{B,A}{n} = g(B,A) \qlog{B,A}{n}
\]
of the quantum dilogarithm.
\begin{lem}
  \label{thm:normalization-properties}
  Let $A, B$ be nonzero complex numbers such that $A^{\nr} + B^{\nr} = 1$.
  Define a function $g$ by
  \[
    \begin{split}
      g(B, A)
      &\defeq
      A^{(\nr-1)/2} \prod_{k=1}^{\nr-1} (1 - \xi^{-2k} B)^{-k/\nr} 
      \\
      &=
      \exp \left( \frac{(\nr-1)}{2\nr}\log(1 - B^\nr) -\frac{1}{\nr} \sum_{k=1}^{\nr-1} k \log(1 - \xi^{-2k} B ) \right)
    \end{split}
  \]
  where the branch of the logarithm is chosen so that
  \[
    \exp\left(\frac{1}{\nr} \log\left((1 - B^\nr)\right)\right) = A.
  \]
  This function satisfies
  \[
    g(\xi^{2m} B, A) = g(B,A) \qlog{B,A}{m}.
  \]
\end{lem}
\begin{proof}
  Write
  \[
    \frac{g(\xi^2 B,A)}{g(B,A)}
    =
    \exp\left(-\frac{1}{\nr} \left( \sum^{\nr-1}_{k=1} k \log(1 - \xi^{-2k} B) - k \log(1 - \xi^{-2(k-1)} B) \right)\right).
  \]
  We can compute that
  \[
    \begin{aligned}
      &\sum^{\nr-1}_{k=1} k \log(1 - \xi^{-2(k-1)} B)  - k \log(1 - \xi^{-2k} B)
      \\
      &=
      \log(1 - \xi^{0} B) - \log(1 - \xi^{-2} B) 
      +2\log(1 - \xi^{-2} B) - 2 \log(1 - \xi^{-4} B) + \\
      &\phantom{=} \cdots 
      \\
      &\phantom{=}+(\nr -1) \log(1 - \xi^{-2(\nr-2)} B) - (\nr-1) \log(1 - \xi^{-2(\nr-1)} B)
      \\
      &=
      \log(1 - B) + \log(1 - \xi^{-2} B) + \cdots + \log(1 - \xi^{-2(\nr -1)} B) \\
      &\phantom{=} + \log(1 - \xi^{-2(\nr -1)} B) - \nr \log(1 - \xi^{-2(\nr -1)}B)
      \\
      &=
      -N \log(1 - \xi^{2} B) + \sum_{k=0}^{\nr -1} \log(1 - \xi^{-2k} B).
    \end{aligned}
  \] 
  Since
  \[
    \sum_{k=0}^{\nr-1} \log(1 - \xi^{-2k} B) =
    \log\left(
      \prod_{k=0}^{\nr-1} (1 - \xi^{-2k} B)
    \right)
    =
    \log(1 - B^\nr)
  \]
  we conclude that
  \[
    \frac{g(\xi^2 B,A)}{g(B,A)}
      =
      \exp\left(
        \log(1 - \xi^{2}B) - \frac{1}{\nr} \log(1 - B^\nr)
      \right)
      =
      \frac{ 1 - \xi^{2}B }{A}
  \]
  since we chose the branch of the logarithm matching $(1 - B^{\nr})^{1/\nr} = A$.
  By induction we have
  \[
    g(\xi^{2m} B, A) = 
    \frac{1 - \xi^{2m} B}{A} g(\xi^{2(m-1)} B, A)
    = \cdots
    = \qlog{B,A}{m} g(B,A).\qedhere
  \]
\end{proof}

\begin{prop}
  \label{thm:qlogn-properties}
  The normalized quantum dilogarithm satisfies:
  \begin{align}
    \label{eq:qlog-shift-B}
    \qlogn{\xi^{2m} B, A}{n} &= \qlogn{B,A}{n + m} 
    \\
    \label{eq:qlog-shift-A}
    \qlogn{B, \xi^{2m} A}{n} &= \xi^{-2nm} \xi^{-m(\nr-1)^2} \qlogn{B,A}{n}
    \intertext{and}
    \label{eq:qlog-product}
    \prod_{k=1}^{\nr} \qlogn{B,A}{k} &= 1.
  \end{align}
\end{prop}
\begin{proof}
  For \eqref{eq:qlog-shift-B}, we have
  \begin{align*}
    \qlogn{\xi^{2m}B, A}{n}
    &=
    g(\xi^{2m} B, A)\qlog{\xi^{2m}B,A}{n}
    \\
    &=
    g(B,A) \qlog{B,A}{m} \frac{\qlog{B,A}{n+m}}{\qlog{B,A}{m}}
    \\
    &= \qlogn{B,A}{n+m}.
  \end{align*}
  We check \eqref{eq:qlog-shift-A} in two parts.
  The unnormalized dilogarithm transforms as
  \[
    \qlog{B,\xi^{2m}A}{n}
    =
    \prod_{k=1}^{n} (\xi^{2m}A)^{-1} (1 - \xi^{2k}B)
    =
    \xi^{-2mk} \qlog{B,A}{n}.
  \]
  The transformation of $g$ is more subtle.
  The branch of the logarithm $\log'$ with $\log'(1 - B^\nr) = \xi^{2m}A$ has
  \[
    \log' = \log + 2\pi i m
  \]
  so
    \begin{align*}
       g(B, \xi^{2m}A)
       &=
       \exp \left( \frac{(\nr-1)}{2\nr}(2\pi i m +  \log(1 - B^\nr)) -\frac{1}{\nr} \sum_{k=1}^{\nr-1} k (2\pi i m + \log(1 - \xi^{-2k} B )) \right)
       \\
       &=
       g(B,A)
      \exp\left(
        \frac{2\pi i}{\nr} m \left[ \frac{\nr -1}{2} - \sum_{k=1}^{\nr-1} k \right]
       \right)
       \\
       &=
       g(B,A)
       \exp\left(
         \frac{\pi i}{\nr} m\left[ \nr -1 - \nr(\nr-1) \right]
       \right)
       \\
       &=
       g(B,A) \xi^{-m(\nr -1)^2}
       .
    \end{align*}
    Finally, for \eqref{eq:qlog-product} we have
    \begin{align*}
        \prod_{k=1}^{\nr} \qlogn{B,A}{k}
        &=
        g(B,A)^\nr \prod_{k=1}^{\nr-1} \qlog{B,A}{k}
        \\
        &=
        A^{\nr(\nr-1)/2}
        \left[
        \prod_{k=1}^{\nr-1} (1 - \xi^{-2k}B)^{-k}
      \right]
      \\
        &\phantom{=}\times
        A^{-\nr(\nr-1)/2}
        \left[
          \prod_{k=1}^{\nr -1} (1 - \xi^{2k}B)^{\nr - k}
        \right]
        \\
        &=
        1.
        \qedhere
    \end{align*}
\end{proof}

\section{Fourier transforms}
\label{sec:fourier-transforms}
Another reason to introduce the normalized quantum dilogarithm is that we can compute its Fourier transform.
\begin{thm}
  \label{thm:qlog-fourier}
  The Fourier transfroms of the cyclic dilogarithm and is inverse are given by
  \begin{align}
    \label{eq:qlog-fourier}
    \sum_{k=0}^{\nr-1} \xi^{-2mk} \qlogn{B,A}{k}
    &=
    \frac{1}{\qlogn{A,\xi^2 B}{m}} \qlogsum{B,A}
    \\
    \label{eq:qlog-recip-fourier}
    \sum_{k=0}^{\nr-1} \xi^{2mk} \frac{1}{\qlogn{B,A}{k}}
    &=
    \qlogn{\xi^{-2} A, B}{m}
    \frac{\nr}{\qlogsum{\xi^{-2} A, B}}
  \end{align}
  where the auxiliary function
  \begin{align}
    \label{eq:qlog-sum-def}
    \qlogsum{B,A}
    &\defeq
    \left( \frac{B}{A} \right)^{\nr-1} g(B,A) \sum_{k=0}^{\nr-1} \qlogn{A , B}{k}
    \\
    \intertext{satisfies}
    \label{eq:qlog-sum-shift-B}
    \qlogsum{\xi^{2m} B, A}
    &=
    \xi^{-m(\nr-1)^2}
    \qlogsum{B,A}
    \\
    \label{eq:qlog-sum-shift}
    \qlogsum{B, \xi^{2m} A}
    &=
    \xi^{2m}\xi^{-m(\nr-1)^2} \qlogsum{B,A}
    \\
    \label{eq:qlog-sum-swap}
    \qlogsum{B,A}
    &=
    \left( \frac{B}{A} \right)^{\nr-1} \qlogsum{A,B}
  \end{align}
\end{thm}
\begin{proof}
  Following \cite{Hikami2014} we view the transforms as terminating $q$-hypergeometric series.
  We first prove \eqref{eq:qlog-fourier}.
  Specifically, we use the transformation formula \cite[eq. III.6]{Gasper2004}
  \begin{equation}
    \label{eq:q-series-id}
    \begin{aligned}
      &\sum^{n}_{k=0} \frac{ \qfac{q^{-n}}[q]{k} \qfac{b}[q]{k} }{ \qfac{c}[q]{k} \qfac{q}[q]{k} } z^k
      \\
      &=
      \frac{\qfac{c/b}[q]{n} }{ \qfac{c}[q]{n} } \left( \frac{bz}{q} \right)^{n}
      \sum^{n}_{k=0} \frac{ \qfac{q^{-n}}[q]{k} \qfac{ q/z }[q]{k} \qfac{ q^{1-n}/c }[q]{k} }{ \qfac{bq^{1-n}/c }[q]{k} \qfac{q}[q]{k} } q^k.
    \end{aligned}
  \end{equation}
  By using the expansion
  \[
    \qfac{x/c}[q]{n} = (-1)^n q^{n(n-1)/2} (x/c)^n + O(1/c^{n-1})
  \]
  we see that as $c \to 0$, \eqref{eq:q-series-id} becomes
  \begin{equation}
    \sum^{n}_{k=0} \frac{ \qfac{q^{-n}}[q]{k} \qfac{b}[q]{k} }{ \qfac{q}[q]{k} }  z^k
    =
    \left( \frac{bz}{q} \right)^{n}
    \sum^{n}_{k=0} \frac{ \qfac{q^{-n}}[q]{k} \qfac{q/z}[q]{k} }{ \qfac{q}[q]{k} }
    \left( \frac{q}{b} \right)^{k}.
  \end{equation}
  Setting $n = \nr-1$ and $q = \zeta = \xi^2$, we obtain
  \begin{equation}
    \label{eq:q-fac-fourier-transform}
    \sum^{\nr-1}_{k=0} \qfac{b}{k} z^k =
    \left( \frac{b z}{\zeta} \right)^{\nr-1}
    \sum^{\nr-1}_{k=0} \qfac{\zeta/z}{k} (\zeta/b)^k.
  \end{equation}
  By substituting $b = \zeta B$ and $z = \zeta^{-m}/A$, we can use \eqref{eq:q-fac-fourier-transform} to compute
  \begin{align*}
    \sum_{k=0}^{\nr-1} \qlogn{B,A}{k} \zeta^{-mk}
    &=
    \sum_{k=0}^{\nr-1} g(B,A) \qlog{B,A}{k} \zeta^{-m k}  
    \\
    &=
    g(B,A) \sum_{k= 0}^{\nr-1} \qfac{\zeta B}{k} A^{-k} \zeta^{-m k}
    \\
    &=
    g(B,A) \left( \frac{B}{A} \zeta^{m} \right)^{\nr-1}
    \sum^{\nr-1}_{k=0} \qfac{A \zeta ^{m + 1}}{k} B^{-k}
    \displaybreak
    \\
    &=
    \frac{g(B, A)}{g(\zeta^m A, B)} \zeta^{-m} \left( \frac{B}{A} \right)^{\nr-1}
    \sum^{\nr-1}_{k=0} \qlogn{A \zeta^{m}, B}{k}
    \\
    &=
    \frac{g(B, A)}{\qlogn{A,B}{m}} \zeta^{-m} \left( \frac{B}{A} \right)^{\nr-1}
    \sum^{\nr-1}_{k=0} \qlogn{A \zeta^{m}, B}{k}.
  \end{align*}
  By \eqref{eq:qlog-shift-B} and periodicity, we can simplify
  \[
    \sum^{\nr-1}_{k=0} \qlogn{A \zeta^{m}, B}{k}
    =
    \sum^{\nr-1}_{k=0} \qlogn{A, B}{k + m}
    =
    \sum^{\nr-1}_{k=0} \qlogn{A, B}{k}.
  \]
  Absorbing the power of $\zeta$ via the identity 
  \[
    \qlogn{A,B}{m} \zeta^{-m} = \qlogn{A, \zeta B}{m}
  \]
  we conclude that
  \[
    \begin{split}
      \sum_{k=0}^{\nr-1} \qlogn{B,A}{k} \zeta^{-mk}
      &=
      \frac{1}{\qlogn{A, \zeta B}{m}}
      \left( \frac{B}{A} \right)^{\nr-1} g(B,A) \sum_{k=0}^{\nr-1} \qlogn{A , B}{k}
      \\
      &=
      \frac{1}{\qlogn{A, \zeta B}{m}} \qlogsum{B,A}.
    \end{split}
  \]
  We now prove the relations for the function $\qlogsum{B,A}$.
  We have
  \begin{align*}
    \qlogsum{B, \zeta^{m} A}
    &=
    \left( \frac{B}{\zeta^m A} \right)^{\nr-1} g(B, \zeta^m A)
    \sum_{k = 0}^{\nr-1} \qlogn{\zeta^m A, B}{k}
    \\
    &=
    \zeta^{m} 
    \left( \frac{B}{A} \right)^{\nr-1} \xi^{-m(\nr-1)^2} g(B, A) \sum_{k=0}^{\nr-1} \qlogn{A, B}{m + k}
    \\
    &=
    \zeta^m \xi^{-m(\nr-1)^2} \qlogsum{B, A}
  \end{align*}
  which gives \eqref{eq:qlog-sum-shift}.
  For \eqref{eq:qlog-sum-swap}, first write
  \begin{align*}
    \qlogsum{A, B}
    &=
    \left( \frac{A}{B} \right)^{\nr-1} g(A,B)
    \sum_{k=0}^{\nr-1} g(B, A) \qlog{B, A}{k}
    \\
    &=
    \left( \frac{A}{B} \right)^{\nr-1} g(A,B) g(B,A)
    \sum_{k=0}^{\nr-1} \qfac{\zeta B}{k} A^{-k}.
  \end{align*}
  By another application of \cref{eq:q-fac-fourier-transform} (with $b = \zeta B$ and $z = A^{-1}$) we see that
  \begin{align*}
    \sum_{k=0}^{\nr-1} \qfac{\zeta B}{k} A^{-k}
    =
    \left( \frac{B}{A} \right)^{\nr-1}
    \sum_{k=0}^{\nr-1} \qfac{\zeta A}{k} B^{-k}
    = 
    \left( \frac{B}{A} \right)^{\nr-1}
    \sum_{k=0}^{\nr-1} \qlog{A,B}{k}
  \end{align*}
  so that
  \begin{align*}
    \qlogsum{A, B}
    &=
    g(A,B) g(B,A) \sum_{k=0}^{\nr-1} \qlog{A, B}{k}
    \\
    &=
    \left( \frac{A}{B} \right)^{\nr-1}
    \left( \frac{B}{A} \right)^{\nr-1}
    g(B,A) \sum_{k=0}^{\nr-1} \qlogn{A, B}{k}
    \\
    &=
    \left( \frac{A}{B} \right)^{\nr-1}
    \qlogsum{B,A}
  \end{align*}
  as claimed.
  \eqref{eq:qlog-sum-shift-B} now follows from \eqref{eq:qlog-sum-shift} and \eqref{eq:qlog-sum-swap}.

  To obtain \eqref{eq:qlog-recip-fourier} we take the Fourier transform of \eqref{eq:qlog-fourier}.
  The left-hand side is
   \begin{align*}
     \sum_{m,k = 0}^{\nr-1} \zeta^{mn - mk} \qlogn{B,A}{k}
     &=
     \nr \qlogn{B,A}{n}
  \end{align*}
  while the right-hand side is
  \begin{align*}
    \qlogsum{B,A} \sum_{m=0}^{\nr-1} \frac{\zeta^{mn}}{\qlogn{A,\zeta B}{m}}.
  \end{align*}
  After the change of variables $A \to B, B \to \zeta^{-1} A$ we obtain
  \[
    \nr \qlogn{\zeta^{-1} A, B}{n}
    =
    \qlogsum{\zeta^{-1} A, B}
    \sum_{m=0}^{\nr-1} \frac{\zeta^{mn}}{\qlogn{B, A}{m}} 
  \]
  or equivalently
  \begin{align*}
    \sum_{m=0}^{\nr-1} \frac{\zeta^{n m}}{ \qlogn{B, A}{m} }
    =
    \qlogn{\zeta^{-1 } A, B}{n}
    \frac{\nr}{\qlogsum{\zeta^{-1} A, B}}
  \end{align*}
  By renaming some indices we get \eqref{eq:qlog-recip-fourier}.
\end{proof}

Finally, we can give a simpler expression for the normalization factor $\qlogsumname$.
\begin{prop}
  Up to a power of $\xi$,
  \label{thm:qlogsum-formula}
  \[
    \qlogsum{B,A} = \frac{\nr}{\gamma(1)} \left( \frac{B}{A} \right)^{(\nr -1)/2}
  \]
  where
  \[
    \gamma(1) = \prod_{k=1}^{\nr-1}(1 - \xi^{-2k})^{\nr/k}.
  \]
\end{prop}
  Since the function $(A/B)^{(\nr -1)/2}$ is not continuous on the curve $A^\nr + B^\nr = 1$ this is the best we could do without changing some normalizations.
\begin{proof}
  We follow \cite[Section 8]{Baseilhac2005classquant}. 
  As above, write $\zeta = \xi^2$ for a primitive $\nr$th root of unity.
  Consider the following variants of $1/g$ and $\qlogsumname$:
  \begin{align*}
    \gamma(B)
    &\defeq \prod_{k=1}^{\nr -1} (1 - B \zeta^{-k})^{k/\nr},
    \\
    S(B | A)
    &\defeq \sum_{m=0}^{\nr-1} \frac{1}{\qlog{B,A}{m}},
  \end{align*}
  where $A^\nr + B^\nr = 1$ as usual.
  The definition of $\gamma$ requires a choice of logarithm branch, so is only defined up to a power of $\zeta$; we will deal with this later.%
  \note{
    This is the reason that we defined $g(B,A)$ in terms of $B$ \emph{and} $A$, instead of just writing $g(B,A) = A^{(\nr-1)/2} \gamma(B)^{-1}$.
  }

  By studying the poles of the function $\prod_{k = 0}^{\nr -1} S(B, \zeta^k A)$ one can derive \cite[Lemma 8.3 (iii)]{Baseilhac2005classquant} the identity 
  \[
    S(B|\zeta A) = \zeta^? \frac{B^{\nr -1} \gamma(1)}{\gamma(B) \gamma(A)}
  \]
  which as indicated holds up to some power of $\zeta$.
  To connect with our notation, observe that by setting $m = 0$ in \eqref{eq:qlog-recip-fourier}, we obtain
  \[
    S(B|\zeta A)
    =
    g(B, \zeta A) \qlogn{A,B}{0} \frac{\nr}{\qlogsum{A,B}}
  \]
  or
  \[
    \qlogsum{A,B} = g(B,\zeta A) g(A,B) \frac{\nr}{S(B|\zeta A)}
  \]
  Since up to a power of $\xi$
  \[
    g(B,A) = A^{(\nr -1)/2} \gamma(B)^{-1},
  \]
  we see that
  \[
    S(B|\zeta A) = \xi^{?} \gamma(1) \left( \frac{B}{A} \right)^{(\nr -1)/2} g(B,A) g(A,B)
  \]
  and thus conclude that
  \begin{align*}
    \qlogsum{A,B}
    = \nr \frac{g(B, \zeta A) g(A,B) }{S(B|\zeta A)}
  \end{align*}
  where
  \[
    \gamma(1) = \prod_{k=1}^{\nr -1} (1 - \zeta^{-k})^{k/\nr}
  \]
  is defined using the standard branch of $\log$ whose imaginary part takes values in $(-\pi, \pi]$.

  We have derived the relation
  \[
    \qlogsum{A,B}
    = \xi^{?} \frac{\nr}{\gamma(1)} \left( \frac{A}{B} \right)^{(\nr -1)/2}
  \]
  which holds up to a power of $\xi$.
\end{proof}

\addcontentsline{toc}{chapter}{Bibliography}
\printbibliography
\end{document}